\documentclass[letterpaper,12pt]{book}

\usepackage{geometry}
\usepackage{amsmath}
\usepackage{amssymb}
\usepackage{amsthm}
\usepackage{mathrsfs}
\usepackage{enumerate}
\usepackage{hyperref}
\usepackage{graphicx}
\usepackage{stmaryrd}
\usepackage{makeidx}
\makeindex

\usepackage{color}

\newcommand{\bC}{\mathbb{C}}
\newcommand{\bF}{\mathbb{F}}	

\newcommand{\bR}{\mathbb{R}}

\newcommand{\bN}{\mathbb{N}}
\newcommand{\bZ}{\mathbb{Z}}

\newcommand{\ds}{\displaystyle}
\newcommand{\ve}{\varepsilon}

\newcommand{\diam}{\operatorname{diam}}
\renewcommand{\Im}{\operatorname{Im}}
\renewcommand{\Re}{\operatorname{Re}}
\newcommand{\supp}{\operatorname{supp}}
\newcommand{\loc}{{\operatorname{loc}}}
\newcommand{\sgn}{\operatorname{sgn}}
\newcommand{\pv}{\operatorname{pv}}
\newcommand{\BMO}{\operatorname{BMO}}
\newcommand{\Ob}{\operatorname{Ob}}
\newcommand{\Hom}{\operatorname{Hom}}
\newcommand{\id}{\operatorname{id}}

\newcommand{\mb}[1]{\mathbb{#1}}
\newcommand{\mc}[1]{\mathcal{#1}}
\newcommand{\mf}[1]{\mathfrak{#1}}
\newcommand{\ms}[1]{\mathscr{#1}}

\newcommand{\hone}{\hspace{.1in}}
\newcommand{\htwo}{\hspace{.2in}}

\newtheorem{theorem}{Theorem}[chapter]
\newtheorem{lemma}[theorem]{Lemma}
\newtheorem{prop}[theorem]{Proposition}
\newtheorem{cor}[theorem]{Corollary}
\newtheorem{conjecture}[theorem]{Conjecture}

\theoremstyle{definition}
\newtheorem{defin}[theorem]{Definition}
\newtheorem{fr}{}[section]

\theoremstyle{remark}

\title{Interpolation Theorems in Harmonic Analysis}
\author{Mark Hyun-Ki Kim \\ \emph{Bachelor of Science} \\ Department of Mathematics \\ Rutgers, the State University of New Jersey}
\date{May 2012 \\ \vspace{2.5in} Advisor: R. Michael Beals}

\begin{document}


\frontmatter

\maketitle 

\thispagestyle{empty}
\pagestyle{empty}

\section*{}

\vspace{3in}

This thesis is ``dedicated'' to the first Rutgers-NYU segway polo champion: come forth and claim your prize!\index{Russell, Matthew C.}

\phantomsection
\addcontentsline{toc}{chapter}{Preface}

\chapter*{Preface}

The present thesis contains an exposition of interpolation
theory in harmonic analysis, focusing on the complex
method of interpolation. Broadly speaking, an interpolation
theorem allows us to guess the ``intermediate'' estimates
between two closely-related inequalities. To give an
elementary example, we take a square-integrable function $f$
on the real line. It is a standard result from real
analysis that $f$ satisfies the \emph{$L^2$-H\"{o}lder inequality}
\begin{equation}\label{preface-L2}
 \int_{-\infty}^\infty \left|f(x) g(x)\right| \, dx
 \leq \left(\int_{-\infty}^\infty |f(x)|^2 \, dx\right)^{1/2}
 \left( \int_{-\infty}^\infty |g(x)|^2 \, dx\right)^{1/2}.
\end{equation}
for every integrable function $g$ on the real line with compact
support. If, in addition, $f$ satisfies the integral inequality
\begin{equation}\label{preface-L1}
 \int_{-\infty}^\infty |f(x)g(x)| \, dx \leq
 \left(\int_{-\infty}^\infty |f(x)|^2 \, dx\right)^{1/2}
 \left( \int_{-\infty}^\infty |g(x)| \, dx\right)
\end{equation}
for all such $g$, it then follows ``from interpolation'' that
the inequality
\begin{equation}\label{preface-Lp}
 \int_{-\infty}^\infty |f(x)g(x)| \, dx \leq
 \left(\int_{-\infty}^\infty |f(x)|^2 \, dx\right)^{1/2}
 \left( \int_{-\infty}^\infty |g(x)|^p \, dx\right)^{1/p}
\end{equation}
holds for all $1 \leq p \leq 2$ and all $g$ on the real line with
compact support. 

From a more abstract viewpoint, we can consider interpolation as
a tool that establishes the continuity of ``in-between'' operators
from the continuity of two endpoint operators. The above example
can be viewed as a study of the ``multiplication by $f$'' operator
\[
 (Tg)(x) = f(x)g(x).
\]
In the language of Lebesgue spaces, inequality (\ref{preface-L2})
implies that $T$ is a continuous mapping from the function space $L^2$ to itself,
and inequality (\ref{preface-L1}) implies that $T$ is a continuous mapping from $L^2$
to another function space $L^1$. The conclusion, then, is that $T$ maps $L^2$
continuously into the ``interpolation spaces'' $L^p$ ($1 \leq p \leq 2$) as well.

Presented herein are a study of four interpolation
theorems, the requisite background material, and a few applications.
The materials introduced in the first three sections of Chapter 1
are used to motivate and prove the Riesz-Thorin interpolation theorem and
its extension by Stein, both of which are presented in the fourth section.
Chapter 2 revolves around Calder\'{o}n's complex
method of interpolation and the interpolation theorem of Fefferman
and Stein, with the material in between providing the necessary
examples and tools. The two theorems are then 
applied to a brief study of linear partial differential
equations, Sobolev spaces, and Fourier integral operators, presented
in the last section of the second chapter.

I have approached the project mainly as an exercise in expository writing.
As such, I have tried to keep a real audience in mind throughout.
Specifically, my aim was to make the present thesis accessible to
Rutgers graduate students who have taken Math 501, 502, and 503.
This means that I have assumed familiarity with the standard 
material in advanced calculus, complex analysis, linear algebra
and point-set topology. In addition, I expect the reader to be
conversant in the language of measure and integration theory
including Lebesgue spaces ($L^p$ spaces), and of functional analysis
up to basic Banach and Hilbert space theory. Beyond those, the required tools
from functional analysis are summarized in the beginning of Chapter 2, and
elements of harmonic analysis are introduced throughout the thesis.

Before I realized how much time it would take to develop each topic
at hand, I had planned to include some harmonic function theory,
maximal function theory of Hardy and Littlewood, 
the interpolation theorem of Marcinkiewicz, the standard material on the theory of
singular integral operators, and the Lions-Peetre method of
real interpolation as a generalization of Maricnkiewicz.
This never happened, and what I had in mind is reduced to
a brief exposition in the further-results section of Chapter 2.
Of course, given the length of the present thesis as is,
I simply would not have had the time and energy to give the
extra materials the care they deserve.

Nevertheless, the inclusion of the theory of singular integral operators
would have helped motivating the section on Fefferman-Stein theory
in Chapter 2, which I believe is extremely condensed and,
frankly, dry as it stands now. Moreover, I was not able to come up
with a coherent narrative for the section on
the functional-analytic prerequisites in the beginning of Chapter 2.
Is there any way to make a ``random collection of things you should
probably know before reading'' section flow pleasantly smooth without
expanding it into a whole chapter or a book? I do not have a good
answer at the present moment.

But, enough excuses. I had a lot of fun writing this thesis, and I hope
that I managed to produce an enjoyable read. Please feel free to send
any comments or corrections to \texttt{markhkim@dimax.rutgers.edu}.

\section*{Acknowledgements}

My deepest gratitude goes to my thesis advisor, Michael Beals.
It is the brief conversation Professor Beals and I had on my
first visit to Rutgers University that gave me the courage to
pursue mathematics, the course he taught in my second-semester
freshman year that convinced me to study analysis,
and the numerous reading courses he gave over the following
years that cultivated my current interests in the field.
From the day I set my foot on campus to the very last
day as an undergraduate, Professor Beals has been the greatest
mentor I could possibly hope for. Indeed, it is he who taught
me most of the mathematics I know, supported me wholeheartedly
in my numerous academic pursuits over the years, and counseled
me ever so patiently in times of trouble.

I would also like to express my gratitude to my academic advisor
and the chair of the honors track, Simon Thomas. There have been
more than a few times I had let myself be consumed by unrealistic,
overly ambitious projects, and Professor Thomas never hesitated
to provide me with a dose of reality and set me on the right path.
He is also one of the best lecturers I know of, and my strong
interest in mathematical exposition was, in part, cultivated
in his course I took as a sophomore. I am truly fortunate to have
had two amazing mentors throughout my undergraduate career.

I have benefited greatly from conversing with other
professors in the department---about the project, and mathematics
at large. Discussions with Eric Carlen, Roe Goodman, Robert
Wilson, and Po Lam Yung have been especially helpful.
The summer school in analysis and geometry at
Princeton University in 2011
also contributed significantly to my understanding of
the background material and their interactions with other fields.
Particularly useful were the lectures by Kevin Hughes,
Lillian Pierce, and Eli Stein. I would like to offer a special
thanks to Professor Stein, who have written the wonderful
textbooks that I have used again and again over the course of the project.

I am also grateful to Itai Feigenbaum, Matt Inverso,
and Jun-Sung Suh for putting up with my endless rants and keeping me sane,
and Matt D'Elia for being a fantastic study buddy.
A warm thank-you goes to my ``graduate officemates'' Katy
Craig and Glen Wilson in Hill 603, and Tim Naumotivz, Matthew Russell,
and Frank Seuffert in Hill 605, who assured me that I am not the
only apprentice navigator in the vast ocean of mathematics.
And last but not least, a bow to my parents for keeping
me alive for the past 23 years and supporting me through
17 years of formal education. Those are awfully big numbers,
if you ask me.

\tableofcontents

\mainmatter

\thispagestyle{headings}
\pagestyle{headings}

\chapter{The Classical Theory of Interpolation}

In the first chapter, we study two interpolation theorems, both of which are
presented in \S\ref{s-interpolation_on_lp_spaces}. Interpolation theory
began with a 1927 theorem of Marcel Riesz, first published in
\cite{Marcel_Riesz:J1927}. \emph{Riesz convexity theorem}, as it is called,
did not arise as a theorem of harmonic analysis, as the paper dealt with
the theory of bilinear forms. It was Riesz's student
G. Olof Thorin with his thesis \cite{Olof_Thorin:T1948} who appropriately
generalized the theorem of Riesz and placed it in its proper context.
The complex-analytic method used in the proof of the Riesz-Thorin
interpolation theorem was then generalized by Elias M. Stein,
allowing for interpolation of families of operators. This result,
known as the \emph{Stein interpolation theorem},
was included in his 1955 doctoral dissertation and was subsequently published
in \cite{Elias_M_Stein:J1956}.

The first two sections of the chapters are devoted to developing the
necessary tools for stating and proving the interpolation theorems.
We review the theory of measure and integration in the first section,
which is included mainly as a convenient reference. In the second section,
we tackle approximation theorems in Lebesgue spaces, which provide a
convenient way of studying function spaces by focusing on small samples of functions.
We then switch gears and present the basic theory of Fourier transform
in the third section. This serves primarily to motivate the
Riesz-Thorin interpolation theorem and to provide a useful example to which
the theorem can be applied. The chapter culminates in the fourth and the
last section, in which we state and prove the Riesz-Thorin interpolation and
its generalization by Stein.
\section{Elements of Integration Theory}\label{s-elements_of_integration_theory}

We begin the chapter by collecting the necessary facts from measure and
integration theory. The present section is meant to serve only as a quick reference,
and so the details will necessarily be sparse. See \cite{Stein_Shakarchi:B2011},
\cite{Stein_Shakarchi:B2005}, \cite{Gerald_B_Folland:B1999}, \cite{Walter_Rudin:B1986},
or any other standard textbook on the subject for a more detailed treatment.
\index{Stein, Elias M.}\index{Shakarchi, Rami}\index{Folland, Gerald B.}\index{Rudin, Walter}

\subsection{Measures and Integration}
Recall that a \emph{$\sigma$-algebra} on a nonempty set $X$ is a collection
$\mf{M}$ of subsets of $X$ such that
\begin{enumerate}[(a)]
 \item $\varnothing \in \mf{M}$ and $X \in \mf{M}$.
 \item If $(E_n)_{n=1}^\infty$ is a sequence in $\mf{M}$, then $\bigcup_n E_n
 \in \mf{M}$.
 \item If $E \in \mf{M}$, then $X \smallsetminus E \in \mf{M}$.
\end{enumerate}
Note that (b) and (c) imply
\begin{enumerate}[(a)]
 \item[(d)] If $(E_n)_{n=1}^\infty$ is a sequence in $\mf{M}$, then $\bigcap_n
 E_n \in \mf{M}$.
\end{enumerate}

The pair $(X,\mf{M})$ is referred to as a \emph{measurable space}.\index{measurable!space} Given
a measurable space $(X,\mf{M})$, we say that a subset of $X$ is \emph{measurable}\index{measurable!set}
if it is an element of $\mf{M}$. A \emph{measure}\index{measure} on $(X,\mf{M})$ is a function
$\mu:\mf{M} \to [0,\infty]$ that is \emph{countably additive}, viz.,
\[
 \mu\left(\bigcup_{n=1}^\infty E_n \right) = \sum_{n=1}^\infty \mu(E_n)
\]
for every pairwise disjoint sequence $(E_n)_{n=1}^\infty$ of measurable sets.
Every measure $\mu$ on $X$ satisfies the following properties:
\begin{enumerate}[(a)]
 \item $\mu(\varnothing) = 0$;
 \item \textbf{Monotonicity.} If $E$ and $F$ are measurable subsets of $X$
 and if $E \subseteq F$, then $\mu(E) \leq \mu(F)$.
 \item \textbf{Countable subadditivity.} If $(E_n)_{n=1}^\infty$ is a sequence
 of measurable subsets of $X$, then
 \[
  \mu\left(\bigcup_{n=1}^\infty E_n \right) \leq \sum_{n=1}^\infty \mu(E_n).
 \]
 \item \textbf{Continuity from below.} If $E_1 \subseteq E_2 \subseteq E_3
 \subseteq \cdots$ is a sequence of measurable subsets of $X$, then
 \[
  \mu\left( \bigcup_{n=1}^\infty E_n \right) = \lim_{n \to \infty} \mu(E_n).
 \]
 \item \textbf{Continuity from above.} If $E_1 \supseteq E_2 \supseteq E_3
 \cdots$ is a sequence of measurable subsets of $X$ such that $\mu(E_1) <
 \infty$, then
 \[
  \mu\left( \bigcap_{n=1}^\infty E_n \right) = \lim_{n \to \infty} \mu(E_n).
 \]
\end{enumerate}

Given a nonempty set $X$, a $\sigma$-algebra $\mf{M}$ on $X$,
and a measure $\mu$ on the measurable space $(X,\mf{M})$, we
refer to triple $(X,\mf{M},\mu)$ as a \emph{measure space}.\index{measure!space} A
measure space $(X,\mf{M},\mu)$ is said to be \emph{finite}\index{measure!finite} if $\mu(X) < \infty$,
\emph{$\sigma$-finite}\index{measure!sigma-finite@$\sigma$-finite} if there exists a sequence $(E_n)_{n=1}^\infty$
of finite-measure sets whose union is $X$, and \emph{complete} if all
subsets of measure-zero sets are measurable. We often talk about a \emph{finite
measure}, a \emph{$\sigma$-finite measure}, or a \emph{complete measure}:
this usage introduces no ambiguity, as specifying a measure picks out
a unique $\sigma$-algebra as its domain, and this $\sigma$-algebra, in turn,
determines a unique base set. Similarly, we usually speak of measures on
the base set $X$, even though the measures are, strictly speaking, defined on
measurable spaces.

If $X$ is a topological space, then we define the \emph{Borel $\sigma$-algebra}\index{Borel, \'{E}mile!sigma-algebra@$\sigma$-algebra}
to be the smallest $\sigma$-algebra on $X$ containing all open subsets of $X$.
A \emph{Borel set}\index{Borel, \'{E}mile!set} in $X$ is an element of the Borel $\sigma$-algebra on $X$,
and a \emph{Borel measure}\index{Borel, \'{E}mile!measure}\index{measure!Borel|see{Borel, \'{E}mile}} on $X$ is a measure on $X$ that renders all Borel
sets measurable. The canonical measure on $\bR^d$, the \emph{$d$-dimensional
Lebesgue measure},\index{Lebesgue, Henri!measure} is the unique complete translation-invariant Borel measure
$\ms{L}^d$ on $\bR^d$ with the normalization $\ms{L}^d([0,1]^d) = 1$. If there is no
danger of confusion, $m(E)$ or $|E|$ is often used in place of $\ms{L}^d(E)$.
We shall have more to say about the Lebesgue measure later in this section.
For now, we merely remark that the Lebesgue measure is $\sigma$-finite.

Given a measure space $(X,\mf{M},\mu)$ and a topological space $Y$, we say
that a function $f:X \to Y$ is \emph{measurable}\index{measurable!function} if each open set $E$ in
$Y$ has a measurable preimage $f^{-1}(E)$. If $X$ is a topological space
and $\mu$ a Borel measure, then the definition renders all continuous
functions measurable. If $Y$ is $\bR$ or $\bC$, then the sums and products
of measurable functions are measurable. We observe that the supremum, the
infimum, the limit superior, and the limit inferior of a sequence of
measurable functions is measurable. This, in particular, implies that
the limit of a pointwise convergent sequence of
measurable functions is measurable. In fact, if the set of divergence is
of measure zero, then this continues to hold. In other worlds, the limit
of a \emph{pointwise almost-everywhere convergent} sequence of measurable
functions is measurable. We say that a property $P$ holds \emph{almost
everywhere}\index{almost everywhere} if the set on which $P$ does not hold is of measure zero.

\index{Lebesgue, Henri!integral|(}
Let $(X,\mf{M},\mu)$ be a measure space. The \emph{characteristic function},
or the \emph{indicator function}, of $E \subseteq X$ is defined to be\index{characteristic function}\index{indicator function|see{characteristic function}}
\[
 \chi_E(x) =
 \begin{cases}
  1 & \mbox{ if } x \in E; \\
  0 & \mbox{ if } x \in X \smallsetminus E.
 \end{cases}
\]
A \emph{simple function}\index{simple function} $s$ on $X$ is a finite linear combination
\[
 s(x)
 = \sum_{n=1}^N \lambda_n \chi_{E_n}
\]
of characteristic functions, where each $\lambda_n$ is a complex number
and $E_n$ a measurable set. Note that simple functions are 
automatically measurable. The \emph{(Lebesgue) integral}
is defined to be the sum
\[
 \int s \, d\mu = \int_X s(x) \, d\mu(x) = \sum_{n=1}^N \lambda_n \mu(E_n).
\]

We extend the definition of the integral to nonnegative measurable
functions $f$ on $X$ by setting
\[
 \int f \, d\mu = \int_X f(x) \, d\mu(x)
 = \sup \left\{ \int s \, d\mu :
 s \mbox{ is simple and } 0 \leq s \leq f\right\}
\]
and call $f$ \emph{integrable} if the integral is finite.
With this definition, we can state one of the fundamental theorems
in measure theory, the \emph{monotone convergence theorem}\index{convergence theorem!monotone}: 
every increasing sequence $(f_n)_{n=1}^\infty$ of nonnegative
integrable functions on $X$ converging pointwise almost everywhere
to a function $f$ on $X$ satisfies the identity
\[
 \lim_{n \to \infty} \int f_n \, d\mu
 = \int \lim_{n \to \infty} f_n \, d\mu
 = \int f \, d\mu.
\]
The theorem allows us to approximate the integral of nonnegative
measurable functions by integrals of simple functions.
Indeed, every nonnegative measurable function $f$ on $X$ admits
an increasing sequence $(s_n)_{n=1}^\infty$ of nonnegative
simple functions that converge pointwise to $f$ and uniformly
to $f$ on all subsets of $X$ on which $f$ is bounded. For
non-increasing sequences of functions, we have \emph{Fatou's lemma}\index{Fatou's lemma},
which states that every sequence $(f_n)_{n=1}^\infty$ of
nonnegative measurable functions on $X$ satisfies the inequality
\[
 \int\liminf_{n \to \infty} f_n \, d\mu
 \leq \liminf_{n \to \infty} \int f_n \, d\mu.
\]

Before we extend the definition of the integral to general
cases, we take a moment to tackle a minor
technical issue. Functions like $f(x) = x^{-1/2}\chi_{[0,1]}$
are ``integrable
over $\bR$'' and have finite integrals, but they are not
functions on $\bR$ in the traditional sense, for $x=0$ must be
excluded from the domain. In order to incorporate such functions
into the framework of Lebesgue integration, we ought to turn them
into measurable functions on their natural ``domain space''. The
solution is to consider the \emph{extended number system}
$\bar{\bR}$, which consists of the real numbers, the negative
infinity $-\infty$, and the positive infinity $\infty$. We
define the arithmetic operations on $\bar{\bR}$ by inheriting
the operations from $\bR$ and then by setting
\[
 x \pm \infty = \pm \infty, \hone
 \frac{x}{\pm\infty} = 0, \hone
 y \cdot (\pm \infty) = \pm \infty, \hone
 (-y) \cdot (\pm \infty) = \mp \infty
\]
for all $x \in \bR$ and $y \in \bR \smallsetminus \{0\}$;
we do not attempt to define $\infty - \infty$. In measure theory,
we typically set
\[
 0 \cdot \pm\infty = 0,
\]
so that the values of an extended real-valued function on a set of measure zero
are negligible. We say that a function $f:X \to \bar{\bR}$ is \emph{measurable}
if $f^{-1}([-\infty,a))$ is measurable in $X$ for each $a \in \bR$.
With the standard topology on $\bar{\bR}$, this definition agrees with the
standard definition of measurable functions given above: see
\S\S\ref{fr-compactification} for a discussion.

We now fix an arbitrary measurable extended real-valued function $f$ on $X$
and define
\[
 f^+(x) = \max\{f(x),0\} \htwo \mbox{and} \htwo f^-(x) = \max\{-f(x),0\}.
\]
$f^+$ and $f^-$ are nonnegative, measurable, extended real-valued functions,
and so we can define the integrals $\int f^+$ and $\int f^-$ by a simple
modification of the definition of integral for nonnegative real-valued functions.
Since $f$ can be written as the difference $f^+ - f^-$, it is natural to define
the \emph{integral} of $f$ to be
\[
 \int f \, d\mu = \int_X f(x) \, d\mu(x)
 = \int f^+ \, d\mu- \int f^- \, d\mu,
\]
provided that the difference is well-defined. We say that $f$ is \emph{integrable}
if and only if the integral of $f$ is finite.

If $f$ is complex-valued, we use the decomposition
\[
 f = \Re f^+ - \Re f^- + i (\Im f^+ - \Im f^-)
\]
to define the integral of $f$ to be
\[
  \int_X f(x) \, d\mu(x)
 = \int \Re f^+ \, d\mu - \int \Re f^- \, d\mu+
 i \left( \int \Im f^+ \, d\mu - \int \Im f^- \, d\mu \right),
\]
where
\begin{eqnarray*}
 \Re f^+(x) &=& \max \{\Re f(x), 0\}; \\
 \Re f^-(x) &=& \max \{-\Re f(x), 0\}; \\
 \Im f^+(x) &=& \max \{\Im f(x), 0\}; \\
 \Im f^-(x) &=& \max \{-\Im f(x), 0\}.
\end{eqnarray*}
Again, the integral of $f$ is defined only when the above sum of integrals is
well-defined, and we say that $f$ is integrable if the integral of $f$ is
finite. The main convergence theorem for this definition is the
\emph{dominated convergence theorem}\index{convergence theorem!dominated}: a sequence $(f_n)_{n=1}^\infty$ of measurable
functions converging pointwise almost everywhere to $f$
and satisfying the bound $|f_n| \leq g$ almost everywhere with an integrable
function $g$ satisfies the following identity:
\[
 \lim_{n \to \infty} \int f_n \, d\mu =
 \int \lim_{n \to \infty} f_n \, d\mu = \int f \, d\mu.
\]

Instrumental in proving the aforementioned convergence theorems are the
following basic properties of the integral:
\begin{enumerate}[(a)]
 \item $\int (f + g) \, d\mu= \int f \, d\mu+ \int g \, d\mu$.
 \item $\int (\lambda f) \, d\mu= \lambda \int f \, d\mu$ for each complex
 number $\lambda$.
 \item If $f \leq g$, then $\int f \, d\mu \leq \int g \, d\mu$.
 \item $| \int f \, d\mu| \leq \int |f| \, d\mu$.
 \item If $f = 0$ almost everywhere, then $\int f \chi_E \, d\mu = 0$
 for all $E$.
 \item If $\mu(E) = 0$, then $\int f \chi_E \, d\mu = 0$ for all $f$.
\end{enumerate}
(a) and (b) imply that the integral is a linear functional on the Lebesgue space
$L^p(X,\mu)$, which we shall define in due course. (e) and (f) can be rephrased
in terms of integrating over subsets: if
$f$ is a complex-valued measurable function on $X$ and $E$ a measurable subset
of $X$, then the \emph{integral of $f$ over $E$} is
\[
 \int_E f \, d\mu = \int_X f \chi_E \, d\mu.
\]
(d) implies that the integrability of $|f|$ establishes the integrability of
$f$. In fact, a simple computation shows that the converse is true as well.
\index{Lebesgue, Henri!integral|)}

\subsection{\texorpdfstring{$L^p$}{Lp} Spaces}

\index{Lebesgue, Henri!space|(}
In light of the above observation, we see that
the collection $L^1(X,\mu)$ of complex-valued measurable functions $f$ on $X$
such that $\int |f| \, d\mu < \infty$ collects all integrable complex-valued
functions on $X$. We thus define the \emph{$L^1$-norm}\index{norm!L1@$L^1$} $\|f\|_1$ of
$f \in L^1(X,\mu)$ to be the integral $\int |f| \, d\mu$.
Note that the $L^1$-norm is
not a norm as it is, since functions that are zero almost everywhere
still have the $L^1$-norm of zero. To rectify this issue, we consider
$L^1(X,\mu)$ to be the quotient vector space\index{quotient space|(} defined by the equivalence
relation
\[
 f \sim g \Leftrightarrow f = g \mbox{ almost everywhere},
\]
at which point the $L^1$ norm becomes a \emph{bona fide} norm on $L^1(X,\mu)$.

We pause to make two remarks. Note first that every integrable function
must be finite almost everywhere, whence
each extended real-valued integrable function is equal almost-everywhere
to a complex-valued integrable function. Therefore, extended real-valued
integrable functions can be put in $L^1(X,\mu)$ without disrupting
the complex-vector-space structure thereof.
We also point out that the equivalence-class definition provides
no real benefit beyond resolving a few technical issues. Therefore,
we shall be intentionally sloppy and speak of \emph{functions}
in $L^1(X,\mu)$, unless structural nit-picking is necessary.

Endowing $L^1(X,\mu)$ with the corresponding norm topology, we can now
consider the dominated convergence theorem\index{convergence theorem!dominated}
as a sufficient condition for turning pointwise almost-everywhere convergence
of integrable functions into convergence in the $L^1$-norm. We also
have a partial converse, which states that every sequence of integrable
functions converging in the
$L^1$-norm admits a subsequence, with a dominating function in $L^1$, that
converges pointwise almost everywhere. We note that the $L^1$-metric
\[
 d_{L^1}(f,g) = \|f-g\|_1
\]
is complete, so that $L^1(X,\mu)$ is a \emph{Banach space}\index{Banach space}, a 
normed linear space whose norm-induced metric topology is complete.

It is also useful to consider the space $L^2(X,\mu)$ of square-integrable
functions on $X$, with the quotient-space construction as above to avoid
technical problems. The bilinear form
\[
 \langle f,g \rangle_2 = \int_X f\bar{g} \, d\mu
\]
is an inner product on $L^2(X,\mu)$, which is well-defined by the
\emph{Cauchy-Schwarz inequality}\index{inequality!Cauchy-Schwarz}:
\[
 |\langle f,g\rangle_2| \leq \|f\|_2\|g\|_2.
\]
Here $\|\cdot\|_2$ is the corresponding $L^2$-norm\index{norm!L2@$L^2$}
\[
 \|f\|_2 = \langle f,f\rangle_2^{1/2}
 = \left( \int_X |f|^2 \, d\mu \right)^{1/2},
\]
which furnishes a complete metric. Therefore, $L^2(X,\mu)$ is a
\emph{Hilbert space}\index{Hilbert space}, an inner product space whose norm-induced metric topology
is complete. Even better, if we set $X$ to be the Euclidean space $\bR^d$ and
$\mu$ the $d$-dimensional Lebesgue measure, then the corresponding $L^2$-space
is \emph{separable}, viz., it contains a countable dense subset. Since all
separable Hilbert spaces are unitarily isomorphic to one another, $L^2$ is,
in a sense, \emph{the} Hilbert space.

Recall that a \emph{linear functional}\index{linear functional} on a real or complex vector space $V$ is a
linear transformation on $V$ into the scalar field\footnote{Since we
primarily work over the complex field $\bC$ in the present thesis,
we will not retain this level of generality for the rest of the
thesis. One exception occurs in \S\ref{s-elements_of_functional_analysis},
where we consider real vector spaces and complex vector spaces
separately.} $\bF$,
which is taken to be either $\bR$ or $\bC$. If $V$ is a normed
linear space, a linear functional $l$ on $V$ is \emph{bounded}\index{bounded!linear functional} in case
it admits a constant $k$ such that
\begin{equation}\label{boundedness-definition}
 |lv| \leq k\|v\|_V
\end{equation}
for all $v \in V$. We note that $l$ is bounded if and only if $l$ is continuous
with respect to the norm topology of $V$. The collection $V^*$ of bounded
linear functionals on $V$ forms a vector space, called the \emph{dual space}\index{dual space}
of $V$. It is a standard result in real analysis that $V^*$ is a Banach space
with the operator norm\index{norm!operator}
\[
 \|l\|_{V^*} = \sup_{\|v\| \leq 1} |lv|,
\]
which, in turn, is the infimum of all possible $k$ in (\ref{boundedness-definition}).

Since many transformations of functions that arise in mathematical analysis
can be understood as bounded linear functionals on function spaces, it is of
interest to describe them as concretely as possible. A common approach,
known as a \emph{representation theorem},\index{representation theorem} is to determine the obvious
bounded linear functionals on the given function space, and then to
investigate the extent in which arbitrary bounded linear
functionals can be represented by the obvious ones.
For $L^2$, we have a wonderfully concrete representation theorem,
due to Frigyes Riesz:

\index{representation theorem!F. Riesz, Hilbert-space version}\index{Hilbert space!dual of|see{representation theorem}}
\begin{theorem}[F. Riesz representation theorem, Hilbert-space version]\label{hilbert-riesz-representation}
If $\mc{H}$ is a Hilbert space, then each bounded linear functional
$l:\mc{H} \to \bC$ admits a unique element $u \in \mc{H}$ such that
\[
 lv = \langle v,u \rangle_{\mc{H}}
\]
for all $v \in \mc{H}$. Moreover, $\|l\|_{\mc{H}^*} = \|u\|_{\mc{H}}$.
\end{theorem}

\noindent It follows that we can identify each element of $\mc{H}^*$ with an element of
$\mc{H}$. In particular, we conclude that
\[
 (L^2(X,\mu))^* = L^2(X,\mu)
\]
in light of the above identification.

Having considered $L^1$ and $L^2$, we now define, for each $p \in [1,\infty)$,
the \emph{Lebesgue space $L^p(X,\mu)$ of order $p$}\index{Lp@$L^p$ space|see{Lebesgue space}} on $X$ by collecting the
complex-valued measurable functions $f$ on $X$ such that\index{norm!Lp@$L^p$}
\[
 \|f\|_p = \left( \int_X |f|^p \, d\mu \right)^{1/p} < \infty.
\]
The standard quotient construction\index{quotient space|)} is applied here as well, turning $\|\cdot\|_p$
into a norm. With the language of Lebesgue spaces, \emph{H\"{o}lder's inequality}\index{inequality!H\"{o}lder's}
can be stated succinctly as
\[
 \|fg\|_1 \leq \|f\|_p\|g\|_{p'},
\]
where $p>1$ and $p'$ is the \emph{conjugate exponent}\index{conjugate exponent}
\[
 p' = \frac{p}{p-1}
\]
of $p$. Note that $1/p + 1/p' = 1$.

Note that if $f \in L^1(X,\mu)$ and $g$ is bounded, then
\[
 \|fg\|_1 \leq \|f\|_1 \sup_{x \in X} |g(x)|.
\]
Expanding on this idea, we introduce the space $L^\infty(X,\mu)$ of
complex-valued measurable functions $f$ on $X$ whose \emph{essential supremum}\index{essential supremum}\index{norm!Li@$L^\infty$|see{essential supremum}}
\[
 \|f\|_\infty = \inf\{\lambda \in \bR : \mu(\{x : |f(x)| > \lambda\})=0\}
\]
is finite. The space $L^\infty(X,\mu)$ can be considered as a ``limiting
space'' of $L^p(X,\mu)$, for if $f \in L^\infty$ is supported on a set
of finite measure, then $f \in L^p$ for all $p < \infty$ and
\[
 \lim_{p \to \infty} \|f\|_p = \|f\|_\infty.
\]
We remark that H\"{o}lder's inequality\index{inequality!H\"{o}lder's} holds for $p =1$ as well,
with the identification $1/\infty = 0$ to yield $p' = \infty$. 

Given $p \in [1,\infty]$, \emph{Minkowski's inequality}\index{inequality!Minkowski's} establishes the
triangle inequality for $\|\cdot\|_p$, thus turning $\|\cdot\|_p$
into a norm on $L^p(X,\mu)$. Moreover, the \emph{Riesz-Fischer theorem}
guarantees that $L^p(X,\mu)$ is a Banach space. A partial converse
to the $L^p$ dominated convergence theorem\index{convergence theorem!dominated} continues to hold,
so that a sequence of functions
converging in the $L^p$ norm admits a pointwise almost-everywhere
convergent subsequence with a dominating function in $L^p$, continues
to hold.

Observe, however, that the dominated convergence theorem fails to
hold on $L^\infty$. The $L^p$ representation theorem for $L^p$,
which yields the identification $(L^p)^* = L^{p'}$, also fails to hold
for $p = \infty$: see \S\S\ref{fr-riesz-representation}.
We shall have more to say about the representation
theorem in the next subsection.
\index{Lebesgue, Henri!space|)}

\subsection{\texorpdfstring{$\sigma$}{Sigma}-Finite Measure Spaces}

\index{measure!sigma-finite@$\sigma$-finite|(}
In this subsection, we review three major theorems of measure and integration theory
that requires the $\sigma$-finiteness
hypothesis. The first is the $L^p$ representation theorem, as
was alluded to above:

\index{representation theorem!F. Riesz, Lp-space version@F. Riesz, $L^p$-space version}\index{Riesz, Frigyes!representation theorem|see{\\ representation theorem}}
\begin{theorem}[F. Riesz representation theorem, $L^p$-space version]\label{Lp-riesz-representation}
Suppose that $(X,\mf{M},\mu)$ is a $\sigma$-finite measure space.
If $p \in [1,\infty)$, then each bounded linear functional $l$ on $L^p(X,\mu)$
admits a unique linear function $u \in L^{p'}(X,\mu)$ such that
\begin{equation}\label{riesz-eq}
 l(f) = \int fu \, d\mu
\end{equation}
for all $f \in L^p(X,\mu)$. Moreover, $\|l\|_{(L^p)^*} = \|u\|_{L^{q'}}$,
whence $(L^p)^*$ is isometrically isomorphic to $L^{p'}$.
\end{theorem}

It is an easy consequence of H\"{o}lder's inequality that every
function of the form (\ref{riesz-eq}) is a bounded linear functional
on $L^p$. The representation theorem states that linear functionals
of the form (\ref{riesz-eq}) are, in fact, \emph{all} bounded linear
functionals on $L^p$.

Since the proof of the representation theorem makes use of a few
key notions that we shall need in later sections, we study it in detail.
To this end, we fix a measurable space $(X,\mf{M})$ and recall that a function
$\nu:\mf{M} \to \bC$ is a \emph{complex measure}\index{measure!complex} if, for each $E \in \mf{M}$
and every countable partition $\{E_n : n \in \bN\}$ of $E$ in $\mf{M}$, the\
function $\nu$ is countably additive, viz.,
\[
 \nu(E) = \sum_{n=1}^\infty \nu(E_n).
\]
Note that the definition forces $\mu(X) < \infty$.

We sometimes use the name \emph{positive measures}\index{measure!positive} for measures proper in order
to distinguish them from complex measures. In fact, there is a canonical
way of assigning a positive measure corresponding to each complex measure
$\nu$: the \emph{total variation}\index{measure!total variation of} of $\nu$ is the positive measure $|\nu|$
defined to be
\[
 |\nu|(E) = \sup_{E \in \mf{M}} \sum_{n=1}^\infty |\nu(E_n)|
\]
for each $E \in \mf{M}$, where the supremum is taken over all partitions
$\{E_n : n \in \bN\}$ of $E$ belonging to $\mf{M}$.
Recalling that a measure $\nu$, complex or positive, on $(X,\mf{M})$ is said to be
\emph{absolutely continuous}\index{measure!absolute continuity of} with respect to a positive measure
$\mu$ on $(X,\mf{M})$ if $\nu(E) = 0$ for all $E \in \mf{M}$ such that
$\mu(E)=0$, we see that $\nu$ is absolutely continuous with respect
to $|\nu|$. In general, we write
\[
 \nu \ll \mu
\]
to denote the absolute continuity of $\nu$ with respect to $\mu$.

A polar opposite notion to absolute continuity is defined
as follows: two measures
$\nu_1$ and $\nu_2$, positive or complex, are said to be \emph{mutually
singular}\index{measure!mutually singular} if there exists a disjoint pair of measurable sets $A$ and $B$ such
that
\[
 \nu_1(E) = \nu_1(A \cap E) \htwo \mbox{and} \htwo
 \nu_2(E) = \nu_2(B \cap E)
\]
for all $E \in \mf{M}$. We write
\[
 \nu_1 \perp \nu_2
\]
to denote the mutual singularity of $\nu_1$ and $\nu_2$.

We are now ready to state the second theorem of this section,
which is the main ingredient of the proof of the representation
theorem.

\begin{theorem}[Lebesgue-Radon-Nikodym]\label{lebesgue-radon-nikodym}\index{Lebesgue-Radon-Nikodym!theorem}\index{Radon-Nikodym!theorem|see{\\ Lebesgue-Radon-Nikodym}}
Let $(X,\mf{M})$ be a measurable \linebreak
space, $\nu$ a complex measure, and $\mu$ a
positive $\sigma$-finite measure. Then there is a unique pair of
complex measures $\nu_a$ and $\nu_s$ such that
\[
 \nu = \nu_a + \nu_s, \hone \nu_a \ll \mu, \hone \nu_s \perp \mu,
\]
and there exists a $u \in L^1(X,\mu)$ such that
\[
 \nu_a(E) = \int_E u \, d\mu
\]
for all $E \in \mf{M}$. Any such function agrees with $u$ almost
everywhere on $X$.
\end{theorem}

Two remarks are in order. First, if $\nu$ is a positive finite measure,
then so are $\nu_a$ and $\nu_s$. Second, if $\nu \ll \mu$,
then $d\nu = u d\mu$ for an $L^1$ function $u$ defined uniquely
almost everywhere. This $u$ is called the
\emph{Radon-Nikodym derivative}\index{Radon-Nikodym!derivative} and is denoted by $\frac{d\nu}{d\mu}$,
so that
\[
 d\nu = \frac{d\nu}{d\mu} d\mu.
\]

Having stated the Lebesgue-Radon-Nikodym theorem, we proceed to the proof
of the $L^p$ representation theorem. In what follows, we use the
\emph{complex signum function}\index{signum function}\index{sign function|see{signum function}}
\[
 \sgn z =
 \begin{cases}
  \frac{z}{|z|} & \mbox{ if } z \in \bC \smallsetminus \{0\}; \\
  0 & \mbox{ if } z = 0.
 \end{cases}
\]

\begin{proof}[Proof of Theorem \ref{Lp-riesz-representation}]
We first claim that the norm of $u \in L^{p'}(X,\mu)$ can be computed by
the identity
\begin{equation}\label{riesz-proof-formula}
 \|u\|_{p'} = \sup_{\|f\|_p \leq 1} \left| \int fu \, d\mu \right|.
\end{equation}
Note first that
\[
 \int |fu| \, d\mu \leq \|f\|_p \|u\|_{p'} \leq \|u\|_{p'}
\]
by H\"{o}lder's inequality, so long as $\|f\|_p \leq 1$.
If $p > 1$, then we set
\[
 f(x) = |u(x)|^{p'-1} \frac{\overline{\sgn u(x)}}{\|u\|^{p'-1}_{p'}}
\]
and observe that
\[
 \int fu \, d\mu = \frac{1}{\|u\|^{p'-1}_{p'}} \int |u(x)|^p \, d\mu = \|u\|_{p'}.
\]
Since $\|f\|_p = 1$, the claim follows. If $p = 1$, then we fix $\ve>0$ and
invoke the $\sigma$-finiteness of $\mu$ to find a set $E$ of finite positive
measure on which
\[
 |u(x)| \geq \|u\|_\infty - \ve.
\]
We then set
\[
 f(x) = \frac{\chi_E(x) \sgn u(x)}{\mu(E)}
\]
and observe that $\|f\|_1 = 1$ and
\[
 \left| \int fu \right| = \frac{1}{\mu(E)} \int_E |u| \, d\mu \geq
 \|u\|_\infty - \ve.
\]
Since $\ve > 0$ was arbitrary, the claim follows.

We now establish a converse to H\"{o}lder's inequality: namely, if $u$ is a
measurable function that is integrable on all sets of finite measure and
satisfies the bound
\[
 \sup_{\substack{\|s\|_p \leq 1 \\ s \mbox{ simple}}}
 \left| \int s u \right| = k< \infty,
\]
then $u \in L^{p'}$ and $\|u\|_{p'} = k$. To this end, we recall that
there exists a sequence $(u_n)_{n=1}^\infty$ of simple functions such
that $|u_n| \leq |u|$ almost everywhere and $u_n \to g$ pointwise
almost everywhere. If $p > 1$, then we set
\[
 f_n(x) = |u_n(x)|^{p'-1} \frac{\overline{\sgn u(x)}}{\|u_n\|^{p'-1}_{p'}}
\]
for each $n \in \bN$ and observe that
\[
 k = \sup_{\substack{\|s\|_p \leq 1 \\ s \mbox{ simple}}} \left| \int s u \right|
 \geq \left|\int f_n u \, d\mu\right|
 = \left|\frac{\int |u_n(x)|^{p'} \, d\mu}{\|u_n\|_{p'}^{p'-1}}\right|
 = \|u_n\|_{p'}.
\]
Fatou's lemma implies that $\|u\|_{p'}^{p'} \leq k^{p'}$, and H\"{o}lder's
inequality establishes the reverse inequality, verifying the claim. If $p=1$,
then we fix $\ve>0$ and let
\[
 E = \{x : |u(x)| \geq k + \ve\}.
\]
Assume for a contradiction that $\mu(E) > 0$, and invoke
the $\sigma$-finiteness of $\mu$ to find a set $F$ of finite positive
measure contained in $E$. We set
\[
 f(x) = \frac{\chi_{F}(x) \overline{\sgn u(x)}}{\mu(F)}
\]
and observe that
\[
 k = \sup_{\substack{\|s\|_1 \leq 1 \\ s \mbox{ simple}}} \left| \int s u \right|
 \geq \left| \int f u \, d\mu \right|
 = \left| \frac{\int_{F} |u| \, d\mu}{\mu(F)} \right|
 \geq k + \ve,
\]
which is absurd. It thus follows that $\|u\|_\infty \leq k$, and the reverse
inequality is established by H\"{o}lder's inequality.

Let us now return to the proof of the theorem. Assume for now that
$\mu$ is a finite measure on $X$, so that $\chi_E \in L^p(X,\mu)$ for every
measurable set $E$. Fix a bounded linear functional $l$ on $L^p(X,\mu)$
and set
\[
 \nu(E) = l(\chi_E)
\]
for each measurable set $E$. We claim that $\nu$ is a complex measure 
on $(X,\mf{M})$ that is absolutely continuous with respect to $\mu$.
To see this, we first note that the linearity of $\varphi$ establishes
the finite additivity of $\nu$. Given a pairwise disjoint sequence
$(E_n)_{n=1}^\infty$ of measurable sets, we set
\[
 E = \bigcup_{n=1}^\infty E_n \htwo \mbox{and} \htwo
 F_N = \bigcup_{n=N+1}^\infty  E_n
\]
for each $N \in \bN$. Observe that $\chi_E = (\chi_{E_1} + \cdots + \chi_{E_N})
+ \chi_{F_N}$, and so
\[
 \nu(E) = \left(\sum_{n=1}^N \nu(E_n) \right) + \nu(F_N).
\]
Since
\begin{equation}\label{absolute-continuity-uh}
 |\nu(F)| =|l(\chi_F)| \leq \|l\|_{(L^q)^*}\|\chi_F\|_p
 = \|l\|_{(L^q)^*} \left( \mu(F) \right)^{1/p},
\end{equation}
for every measurable set $F$, we see that $\nu(F_N) \to 0$ as
$N \to \infty$. Therefore, $\nu$ is countably additive, and (\ref{absolute-continuity-uh})
shows that $\nu \ll \mu$.

We now invoke the \hyperref[lebesgue-radon-nikodym]{Lebesgue-Radon-Nikodym theorem}\index{Lebesgue-Radon-Nikodym!theorem}
to find the unique $u \in L^1(X,\mu)$ such that
\[
 \nu(E) = \int_E u \, d\mu
\]
for all measurable sets $E$. Therefore,
\[
 l(\chi_E) = \int \chi_E u \, d\mu
\]
and the linearity of the integral implies that
\[
 l(s) = \int s u \, d\mu
\]
for each simple function $s$ on $X$. Recalling that every $L^p$ function can be
approximated by simple functions, we conclude that
\[
 \varphi(f) = \int f u \, d\mu
\]
for all $f \in L^p(X,\mu)$. Furthermore, we have
\[
 \|u\|_{p'} = \sup_{\|f\|_p \leq 1} \left| \int fu \, d\mu \right|
 = \sup_{\|f\|_p \leq 1} |l(f)| = \|l\|_{(L^p)^*}
\]
by formula (\ref{riesz-proof-formula}). This establishes the theorem for $\mu(X) < \infty$.

We now lift the assumption that $\mu$ is finite. By the $\sigma$-finiteness of
$\mu$, we can find an increasing sequence $(E_n)_{n=1}^\infty$ of finite-measure
sets whose union is $X$. On each $E_n$, we invoke the representation theorem for
finite measures to find an integrable function $u_n$ on $E_n$ such that
\[
 l(f\chi_{E_n}) = \int_{E_n} fu_n \, d\mu
\]
for all $f \in L^p(X,\mu)$. We extend $u_n$ onto $X$ by setting it to be zero
on $X \smallsetminus E_n$ and invoke the converse of H\"{o}lder's inequality
to see that
\[
 \|u_n\|_q \leq \|;\|_{(L^p)^*}.
\]

Note that $(u_n)_{n=1}^\infty$ is a pointwise almost-everywhere convergent
sequence of integrable functions. We set the limit to be $u$ and apply
Fatou's lemma to conclude that
\[
 \|u\|_q \leq \|l\|_{(L^p)^*}.
\]
It now follows that
\[
 l(f\chi_{E_n}) = \int f \chi_{E_n} u \, d\mu
\]
for each $f \in L^p(X,\mu)$ and every $n \in \bN$, whence taking the limit
yields
\[
 l(f) = \int f u \, d\mu.
\]
We now apply H\"{o}lder's inequality to establish the reverse inequality
\[
 \|u\|_q \geq \|l\|_{(L^p)^*},
\]
and the proof is complete.
\end{proof}

Finally, we review integration on product spaces. Given two measure
spaces $(X,\mf{M},\mu)$ and $(Y,\mf{N},\nu)$, we define the \emph{product
$\sigma$-algebra} $\mf{M} \otimes \mf{N}$ to be the smallest $\sigma$-algebra
containing the collection
\[
 \{E \times F : E \in \mf{M} \mbox{ and } F \in \mf{N}\}.
\]
of \emph{measurable rectangles}. It is a standard fact that the set function
\[
 (\mu \times \nu)(E \times F) = \mu(E)\nu(F),
\]
initially defined on the collection of measurable rectangles, can be extended
to a measure on $(X \times Y, \mf{M} \otimes \mf{N})$, forming a measure space
$(X \times Y, \mf{M} \otimes \mf{N}, \mu \times \nu)$.\index{measure!product}

If $E \subseteq X \times Y$, $x \in X$, and $y \in Y$, we define the
\emph{$x$-section} $E_x$ and the \emph{$y$-section} $E^y$ of $E$ as follows:
\[
 E_x = \{y' \in Y : (x,y') \in E\} \htwo \mbox{and} \htwo
 E^y = \{x' \in X : (x',y) \in E\}.
\]
Analogously, given a function $f:X \times Y \to \bC$, we define the
\emph{$x$-section} $f_x$ and the \emph{$y$-section} $f^y$ of $f$ as follows:
\[
 f_x(y) = f^y(x) = f(x,y).
\]
The main theorem, due to Guido Fubini and Leonida Tonelli, gives sufficient
conditions for which the order of integration may be exchanged:

\begin{theorem}[Fubini-Tonelli]\label{fubini-tonelli}\index{Fubini-Tonelli theorem}\index{Fubini's theorem|see{Fubini-Tonelli theorem}}\index{Tonelli's theorem|see{Fubini-Tonelli theorem}}
Let $(X,\mf{M},\mu)$ and $(Y,\mf{N},\nu)$ are $\sigma$-finite measure spaces.
\begin{enumerate}[(a)]
 \item \emph{\textbf{Tonelli's theorem.}} If $f$ is a nonnegative integrable
 function on $X \times Y$, then the functions $x \mapsto \int f_x \, d\nu$
 and $y \mapsto \int f^y \, d\mu$ are nonnegative integrable functions on
 $X$ and $Y$, respectively, and
 \[
  \int_{X \times Y} f \, d(\mu \times \nu)
 = \int_X \int_Y f(x,y) \, d\nu(y) \, d\mu(x)
 = \int_Y \int_X f(x,y) \, d\mu(x) \, d\nu(y).
 \]
 \item \emph{\textbf{Fubini's theorem}.} If $f$ is integrable on $X \times Y$,
 then $f_x$ is integrable on $Y$ for almost every $x \in X$, $f^y$ is integrable
 on $X$ for almost every $y \in Y$, the function $x \mapsto \int f_x \, d\nu$
 is integrable on $X$, the function $y \mapsto \int f^y \, d\mu$ is integrable
 on $Y$, and
 \[
  \int_{X \times Y} f \, d(\mu \times \nu)
 = \int_X \int_Y f(x,y) \, d\nu(y) \, d\mu(x)
 = \int_Y \int_X f(x,y) \, d\mu(x) \, d\nu(y).
 \]
\end{enumerate}
\end{theorem}
\index{measure!sigma-finite@$\sigma$-finite|)}

\subsection{The Lebesgue Measure}

\index{Lebesgue, Henri!measure|(}
We conclude our review by presenting a rapid treatment of the basic properties
of the canonical measure on the Euclidean space, the Lebesgue measure.
We adopt a particularly constructive approach
from \cite{Stein_Shakarchi:B2005},\index{Stein, Elias M.}\index{Shakarchi, Rami} hinging on a decomposition theorem of
Hassler Whitney. In what follows, a \emph{cube} is an $n$-fold product of closed
intervals of the same length, and two cubes in $\bR^d$ are \emph{almost disjoint}
if their interiors are disjoint.

  \begin{theorem}[Whitney decomposition theorem]\label{whitney-decomposition}\index{Whitney decomposition theorem}
Every open set in $\bR^d$ can be decomposed into a union of
countably many almost-disjoint cubes.
\end{theorem}

Our version of the theorem omits the estimate on the sizes of the cubes.
See \S\S\ref{fr-whitney-decomposition} for the precise version.
We shall have more occasions to use the decomposition theorem, so we present
a full proof of the theorem.

\begin{figure}[!htb]
\begin{center}
\includegraphics[scale=0.3]{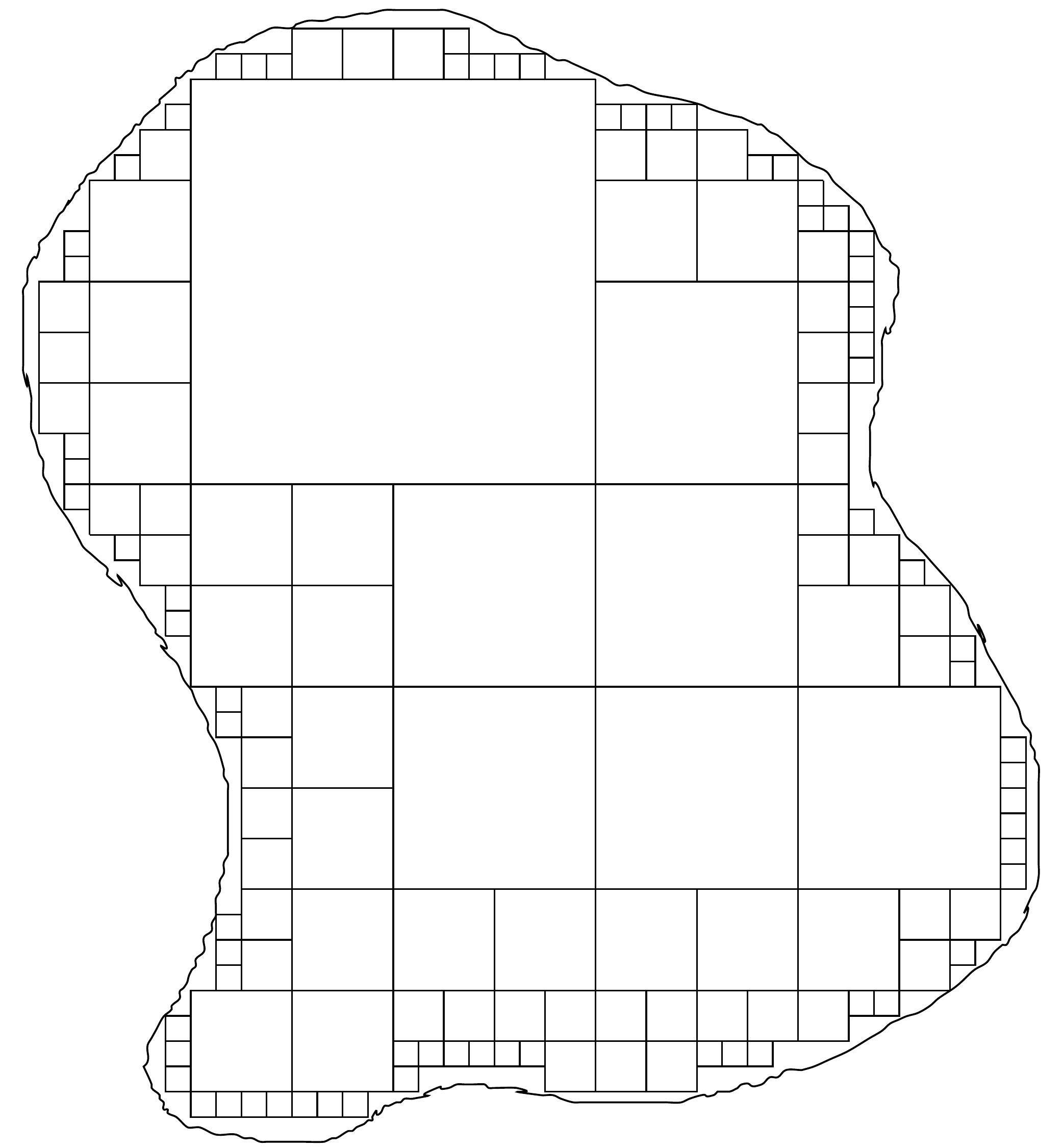}
\caption{A Whitney decomposition of a two-dimensional figure}\index{pancake!an ugly one}
\end{center}
\end{figure}

\begin{proof}
Let $O$ be an open subset of $\bR^d$. For each $n$, we consider the grid
formed by cubes of side length $2^{-n}$, whose vertices have coordinates
in
\[
 2^{-n} \bZ^d = \{ (k_1,\ldots,k_d) : 2^n k_j \in \bZ
 \mbox{ for all } 1 \leq j \leq d\}.
\]
Note that the grid formed at the $n$th stage is obtained by bisecting the
cubes that formed the grid at the $(n-1)$th stage. We define $\mc{C}_n$ 
to be the collection of all such cubes, of side length $2^{-n}$, that 
intersect $O$. Note that $\bigcup \mc{C}_n$ is a countable collection of
cubes, and that the union of all cubes in each $\mc{C}_n$ contains $O$.

We now extract a collection $\mc{C}$ of almost-disjoint cubes from $\bigcup
\mc{C}_n$ as follows. We begin by declaring every cube in $\mc{C}_1$ to
be a member of $\mc{C}$. For each $n$, we throw away all cubes in $\mc{C}$ that
intersect nontrivially with some cubes in $\mc{C}_n$ and add in all cubes
in $\mc{C}_n$ that are almost disjoint from every remaining cube in $\mc{C}$.
The resulting collection $\mc{C}$ clearly consists of almost-disjoint cubes.
Since $\mc{C} \subseteq \bigcup \mc{C}_n$, the collection is countable as well.

It now remains to show that the union of all cubes in $\mc{C}$ is $O$. Since
the union evidently contains $O$, it suffices to show that no point in $\bR^d
\smallsetminus O$ is covered by $\mc{C}$. Let $x$ be such a point, and suppose
for a contradiction that there is a cube $Q \in \mc{C}$ containing $x$. This,
in particular, implies that $Q \not\subseteq O$. Fix
$y \in Q \cap O$. Since $O$ is open, a sufficiently large integer $N$
guarantees that the grid at the $N$th stage admits a cube $Q'$ of side length
$2^{-N}$ that contains $y$ and is contained entirely in $O$. But $Q'$ is a
smaller cube than $Q$ that intersects $Q$ nontrivially, whence $Q$ cannot be
an element of $\mc{C}$ in the first place. This is absurd, and the proof is now
complete.
\end{proof}

The decomposition provides a natural way of assigning a volume to each
open set in $\bR^n$: we look at the sum of the volumes of the cubes in
each Whitney decomposition and take the infimum as the volume of the
open set.
We then define the \emph{Lebesgue outer measure} $m^*$ of an arbitrary
subset to be the infimum of the volumes of the open supersets of the set.
To ensure countable additivity, we restrict the outer measure to the
subsets $E$ of $\bR^d$ such that each $\ve>0$ admits an open superset $O$ of
$E$ with the estimate $m^*(O \smallsetminus E) < \ve$. Such a set is called
a \emph{Lebesgue-measurable set}, and the restriction of $m^*$
onto the collection of Lebesgue-measurable sets is referred to as
the \emph{Lebesgue measure}. We denote the Lebesgue measure of $E$
by $m(E)$, or $|E|$ if there is no danger of confusion.

Before we review the basic properties of the Lebesgue measure,
we remark that any reference to measurability of subsets of $\bR^d$ or
functions on $\bR^d$ in this thesis shall be
for the Lebesgue measure, unless otherwise
specified. We now recall that the Lebesgue measure of an arbitrary measurable
set can be approximated by that of open sets and closed sets:

\begin{prop}\label{approximation-by-open-and-closed-sets}
If $E$ is a measurable subset of $\bR^d$, then each $\ve>0$ has
a corresponding closed set $F \subseteq E$ and an open
set $O \supseteq E$ such that $|E \smallsetminus F| \leq \ve$
and $|O \smallsetminus E| \leq \ve$. If $|E| < \infty$, we may take $F$
to be a compact set.
\end{prop}

It follows from the above proposition and the continuity of measure that
the Lebesgue measure is \emph{Borel regular}:\index{Borel, \'{E}mile!regular} $m$ is a Borel measure,\index{Borel, \'{E}mile!measure}
and each measurable subset $E$ of $\bR^d$ has a corresponding Borel
subset\index{Borel, \'{E}mile!set} $B$ of $\bR^d$ such that $E \subseteq B$ and $|E| = |B|$.
Even better, it turns out that countable intersections of open sets,
known as \emph{$G_\delta$
sets},\index{G-delta@$G_\delta$ set} and countable unions of closed sets, known as \emph{$F_\sigma$
sets},\index{F-sigma@$F_\sigma$ set} are quite enough:

\begin{prop}\label{approximation-by-borel-sets}
$E \subseteq \bR^d$ is measurable
\begin{enumerate}[(a)]
 \item if and only if there exists a $G_\delta$ set $G \subseteq \bR^d$
 such that $|G \smallsetminus E| = 0$;
 \item if and only if there exists an $F_\sigma$ set $F \subseteq \bR^d$
 such that $|E \smallsetminus F| = 0$.
\end{enumerate}
\end{prop}

The Lebesgue measure behaves well under linear endomorphisms on $\bR^d$:

\begin{theorem}
If $E \subseteq \bR^d$ is measurable and $T:\bR^d \to \bR^d$ a linear
transformation, then
\[
 m(T(E)) = |\det T|m(E).
\]
\end{theorem}

This, in particular, implies that the Lebesgue measure is invariant
under translation and rotation, and scales in tune with the usual
geometric intuition under dilation.

We frequently denote the integral of a measurable function $f:\bR^d \to \bC$
with respect to the Lebesgue measure by
\[
 \int f(x) \, dx = \int_{\bR^d} f(x) \, dx
\]
instead of the more cumbersome
\[
 \int_{\bR^d} f(x) \, dm(x).
\]

We also recall that the change-of-variables formula continues to hold for
the Lebesgue integral:

\begin{theorem}[Change-of-variables formula]\label{change-of-variables}\index{change-of-variables formula}
If $O$ is an open subset of $\bR^d$ and $\phi:O \to \bR^d$ an injective
differentiable function, then, for each $f \in L^1(O)$, we have
$f \circ \varphi \in L^1(\phi(O))$ and 
\[
 \int_{\phi(O)} f(x) \, dx = \int_O f(\phi(x)) |\det D\phi(x)| \, dx,
\]
where $D \phi(x)$ is the total derivative of $\phi$ at $x$.
\end{theorem}

This, in particular, implies that
\begin{eqnarray*}
 \int_{\bR^d} f(x+h) \, dx &=& \int_{\bR^d} f(x) \, dx \\
 \int_{\bR^d} f(-x) \, dx &=& \int_{\bR^d} f(x) \, dx \\ 
 \delta^d \int_{\bR^d} f(\delta x) \, dx &=& \int_{\bR^d} f(x) \, dx
\end{eqnarray*}
for all $h \in \bR^d$ and $\delta > 0$.

With this, we conclude the review. We refer the reader to \S\S\ref{fr-littlewoods-three-principles}
for a discussion of some other nice properties of the Lebesgue measure.
\index{Lebesgue, Henri!measure|)}\index{measure!Lebesgue|see{Lebesgue, Henri}}
\section{Approximation in \texorpdfstring{$L^p$}{Lp} Spaces}\label{s-approximation_in_lp_spaces}

\index{Banach space|(}
The central objects of study in interpolation theory are function spaces
and linear operators between function spaces. Typically, the function
spaces are vector spaces equipped with topologies that are compatible
with the vector-space structure. We can then require the operators
to be continuous, so as to have them behave well under various limiting
processes. Recall that a linear operator $T:V \to W$ between normed linear
spaces $V$ and $W$ is \emph{bounded}\index{bounded!linear operator}\index{norm!operator} if there exists a constant $k>0$ such that
\[
 \|Tv\|_W \leq k\|v\|_V
\]
for all $v \in V$. It is easy to show that $T$ is bounded
if and only if $T$ is continuous with respect to the norm topologies
of $V$ and $W$, and that the collection $\ms{L}(V,W)$ of bounded
linear operators from $V$ to $W$ with the \emph{operator norm}
\[
 \|T\|_{V \to W} = \sup_{\|v\|_W \leq 1} \|Tv\|_W
\]
is a Banach space if $W$ is a Banach space.

It is, however, cumbersome to specify the value of an operator
at all points on its domain. We therefore seek to find a suitable
subset of the domain that is essentially the whole space.

\begin{defin}
A subset $D$ of a topological space $X$ is \emph{dense}\index{density} if
the closure of $D$ is in $X$ is $X$.
\end{defin}

As it turns out, it is enough in many cases to specify the value
of an operator on a dense subset of its domain. Even better, the
extension is norm-preserving.

\begin{theorem}\label{norm-preserving-extension}\index{norm!-perserving extension}
Let $V$ and $W$ be normed linear spaces and $D$ a dense linear subspace
of $V$. If $W$ is a Banach space and $T:D \to W$ is a bounded linear
operator, then there exists a unique linear operator $T_1:V \to W$
such that $\|T_1\|_{V \to W} = \|T\|_{D \to W}$ and $T_1|_D = T$.
\end{theorem}

\begin{proof}
For each $v \in V$, we find a sequence $(v_n)_{n=1}^\infty$ in $D$
that converges to $v$. Since $T$ is bounded, $(Tv_n)_{n=1}^\infty$
is a Cauchy sequence in $W$, whence it converges to a vector $w_v \in W$.
If $(v_n')_{n=1}^\infty$ is another sequence in $D$ that converges
to $v$, then
\begin{eqnarray*}
 \|Tv_n' - w\|
 &\leq& \|Tv_n' - Tv_n\| + \|Tv_n - w_v\| \\
 &\leq& \|T\|\|v_n' - v_n\| + \|Tv_n - w_v\| \\
 &\leq& \|T\|\left( \|v_n' - v\| + \|v - v_n\| \right) + \|Tv_n - w_v\|,
\end{eqnarray*}
and so $(Tv_n')_{n=1}^\infty$ converges to $w_v$ as well. The operator
\[
 T_1 v =
 \begin{cases}
  T v & \mbox{ if } v \in D \\
  w_v & \mbox{ if } v \in V \smallsetminus D
 \end{cases}
\]
is therefore well-defined, and its linearity is a trivial consequence of
the linearity of $T$. Furthermore,
\[
 \|T_1 v\|
 = \lim_{n \to \infty} \|T_1 v_n\| 
 = \lim_{n \to \infty} \|T v_n\|
 \leq \lim_{n \to \infty} \|T\|\|v_n\| 
 = \|T\|\|v_n\|,
\]
for each $v \in V$, so that $\|T_1\| \leq \|T\|$. Since
\[
 \|T_1\|
 = \sup_{\|v\| \leq 1} \|T_1 v\|
 \geq \sup_{\substack{\|v\| \leq 1 \\ v \in D}} \|T_1 v\|
 = \sup_{\substack{\|v\| \leq 1 \\ v \in D}} \|T v\|
 = \|T\|,
\]
we have $\|T_1\| = \|T\|$, as was to be shown.
\end{proof}
\index{Banach space|)}

\subsection{Approximation by Continuous Functions}

It is therefore useful to have several examples of dense subspaces
of frequently used function spaces. We know, for example, that
we can approximate integrable functions by simple functions, whence
the space of simple functions is dense in $L^1$. Moreover,
we can approximate just as well if after restricting ourselves to
simple functions on sets of finite measure.

\begin{prop}\label{approximation-by-simple-functions}\index{simple function}\index{density!of simple functions}
Let $(X,\mf{M},\mu)$ be a $\sigma$-finite measure space.
The space of simple functions with finite-measure support is
dense in $L^1(X,\mu)$.
\end{prop}

\begin{proof}
Let $(X_n)_{n=1}^\infty$ be an increasing sequence of finite-measure
subsets of $X$ whose union is $X$. Given an $f \in L^1(X,\mu)$,
the monotone convergence theorem implies that
\[
 \lim_{n \to \infty} \|f - f\chi_n\| = 0.
\]
Therefore, each $\ve>0$ admits an integer $N$ such that
\[
 \int_X |f-f\chi_{X_N}| \, d\mu < \ve.
\]
We now find a sequence $(s_n)_{n=1}^\infty$ of simple functions
in $L^1(X_N,\mu)$ such that $|s_n| \leq |f\chi_{X_N}|$ 
almost everywhere and $s_n \to f$ pointwise almost everywhere.
Arguing as above, we can find an integer $M$ such that
\[
 \int_{X_N} |s_M -f\chi_{X_N}| \, d\mu < \ve. 
\]
Extending $s_M$ onto $X$ by defining $s_M(x) = 0$ on
$X \smallsetminus X_N$, we see that
\[
 \|f-s_M\|_1 \leq \|f - f\chi_{X_N}\|_1 + \|f\chi_{X_N} - s_M\|_1
 < 2\ve,
\]
as desired.
\end{proof}

While the above proposition is useful, the domain of the simple functions
in question can be quite complicated. In order to obtain more refined
approximations, it is necessary to confine ourselves to nicer measure
spaces. For simplicity, we shall work on $\bR^d$, but the main theorem
of this subsection (Theorem \ref{approximation-by-continuous-functions})
can be established on more general measure spaces.
See \S\S\ref{fr-approximation-by-Cc-on-locally-compact-spaces} for a discussion.

First, we observe that it suffices to deal with simple functions on
very nice domains.

\begin{theorem}\label{approximation-by-step-functions}\index{density!of simple functions}
The space of simple functions over cubes is dense in $L^1(\bR^d)$.
\end{theorem}

\begin{proof}
In light of Proposition \ref{approximation-by-simple-functions}, it suffices to
approximate characteristic functions over finite-measure sets by simple
functions over cubes. We therefore fix a set $E$ of finite measure.
Pick $\ve>0$, and invoke Proposition \ref{approximation-by-open-and-closed-sets}
to find an open set $O$ containing $E$ such that $|O \smallsetminus E| < \ve$.
We then have
\[
 \|\chi_O - \chi_E\|_1 =|\chi_{O \smallsetminus E}\|_1 < \ve.
\]

Let $(Q_n)_{n=1}^\infty$ be a \hyperref[whitney-decomposition]{Whitney decomposition}\index{Whitney decomposition theorem}
of $O$. Since the intersection of two almost-disjoint cubes is of measure zero, we have
\[
 |O| = \sum_{n=1}^\infty |Q_n|.
\]
Noting that $|O| \leq |E| + \ve$, we can find an integer $N$ such that
\[
 \sum_{n=N+1}^\infty |Q_n| < \ve.
\]
This, in particular, implies that
\[
 \left|\bigcup_{n=1}^N Q_n \right| = \sum_{n=1}^N |Q_n| < |O| - \ve,
\]
and so
\[
 \|\chi_O - \chi_{Q_1 \cup \cdots \cup Q_N}\|_1
 = |O| - \sum_{n=1}^N |Q_n| < \ve.
\]
Therefore,
\[
 \|\chi_E - \chi_{Q_1 \cup \cdots \cup Q_N}\|_1
 \leq \|\chi_E - \chi_O\|_1 + \|\chi_O - \chi_{Q_1 \cup \cdots \cup Q_N}\|_1
 < \frac{2}{3}\ve.
\]

It remains to ``disjointify'' the cubes $Q_1,\ldots,Q_N$, so as to turn
$\chi_{Q_1 \cup \cdots \cup Q_N}$ into a finite sum of characteristic
functions over cubes. For each $1 \leq n \leq N$, we fix a cube $R_n$
in the interior of $Q_n$ such that $|Q_n \smallsetminus R_n| < \ve$. Then
$\{R_1,\ldots,R_N\}$ is a pairwise-disjoint collection of cubes, and
\begin{eqnarray*}
 \left\|\chi_{Q_1 \cup \cdots \cup Q_N} - \sum_{n=1}^N \chi_{R_n} \right\|_1
 &=& \|\chi_{Q_1 \cup \cdots \cup Q_N} - \chi_{R_1 \cup \cdots \cup R_N}\|_1 \\
 &=& \left\|\chi_{(Q_1 \smallsetminus R_1) \cup \cdots \cup (Q_N \smallsetminus R_N)}\right\|_1 \\
 &=& \sum_{n=1}^N |Q_n \smallsetminus R_n|  \\
 &\leq& N\ve.
\end{eqnarray*}
It now follows that
\begin{eqnarray*}
 \left\|\chi_E - \sum_{n=1}^N \chi_{R_n} \right\|_1
 &\leq& \|\chi_E - \chi_{Q_1 \cup \cdots \cup Q_N}\|_1
 + \left\|\chi_{Q_1 \cup \cdots \cup Q_N} - \sum_{n=1}^N \chi_{R_n} \right\|_1 \\
 &\leq& (N+2)\ve,
\end{eqnarray*}
as was to be shown.
\end{proof}

We note that a characteristic function over a cube can be approximated quite easily
with a continuous function: we just draw steep lines from the boundary of the graph
down to zero, thereby producing a function with a tent-like graph. \index{tent function|(}Precisely, 
we construct a \emph{tent function}, which is 1 on a nice set---a cube
in our case---and 0 outside of a small dilation of the set. Once we approximate
a characteristic function over an arbitrary cube with a continuous function, we
can then appeal to the density of simple functions over cubes to show that
integrable functions can be approximated with continuous functions. This is the
content of the following theorem: 

\begin{theorem}\label{approximation-by-continuous-functions}\index{density!of continuous functions}
The space $\mc{C}_c(\bR^d)$ of continuous functions on $\bR^d$ with
compact support is dense in $L^p(\bR^d)$ for each $1 \leq p < \infty$. 
\end{theorem}

\begin{proof}
We first prove the theorem for $p=1$. By Theorem \ref{approximation-by-step-functions},
it suffices to approximate characteristic functions over cubes by continuous
functions with compact support. This is done by constructing a tent function
over the generic cube $Q = [a_1,b_1] \times \cdots \times [a_d,b_d]$. 
To do so, we fix $\delta>0$, and let
\[
 f_\delta(x) = 
 \begin{cases}
  1 & \mbox{ if } |x| \leq 1;\\
  \ds 1 - \frac{|x|-1}{\delta} & \ds \mbox{ if } 1 < |x| < 1 + \delta; \\
  0 & \mbox{ if } |x| > \ds 1 + \delta;
 \end{cases}
\]
This is a tent function over $[-1,1]$, with decay taking place
on intervals of length $\delta$ to make the function continuous. We observe that
\[
 \|f_\delta - \chi_{[-1,1]}\|_{L^1(\bR)}
 < \|\chi_{[-1+\delta,1+\delta]} - \chi_{[-1,1]}\|_{L^1(\bR)}
 = 2\delta,
\]
whence $f_\delta$ is a continuous approximation of the characteristic function
$\chi_{[-1,1]}$.

\begin{figure}[!htb]
\begin{center}
\includegraphics[scale=1.2]{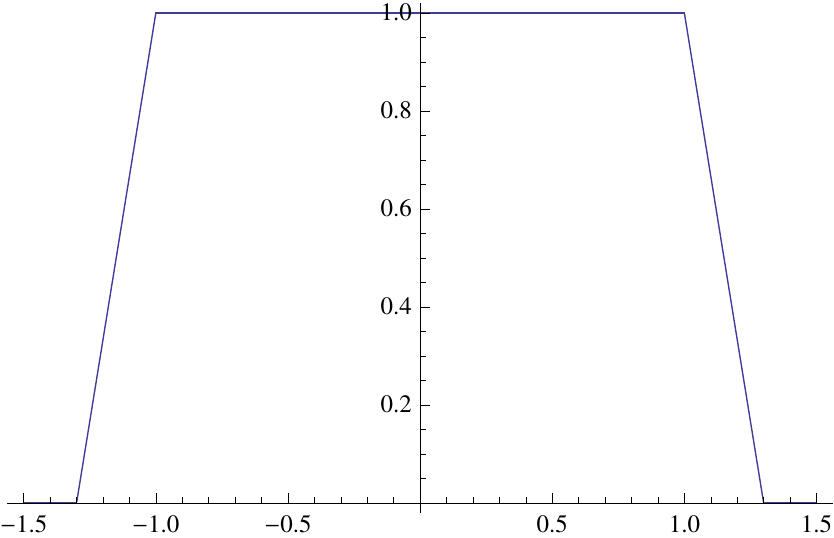}
\caption{A one-dimensional tent function, with $\delta=0.3$}
\end{center}
\end{figure}

For each $1 \leq n \leq d$, we consider the function
\[
 g^n_\delta(x) = f \left( \frac{b_n - a_n}{2} x + \frac{b_n + a_n}{2} \right).
\]
This is a tent function over $[a_n,b_n]$, viz., $g^n_\delta$ is a continuous
function that is 1 on $[a_n,b_n]$ and vanishes outside an interval slightly
bigger than $[a_n,b_n]$. Precisely, the decay to zero takes place on the intervals
$[a_n - (b_n-a_n)\delta/2, a_n]$ and $[b_n, b_n + (b_n-a_n)\delta/2]$, so that
a similar computation as above yields
\begin{equation}\label{tent-estimate}
 \|g^n_\delta - \chi_{[a_n,b_n]}\|_{L^1(\bR)} < (b_n-a_n)\delta.
\end{equation}
\index{tent function|)}
We now set
\[
 g_\delta(x_1,\ldots,x_d) = g^1_\delta(x_1) \cdots g^d_\delta(x_d).
\]
By construction, $g_\delta$ is clearly 1 on $Q$ and 0 outside a cube slightly
bigger than $Q$. By \hyperref[fubini-tonelli]{Tonelli's theorem} and the
(\ref{tent-estimate}), we have
\[
 \|g_\delta - \chi_Q\|_{L^1(\bR^d)} < \delta^d \prod_{n=1}^d (b_n - a_n).
\]
Since $\delta$ can be made arbitrarily small, we have successfully produced
a continuous approximation of $\chi_Q$. This proves the theorem for $p=1$.

We move onto the $p>1$ case. We fix $\ve>0$ and claim that each $f \in L^p(\bR^d)$
furnishes a $g \in L^\infty(\bR^d)$ and a compact set $K \subseteq (\bR^d)$
such that $\supp g \subseteq K$ and $\|f-g\|_p < \ve$. To see this,
we define the truncation\index{truncation} operator $T_r:\bC \to \bC$ at $r>0$ 
by setting
\[
 T_n z =
 \begin{cases}
  z & \mbox{ if } |z| \leq r; \\
  \frac{rz}{|z|} & \mbox{ if } |z| > r;
 \end{cases}
\]
and set
\[
 f_n = \chi_{B_n(0)} T_n f
\]
for each $n \in \bN$.
The dominated convergence theorem implies that $\|f_n - f\|_p \to 0$ as
$n \to \infty$, and so we can pick an integer $N$ such that $\|f_N - f\|_p
< \ve/2$.

Fix a second constant $\ve_1>0$.
Since $f_n \in L^1(\bR^d)$, we can find $g' \in \ms{C}_c(\bR^d)$ such that
$\|g'-f_N\|_1 < \ve_1$. We set $g = T_{\|g\|_\infty} g'$ and note that
\[
 g \in \ms{C}_c(\bR^d),
 \hone \|g\|_\infty \leq \|f_N\|_\infty,
 \hone \mbox{and} \hone  \|g - f_N\|_1 < \ve_1.
\]
It now follows from H\"{o}lder's inequality that
\begin{eqnarray*}
 \|f - g\|_p	
 &\leq& \|f-f_N\|_p + \|f_N - g\|_p \\
 &<& \frac{\ve}{2} + \|g-g_1\|^{1/p}_1 \|g-g_1\|^{1 - (1/p)}_\infty \\
 &<& \frac{\ve}{2} + \ve_1^{1/p} \left( 2\|g\|_\infty \right)^{1-(1/p)},
\end{eqnarray*}
whence picking a sufficiently small $\ve_1>0$ yields
\[
 \|f-g\|_p < \ve.
\]
This completes the proof of the theorem.
\end{proof}

\subsection{Convolutions}

To refine our approximation techniques even further, we now introduce a
widely used ``smoothing'' operation.

\begin{defin}\index{convolution}
The \emph{convolution} of measurable functions $f$ and $g$ on $\bR^d$
at $x \in \bR^d$ is defined to be
\[
 (f*g)(x) = \int f(x-y)g(y) \, dy,
\]
whenever the expression is well-defined.
\end{defin}

Convolutions can be thought of as a kind of weighted average. Indeed,
if $f(x) = 1$, then $f*g$ corresponds to the integral mean value of
$g$ over the entire space. Before we discuss why convolutions are
smoothing operations, we establish a few basic properties thereof.

\begin{theorem}[Properties of convolutions]\label{convolution}
Let $1 \leq p \leq \infty$.
\begin{enumerate}[(a)]
 \item The convolution of two measurable functions is measurable.
 \item If $f$ and $g$ are measurable, then $f*g = g*f$.
 \item \emph{\textbf{Young's inequality.}}\index{inequality!Young's} If $f \in L^p(\bR^d)$ and $g \in L^1
 (\bR^d)$, then $f*g$ is well-defined almost everywhere and
 \[
   \|f*g\|_p \leq \|f\|_p \|g\|_1.
 \]
\end{enumerate}
\end{theorem}

The following inequality of Hermann Minkowski,\index{Minkowski, Hermann} which we shall use frequently
in the remainder of the thesis, plays a crucial role in the proof of (c).

\begin{theorem}[Minkowski's integral inequality]\label{minkowski-integral-inequality}\index{inequality!Minkowski's integral}
Let $(X,\mf{M},\mu)$ and $(Y,\mf{N}, \linebreak \nu)$ be $\sigma$-finite measure spaces and
$f$ an $(\mu \times \nu)$-measurable function on $X \times Y$. If $f \geq 0$ and
$1 \leq p < \infty$, then
\[
 \left( \int \left( \int f(x,y) \, d\nu(y) \right)^p \, d\mu(x) \right)^{1/p}
 \leq \int \left( \int f(x,y)^p \, d\mu(x) \right)^{1/p} \, d\nu(y).
\]
\end{theorem}

\begin{proof}
The $p=1$ case is \hyperref[fubini-tonelli]{Tonelli's theorem}. If
$1 < p < \infty$, then \hyperref[fubini-tonelli]{Tonelli's theorem} and H\"{o}lder's inequality
imply that each $g \in L^{p'}(X,\mu)$ satisfies the
following inequality:
\begin{eqnarray*}
 \iint f(x,y) \, d\nu(y) |g(x)| \, d\mu(x)
 &=& \iint f(x,y) |g(x)| \, d\mu(x) \, d\nu(y) \\
 &\leq& \int \left( \int f(x,y)^p \, d\mu(x) \right)^{1/p} \|g\|_{p'} \, d\nu(y) \\
 &=&  \|g\|_{p'} \int \left( \int f(x,y)^p \, d\mu(x) \right)^{1/p} \, d\nu(y).
\end{eqnarray*}
Let
\[
 \phi(x) = \int \int f(x,y) \, d\nu(y).
\]
We know from the \hyperref[Lp-riesz-representation]{Riesz representation theorem}\index{representation theorem!F. Riesz, Lp-space version@F. Riesz, $L^p$-space version}
that
\[
 \|\phi\|_p = \sup_{\|g\|_{p'} \leq 1} \left| \int \phi g \, d\mu \right|,
\]
whence the above inequality implies that
\begin{eqnarray*}
 & & \left( \int \left( \int f(x,y) \, d\nu(y) \right)^p \, d\mu(x) \right)^{1/p} \\
 &=& \|\phi\|_p \\
 &=& \sup_{\|g\|_{p'} \leq 1} \left| \int \phi g \, d\mu \right| \\
 &\leq& \sup_{\|g\|_{p'} \leq 1} \|g\|_{p'} \int \left( \int f(x,y)^p \, d\mu(x) \right)^{1/p} \, d\nu(y) \\
 &=& \int \left( \int f(x,y)^p \, d\mu(x) \right)^{1/p} \, d\nu(y),
\end{eqnarray*}
as was to be shown.
\end{proof}

We proceed to the proof of the basic properties of convolutions.

\begin{proof}[Proof of Theorem \ref{convolution}]
 (a) Let $f$ and $g$ be measurable functions on $\bR^d$. We first show that
\[
 f_1(x,y) = f(x-y)
\]
is measurable on $\bR^d \times \bR^d$. It clearly suffices to prove that
$f_1^{-1}(B_r(z))$ is measurable for each $r>0$ and every $z \in \bC$.
For each subset $E$ of $\bR^d$, we define
\[
 \tilde{E} = \{(x,y) \in \bR^d \times \bR^d : x-y \in E\}.
\]
Since the subtraction operation
\[
 (x,y) \mapsto x-y
\]
is a continuous map from $\bR^d \times \bR^d$ into $\bR^d$, the set
$\tilde{E}$ is open whenever $E$ is open. By taking a countable intersection,
we see that $\tilde{E}$ is a $G_\delta$ set if $E$ is.

We also claim that $\tilde{E}$ is of measure zero whenever $E$ is of
measure zero. Indeed, if $|E| = 0$, then we can find a sequence
$(O_n)_{n=1}^\infty$ of open sets such that $O_n \supseteq E$ for each
$n$ and $|O_n| \to 0$ as $n \to \infty$. Given $n,k \in \bN$,
\hyperref[fubini-tonelli]{Tonelli's theorem} and the translation invariance
of the Lebesgue measure imply that
\begin{eqnarray*}
 |\tilde{O}_n \cap B_k(0)|
 &=& \iint \chi_{O_n}(x-y) \chi_{B_k(0)}(y) \, dy \, dx \\
 &=& \int \left( \int \chi_{O_n}(x-y) \, dx \right) \chi_{B_k(0)} \, dy \\
 &=& \int \left( \int \chi_{O_n}(x) \, dx \right)  \chi_{B_k(0)} \, dy \\
 &=& |O_n||B_k|.
\end{eqnarray*}
Therefore, if we set $\tilde{E}_k = \tilde{E} \cap B_k(0)$ for each positive
integer $k$, then $\tilde{E}_k \subseteq \tilde{O}_n \cap B_k(0)$ for all $n$ and
$|\tilde{O}_n \cap B_k(0)| \to 0$ as $n \to \infty$. It follows that
$|\tilde{E}_k| = 0$, whence by continuity of measure we have $|\tilde{E}| = 0$,
as desired.

We now fix an $r>0$ and a $z \in \bC$ and set $E = f^{-1}(B_r(z))$,
so that $\tilde{E} = f_1^{-1}(B_r(z))$.
Since $B_r(z)$ is open, the measurability of $f$ implies the measurability
of $E$, whence Proposition \ref{approximation-by-borel-sets} furnishes
a $G_\delta$ set $G$ such that $G \supseteq E$ and $|G \smallsetminus E| =0$.
Setting $F = G \smallsetminus E$, we see that
\[
 \tilde{G} = \tilde{E} \cup \tilde{F}
\]
is a $G_\delta$ set. Since $|F| = 0$, the above argument shows that
$|\tilde{F}| = 0$, whereby we appeal once again to Proposition
\ref{approximation-by-borel-sets} to conclude that $\tilde{E}$ is measurable.

It follows that if $f$ and $g$ are measurable functions on $\bR^d$, then
\[
 (x,y) \mapsto f(x-y) g(y)
\]
is measurable on $\bR^d \times \bR^d$. We now invoke
\hyperref[fubini-tonelli]{Fubini's theorem} to conclude that
\[
 (f*g)(x) = \int f(x-y)g(y) \, dy
\]
is measurable on $\bR^d$.

(b) This is a trivial consequence of the commutativity of multiplication in $\bC$
and the translation invariance of the Lebesgue measure.

(c) Since $|(f*g)(x)| \leq \int |f(x-y)||g(y)| \, dy$, we invoke
\hyperref[minkowski-integral-inequality]{Minkowski's integral inequality}\index{inequality!Minkowski's integral}
to conclude that
\begin{eqnarray*}
 \|f*g\|_p
 &=& \left( \int \left| \int f(x-y)g(y) \, dy \right|^p \, dx \right)^{1/p} \\
 &\leq& \int \left( \int |f(x-y)|^p \, dx \right)^{1/p} |g(y)| \, dy \\
 &=& \|f\|_p \|g\|_1,
\end{eqnarray*}
as was to be shown.
\end{proof}

We now return to the task of justifying the ``smoothing operator'' nickname
that convolutions possess. We begin by showing that the convolution of
two compactly supported functions is compactly supported.

\begin{theorem}\label{support-of-convolution}\index{support!of convolutions}
Let $1 \leq p \leq \infty$. If $f \in L^p(\bR^d)$ and $g \in L^1(\bR^d)$, then
\[
 \supp (g*h) \subseteq \overline{ \supp f + \supp g}
 = \overline{\{ x + y : x \in \supp f \mbox{ and } y \in \supp g\}}
\]
\end{theorem}

As it stands now, however, it is not entirely clear how we should interpret
the statement of the above theorem. While two functions that are almost everywhere
are considered to be ``the same'' in integration theory, the traditional
notion of support can fail to assign the same support to both functions.
To rectify this issue, we adopt a new definition:

\begin{defin}\index{support!of Lp-functions@of $L^p$-functions}
Let $f$ be a complex-valued function on $\bR^d$ and $O$
the union of all open sets in $\bR^d$ on which $f$
vanishes almost everywhere. We define the \emph{support} of $f$,
denoted $\supp f$, to be the complement of $O$.
\end{defin}

Of course, we must justify the new terminology:

\begin{prop}
Let $f$ and $O$ be defined as above.
Then $f$ vanishes almost everywhere on $O$. If $g$ is another function
that is equal to $f$ almost everywhere, then $\supp f = \supp g$. Furthermore,
this definition of support agrees with the old definition of support for
continuous functions.
\end{prop}

\begin{proof}
Since $\bR^d$ is second-countable, we can find a sequence $(O_n)_{n=1}^\infty$
of of open sets such that $f$ vanishes almost everywhere on each $O_n$ and
that the union of all $O_n$ is $O$. The union is countable, and so $f$ vanishes
almost everywhere on $O$. If $f = g$ almost everywhere, then $f$ vanishes almost
everywhere on a vanishing set of $g$, and vice versa, whence $\supp f$ and
$\supp g$ must agree. If $f$ is continuous, then
$O = f^{-1}(\bC \smallsetminus \{0\})$, and so $\supp f = f^{-1}(\{0\})$,
as was to be shown.
\end{proof}

With the new definition, we proceed to the proof of the theorem.

\begin{proof}[Proof of Theorem \ref{support-of-convolution}]
By \hyperref[convolution]{Young's inequality}, the map $y \mapsto f(x-y)g(y)$
is integrable for each $x \in \bR^d$. 
Writing $x - \supp f$ to denote the set $\{x - y : y \in \supp f\}$, we see
that
\[
 (f*g)(x) = \int_{(x - \supp f) \cap \supp g} f(x-y) g(y) \, dy.
\]
Let $E = \supp f + \supp g$ for notational simplicity. We
note that $x \notin E$ implies $(x-\supp f) \cap \supp g = \varnothing$, so that
$(f*g)(x) = 0$.  Therefore, $(f*g)(x) = 0$ for almost every $x \in \bR^d
\smallsetminus E$. In particular, $(f*g)(x) = 0$ for almost every $x$ in the
interior of $\bR^d \smallsetminus E$, and so
\[
 \supp (f*g) \subseteq \overline{E}
\]
by the new definition of support.
\end{proof}

We are now ready to supply the promised justification of the smoothing-operations
nickname. For notational simplicity we define the following shorthand:

\begin{defin}\index{multi-index}
A $d$-dimensional \emph{multi-index} is a $d$-tuple
\[
 \alpha = (\alpha_1,\ldots,\alpha_d)
\]
consisting of nonnegative integers. We employ the following notations
for multi-indices; here $\alpha$ and $\beta$ are multi-indices, and
$x$ an element of $\bR^d$:
\begin{eqnarray*}
 |\alpha| &=& \alpha_1 + \cdots + \alpha_d \\
 x^\alpha &=& x_1^{\alpha_1} \cdots x_d^{\alpha_d} \\
 D^\alpha &=& \left(\frac{\partial}{\partial x_1}\right)^{\alpha_1}
 \cdots \left(\frac{\partial}{\partial x_d}\right)^{\alpha_d} \\
 \alpha \pm \beta &=& (\alpha_1 \pm \beta_1 , \cdots , \alpha_d \pm \beta_d).
\end{eqnarray*}
\end{defin}

The main theorem can now be stated as follows:

\begin{theorem}[Convolution as a smoothing operation]\label{convolution-smoothing}\index{convolution!as a smoothing operation}
Convolutions are ``smoothing operations'' in the following sense:
\begin{enumerate}[(a)]
 \item If $f \in \ms{C}_c(\bR^d)$ and $g \in L^1_\loc(\bR^d)$, then
 $f*g$ is well-defined everywhere and $f*g \in \ms{C}(\bR^d)$.
 \item If $f \in \ms{C}^k_c(\bR^d)$ and $g \in L^1_\loc(\bR^d)$, then
 $f*g \in \ms{C}^k(\bR^d)$ and 
 \[
  D^\alpha (f*g) = (D^\alpha f) * g
 \]
 for each multi-index $|\alpha| \leq k$. The result holds for
 $k = \infty$ as well.
 \item If $f \in \ms{C}^k(\bR^d)$ and $g \in \ms{C}^l(\bR^d)$, then
 $f*g \in \ms{C}^{k+l}(\bR^d)$ and
 \[
  D^{\alpha + \beta} (f*g) = (D^\alpha f) * (D^\beta g)
 \]
 for all multi-index $|\alpha| \leq k$ and $|\beta| \leq l$.
\end{enumerate}
\end{theorem}

\begin{proof}
(a) For each $x \in \bR^d$, the map $x \mapsto f(x-y)g(y)$ is measurable and has
compact support, hence integrable. Therefore, $(f*g)(x)$ is defined for all
$x \in \bR^d$. We now fix $x \in \bR^n$, pick a sequence $(x_n)_{n=1}^\infty$
in $\bR^d$ converging to $x$, and find a compact subset $K$ of $\bR^d$ such that
\[
 x_n - \supp f = \{x_n - y : y \in \supp f\} \subseteq K
\]
for each $n \in \bN$. It then follows that $f$ is uniformly continuous on $K$
and $f(x_n - y) = 0$ for all $n \in \bN$ and $y \in \bR^d \smallsetminus K$.
We can thus pick a sequence $(\ve_n)_{n=1}^\infty$ of positive real numbers
converging to zero such that
\[
 |f(x_n - y) - f(x-y)| \leq \ve_n \chi_K(y)
\]
for each $n \in \bN$ and every $y \in \bR^d$. Multiplying through by $|g(y)|$
and integrating with respect to $y$, we obtain
\[
 |(f*g)(x_n) - (f*g)(x)| \leq \ve_n \int_K |g(y)| \, dy.
\]
Since the right-hand side converges to zero as $n \to \infty$, we conclude that
\[
 \lim_{n \to \infty} (f*g)(x_n) = (f*g)(x),
\]
as was to be shown.

(b) We suppose for now that $k = 1$. The task at hand then reduces to
establishing the claim that $f*g$ is continuously differentiable at
each $x \in \bR^d$ and
\[
 \nabla (f*g)(x) = (\nabla f * g)(x).
\]
To this end, we pick $x \in \bR^d$. For each $y \in \bR^d$,
we observe that
\[
 \lim_{|h| \to 0} \frac{|f((x-y)+h) - f(x-y) - \nabla f(x-y) \cdot h|}{|h|}
 = 0,
\]
whence every $\ve>0$ admits $M_\ve>0$ such that
\[
 |f((x-y)+h) - f(x-y) - \nabla f(x-y) \cdot h| \leq \ve|h|
\]
for all $|h| < M_\ve$.

Fix a compact subset $K$ of $\bR^d$ such that
\[
 x  - \supp f + B_{M_\ve}(0)
 = \{ (x-y)+h : y \in \supp f \mbox{ and } |h|<M_\ve\}
 \subseteq K.
\]
Since
\[
 f((x-y)+h) - f(x-y) - \nabla f(x-y) \cdot h = 0
\]
for all $|h| < M_\ve$ and $y \in K$, we have
\[
 |f((x-y)+h) - f(x-y) - \nabla f(x-y) \cdot h| \leq \ve|h|\chi_K(y)
\]
for all $y \in \bR^d$. Multiplying through by $|g(y)|$ and integrating
with respect to $y$, we see that
\[
 |(f*g)(x+h) - (f*g)(x) - (\nabla f*g)(x) \cdot h|
 \leq \ve|h| \int_K |g(y)| \, dy.
\]
It follows that $f*g$ is differentiable at $x$, with the gradient
\[
 \nabla (f*g)(x) = (\nabla f * g)(x).
\]
$f \in \ms{C}^1_c(\bR^d)$ implies that $\nabla f \in \ms{C}_c(\bR^d)$,
whence $\nabla f * g \in \ms{C}(\bR^d)$ by (a). This completes the
proof for $k=1$. The case for $k>1$ now follows from induction.

(c) is a trivial consequence of (b) and the commutativity of convolution,
and the proof is now complete.
\end{proof}

\subsection{Approximation by Smooth Functions}

We shall now establish the final approximation theorem of this section:
namely, the approximation of $L^p$ functions by smooth functions.
As was hinted at in the previous subsection, we shall use convolutions
to smooth out the approximating functions. The key result, known as
\emph{approximations to the identity}\footnote{See
\S\S\ref{fr-approximations-to-the-identity} for a discussion on the
name ``approximations to the identity''.}, provides a widely applicable
tool for generating a collection of approximating functions for any
given $L^p$ function.

\begin{theorem}[Approximations to the identity]\label{approximations-to-the-identity}\index{approximations to the identity}
Let $1 \leq p < \infty$. If $f \in L^p(\bR^d)$ and $\rho \in L^1(\bR^d)$
such that $\int \rho = 1$, then $\|f*\rho_\ve - f\|_p \to 0$ as $\ve
\to 0$, where $\rho_\ve(x) = \ve^{-d}\rho(\ve^{-1}x)$ for each
$\ve>0$.
\end{theorem}

As per Theorem \ref{convolution-smoothing}, we can make the approximating
functions as well-behaved as we would like. Indeed, we can construct
\emph{smooth} approximations to the identity, which we shall furnish
after the proof of the theorem.

\begin{proof}
We set $\Delta_f(y) = \|f(x-y) -f(x)\|_p$ for each $y \in \bR^d$. Fix $\delta>0$
and invoke Theorem \ref{approximation-by-continuous-functions} to find
$f_1 \in \ms{C}_c(\bR^d)$ with $\|f-f_1\|_p \leq \delta$. Set $f_2 = f - f_1$.
Since $f_1(x-y)$ converges uniformly to $f_1(x)$ as $y \to 0$, we see that
$\Delta_{f_1}(y) \to 0$ as $y \to 0$. Moreover, $\Delta_{f_2}(y) \leq 2\delta$, whence
\[
 \Delta_f(y) \leq \Delta_{f_1}(y) + \Delta_{f_2}(y) \to 0
\]
as $y \to 0$. By \hyperref[minkowski-integral-inequality]{Minkowski's integral inequality}\index{inequality!Minkowski's integral},
we have the following estimate:
\begin{eqnarray*}
 \|f* \rho_\ve - f\|_p
 &=& \left\| \int f(x-y)\rho_\ve(y) \, dy - f(x) \right\|_p \\
 &=& \left\| \int f(x-y)\rho_\ve(y) \, dy - f(x) \int \rho_\ve(y) \, dy \right\|_p \\
 &=& \left\| \int [f(x-y) - f(x)]\rho_\ve(y) \, dy \right\|_p \\
 &\leq& \int \| f(x-y) - f(x)\|_{L^p(x)} |\rho_\delta(y)| \, dy \\
 &=& \int \Delta_f(y) |\rho_\ve(y)| \, dy \\
 &=& \int \Delta_f (\ve y) |\rho(y)| \, dy;
\end{eqnarray*}
the last inequality follows from the \hyperref[change-of-variables]{change-of-variables formula}.

We have shown above that $\Delta_f(\ve y) \to 0$ as $\ve \to 0$. Furthermore,
we have the bound
\[
 |\Delta_f(\delta y) \rho(y)| \leq \|2 f\|_p |\rho(y)|,
\]
whence by the dominated convergence theorem we obtain
\begin{eqnarray*}
 \lim_{\ve \to 0} \|f * \rho_\ve - f\|_p
 &\leq& \lim_{\ve \to 0} \int \Delta_f(\ve y) |\rho(y)| \, dy \\
 &=& \int \lim_{\ve \to 0} \Delta_f(\ve y) |\rho(y)| \, dy,
\end{eqnarray*}
as was to be shown.
\end{proof}

\begin{cor}[Smooth approximations to the identity]\label{smooth-approximations-to-the-identity}\index{approximations to the identity!smooth}
There exists a sequence of \emph{mollifiers}\index{mollifiers} on $\bR^d$, which is a sequence
$(\rho_n)_{n=1}^\infty$ of nonnegative $\ms{C}^\infty$-maps on $\bR^d$ such that
$\supp \rho_n \subseteq \overline{B_{1/n}(0)}$ and $\int \rho_n = 1$ for
each $n \in \bN$. Furthermore, if $f \in L^p(\bR^d)$, then
$\|f*\rho_n - f\|_p \to 0$ as $n \to \infty$.
\end{cor}

\begin{figure}[!htb]
\begin{center}
\includegraphics[scale=1.4]{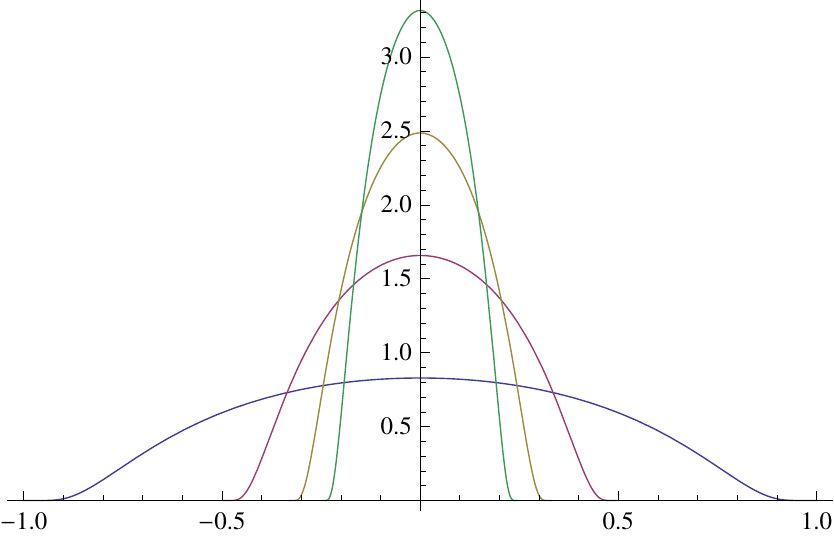}
\caption{The first four mollifiers}
\end{center}
\end{figure}

\begin{proof}
We set
\[
 \phi(x) =
 \begin{cases}
  e^{1/(|x|^2-1)} & \mbox{ if } |x| < 1; \\
  0 & \mbox{ if } |x| \geq 1
 \end{cases}
\]
and
\[
 \phi_\ve(x) = \frac{\ve^{-d} \phi(\ve^{-1} x)}{\int \phi(x) \, dx}.
\]
for each $\ve > 0$ Then each $\phi_\ve$ is a compactly supported smooth function
whose integral is 1, whence by Theorem \ref{approximations-to-the-identity}
we have
\[
 \lim_{\ve \to 0} \|f * \rho_\ve - f\|_p = 0.
\]

We now define a sequence $(\rho_n)_{n=1}^\infty$ of functions by
setting $\rho_n = \phi_{1/n}$ for each $n \in \bN$. It immediately follows from
the above construction that this is a sequence of mollifiers.
\end{proof}

The approximation theorem now follows as a simple corollary.

\begin{cor}\label{approximation-by-smooth-functions}\index{density!of smooth functions}
$\ms{C}^\infty_c(\bR^d)$ is dense in $L^p(\bR^d)$ for each $1 \leq p < \infty$.
\end{cor}

\begin{proof}
Fix $1 \leq p < \infty$. Let $(\rho_n)_{n=1}^\infty$ be a sequence of mollifiers,
set $B_n = \overline{B_n(0)}$ for each $n \in \bN$, and define a sequence
$(f_n)_{n=1}^\infty$ by
\[
 f_n = (f\chi_{B_n}) * \rho_n.
\]
Then \hyperref[convolution]{Young's inequality} implies that
\begin{eqnarray*}
 \|f-f_n\|_p
 &\leq& \|f - f*\rho_n\|_p + \|\rho_n * f - \rho_n * (f\chi_{B_n})\|_p \\
 &=& \|f - f*\rho_n\|_p + \|\rho_n * (f- f\chi_{B_n})\|_p \\
 &\leq& \|f - f*\rho_n\|_p + \|f-f\chi_{B_n}\|_p \|\rho_n\|_1 \\
 &=& \|f - f*\rho_n\|_p + \|f-f\chi_{B_n}\|_p.
\end{eqnarray*}
By Corollary \ref{smooth-approximations-to-the-identity}, we have
$\|f-(f*\rho_n)\|_p \to 0$ as $n \to \infty$, and the dominated convergence
theorem implies that $\|f - f\chi_{B_n}\|_p \to 0$ as $n \to \infty$.
It follows that
\[
 \lim_{n \to \infty} \|f-f_n\|_p = 0,
\]
as was to be shown.
\end{proof}

We conclude the section with another instant of convolutions
as smoothing operations. This time, we are able to recover
continuity without any smoothness on either side.

\begin{cor}\label{convolution-of-dual-Lp-functions}\index{convolution!as a smoothing operation}\index{inequality!H\"{o}lder's}
Let $1 < p < \infty$. If $f \in L^p(\bR^d)$ and $g \in L^{p'}(\bR^d)$,
then $f*g$ belongs to the space $\ms{C}_0(\bR^d)$ of continuous functions
vanishing at infinity.
\end{cor}

\begin{proof}
By H\"{o}lder's inequality, $f*g$ is well-defined everywhere on $\bR^d$.
For each $\ve>0$, Corollary \ref{approximation-by-smooth-functions} furnishes
$f_\ve,g_\ve \in \ms{C}^\infty_c(\bR^d)$ such that
\[
 \|f-f_\ve\|_p \leq \frac{\ve}{2(\|f\|_p + \|g\|_{p'})}
 \htwo \mbox{and} \htwo
 \|g-g_\ve\|_{p'} \leq \frac{\ve}{2(\|f\|_p + \|g\|_{p'})}.
\]
It then follows from H\"{o}lder's inequality that
\begin{eqnarray*}
 \|f*g - f_\ve*g_\ve\|_\infty
 &\leq& \|(f-f_\ve) * g\|_\infty + \|f_\ve * (g-g_\ve)\|_\infty \\
 &\leq& \left\| \|f-f_\ve\|_p \|g\|_{p'} \right\|_\infty
 + \left\| \|f_\ve\|_p \|g-g_\ve\|_{p'} \right\|_\infty \\
 &\leq& \|f-f_\ve\|_p \|g\|_{p'} + \|f_\ve\|_p \|g-g_\ve\|_{p'} \\
 &\leq& (\|f\|_p + \|g\|_{p'})(\|f-f_\ve\|_p + \|g-g_\ve\|_{p'} \\
 &\leq& (\|f\|_p + \|g\|_{p'}) \left( \frac{2\ve}{2(\|f\|_p + \|g\|_{p'})} \right) \\
 &=& \ve,
\end{eqnarray*}
whence $f*g$ is a uniform limit of smooth functions $f_\ve * g_\ve$ with
compact support. This establishes the corollary.
\end{proof}
  
\section{The Fourier Transform}\label{s-the_fourier_transform}

\index{Fourier, Joseph}
\index{Fourier transform|(}
We now restrict our attention to the famous operator of Joseph Fourier,
the Fourier transform. To motivate the definition, we consider the
``limiting case'' of the classical Fourier series
\[
 \sum_{n=-\infty}^\infty \hat{f}(n) e^{2 \pi i n x / L}, 
\]
of $L$-periodic functions $f:[-L/2,L/2] \to \bR$, whose Fourier
coefficients are given by the formula
\[
 \hat{f}(n) = \frac{1}{L} \int_{-L/2}^{L/2} f(x) e^{-2 \pi i n x / L} \, dx
\]
Indeed, we make a simple change of variable in the above formula to obtain
\[
 \hat{f}(n) = \int_{-1/2}^{1/2} f(L x) e^{-2 \pi i n x} \, dx,
\]
and ``sending $L$ to infinity'' leads us to the following:
\[
 \hat{f}(n) = \int_{-\infty}^{\infty} f(x) e^{-2 \pi i n x} \, dx.
\]
So long as $f$ decays suitably at infinity, the integral makes sense
even when $n$ is not an integer. Therefore, we replace $n$ with a real variable
$\xi$:
\[
 \hat{f}(\xi) = \int_{-\infty}^{\infty} f(x) e^{-2 \pi i \xi x} \, dx.
\]

We promptly generalize the above ``transform'' to higher dimensions; this,
of course, requires us to take the scalar product of multi-dimensional
variables $x$ and $\xi$, which we do by taking the standard dot product:
\[
 \hat{f}(\xi) = \int_{\bR^d} f(x) e^{-2 \pi i \xi \cdot x} \, dx.
\]

\subsection{The \texorpdfstring{$L^1$}{L1} Theory}

The above expression makes sense only if $f(x) e^{-2 \pi i \xi \cdot x}$
is in $L^1$ for all $\xi$. Since $|e^{-2 \pi i \xi \cdot x}| = 1$
for all $x$ and $\xi$, this is equivalent to the condition that
$f$ is in $L^1(\bR^d)$. We are thus led to the following definition:

\begin{defin}\index{Fourier transform!on L1@on $L^1$}
The \emph{Fourier transform} of $f \in L^1(\bR^d)$ is the function $\hat{f}$
given by
\[
 \hat{f}(\xi) = \int f(x) e^{-2 \pi i \xi \cdot x} \, dx
\]
for each $\xi \in \bR^d$. We also write $\ms{F}f$ to denote the Fourier
transform of $f$.
\end{defin}

Note that $\ms{F}$ can be thought of as an operator on $L^1(\bR^d)$.
By the linearity of the integral, $\ms{F}$ is a linear operator.
The target space of $\ms{F}$, as well as a few other basic properties
of $\ms{F}$, are established in the following proposition.

\begin{prop}\label{basic-properties-of-ft}
The Fourier transform of $f \in L^1(\bR^d)$ satisfies the following properties:
\begin{enumerate}[(a)]
 \item $\|\hat{f}\|_\infty \leq \|f\|_1$. Therefore, $\ms{F}$ is a bounded
 linear operator from $L^1(\bR^d)$ into $L^\infty(\bR^d)$.
 \item $\hat{f}$ is uniformly continuous on $\bR^d$.
 \item \emph{\textbf{Riemann-Lebesgue lemma.}}\index{Riemann-Lebesgue lemma}\index{Fourier transform!decay of|see{Riemann-Lebesgue lemma}} $\hat{f}$ vanishes at infinity,
 viz., $\hat{f}(\xi) \to 0$ as $|\xi| \to \infty$.
\end{enumerate}
\end{prop}

\begin{proof}
(a) It suffices to observe that
\[
 \|\hat{f}\|_\infty =
 \left\| \int f(x) e^{-2 \pi i \xi \cdot x} \, dx \right\|_\infty
 \leq \int |f(x)| |e^{-2 \pi i \xi \cdot x}| \, dx
 \leq \|f\|_1.
\]

(b) Since $|f(x+h)-f(x)| \leq 2|f(x)|$ for all sufficiently small $h \in \bR^d$,
it follows from the dominated convergence theorem that
\begin{eqnarray*}
 \lim_{h \to infty} |\hat{f}(\xi+h) - \hat{f}(\xi)|
 &\leq& \lim_{h \to infty} \int |f(x+h) - f(x)||e^{-2 \pi i \xi \cdot x}| \, dx \\
 &=& \int \lim_{h \to infty} |f(x+h) - f(x)| \, dx \\
 &=& 0.
\end{eqnarray*}

(c) Let $Q = [0,1]^d$. By \hyperref[fubini-tonelli]{Tonelli's theorem},
\[
 \widehat{\chi_Q(\xi)} =
 \int \chi_Q(x) e^{- 2 \pi i \xi \cdot x} \, dx
 = \prod_{n=1}^d \int_0^1 e^{-2 \pi i x_n \xi_n} \, dx_n
 = \prod_{n=1}^d \frac{e^{-2 \pi i \xi} - 1}{ - 2 \pi i \xi},
\]
which tends to zero as $|\xi| \to \infty$. By linearity, the Riemann-Lebesgue
lemma holds for all simple functions over cubes. Given a general integrable
function $f$ on $\bR^d$, we can invoke Theorem \ref{approximation-by-step-functions}
to find a simple function $s_\ve$ over cubes corresponding to each $\ve>0$,
satisfying the estimate
\[
 \|f-f_\ve\|_1 < \ve.
\]
Since $\hat{f_\ve}(\xi) \to 0$ as $|\xi| \to \infty$, we can find a constant $M$
such that $|\hat{f_\ve}(\xi)| < \ve$ for all $|\xi|>M$. It then follows that
\begin{eqnarray*}
 |\hat{f}(\xi)|
 &\leq& |\widehat{f_\ve}(\xi)| + \int |f(x)-f_\ve(x)||e^{-2 \pi i \xi \cdot x}| \, dx \\
 &=& |\widehat{f_\ve}(\xi)| + \|f-f_\ve\|_\ve \\
 &\leq& 2 \ve
\end{eqnarray*}
for all $|\xi| > M$, whence $\hat{f}(\xi) \to 0$ as $|\xi| \to \infty$.
\end{proof}

The Fourier transform behaves well under a number of symmetry operations in
the Euclidean space. The proof of the following proposition consists of
trivial computations and is thus omitted.

\begin{prop}\label{symmetry-invariance-of-fourier-transform}\index{Fourier transform!symmetry invariance of}
Let $f \in L^1(\bR^d)$ and $\tau \in \bR^d$.
\begin{enumerate}[(a)]
 \item $\ms{F}$ turns translation into rotation:\index{translation!of functions}
 if $\tau_hf(x)=f(x-h)$, then
 \[
  \widehat{\tau_hf}(\xi) = e^{-2 \pi i h \cdot \xi} \hat{f}(\xi).
 \]
 \item $\ms{F}$ turns rotation into translation:\index{rotation!of functions}
 if $e_hf(x) = e^{2 \pi i x \cdot h}f(x)$, then
 \[
  \widehat{e_h f}(\xi) = \tau_h\hat{f}(\xi).
 \]
 \item $\ms{F}$ commutes with reflection:\index{reflection!of functions}
 if $\tilde{f}(x) = f(-x)$, then
 \[
  \hat{\tilde{f}}(\xi) = \tilde{\hat{f}}(\xi).
 \]
 \item $\ms{F}$ scales nicely under dilation:\index{dilation!of functions}
 if we set $\delta_af(x) = f(ax)$ for each $a>0$, then
 \[
  \widehat{\delta_a f}(\xi) = a^{-d} \delta_{a^{-1}} \hat{f}(\xi)
 = a^{-d}\hat{f}(a^{-1}\xi)
 \]
 for all $a > 0$. \qed
\end{enumerate}
\end{prop}

The Fourier transform also behaves quite nicely under
differentiation. Indeed, the Fourier transform turns
differentiation into multiplication by a polynomial.

\begin{prop}\label{differentiation-to-multiplication}\index{Fourier transform!differentiation of}\index{differentiation!of the Fourier transform|see{\\ Fourier transform}}
Let $f \in L^1(\bR^d)$ and suppose that $x_n f(x_1,\ldots,x_n,\ldots,x_d)$
is an $L^1$ function as well. Then $\hat{f}(\xi_1,\ldots,\xi_n,\ldots,x_d)$ is continuously
differentiable with respect to $\xi_n$ and
\[
 \frac{\partial}{\partial \xi_k} \hat{f}(\xi)
 = \ms{F}(-2 \pi i x_n f(x))(\xi).
\]
More generally, if $P$ is a polynomial in $d$ variables, then
\[
 P(D)\hat{f}(\xi) = \ms{F}(P(-2 \pi i x)f(x))(\xi)
 \htwo \mbox{and} \htwo 
 \ms{F}(P(D)f)(\xi) = P(2 \pi i \xi)\hat{f}(\xi).
\]
\end{prop}

\begin{proof}
Let $h = (0,\ldots,h_n,\ldots,0)$ be a nonzero vector along the $n$th coordinate
axis. By Proposition \ref{symmetry-invariance-of-fourier-transform} (ii) and
the dominated convergence theorem, we have
\begin{eqnarray*}
 \lim_{h_n \to 0} \frac{\hat{f}(\xi+h) - \hat{f}(\xi)}{h_n}
 &=& \lim_{h_n \to 0}
 \ms{F} \left( \left( \frac{e^{-2 \pi i x \cdot h} - 1}{h_n} \right) f(x) \right)(\xi) \\
 &=& \ms{F} (- 2 \pi i x_n f(x)) (\xi),
\end{eqnarray*}
as was to be shown. The second assertion now follows from linearity of the
differential operator.
\end{proof}

To rid ourselves of technical issues that arise in dealing with non-smooth functions,
it will be convenient to work in a space of smooth functions that behaves well under
the key operations in harmonic analysis. Certainly, we would like the space to be
closed under the Fourier transform. Proposition \ref{differentiation-to-multiplication}
then implies that the space must be closed under multiplication by polynomials as well.
We are thus led to the following definition, named after Laurent Schwartz:

\begin{defin}\label{defin-schwartz-space}\index{Schwartz, Laurent!space}
The \emph{Schwartz space} $\ms{S}(\bR^d)$ consists of
functions $f \in \ms{C}^\infty(\bR^d)$ with the decay condition
\[
 \sup_{x \in \bR^d} |x^\alpha D^\beta f(x)| < \infty
\]
for each pair of multi-indices $\alpha$ and $\beta$.
\end{defin}

We remark that
\[
 \ms{C}_c^\infty(\bR^d) \subseteq \ms{S}(\bR^d) \subseteq L^p(\bR^d)
\]
for all $1 \leq p \leq \infty$. Since $\ms{C}^\infty_c(\bR^d)$ contains
\hyperref[smooth-approximations-to-the-identity]{mollifiers},\index{mollifiers}
$\ms{S}(\bR^d)$ is nonempty. In fact, $\ms{C}^\infty_c(\bR^d)$ is
dense in $L^p(\bR^d)$, whence so is $\ms{S}(\bR^d)$.\index{density!of the Schwartz space} 

An equivalent definition for a Schwartz
function is a function $f \in \ms{C}^\infty(\bR^d)$ that satisfies
the growth condition
\[
 \sup_{x \in \bR^d} \langle x \rangle^n |D^\alpha f(x)| < \infty
\]
for all natural numbers $n$ and multi-indices $\alpha$, where
\[
 \langle x \rangle = \sqrt{1+x^2}.
\]

As noted, the Schwartz space is closed under the action of
the Fourier transform. This basic fact is an immediate
corollary of Proposition \ref{differentiation-to-multiplication}.

\begin{prop}\label{fourier-transform-of-schwartz-is-schwartz}\index{Fourier transform!on the Schwartz space}
If $f \in \ms{S}(\bR^d)$, then $\hat{f} \in \ms{S}(\bR^d)$. \qed
\end{prop}

The Schwartz space behaves well under other important operations
in harmonic analysis as well. We shall take up this matter
in \S\ref{s-generalized_functions}.
 
We now turn to one of the fundamental questions in classical
Fourier analysis: given the Fourier transform of a
function, can we find the function itself? We begin with a useful
proposition that allows us to ``push the hat around'':

\begin{prop}[Multiplication formula]\label{multiplication-formula}\index{multiplication formula}
If $f,g \in L^1(\bR^d)$, then
\[
 \int \hat{f}(t) g(t) \, dt = \int f(t) \hat{g}(t) \, dt.
\]
\end{prop}

\begin{proof}
By \hyperref[fubini-tonelli]{Fubini's theorem},
\begin{eqnarray*}
 \int \hat{f}(t) g(t) \, dt
 &=& \int \left( \int f(x) e^{-2 \pi i t \cdot x} \, dx \right) g(t) \, dt \\
 &=& \int \left( \int g(t) e^{-2 \pi i t \cdot x} \, dt \right) f(x) \, dx \\
 &=& \int \hat{g}(x) f(x) \, dx \\
 &=& \int f(t) \hat{g}(t) \, dt,
\end{eqnarray*}
as desired.
\end{proof}

We shall also need the following computation:

\begin{prop}\label{gaussian}\index{Fourier transform!of the Gaussian}
For all $\ve > 0$, we have
\[
 \ms{F}\left(e^{-\ve \pi |x|^2}\right)(\xi) = \ve^{-d/2}e^{-\ve^{-1}\pi|\xi|^2}.
\]
This, in particular shows that the Fourier transform of the \emph{Gaussian}
\[
\Gamma(x) = e^{-\pi |x|^2}
\]
is the Gaussian itself.
\end{prop}

\begin{proof}
Recall that
\[
 \int_{-\infty}^\infty e^{-x^2} dx = \sqrt{\pi}.
\]
We first consider the one-dimensional case. In fact, we fix
positive real numbers $p$ and $q$ and compute a more general
integral:
\begin{eqnarray*}
 \int_{-\infty}^\infty e^{-p x^2} e^{- q i x} \, dx
 &=& \int_{-\infty}^\infty e^{-p(x^2 + \frac{qi}{p}x)} \, dx \\
 &=& \int_{-\infty}^\infty e^{-p(x + \frac{qi}{2p})^2 - \frac{q^2}{4p}} \, dx \\
 &=& e^{-q^2 / 4p} \int_{-\infty}^\infty e^{-p(x + \frac{qi}{2p})^2} \\
 &=& e^{-q^2 / 4p} \int_{-\infty}^\infty e^{-(\sqrt{p}x)^2} \\
 &=& \frac{e^{-q^2 / 4p}}{\sqrt{p}} \int_{-\infty}^\infty e^{-x^2} \\
 &=& e^{-q^2 / 4p} \sqrt{\frac{\pi}{p}}.
\end{eqnarray*}
Plugging in $p = \ve \pi$ and $q = 2 \pi \xi$, we have
\[
 \ms{F}\left(e^{-\ve \pi x^2}\right)(\xi) = \ve^{-1/2} e^{-\ve^{-1} \pi \xi^2}
\]
whenever $x,\xi \in \bR$.

It now suffices to observe that
\begin{eqnarray*}
 \ms{F}\left(e^{-\ve \pi |x|^2}\right)(\xi)
 &=& \int e^{-\ve \pi |x|^2} e^{-2 \pi i \xi \cdot x} \, dx \\
 &=& \prod_{n=1}^d \int e^{-\ve \pi x_n^2} e^{-2 \pi i \xi_nx_n} \, dx_n \\
 &=& \prod_{n=1}^d \ms{F}\left(e^{-\ve \pi x_n^2}\right)(\xi_n) \\
 &=& \prod_{n=1}^d \ve^{-1/2} e^{-\ve^{-1} \pi \xi_n^2} \\
 &=& \ve^{-d/2}e^{-\ve^{-1}\pi|\xi|^2},
\end{eqnarray*}
as desired.
\end{proof}

We now present a preliminary solution to the inversion problem, which is
sufficient for the present thesis. A more detailed discussion can be
found in \S\S\ref{fr-fourier-inversion-problem}.

\begin{theorem}[Fourier inversion theorem]\label{fourier-inversion-formula}\index{Fourier inversion!on L1@on $L^1$|(}\index{Fourier transform!inverse of|see{Fourier inversion}}
If $f \in L^1(\bR^d)$ and $\hat{f} \in L^1(\bR^d)$, then
\[
 f(x) = \int \hat{f}(\xi) e^{2 \pi i \xi \cdot x} \, d\xi
\]
for almost every $x \in \bR^d$.
\end{theorem}

By Proposition \ref{fourier-transform-of-schwartz-is-schwartz},
Schwartz functions satisfy the hypothesis of the above theorem.
In general, Proposition \ref{basic-properties-of-ft} indicates
that $f$ must necessarily be a $C_0$ map, although this is not
a sufficient condition.

\begin{proof}[Proof of Theorem \ref{fourier-inversion-formula}]
We consider the following modification of the inversion theorem:
\[
 I_\ve(x) = \int \hat{f}(\xi) e^{- \pi \ve^2 |\xi|^2} e^{2 \pi i \xi \cdot x} \, d\xi.
\]
Since $\hat{f} \in L^1(\bR^d)$, the dominated convergence theorem implies that
\[
 \lim_{\ve \to 0} I_\ve(x) = \int \hat{f}(\xi) e^{2 \pi i \xi \cdot x} \, d\xi.
\]
We now set $g_\ve(\xi) =  e^{- \pi \ve^2 |\xi|^2} e^{2 \pi i \tau \cdot x}$
for a fixed $\tau$. By the
\hyperref[multiplication-formula]{multiplication formula}, we have
\[
 I_\ve(x) = \int \hat{f}(t) g_\ve(t) \, dt = \int f(t) \widehat{g_\ve}(t) \, dt.
\]
Proposition \ref{symmetry-invariance-of-fourier-transform}(b) and
Proposition \ref{gaussian} imply that
\begin{eqnarray*}
 \widehat{g_\ve}(t)
 &=& \ms{F}\left(e^{-\pi\ve^2|\xi - \tau|^2}\right)(t) \\
 &=& \ve^{-d} e^{-\ve^{-2} |t-\tau|^2} \\
 &=& \ve^{-d} \Gamma( \ve^{-1} (\tau-t)).
\end{eqnarray*}
Setting $\Gamma_\ve(s) = \ve^{-d} \Gamma (\ve^{-1} s)$, we see that
\[
 I_\ve(s) = \int f(t) \Gamma_{\ve}(s-t) \, dt = (f * \Gamma_{\ve})(s).
\]
$(\Gamma_{\ve})_{\ve>0}$ is an \hyperref[approximations-to-the-identity]{approximation
to the identity}, and so
\[
 \lim_{\ve \to 0} \|I_\ve - f\|_1 = \lim_{\ve \to 0} \|f*\Gamma_\ve - f\|_1 = 0.
\]
It follows that $(I_\ve)_{\ve>0}$ converges to
$\int \hat{f}(\xi) e^{2 \pi i \xi \cdot x} \, d\xi$ pointwise and to $f$ in $L^1$,
whence
\[
 f(x) = \int \hat{f}(\xi) e^{2 \pi i \xi \cdot x} \, d\xi,
\]
as was to be shown.
\end{proof}
\index{Fourier inversion!on L1@on $L^1$|)}

We often write $f^\vee$ to denote the inverse Fourier transform of $f$. Note that
\[
 f^\vee(x) = \hat{f}(-x).
\]
The inversion formula, combined with Proposition \ref{fourier-transform-of-schwartz-is-schwartz},
implies that the Fourier transform operator $\ms{F}$ maps $\ms{S}(\bR^d)$
onto itself, with the inverse
\[
 \ms{F}^{-1}(f)(x) = \ms{F}(\hat{f})(-x).
\]
Since $\ms{F}$ is also linear, we see that $\ms{F}$ is a linear automorphism
of $\ms{S}(\bR^d)$. We shall see that $\ms{F}$ also preserves the natural
topological structure on $\ms{S}(\bR^d)$, hence turning $\ms{F}$ into a
\emph{Fr\'{e}chet-space automorphism}.\index{Fourier transform!as a Fr\'{e}chet-space \\ automorphism} Fr\'{e}chet spaces\index{Fr\'{e}chet space} are discussed in
\S\ref{s-elements_of_functional_analysis}; topological properties of the Schwartz
space are discussed in \S\ref{s-generalized_functions}.\index{Fourier transform!on the Schwartz space}

\subsection{The \texorpdfstring{$L^2$}{L2} Theory}

\index{Fourier transform!on L2@on $L^2$|(}
We now recall that $L^2(\bR^d)$ is a Hilbert space, with the inner product
\[
 \langle f,g \rangle_2 = \int f \bar{g}.
\]
Since $\ms{S}(\bR^d)$ is a linear subspace of $L^2(\bR^d)$, it inherits
the inner product as well. As it turns out, the Fourier transform
operator preserves the inner product:

\begin{lemma}[Plancherel, Schwartz-space version]\label{schwartz-plancherel}\index{Fourier transform!on the Schwartz space}
$\ms{F}:\ms{S}(\bR^d) \to \ms{S}(\bR^d)$ is a unitary operator.\index{unitary!operator}
In other words, if $\varphi,\phi \in \ms{S}(\bR^d)$, then
\[
 \langle \hat{f}, \hat{g}\rangle_2 = \langle f,g\rangle_2.
\]
In particular, $\|\hat{\varphi}\|_2 = \|\varphi\|_2$.
\end{lemma}

\begin{proof}
By the \hyperref[fourier-inversion-formula]{Fourier inversion formula}
and Proposition \ref{symmetry-invariance-of-fourier-transform}(c),
\[
 \int \varphi(t) \bar{\phi}(t) \, dt
 = \int \widehat{\hat{\varphi}}(-t) \bar{\phi}(t) \, dt
 = \int \widehat{\hat{\varphi}}(-t)\overline{\phi}(-t) \, dt,
\]
and so
\[
 \int \hat{\varphi}\overline{\hat{\phi}}
 = \int \widehat{\hat{\varphi}}\tilde{\varphi}.
\]
It now follow from the \hyperref[multiplication-formula]{multiplication formula}
that
\[
 \langle \varphi,\phi\rangle_2
 = \int \varphi \bar{\phi}
 = \int \hat{\varphi} \widehat{\tilde{\phi}}
 = \int \hat{\varphi} \overline{\hat{\phi}}
 = \langle \hat{\varphi}, \hat{\phi} \rangle_2,
\]
as desired.
\end{proof}

Could we do better? By Theorem \ref{norm-preserving-extension},
the Fourier transform operator $\ms{F}$, defined on the dense subspace
$\ms{S}(\bR^d)$ of $L^2(\bR^d)$, admits a unique norm preserving extension
on $L^2(\bR^d)$. We call this extension the \emph{$L^2$ Fourier transform}
and denote it by $\ms{F}$ as well. The norm-preserving property implies that
the $L^2$ Fourier transform is an isometry into itself. Since an isometric
operator on a Hilbert space is also a unitary operator,\index{unitary!operator} the Fourier
transform preserves the $L^2$-inner product as well. We summarize
the foregoing discussion in the following theorem:

\begin{theorem}[Plancherel]\label{plancherel}\index{Plancherel's theorem}\index{Fourier transform!as a unitary automorphism on L2@as a unitary automorphism \\ on $L^2$|see{Plancherel's theorem}}
The $L^2$ Fourier transform $\ms{F}$ is a unitary\index{unitary!automorphism}
automorphism on $L^2(\bR^d)$. In other words, the $L^2$ Fourier transform
is linear, maps $L^2(\bR^d)$ onto itself, and preserves the inner-product
structure of $L^2(\bR^d)$. Furthermore, the $L^2$ Fourier transform on
$L^1(\bR^d) \cap L^2(\bR^d)$ agrees with the $L^1$ transform.
\end{theorem}

\begin{proof}
The linearity of $\ms{F}:L^2(\bR^d) \to L^2(\bR^d)$ has already been
established by Theorem \ref{norm-preserving-extension}. Since
$\ms{S}$ is dense in both $L^1(\bR^d)$ and $L^2(\bR^d)$, the uniqueness
clause in Proposition \ref{norm-preserving-extension} also guarantees
that the $L^1$ and $L^2$ Fourier transforms must agree on $L^1 \cap L^2$.

We claim that $\ms{F}(L^2(\bR^d))$ is closed and dense. Indeed, if
$(f_n)_{n=1}^\infty$ is a sequence of functions in $\ms{F}(L^2(\bR^d))$
that converges to $f \in L^2(\bR^d)$, then we can find a sequence
$(g_n)_{n=1}^\infty$ of functions in $L^2(\bR^d)$ such that $\widehat{g_n}
= f_n$ for each integer $n$. Since $\ms{F}$ is an isometry, $(g_n)_{n=1}^\infty$
is Cauchy in $L^2(\bR^d)$, hence converges to $g \in L^2(\bR^d)$. Of course,
$\hat{g} = f$, and the range is closed. To establish the density of
$\ms{F}(L^2(\bR^d))$ in $L^2(\bR^d)$, it suffices to observe that
\[
 \ms{S}(\bR^d) =\ms{F}(\ms{S}(\bR^d)) \subseteq \ms{F}(L^2(\bR^d)).
\]
This proves the claim, and it now follows that $\ms{F}(L^2(\bR^d)) = L^2(\bR^d)$.
For each $f,g \in L^2(\bR^d)$, we invoke the polarization identity of
inner-product spaces to see that
\begin{eqnarray*}
 \langle f,g\rangle_2
 &=& \frac{1}{4} \left( \|f+g\|_2 - \|f-g\|_2 +i\|f+ig\|_2 - i\|f - ig\|_2 \right) \\
 &=& \frac{1}{4} \left( \|\widehat{f+g}\|_2 - \|\widehat{f-g}\|_2
 +i\|\widehat{f+ig}\|_2 - i\|\widehat{f - ig}\|_2 \right) \\
 &=& \frac{1}{4} \left( \|\hat{f}+\hat{g}\|_2 - \|\hat{f}-\hat{g}\|_2
 +i\|\hat{f}+i\hat{g}\|_2 - i\|\hat{f} - i\hat{g}\|_2 \right) \\
 &=& \langle \hat{f},\hat{g} \rangle_2,
\end{eqnarray*}
whence $\ms{F}$ is a unitary automorphism on $L^2(\bR^d)$. 
\end{proof}
\index{Fourier transform!on L2@on $L^2$|)}

\subsection{The \texorpdfstring{$L^p$}{Lp} Theory}

\index{Fourier transform!on L1 + L2@on $L^1+L^2$|(}
Thus far, we have seen that the Fourier transform can be defined on
$L^1(\bR^d)$ and $L^2(\bR^d)$. In the final subsection of this section,
we shall extend the Fourier transform operator onto other $L^p$ spaces.
To this end, we establish our first interpolation result:

\begin{prop}\label{intermediate-lp-spaces}\index{interpolation space}
Let $(X,\mf{M},\mu)$ be a measure space.
If $1 \leq p < r < q \leq \infty$, then $L^r(X,\mu) \subseteq (L^p+L^q)(X,\mu)$.
\end{prop}

\begin{proof}
Let $f \in L^r(X,\mu)$, set $E = \{x : |f(x)| > 1\}$, and define
$g = f \chi_E$ and $h = \chi_{X \smallsetminus E}$. Note that
$|g|^p = |f|^p \chi_E \leq |f|^r \chi_E$, and so $g \in L^p(X,\mu)$.
If $q < \infty$, then $|h|^p = |f|^p\chi_{X \smallsetminus E}
\leq |f|^r \chi_{X \smallsetminus E}$, and so $h \in L^q(X,\mu)$.
If $q = \infty$, then $\|h\|_\infty \leq 1$, and so $h \in L^q(X,\mu)$.
It thus follows that
\[
 f = g+h
\]
is in $(L^p+L^q)(X,\mu)$.
\end{proof}

In view of the above proposition, we extend the domain of Fourier transform to
all $L^p(\bR^d)$ for $1 < p < 2$ by defining the \emph{$L^1+L^2$ Fourier transform}.
Indeed, we set
\[
 \hat{f} = \hat{g} + \hat{h}
\]
for each $f \in L^p(\bR^d)$, where $g \in L^1(\bR^d)$ and $h \in L^2(\bR^d)$.
The decomposition is not unique, of course, but the $L^1 + L^2$ Fourier transform
is nevertheless well-defined. Indeed, $g_1 + h_1 = g_2 + h_2$ implies
that $g_1 - g_2 = h_2 - h_1$ is in $L^1(\bR^d) \cap L^2(\bR^d)$. Since
the $L^1$ and $L^2$ Fourier transforms coincide on $L^1 \cap L^2$, it follows
that $\widehat{g_1} - \widehat{g_2} = \widehat{h_2} - \widehat{h_1}$, or
\[
 \widehat{g_1} + \widehat{h_1} = \widehat{g_2} + \widehat{h_2}.
\]
We can now restrict the $L^1+L^2$ Fourier transform operator onto each
$L^p(\bR^d)$ to define the $L^p$ Fourier transform.

Alternatively, we could use the density of $\ms{S}(\bR^d)$ to extend the
Fourier transform on $\ms{S}(\bR^d)$ onto $L^p(\bR^d)$, as
we shall show below that the Fourier transform extends to a bounded
operator. Since the $L^1+L^2$
definition of the $L^p$ Fourier transform must agree with the usual Fourier
transform on $\ms{S}(\bR^d)$, Theorem \ref{norm-preserving-extension}
implies that these two definitions coincide.

Implicit in the above argument is the second conclusion of
\hyperref[plancherel]{Plancherel's theorem}, which asserts that the $L^1$
Fourier transform and the $L^2$ Fourier transform agree on $L^1 \cap L^2$.
In fact, it is possible to carry out the argument directly on the intersection,
as the next proposition shows.

\begin{prop}\label{norm-estimates-of-intermediate-lp-spaces}
Let $(X,\mf{M},\mu)$ be a measure space.
If $1 \leq p < r < q \leq \infty$, then $L^p(X,\mu) \cap L^q(X,\mu)
\subseteq L^r(X,\mu)$ and
\begin{equation}\label{norm-estimate}
 \|f\|_r \leq \|f\|_p^{1-\theta}\|f\|_q^\theta
\end{equation}
for all $f \in L^p(X,\mu) \cap L^q(X,\mu)$, where $\theta$ is
the unique real number in $(0,1)$ satisfying the identity
\begin{equation}\label{conjugate-exponent-interpolation}
 r^{-1} = (1-\theta)p^{-1} + \theta q^{-1}.
\end{equation}
\end{prop}

\begin{proof}
If $q = \infty$, then $|f|^r = |f|^p|f|^{r-p} \leq |f|^p\|f\|_\infty^{r-p}$.
With $\theta = 1-p/r$, we have
\[
 \|f\|_r \leq \|f\|^{p/r}_p\|f\|^{1-p/r}_\infty
 = \|f\|^{1-\theta}_p \|f\|^\theta_\infty.
\]
If $q < \infty$, then we observe that
\[
 1 = \frac{1}{p/(1-\theta)r} + \frac{1}{q/\theta r},
\]
whence we may apply H\'{o}lder's inequality:
\begin{eqnarray*}
 \|f\|_r^r
 &=& \int |f|^r \, d\mu \\
 &=& \int |f|^{(1-\theta) r} |f|^{\theta r} \, d\mu \\
 &\leq& \|f\|_{p/(1-\theta)r}^{(1-\theta) r} \|f\|_{q/\theta r}^{\theta r} \\
 &=& \left( \int |f|^p \, d\mu \right)^{(1-\theta)r / p}
 \left( \int |f|^q \, d\mu \right)^{\theta r / q} \\
 &=& \|f\|_p^{(1-\theta)r} \|f\|_q^{\theta r}.
\end{eqnarray*}
Taking the $r$th roots, we obtain the desired inequality.
\end{proof}

The above proposition also suggests what we should expect
the codomain of the $L^p$ Fourier transform to be.
The $L^1$ Fourier transform maps into $L^\infty$, and the $L^2$ Fourier
transform maps into $L^2$. For any given $1 < p < 2$, then
we might expect the $L^p$ Fourier transform to map into $L^{p'}$,
as the constant $\theta$ that satisfies the identity (\ref{conjugate-exponent-interpolation})
\[
 p^{-1} = \frac{1-\theta}{1} + \frac{\theta}{2},
\]
with 1 and 2 plugged in for the $L^1$ and $L^2$ Fourier transforms, yields
\[
 (p')^{-1} = \frac{1-\theta}{\infty} + \frac{\theta}{1}
\]
when we plug in 2 and $\infty$, as per the orders of the the target spaces
for the $L^1$ and $L^2$ Fourier transforms. Following this line of reasoning,
we could also conjecture that the norm estimate (\ref{norm-estimate}) holds
for operators as well, which, in this case, implies that
\begin{equation}\label{hausdorff-young-preliminary}\index{inequality!Hausdorff-Young}
 \|\ms{F}\|_{L^p \to L^{p'}}
 \leq \|\ms{F}\|_{L^1 \to L^{\infty}}^{1-\theta}
 \|\ms{F}\|_{L^2 \to L^2}^\theta
 \leq 1^{1-\theta} 1^\theta = 1
\end{equation}
for all $1 < p < 2$. This, in fact, turns out to be true.

Proposition \ref{norm-estimates-of-intermediate-lp-spaces} and
the conjectured inequality (\ref{hausdorff-young-preliminary}) are
special cases of our first main theorem of the thesis, the
\hyperref[riesz-thorin]{Riesz-Thorin interpolation theorem}.\index{Riesz-Thorin interpolation \\ theorem}
We shall study the theorem and its consequences in the next section.
For now, we shall apply our newly established $L^p$ Fourier transform
to convolutions and study their behaviors. First, we establish
a preliminary result:

\begin{prop}[Convolution theorem, $L^1$ version]\label{convolution-theorem-preliminary}\index{convolution theorem!on L1@on $L^1$}
If $f,g \in L^1(\bR^d)$, then $\widehat{f*g} = \hat{f}\hat{g}$.
\end{prop}

\begin{proof}
\hyperref[convolution]{Young's inequality} shows that $f*g \in L^1(\bR^d)$, so
we can make sense of the Fourier transform of $f*g$. The proposition is
now an easy consequence of \hyperref[fubini-tonelli]{Fubini's theorem}
and Proposition  \ref{symmetry-invariance-of-fourier-transform}(a):
\begin{eqnarray*}
 \widehat{f*g}(\xi)
 &=& \int \left( \int f(x-y)g(y) \, dy \right) e^{-2 \pi i \xi \cdot x} \, dx \\
 &=& \int \left( \int f(x-y) e^{-2 \pi i \xi \cdot x} \, dx \right) g(y) \, dy \\
 &=& \int \hat{f}(\xi) g(y) e^{-2 \pi i y \cdot \xi} \, dy \\
 &=& \hat{f}(\xi)\hat{g}(\xi).
\end{eqnarray*}
\end{proof}

Since the Fourier transform is linear, the result extends easily to the $L^p$ case.
The main idea is to consider the convolution operator\index{convolution!operator} as a bounded operator.
We shall have more to say about this in \S\ref{s-generalized_functions}.
See also Theorem \ref{young} in the next section for another example of a
convolution operator.

\begin{theorem}[Convolution theorem, $L^p$ version]\label{convolution-theorem}\index{convolution theorem!on Lp, for 1 p 2@$L^p$, for $1 \leq p \leq 2$}
Let $1 \leq p \leq 2$. If $f \in L^p(\bR^d)$ and $g \in L^1(\bR^d)$, then
$\widehat{f*g}=\hat{f}\hat{g}$.
\end{theorem}

\begin{proof}
Fix $g \in L^1(\bR^d)$ and $1 \leq p \leq 2$. For each $f \in L^p(\bR^d)$,
\hyperref[convolution]{Young's inequality} shows that $f*g \in L^p(\bR^d)$,
so we can apply the $L^p$ Fourier transform to $f*g$. Consider the ``convolution
operator'' $T:L^p(\bR^d) \to L^p(\bR^d)$ defined by
\[
 Tf = f*g.
\]
By \hyperref[convolution]{Young's inequality}, $T$ is bounded with operator norm
at most $\|g\|_1$. Therefore, the operator $T_1 = \ms{F}T:L^p(\bR^d) \to L^{p'}(\bR^d)$
is bounded as well.

Proposition \ref{convolution-theorem-preliminary}
now implies that
\begin{equation}\label{yay}
 T_1 f = \widehat{f*g} = \widehat{f}\widehat{g}
\end{equation}
for all $f \in \ms{S}(\bR^d)$. $\ms{S}(\bR^d)$ is dense in $L^p(\bR^d)$,
and $\ms{F}$ is continuous, hence the uniqueness statemente in
Theorem \ref{norm-preserving-extension} guarantees that (\ref{yay}) holds
for all $f \in L^p(\bR^d)$.
\end{proof}
\index{Fourier transform!on L1 + L2@on $L^1+L^2$|)}

What can be say about the $L^p$ Fourier transform for $p > 2$?\index{Fourier transform!on Lp, for p>2@on $L^p$, for $p>2$}
In order to extend $L^1$ Fourier transform on $\ms{S}(\bR^d)$
to $L^p(\bR^d)$ via Theorem \ref{norm-preserving-extension},
we must have a Banach space as a codomian. As it stands now, however,
it is not at all clear what the target space should be.
In fact, the Fourier transform of an $L^p$ function may
not even be a function but a tempered distribution,
which we shall define in \S\ref{s-generalized_functions}.
\index{Fourier transform|)}

\section{Interpolation on \texorpdfstring{$L^p$}{Lp} Spaces}\label{s-interpolation_on_lp_spaces}

We now come to the first major theorem of this thesis. We have seen
in the last section that the Fourier transform, defined as a
bounded operator on $(L^1 + L^2)(\bR^d)$ into $(L^\infty + L^2)(\bR^d)$,
can be ``interpolated'' to yield a bounded operator on $L^p(\bR^d)$ into
$L^{p'}(\bR^d)$. We shall see that a general theorem of the kind
holds. Specifically, if an operator is ``defined'' on both $L^{p_0}$ and
$L^{p_1}$ and maps boundedly into $L^{q_0}$ and $L^{q_1}$, respectively,
then we shall prove that the operator can be interpolated to yield
a bounded operator on $L^{p_\theta}$ into $L^{q_\theta}$, where
$p_\theta$ and $q_\theta$ are appropriately defined intermediate exponents.

\subsection{The Riesz-Thorin Interpolation Theorem}

\index{Riesz-Thorin interpolation \\ theorem|(}\index{Riesz, Marcel!convexity theorem|see{\\ Riesz-Thorin interpolation \\ theorem}}
To state the theorem, we first need to make sense of an operator defined
on two separate domains. Taking a cue from Theorem \ref{norm-preserving-extension},
we define our operator on a dense subset of each Lebesgue space in question.

\begin{defin}\label{strong-type}\index{operator!of type (p,q)@of type $(p,q)$}\index{operator!of strong type|see{type $(p,q)$}}\index{operator!norm|see{norm}}
Let $(X,\mu)$ and $(Y,\nu)$ be $\sigma$-finite measure spaces. We fix
a vector space $D$ of $\mu$-measurable complex-valued functions on $X$
that contains the simple functions with finite-measure support.
We also assume that $D$ is closed under \emph{truncation},\index{truncation} viz.,
if $f \in D$, then the function
\[
 g_{r_1,r_2}(x) =
 \begin{cases}
  f(x) & \mbox{ if } r_1 < |f(x)| \leq r_2; \\
  0 & \mbox{ otherwise;}
 \end{cases}
\]
defined for each $0<r_1 \leq r_2$, is also in $D$. Given $1 \leq p,q \leq
\infty$, we say that a linear operator $T$ on $D$ into the vector space
$\mc{M}(Y,\mu)$ of $\nu$-measurable complex-valued functions on $Y$ is
\emph{of type $(p,q)$} if there exists a constant $k>0$ such that
\[
 \|Tf\|_q \leq k\|f\|_p
\]
for all $f \in D \cap L^p(X,\mu)$. The infimum of all such $k$ is referred
to as the \emph{$(p,q)$ norm} of $T$ and is denoted by $\|T\|_{L^p \to L^q}$.
\end{defin}

Let $T$ be an operator of type $(p_0,q_0)$. We first remark that we can restrict
the codomain of $T:D \cap L^{p_0}(X,\mu) \to \mc{M}(Y,\nu)$ to $L^{q_0}(Y,\mu)$,
which is a Banach space. Since Proposition \ref{approximation-by-simple-functions}
implies that $D \cap L^{p_0}(X,\mu)$ is dense in $L^{p_0}(X,\mu)$,  we may
invoke Theorem \ref{norm-preserving-extension} to construct a unique
norm-preserving extension $T:L^{p_0}(X,\mu) \to L^{q_0}(Y,\nu)$. If, in
addition, $T$ is of type $(p_1,q_1)$, then a similar argument yields the
extension $T:L^{p_1}(X,\mu) \to L^{q_1}(Y,\nu)$.

We now state the interpolation theorem, due to M. Riesz and O. Thorin:

\begin{theorem}[Riesz-Thorin interpolation]\label{riesz-thorin}
Let $1 \leq p_0,p_1,q_0,q_1 \leq \infty$. If $T$ is a linear operator
simultaneously of type $(p_0,q_0)$ and of type $(p_1,q_1)$, then
$T$ is of type $(p_\theta,q_\theta)$ with the norm estimate
\[
 \|T\|_{L^{p_\theta} \to L^{q_\theta}}
 \leq \|T\|^{1-\theta}_{L^{p_0} \to L^{q_0}} \|T\|^\theta_{L^{p_1} \to L^{q_1}}
\]
for each $\theta \in [0,1]$, where
\begin{eqnarray*}
 p_\theta^{-1} &=& (1-\theta)p_0^{-1} + \theta p_1^{-1};\\
 q_\theta^{-1} &=& (1-\theta)q_0^{-1} + \theta q_1^{-1}.
\end{eqnarray*}
\end{theorem}

We remark that the interpolation result can be described pictorially
in a so-called \emph{Riesz diagram} of $T$, which is the collection
of all points $(1/p,1/q)$ in the unit square such that $T$ is
of type $(p,q)$. In this context, the above theorem implies that the Riesz
diagram of a linear operator is a convex set: for any two points in
the Riesz diagram, the Riesz-Thorin interpolation theorem guarantees
that the line connecting them is also in the Riesz diagram.

\begin{figure}[!htb]
\begin{center}
\includegraphics[scale=1.3]{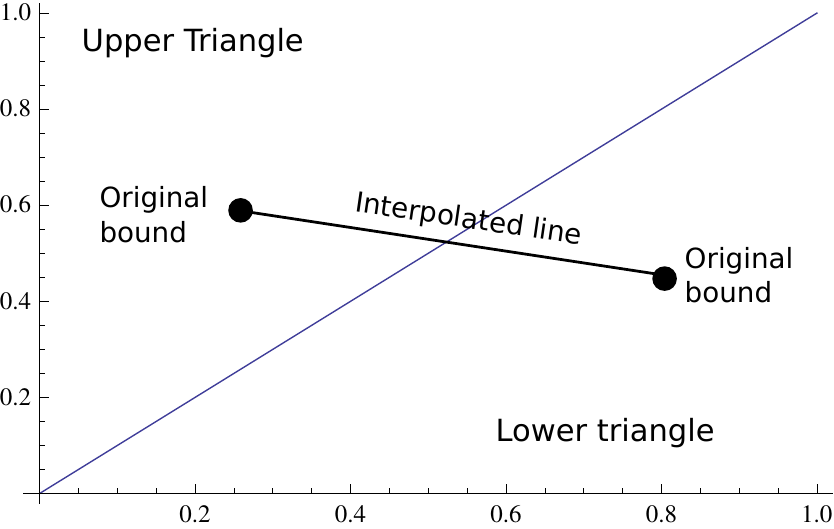}
\caption{A Riesz diagram}\label{riesz-diagram}\index{Riesz, Marcel!diagram}
\end{center}
\end{figure}

The interpolation theorem was originally stated by M. Riesz
in \cite{Marcel_Riesz:J1927}.\index{Riesz, Marcel} In the paper, the theorem was
stated only for the lower triangle in the Riesz diagram, i.e.,
for $p_n \leq q_n$. Since the proof of the theorem made use of
convexity results for bilinear forms, the interpolation theorem is
often referred to as the \emph{Riesz convexity theorem}. The extension
of the theorem to the entire square is due to O. Thorin, originally
published in his 1938 paper and explicated further in \cite{Olof_Thorin:T1948}.\index{Thorin, Olof}
The modern proof of the theorem was first provided by J. Tamarkin\index{Tamarkin, Jacob D.} and A. Zygmund\index{Zygmund, Antoni}
in \cite{Tamarkin_Zygmund:J1944}
and makes use of an extension of the maximum modulus principle\index{maximum modulus principle} from
complex analysis:

\begin{theorem}[Hadamard's three-lines theorem]\label{hadamard}\index{Hadamard's three-lines theorem}
Let $\Phi$ be a holomorphic function in the interior of the closed strip
\[
 S = \{z \in \bC : 0 \leq \Re z \leq 1\}
\]
and bounded and continuous on $S$. If $|\Phi(z)| \leq k_0$ on the line
$\Im z = 0$ and $|\Phi(z)| \leq k_1$ on the line $\Im z = 1$, then,
for each $0 \leq \theta \leq 1$, we have the inequality
\[
 |\Phi(z)| \leq k_0^{1-\theta} k_1^\theta
\]
on the line $\Im z = \theta$.
\end{theorem}

\begin{proof}
We define holomorphic functions
\[
 \Psi(z) = \frac{\Phi(z)}{k_0^{1-z}k_1^z}
 \htwo \mbox{and} \htwo
 \Psi_n(z) = \Psi(z)e^{(z^2-1)/n},
\]
so that $\Psi_n \to \Psi$ as $n \to \infty$. We wish to show that $|\Psi(z)| \leq 1$
on $S$.

To this end, we first note that $|\Psi(z)| \leq 1$ on the lines $\Im z = 0$ and
$\Im z = 1$. Moreover, $\Phi$ is bounded above on $S$, and $k_0^{1-z} k_1^z$ is
bounded below on $S$, hence we have the bound $|\Psi(z)| \leq M$ on $S$.
Noting the inequality
\[
 |\Psi_n(x+iy)| \leq Me^{-y^2 / n} e^{(x^2-1)/n} \leq Me^{-y^2/n},
\]
we see that $\Psi_n(x+iy)$ converges uniformly to 0 as $|y| \to \infty$. 
We can therefore find $y_n > 0$ such that $|\Psi_n(x+iy)| \leq 1$ for
all $|y| \geq y_n$ and $0 \leq x \leq 1$, whence by the maximum modulus
principle\index{maximum modulus principle} we have $|\Psi_n(z)| \leq 1$ on the rectangle
\[
 \{z = x+iy : 0 \leq x \leq 1 \mbox{ and } -y_n \leq y \leq y_n\}.
\]
It follows that $|\Psi_n(z)| \leq 1$ on $S$ for each $n \in \bN$, and
sending $n \to \infty$ yields the desired result.
\end{proof}

We now come to the proof of the interpolation theorem.
In the words of C. Fefferman\index{Fefferman, Charles|(} in \cite{Fefferman_Fefferman_Wainger:B1995},
the proof essentially amounts to providing an estimate for the integral
\[
 \int_Y (T f)g \, d\nu,
\]
where $f \in L^{p_\theta}(X,\mu)$ and $g \in L^{q_\theta'}(Y,\nu)$.
In light of the \hyperref[Lp-riesz-representation]{Riesz representation theorem},\index{representation theorem!F. Riesz, Lp-space version@F. Riesz, $L^p$-space version}
the supremum of the above expression equals $\|Tf\|_{q_\theta}$.
To give an upper bound for the above integral, we begin by finding nonnegative
real numbers $F$ and $G$ and real numbers
$\phi$ and $\psi$ such that $f = Fe^{i\phi}$ and $g = Ge^{i\psi}$,
and consider the entire function
\[
 \Phi(z) = \int_Y (Tf_z)g_z \, d\nu,
\]
where $f_z = F^{az+b} e^{i\phi}$ and $g_z = G^{cz+d} e^{i\psi}$ for
suitably chosen real numbers $a,b,c,d$. We pick $a,b,c,d$ such that
\[
 |f_z|^{p_0} = |f|^{p_\theta} \htwo \mbox{and} \htwo
 |g_z|^{q_0'} = |g|^{q_\theta'}
\]
on the line $\Re z = 0$,
\[
 |f_z|^{p_1} = |f|^{p_\theta} \htwo \mbox{and} \htwo
 |g_z|^{q_1'} = |g|^{q_\theta'}
\]
on the line $\Re z = 1$, and
\[
 f_z = f \htwo \mbox{ and } \htwo g_z = g
\]
at the point $z = \theta$.

The first assumption yields an upper bound
$k_0'$ for both $\|f_z\|_{p_0}$ and $\|g_z\|_{q_0'}$ on the line $\Re z = 1$,
and so the $L^{p_0} \to L^{q_0}$ norm inequality of $T$ yields the estimate
\[
 |\Phi(z)| \leq k_0
\]
on the line $\Re z = 0$ for some constant $k_0$. Similarly, the second
assumption furnishes a constant $k_1$ such that
\[
 |\Phi(z)| \leq k_1
\]
on the line $\Re z = 1$. By \hyperref[hadamard]{Hadamard's three-lines theorem},
we now have the bound $|\Phi(z)| \leq k_0^{1-\theta}k_1^\theta$, and the third
assumption implies that
\[
 \left| \int_Y (Tf) g \, d\nu \right| \leq k_0^{1-\theta}k_1^\theta.
\]
\index{Fefferman, Charles|)}

We now present the proof in full detail.

\begin{proof}[Proof of Theorem \ref{riesz-thorin}]
Fix $0 \leq \theta \leq 1$. For notational convenience, we let
\[
 \alpha_0 = \frac{1}{p_0}, \hone \alpha_1 = \frac{1}{p_1},
 \hone \alpha = \frac{1}{p_\theta}, \hone
 \beta_0 = \frac{1}{q_0}, \hone \beta_1 = \frac{1}{q_1}, \hone
 \beta = \frac{1}{q_\theta}.
\]
With this notation, we set
\[
 \alpha(z) = (1-z)\alpha_0 + z\alpha_1 \htwo \mbox{and} \htwo
 \beta(z) = (1-z)\beta_0 + z\beta_1,
\]
so that
\[
 \alpha(0) = \alpha_0, \hone \alpha(1) = \alpha_1,
 \hone \alpha(\theta) = \alpha, \hone
 \beta(0) = \beta_0, \hone \beta(1) = \beta_1, \hone
 \beta(\theta) = \beta.
\]

We shall first prove the theorem for simple functions with finite-measure
support, which evidently belong to $D \cap L^{p_\theta}(X,\mu)$.
By the \hyperref[Lp-riesz-representation]{Riesz representation theorem}, we have
\[
 \|Tf\|_{q_\theta} = \sup \left| \int_Y (Tf) g \, d\nu \right|
\]
for each simple function $f$, where the supremum is taken over all simple
functions $g$ of $L^{q_\theta'}$-norm at most one. Therefore, it suffices to
show that
\[
 \left| \int_Y (Tf)g \, d\nu \right| \leq k_0^{1-\theta}k_1^{\theta}\|f\|_{p_\theta}
\]
for each such $g$. If $\|f\|_{p_\theta} = 0$, then there is nothing to prove, and so
we can assume by renormalization of $f$ that $\|f\|_{p_\theta} = 1$. We thus
set out to establish
\[
 \left| \int_Y (Tf)g \, d\nu \right| \leq k_0^{1-\theta}k_1^{\theta}
\]
for all simple functions $g$ with $\|g\|_{q_\theta'}=1$. 

We now suppose that
\[
 f = \sum_{j=1}^m a_j \chi_{E_j} \htwo\mbox{ and }\htwo
 g = \sum_{k=1}^n b_k \chi_{F_k}
\]
are two simple functions satisfying the above conditions. We also assume 
without loss of generality that $p_\theta < \infty$ and $q_\theta > 1$, so that
$\alpha > 0$ and $\beta < 1$. Write $a_j = |a_j| e^{i\theta_j}$
and $b_k = |b_k| e^{i\varphi_k}$ and set 
\[
 f_z = \sum_{j=1}^m |a_j|^{\alpha(z)/\alpha}e^{i\theta_j} \chi_{E_j}
 \htwo\mbox{and}\htwo
 g_z = \sum_{k=1}^m |b_k|^{(1-\beta(z))/(1-\beta)}e^{i\varphi_k}\chi_{F_k}
\]
for each $z \in \bC$. Then
\[
 \Phi(z) = \int_Y (Tf_z) g_z \, d\nu
\]
is an entire function such that
\[
 \Phi(\theta) = \int_Y (Tf) g \, d\nu.
\]
We note, in particular, that $\Phi$ is holomorphic in the interior of $S$
and is continuous on $S$. Since $T$ is linear, we see that
\[
 \Phi(z) = \sum_{j=1}^m\sum_{k=1}^n
 |a_j|^{\alpha(z)/\alpha} |b_k|^{(1-\beta(z))/(1-\beta)}
 \left(e^{i(\theta_j + \varphi_k)} \int_Y (T\chi_{E_j}) \chi_{F_k} \right),
\]
whence $F$ is bounded on $S$. 

We now furnish a bound for $\Phi$ on the lines $\Re z = 0$ and $\Re z = 1$.
Note first that $|\Phi(z)| \leq \|Tf_z\|_{q_0} \|g_z\|_{q_\theta'}$ by H\"{o}lder's
inequality.
Observing the identities $\alpha(z) = \alpha_0 + z(\alpha_1 -\alpha_0)$ and
$1-\beta(z) = (1-\beta_0) - z(\beta_1 - \beta_0)$ on the line $\Re z =0$,
we see that
\begin{eqnarray*}
 |f_z|^{p_0}
 &=& |e^{i \arg f} |f|^{z (\alpha_1-\alpha_0)/\alpha} |f|^{p_\theta/p_0}|^{p_0}
 = |f|^{p_\theta} \\
 |g_z|^{q_0'}
 &=& |e^{i \arg g} |g|^{-z (\beta_1 - \beta_0)/(1-\beta)} |g|^{q_\theta'/q_0'}|^{q_0'}
 = |g|^{q_\theta'}.
\end{eqnarray*}
Since $T$ is of type $(p_0,q_0)$, it thus follows that
\begin{eqnarray*}
 |\Phi(z)|
 &\leq& \|Tf_z\|_{q_0} \|g_z\|_{q_\theta'} \\
 &\leq& k_0 \|f_z\|_{p_0}\|g_z\|_{q_0'} \\
 &=& k_0 \left( \int_X |f|^p \, d\mu \right)^{1/p_0}
 \left( \int_Y |g|^{q_0'} \, d\nu \right)^{1/q_0} \\
 &\leq& k_0 \|f\|_p^{p_\theta/p_0} \|g\|_q^{q_\theta'/q_0'} \\
 &\leq& k_0
\end{eqnarray*}
on the line $\Re z = 0$. A similar computation establishes the bound $|\Phi(z)| \leq k_1$
on the line $\Re z = 1$, whence by \hyperref[hadamard]{Hadamard's three-lines theorem}
we have the inequality
\[
 |\Phi(z)| \leq k_0^{1-\theta}k_1^\theta
\]
on the line $\Re z = \theta$. Setting $z = \theta$, we have
\[
 \left|\int_Y (Tf) g \, d\nu\right| = |\Phi(\theta)| \leq k_0^{1-\theta} k_1^\theta,
\]
which is the desired inequality.

Having established the theorem for simple functions, we now prove the theorem for
the general function $f \in D \cap L^p(X,\mu)$. To this end, we shall furnish
a sequence $(f_n)_{n=1}^\infty$ of simple functions such that
\[
 \lim_{n \to \infty} \|f_n - f\|_{p_\theta} = 0
 \htwo \mbox{and} \htwo
 \lim_{n \to \infty} (T f_n)(x) = (Tf)(x),
\]
for then Fatou's lemma yields the inequality
\[
 \|Tf\|_{q_\theta}
 \leq \lim_{n \to \infty} \|Tf_n\|_{q_\theta}
 \leq \lim_{n \to \infty} k_0^{1-\theta} k_1^\theta \|f_n\|_{p_\theta}
 = k_0^{1-\theta} k_1^\theta \|f\|_{p_\theta}.
\]
Therefore, the task of proving the theorem reduces to finding such a sequence.

We assume without loss of generality that $f \geq 0$ and $p_0 \leq p_1$. Let
$f^0$ be the truncation of $f$, defined as
\[
 f^0(x) =
 \begin{cases}
  f(x) & \mbox{ if } f(x) > 1; \\
  0 & \mbox{ if } f(x) \leq 1;
 \end{cases}
\]
and $f^1=f-f^0$ another truncation. $D$ contains all truncations of $f$,
and $(f^0)^{p_0}$ and $(f^1)^{p_1}$ are bounded by $f^p$, hence
$f^0 \in D \cap L^{p_0}(X,\mu)$ and $f^1 \in D \cap L^{p_1}(X,\mu)$. We can find
a monotonically increasing sequence $(g_m)_{m=1}^\infty$ converging to
$f$, which satisfies
\[
 \lim_{m \to \infty} \|g_m - f\|_{p_\theta} = 0
\]
by the monotone convergence theorem. If $g_m^0$ and $g_m^1$ are truncations
of $g_m$ defined in the same way as $f^0$ and $f^1$, then we have
\[
 \lim_{m \to \infty} \|g_m^0 - f^0\|_{p_\theta}
 = \lim_{m \to \infty} \|g_m^1 f^1\|_{p_\theta}
 = 0.
\]

Since $T$ is of types $(p_0,q_0)$ and $(p_1,q_1)$, we have
\[
 \lim_{m \to \infty} \|Tg_m^0 - Tf^0\|_{q_0}
 = \lim_{m \to \infty} \|Tg_m^1 - Tf^1\|_{q_1}
 = 0.
\]
We can then find a subsequence of $(T g_m^0)_{m=1}^\infty$ converging
almost everywhere to $Tf^0$, whence we may as well assume that
the full sequence converges almost everywhere to $Tf^0$. Similarly,
we can find a subsequence $(g_{m_n})_{n=1}^\infty$ of $(g_m)_{m=1}^\infty$
such that $(T g_{m_n}^1)_{n=1}^\infty$ converges almost everywhere to $Tf^1$,
whence the sequence $(f_n)_{n=1}^\infty$ defined by setting
\[
 f_n = g_{m_n}^0 + g_{m_n}^1
\]
is the desired sequence. This completes the proof of the Riesz-Thorin
interpolation theorem.
\end{proof}

\subsection{Corollaries of the Interpolation Theorem}

We now recall the Hausdorff-Young inequality from the last section:

\begin{theorem}[Hausdorff-Young inequality]\label{hausdorff-young}\index{Hausdorff-Young}
For each $1 \leq p \leq 2$, the $L^p$ Fourier transform is a bounded
linear operator from $L^p(\bR^d)$ into $L^{p'}(\bR^d)$. Specifically,
we have the inequality
\[
 \|\hat{f}\|_{p'} \leq \|f\|_p,
\]
whence the $L^p \to L^{p'}$ operator norm of $\ms{F}$ is at most 1.
\end{theorem}

The inequality is now a trivial consequence of the interpolation theorem,
for the Fourier transform is a bounded operator from $L^1$ into $L^\infty$
and from $L^2$ into $L^2$,\index{Plancherel's theorem} whose operator norm is at most 1 in both cases.

Another application is the following generalization of
\hyperref[convolution]{Young's inequality}:

\begin{theorem}[Young's inequality]\label{young}\index{inequality!Young's}
If $1 \leq p,q,r \leq \infty$ such that
\[
 \frac{1}{p} + \frac{1}{q} = 1 + \frac{1}{r},
\]
then
\[
 \|f*g\|_r \leq \|f\|_p \|g\|_q
\]
for all $f \in L^p(\bR^d)$ and $g \in L^q(\bR^d)$.
\end{theorem}

Again, the proof is a routine application of the interpolation theorem
on the convolution operator
\[
 Tg = f*g,
\]
once we recall the \hyperref[convolution]{Young's inequality} to establish
that $T$ is of type $(1,p)$ and invoke the bound
\[
 \|f*g\|_\infty \leq \|f\|_p \|g\|_{p'}
\]
from Theorem \ref{convolution-of-dual-Lp-functions} to prove that $T$ is of
type $(p',\infty)$. We shall have more to say about convolution operators
in \S\ref{s-generalized_functions}. The constants in the above inequalities
can be improved: see \S\S\ref{fr-sharp-inequalities} for the sharp versions.
\index{Riesz-Thorin interpolation \\ theorem|)}

\subsection{The Stein Interpolation Theorem}

\index{Stein, Elias M.!interpolation theorem|(}
Let $(X,\mf{M},\mu)$ and $(Y,\mf{N},\nu)$ be $\sigma$-finite measure spaces
and $T$ a linear operator of types $(p_0,q_0)$ and $(p_1,q_1)$.
Recall that the proof of the \hyperref[riesz-thorin]{Riesz-Thorin interpolation theorem}
essentially amounted to giving an estimate of the entire function
\[
 z \mapsto \int_Y (Tf_z) g_z \, d\nu
\]
What if we let the operator $T$ vary as well? \index{Fefferman, Charles|(}C. Fefferman points out in
\cite{Fefferman_Fefferman_Wainger:B1995} that
the net result of adding a letter $z$ to the operator $T$ is the new entire function
\[
 \Phi(z) = \int_Y (T_zf_z) g_z \, d\nu,
\]
whence establishing an estimate of $\Phi$ produces another interpolation theorem.
The similarity of the operators suggest that the new proof should closely mimic
that of the \hyperref[riesz-thorin]{Riesz-Thorin interpolation theorem},
and this is indeed the case.

We thus obtain an interpolation theorem that
allows the operator to vary in a holomorphic manner, as the letter $z$ suggests.\index{Fefferman, Charles|)}
The interpolation theorem was first established by E. Stein in
\cite{Elias_M_Stein:J1956} and is dubbed
\emph{interpolation of analytic families of operators}\index{interpolation of analytic families of operators|see{Stein interpolation theorem}} in
the standard reference \cite{Stein_Weiss:B1971}\index{Stein, Elias M.}\index{Weiss, Guido} of E. Stein and G. Weiss. In the
present thesis, we shall refer to it as the \emph{Stein interpolation theorem}.

In order for $\Phi$ to be holomorphic, we must impose a restriction on how
the operator can vary:

\begin{defin}
Let $(X,\mf{M},\mu)$ and $(Y,\mf{N},\nu)$ be $\sigma$-finite measure spaces
and
\[
 S = \{z \in \bC : 0 \leq \Re z \leq 1\}.
\]
We suppose that we are given a linear operator $T_z$, for each $z \in S$,
on the space of simple functions in $L^1(M,\mu)$ into the space of
measurable functions on $N$. If $f$ is a simple function in $L^1(M,\mu)$
and $g$ a simple function in $L^1(N,\nu)$, we assume furthermore that
$(T_z f)g \in L^1(N,\nu)$. The family $\{T_z\}_{z \in S}$ of operators
is said to be \emph{admissible} if, for each such $f$ and $g$, the map
\[
 z \mapsto \int_N (T_z f)g \, d\nu
\]
is holomorphic in the interior of $S$ and continuous on $S$, and if
there exists a constant $k < \pi$ such that
\[
 \sup_{z \in S} e^{-k|\Im z|} \left| \int_N (T_z f) g \, d\nu \right| < \infty.
\]
\end{defin}

With this hypothesis, we can state the interpolation theorem as follows:

\begin{theorem}[Stein interpolation theorem]\label{stein}
Let $(X,\mf{M},\mu)$ and $(Y,\mf{N},\nu)$ be $\sigma$-finite measure spaces
and $\{T_z\}_{z \in S}$ an admissible family of linear operators.
Fix $1 \leq p_0,p_1,q_0,q_1 \leq \infty$ and assume, for each real number $y$, that
there are constants $M_0(y)$ and $M_1(y)$ such that
\[
 \|T_{iy} f\|_{q_0} \leq M_0(y) \|f\|_{p_0} \htwo \mbox{and} \htwo
 \|T_{1+iy} f\|_{q_1} \leq M_1(y) \|f\|_{p_1}
\]
for each simple function $f$ in $L^1(M,\mu)$. If, in addition, the constants $M_j(y)$
satisfy
\[
 \sup_{-\infty < y < \infty} e^{-k|y|} \log M_j(y) < \infty
\]
for some $k < \pi$, then each $\theta \in [0,1]$ furnishes a constant $M_\theta$ such that
\[
 \|T_\theta f\|_{q_\theta} \leq M_\theta \|f\|_{p_\theta}
\]
for each simple function $f$ in $L^1(M,\mu)$, where
\begin{eqnarray*}
 p_\theta^{-1} &=& (1-\theta)p_0^{-1} + \theta p_1^{-1};\\
 q_\theta^{-1} &=& (1-\theta)q_0^{-1} + \theta q_1^{-1}.
\end{eqnarray*}
\end{theorem}

Theorem \ref{norm-preserving-extension} then implies that
$T_\theta$ can be extended to a bounded operator on $L^{p_\theta}(\bR^d)$
into $L^{q_\theta}(\bR^d)$. For the sake of convenience, we shall use
the term \emph{operator of type $(p_\theta,q_\theta)$} for $T_\theta$
as well. The proof of the interpolation theorem
makes use of an extension of \hyperref[hadamard]{Hadamard's three-lines theorem}
due to I. Hirschman:

\begin{lemma}[Hirschman]\label{hirschman}\index{Hirschman's lemma|(}
If $\Phi$ is a continuous function on the strip $S$ that is holomorphic in
the interior of $S$ and satisfies the bound
\begin{equation}\label{hirschman1}
 \sup_{z \in S} e^{-k|\Im z|} \log |\Phi(z)| < \infty
\end{equation}
for some constant $k < \pi$, then
\begin{equation}\label{hirschman2}
 \log|\Phi(\theta)|
 \leq \frac{\sin \pi \theta}{2}
 \int_{-\infty}^\infty \frac{\log|\Phi(iy)|}{\cosh \pi y - \cos \pi \theta}
 + \frac{\log|\Phi(1+iy)|}{\cosh \pi y + \cos \pi \theta} \, dy
\end{equation}
for all $\theta \in (0,1)$. 
\end{lemma}

Why is this an extension of \hyperref[hadamard]{Hadamard's three-lines theorem}?\index{Hadamard's three-lines theorem}
If $\Phi$ is bounded and continuous on $S$, $|\Phi(z)| \leq k_0$ on the line
$\Im z = 0$ and $|\Phi(z)| \leq k_1$ on the line $\Im z = 1$, then
the bound
\[
 \sup_{z \in S} e^{-k|\Im z|} \log |\Phi(z)| < \infty
\]
is satisfied for $k = 0$. Furthermore, we have
\[
 \frac{1}{2}\int_{-\infty}^\infty \frac{\log|\Phi(iy)|}{\cosh \pi y - \cos \pi \theta} \, dy = \theta
 \htwo \mbox{and} \htwo
 \frac{1}{2} \int_{-\infty}^\infty \frac{\log|\Phi(1+iy)|}{\cosh \pi y + \cos \pi \theta} \, dy
 = 1-\theta,
\]
whence Hirschman's lemma implies that
\[
 |\Phi(x)| \leq k_0^{1-\theta}k_1^\theta.
\]
We have thus recovered the three-lines theorem from Hirschman's lemma.

\begin{proof}[Proof of Lemma \ref{hirschman}]
Let $D$ be the closed unit disk in $\bC$. For each $\zeta$ in $D \smallsetminus
\{-1,1\}$, we define
\[
 \varphi(\zeta) = \frac{1}{\pi i} \log \left( i\frac{1+\zeta}{1-\zeta} \right).
\]
$\varphi$ is a composition of the conformal mapping
\[
 \zeta \mapsto w = i\frac{1+\zeta}{1-\zeta}
\]
of $D \smallsetminus \{-1,1\}$ onto the closed upper-half plane $\mb{H}$
and of the conformal mapping
\[
 w \mapsto z = \frac{1}{\pi i} \log w
\]
of $\mb{H}$ onto the closed strip $S$. Therefore, $h$ maps $D
\smallsetminus \{-1,1\}$ conformally onto $S$, and the inverse
\[
 \zeta = \varphi^{-1}(z) = \frac{e^{\pi i z} - i}{e^{ \pi i z } + i}
\]
is conformal as well. $\Psi = \Phi \circ \varphi$ is holomorphic
on open unit disk and is continuous on $D \smallsetminus \{-1,1\}$.

We now recall the following standard result\footnote{See pages 206-208 of
\cite{Lars_V_Ahlfors:B1979} for a detailed discussion} from complex analysis:

\begin{lemma}[Poisson-Jensen formula]
Let $\Psi$ be a holomorphic function on an open disk of radius $R$
centered at 0. If, for some $0 < \rho < R$, we write
$a_1,\ldots,a_N$ to denote the zeroes of $\Psi$ in the open
disk $|z| < \rho$, then
\[
 \log|\Psi(z)| = -\sum_{n=1}^N \log \left| \frac{\rho^2 - \bar{a}_nz}{\rho(z-a_n)} \right|
 + \frac{1}{2\pi} \int_0^{2\pi} \Re \frac{\rho e^{i\theta} + z}{\rho e^{i\theta} - z}
 \log |f(\rho e^{i\theta})| \, d\theta 
\]
for all $|z| < r$ such that $f(z) \neq 0$.
\end{lemma}

For $0 \leq \rho < R < 1$, we can write $\zeta = \rho e^{i\omega}$ and apply
the above lemma to obtain
\begin{equation}\label{poisson-jensen}
 \log|\Psi(\zeta)|
 \leq \frac{1}{2\pi}
 \int_{-\pi}^\pi \frac{R^2 - \rho^2}{R^2 - 2 R \rho \cos (\omega - \phi) + \rho^2}
 \log|\psi(Re^{i\phi})| \, d\phi.
\end{equation}
Rewriting (\ref{hirschman1}) in terms of $\Psi$ and $\eta = \varphi^{-1}(x+iy)$,
we have
\[
 \log|\Psi(\eta)| \leq C \left[ |1+\eta|^{-k/\pi} + |1-\eta|^{-k/\pi} \right]
\]
for some constant $C$ independent of $\eta \in D$. Plugging in $\eta = Re^{i\phi}$
and noting that $k/\pi < 1$, we see that the above inequality permits us to
use the dominated convergence theorem to the integral in (\ref{poisson-jensen}).
Therefore, by sending $R \to 1$, we obtain
\begin{equation}\label{uh2}
\log|\Psi(\zeta)|
 \leq \frac{1}{2\pi}
 \int_{-\pi}^\pi \frac{1 - \rho^2}{1 - 2 \rho \cos (\omega - \phi) + \rho^2}
 \log|\psi(e^{i\phi})| \, d\phi.
\end{equation}

We now apply a \hyperref[change-of-variables]{change of variables} to (\ref{uh2})
inequality to obtain (\ref{hirschman2}). First, we convert the condition
\begin{equation}\label{uh}
 0 < \omega = \varphi(\rho e^{i \omega}) < 1
\end{equation}
into a condition on $\rho$ and $\omega$. Indeed, we note that
\[
 \rho e^{i \omega}
 = \varphi^{-1}(\theta) 
 = \frac{e^{\pi i \theta} - i}{e^{\pi i \theta} + i} 
 = - i \frac{\cos \pi \theta}{1 + \sin \pi \theta} 
 = \left( \frac{\cos \pi \theta}{1 + \sin \pi \theta} \right) e^{i\pi/2},
\]
whence (\ref{uh}) implies that
\[
 \rho = 
 \begin{cases}
  \frac{\cos \pi \theta}{1 + \sin \pi \theta} & \mbox{ if } 0 < \theta \leq \frac{1}{2}; \\
  -\frac{\cos \pi \theta}{1 + \sin \pi \theta} & \mbox{ if } \frac{1}{2} \leq \theta < 1;
 \end{cases}
 \htwo \mbox{and} \htwo
 \omega =
 \begin{cases}
  -\frac{\pi}{2} & \mbox{ if } 0 < \theta \leq \frac{1}{2}; \\
  \frac{\pi}{2} & \mbox{ if } \frac{1}{2} \leq \theta < 1.
 \end{cases}
\]
Therefore, if $0 < \theta \leq \frac{1}{2}$, then
\begin{eqnarray*}
 \frac{1 - \rho^2}{1 - 2 \rho \cos (\omega - \phi) + \rho^2}
 &=& \frac{1 - \rho^2}{1 - 2 \rho \cos(\phi + \pi/2) + \rho^2} \\
 &=& \frac{1-\rho^2}{1+2\rho\sin\phi + \rho^2} \\
 &=& \frac{\sin \pi \theta}{1 + \cos \pi \theta \sin \phi}.
\end{eqnarray*}
Furthermore, the above equality holds for $\frac{1}{2} \leq \theta < 1$
as well.

We now observe from the identity
\[
 e^{i\phi} = \varphi^{-1}(iy) = \frac{e^{-\pi y}}{e^{-\pi y} + i}
\]
that $y$ ranges from $+\infty$ to $-\infty$ as $\phi$ ranges from $-\pi$
to $0$. Moreover,
\[
 \sin \phi = -\frac{1}{\cosh \pi y}
 \htwo \mbox{and} \htwo
 d\phi = -\frac{\pi}{\cosh \pi y} dy,
\]
and so the \hyperref[change-of-variables]{change-of-variables formula}
yields
\begin{eqnarray*}
 & & \frac{1}{2\pi}
 \int_{-\pi}^0 \frac{1 - \rho^2}{1 - 2 \rho \cos (\omega - \phi) + \rho^2}
 \log|\Psi(e^{i\phi})| \, d\phi \\
 &=& \frac{\sin \pi \theta}{2} \int_{-\infty}^\infty
 \frac{\log|\Phi(iy)|}{\cosh \pi y - \cos \pi \theta}  \, dy.
\end{eqnarray*}
Similarly, as $\phi$ ranges from $0$ to $\pi$, the function
$\varphi(e^{\phi})$ produces the points $1+iy$ with $-\infty < y < \infty$,
whence
\begin{eqnarray*}
& & \frac{1}{2\pi}
 \int_0^{\pi} \frac{1 - \rho^2}{1 - 2 \rho \cos (\omega - \phi) + \rho^2}
 \log|\Psi(e^{i\phi})| \, d\phi \\
 &=& \frac{\sin \pi \theta}{2} \int_{-\infty}^\infty
 \frac{\log|\Phi(1+iy)|}{\cosh \pi y + \cos \pi \theta}  \, dy.
\end{eqnarray*}
We add up the two quantities and plug the sum into (\ref{uh2})
to obtain (\ref{hirschman2}), which was the desired inequality.
\end{proof}
\index{Hirschman's lemma|)}

We are now ready to present a proof of the interpolation theorem.

\begin{proof}[Proof of Theorem \ref{stein}]
Fix $0 \leq \theta \leq 1$. As in the proof of the
\hyperref[riesz-thorin]{Riesz-Thorin interpolation theorem}, we let
\[
 \alpha_0 = \frac{1}{p_0}, \hone \alpha_1 = \frac{1}{p_1},
 \hone \alpha = \frac{1}{p_\theta}, \hone
 \beta_0 = \frac{1}{q_0}, \hone \beta_1 = \frac{1}{q_1}, \hone
 \beta = \frac{1}{q_\theta}.
\]
With this notation, we set
\[
 \alpha(z) = (1-z)\alpha_0 + z\alpha_1 \htwo \mbox{and} \htwo
 \beta(z) = (1-z)\beta_0 + z\beta_1,
\]
so that
\[
 \alpha(0) = \alpha_0, \hone \alpha(1) = \alpha_1,
 \hone \alpha(\theta) = \alpha, \hone
 \beta(0) = \beta_0, \hone \beta(1) = \beta_1, \hone
 \beta(\theta) = \beta.
\]

Let 
\[
 f = \sum_{j=1}^m a_j \chi_{E_j} \htwo\mbox{ and }\htwo
 g = \sum_{k=1}^n b_k \chi_{F_k}
\]
be simple functions such that $f \in L^1(X,\mu)$, $g \in L^1(Y,\nu)$,
and
\[
\|f\|_{p_\theta} = 1 = \|g\|_{q_\theta'}.
\]
Write $a_j = |a_j| e^{i\theta_j}$ and $b_k = |b_k| e^{i\varphi_k}$ and set 
\[
 f_z = \sum_{j=1}^m |a_j|^{\alpha(z)/\alpha}e^{i\theta_j} \chi_{E_j}
 \htwo\mbox{and}\htwo
 g_z = \sum_{k=1}^m |b_k|^{(1-\beta(z))/(1-\beta)}e^{i\varphi_k}\chi_{F_k}
\]
for each $z \in \bC$. Then
\[
 \Phi(z) = \int_Y (T_zf_z) g_z \, d\nu
\]
is an entire function such that
\[
 \Phi(\theta) = \int_Y (T_\theta f) g \, d\nu.
\]
We note, in particular, that $\Phi$ is holomorphic in the interior of $S$
and is continuous on $S$. Since $T_z$ is linear, we see that
\[
 \Phi(z) = \sum_{j=1}^m\sum_{k=1}^n
 |a_j|^{\alpha(z)/\alpha} |b_k|^{(1-\beta(z))/(1-\beta)}
 \left(e^{i(\theta_j + \varphi_k)} \int_Y (T_z \chi_{E_j}) \chi_{F_k} \right).
\]

It follows that $\{T_z\}_{z \in S}$ is an admissible family, whence
$\Phi$ satisfies the bound
\[
 \sup_{z \in S} e^{-k|\Im z|} \log |\Phi(z)| < \infty.
\]
Furthermore, we have
\[
 |f_{iy}|^{p_0} = |f|^{p_\theta} = |f_{1+iy}|^{p_1}
 \htwo \mbox{and} \htwo
 |g_{iy}|^{q_0'} = |g|^{q_\theta'} = |g_{1+iy}|^{q_1'}
\]
for all $y \in \bR$, and so H\"{o}lder's inequality implies that
$|\Phi(iy)| \leq M_0(y)$ and $|\Phi(1+iy)| \leq M_1(y)$.
\hyperref[hirschman]{Hirschman's lemma} now establishes the bound
\begin{eqnarray*}
 |\Phi(\theta)| \\
 &\leq& \exp \left( \frac{\sin (\pi \theta)}{2}
 \int_{-\infty}^\infty \frac{\log|\Phi(iy)|}{\cosh \pi y - \cos \pi \theta}
 + \frac{\log|\Phi(1+iy)|}{\cosh \pi y + \cos \pi \theta} \, dy
 \right) \\
 &=& M_\theta,
\end{eqnarray*}
whence we invoke the \hyperref[Lp-riesz-representation]{Riesz representation theorem}
to conclude that
\[
 \|T_\theta f\|_{q_\theta}
 = \sup_{\|g\|_{q_\theta'} = 1}\left| \int_Y (T_\theta f) g \, d\nu \right|
 \leq M_\theta 
 = M_\theta \|f\|_{p_\theta},
\]
as was to be shown.
\end{proof}

See \S\S\ref{fr-fourier-inversion-problem} for an application of the Stein
interpolation theorem in the context of the Fourier inversion problem.
In \S\ref{s-hardy-spaces-and-bmo}, we shall prove Fefferman's
generalization of the Stein interpolation theorem
that allows us to take a particular subspace
of $L^1$ as the domain for one of the endpoint estimates. The theory of
\emph{complex interpolation}, which provides an abstract framework for
Riesz-Thorin, Stein, and Fefferman-Stein, will be developed in \S\ref{s-the_complex_interpolation_method}.\index{complex interpolation}
\index{Stein, Elias M.!interpolation theorem|)}
  
\section{Additional Remarks and Further Results}

In this section, we collect miscellaneous comments that provide
further insights or extension of the material discussed in the
chapter. No result in the main body of the thesis relies on
the material presented herein.

\begin{fr}\label{fr-compactification}\index{compactification}
If $X$ is a topological space, then a
\emph{compactification}\footnote{See \S29 and \S38 in \cite{James_R_Munkres:B2000}\index{Munkres, James R.} or
Proposition 4.36 and Theorem 4.57 in \cite{Gerald_B_Folland:B1999}\index{Folland, Gerald B.}
for standard methods of compactification in point-set topology.}
of $X$ is defined to be a compact Hausdorff space $Y$ such that
$X$ embeds into $Y$ as a dense subset of $Y$. The extended
real line $\bar{\bR}$ can be considered as a compactification
of $\bR$, if, in addition to the open sets in $\bR$,
we declare the intervals of the form $[-\infty,a)$ and
$(a,\infty]$ as open sets in $\bar{\bR}$. With this topology,
an extended real-valued function $f$ is measurable if and only
if $f^{-1}(U)$ is measurable for every open set $U$ in $\bar{\bR}$,
thus conforming to the standard definition of measurability.
\end{fr}

\begin{fr}\label{fr-riesz-representation}\index{representation theorem!F. Riesz, Lp-space version@F. Riesz, $L^p$-space version}
For $1<p<\infty$, the \hyperref[Lp-riesz-representation]{Riesz representation theorem} continues to hold on
non-$\sigma$-finite
measure spaces: see Theorem 6.15 in \cite{Gerald_B_Folland:B1999}\index{Folland, Gerald B.}.
It thus follows that $L^p$ spaces for $1 < p <\infty$ are always reflexive
Banach spaces.
\index{representation theorem!on L1@on $L^1$|(}
As for $p=1$, the representation theorem continues to hold
for a slightly milder condition that $\mu$ be decomposable:\index{measure!decomposable}
we say that $\mu$ is \emph{decomposable} if there is a pairwise disjoint
collection $\ms{F}$ of finite $\mu$-measure such that
\begin{enumerate}[(a)]
 \item The union of all members of $\ms{F}$ is the base set $X$;
 \item If $E$ is of finite $\mu$-measure, then
 \[
  \mu(E) = \sum_{F \in \ms{F}} \mu(E \cap F);
 \]
 \item If the intersection a subset $E$ of $X$ and each element of
 the collection $\ms{F}$  is $\mu$-measurable, then $E$ is $\mu$-measurable.
\end{enumerate}
See Theorem 20.19 in \cite{Hewitt_Stromberg:B1965}\index{Hewitt, Edwin}\index{Stromberg, Karl} for a proof.
\index{representation theorem!on L1@on $L^1$|)}

\index{representation theorem!on Linfty@on $L^\infty$|(}
The dual of $L^\infty$ is strictly bigger than $L^1$, even on the Euclidean
space. We define a bounded linear functional $l_0$ on $\mc{C}_c(\bR^d)$
by setting
\[
 l_0(f) = f(0)
\]
for each $f \in \mc{C}_c(\bR^d)$. By Theorem \ref{extension-of-linear-functionals-on-banach-spaces},\index{Hahn-Banach theorem} 
there exists a bounded extension $l$ of $l_0$ on $L^\infty(\bR^d)$. Assume for
a contradiction that
\[
 l(f) = \int fu \, dx
\]
for some $u \in L^1(\bR^d)$. Then $\int fu = 0$ for all $f \in \mc{C}_c(\bR^d)$
such that $f(0)=0$, whence the discussion in \S\ref{fr-smooth-approximations-to-the-identity}\index{approximations to the identity!smooth}
establishes that $u=0$ almost everywhere on $\bR^d \smallsetminus \{0\}$. Therefore, $u=0$
almost everywhere on $\bR^d$, and so
\[
 \int fu \, dx = 0
\]
for all $f \in L^\infty(\bR^d)$, which is evidently absurd.

In general, the dual of $L^\infty(X,\mf{M},\mu)$ is isometrically isomorphic
to the Banach space of finitely additive measures of bounded total variation
that are absolutely continuous to $\mu$, which \cite{Dunford_Schwartz:B1958}\index{Dunford, Nelson}\index{Schwartz, Jacob T.}
denotes as $ba(X,\mf{M},\mu_1)$. See \S IV.8.16 in \cite{Dunford_Schwartz:B1958}
or \S IV.9, Example 5 in \cite{Kosaku_Yosida:B1980}\index{Yosida, Kosaku} for a proof of this
representation theorem. Another ``representation theorem'' can be established
if we consider $L^\infty$ as a $C^*$-algebra: see \S12.20 in \cite{Walter_Rudin:B1991}.\index{Rudin, Walter}
\index{representation theorem!on Linfty@on $L^\infty$|)}
\end{fr}

\begin{fr}\label{fr-whitney-decomposition}\index{Whitney decomposition theorem}
The \emph{Whitney decomposition theorem} states that every nonempty \linebreak
closed set $F$
in $\bR^d$ admits a countable collection $\{Q_n:n \in \bN\}$ of almost-disjoint cubes
such that the union of the cubes is $\bR^d \smallsetminus F$ and
\[
 \diam Q_n \leq d(Q_n,F) \leq 4 \diam(Q_n)
\]
for each $n \in \bN$. See \S VI.1.2 in \cite{Elias_M_Stein:B1970}\index{Stein, Elias M.} or Appendix J
in \cite{Loukas_Grafakos:B2008}\index{Grafakos, Loukas} for a proof. 
Compare the decomposition theorem with the \emph{Calder\'{o}n-Zygmund lemma}:\index{Calder\'{o}n-Zygmund!lemma}
if $f$ is a nonnegative integrable function on $\bR^d$ and $\alpha$ a positive
constant, then there exists a decomposition $\bR^d = F \cup \Omega$ such that
\begin{enumerate}[(a)]
 \item $F \cap \Omega = \varnothing$;
 \item $f(x) \leq \alpha$ almost everywhere on $F$;
 \item $\Omega$ is the union of almost-disjoint cubes $\{Q_n : n \in \bN\}$
 such that
 \[
  \alpha < \frac{1}{|Q_n|} \int_{Q_n} f(x) \, dx \leq 2^d \alpha.
 \]
\end{enumerate}
This is of paramount importance in the theory of singular integrals,
which we briefly touch upon in \S\ref{s-the_hilbert_transform}.
A proof of the lemma can be found in \cite{Elias_M_Stein:B1970},\index{Stein, Elias M.} I.3.3,
although the lemma is usually integrated into the
\emph{Calder\'{o}n-Zygmund decomposition} in many expositions:
see Theorem \ref{calderon-zygmund} for details.
\end{fr}

\index{Littlewood's three principles|(}
\begin{fr}\label{fr-littlewoods-three-principles}
Littlewood's three principles serve as a useful guide for studying
the properties of the Lebesgue measure. The first principle states
that \emph{every set is nearly a finite sum of intervals}. To make
this precise, we recall that the \emph{symmetric difference} of two
sets $E$ and $F$ is defined by
\[
 E \Delta F = (E \smallsetminus F) \cup (F \smallsetminus E).
\]

\begin{prop}
If $E \subseteq \bR^d$ is of finite measure, then, for each
$\ve>0$, there exists a finite sequence $(Q_n)_{n=1}^N$ of
cubes such that
\[
 \left| E \Delta \bigcup_{n=1}^N Q_n \right| \leq \ve.
\]
\end{prop}

The second principle, which states that \emph{every function is nearly
continuous}, can be stated as follows:

\begin{theorem}[Lusin]\index{Lusin's theorem!on Euclidean spaces}
If $E$ is a finite-measure subset of $\bR^d$ and $f:E \to \bC$ a measurable
function, then, for each $\ve>0$, there exists a closed subset $F$ of $E$
such that $|E \smallsetminus F| \leq \ve$ and $f|_{F}$ is
continuous. 
\end{theorem}

The third and the final principle states that \emph{every convergent
sequence of functions is nearly uniformly convergent} and can be
formulated precisely as follows:

\begin{theorem}[Egorov]\index{Egorov's theorem}\index{Egoroff's thoerem|see{Egorov's theorem}}
If $E$ is a finite-measure subset of $\bR^d$ and $(f_n)_{n=1}^\infty$ a
sequence of measurable functions on $E$ that converge pointwise
almost everywhere to $f:E \to \bC$, then, for each $\ve>0$, there
exists a closed subset $F$ of $E$ such that $|E \smallsetminus F|
\leq \ve$ and $f_n \to f$ uniformly on $F$.
\end{theorem}

Lusin's theorem can be established on more general domains, which can
then be used to generalize Theorem \ref{approximation-by-continuous-functions}.
We shall take up on this matter in the next subsection.
\end{fr}

\begin{fr}\label{fr-approximation-by-Cc-on-locally-compact-spaces}
Theorem \ref{approximation-by-continuous-functions} can be generalized to
locally compact Hausdorff domains with complete Borel regular measures.
This is a direct consequence of the generalized Lusin's theorem:

\begin{theorem}[Lusin]\index{Lusin's theorem!on locally compact spaces}
Let $\mu$ be a complete Borel regular measure on a locally compact Hausdorff
space $X$. If $f$ is a complex-valued measurable function on $X$, $A$
a finite-measure subset of $X$, and $\supp f \subseteq A$, then each
$\ve>0$ admits a function $g \in \mc{C}_c(X)$ such that
\[
 \mu(\{x : f(x) \neq g(x)\}) < \ve
 \htwo \mbox{and} \htwo
 \sup_{x \in X} |g(x)| \leq \sup_{x \in X}|f(x)|.
\]
\end{theorem}

See \S2.24 in \cite{Walter_Rudin:B1986}\index{Rudin, Walter} for a proof. Crucial
in proving the above theorem is \emph{Urysohn's lemma} for locally
compact Hausdorff spaces, which can be stated as follows:

\begin{theorem}[Urysohn's lemma]\index{Urysohn's lemma}
Let $X$ be a locally compact Hausdorff space. For each open subset $O$
of $X$ and a compact subset $K$ of $O$, we can find a function 
$f \in \mc{C}_c(X)$ such that $0 \leq f(x) \leq 1$ on $X$,
$\supp f \subseteq O$, and $f(x) = 1$ on $K$.
\end{theorem}

A proof of the above theorem can be found in \cite{Walter_Rudin:B1986},\index{Rudin, Walter}
\S2.12. An even more general version of Lusin's theorem for Radon
measures appears in \cite{Gerald_B_Folland:B1999}\index{Folland, Gerald B.} as Theorem 7.10.
Urysohn's lemma continues to hold for normal spaces: see \S33 in
\cite{James_R_Munkres:B2000}\index{Munkres, James R.} or Theorem 4.15 in \cite{Gerald_B_Folland:B1999}\index{Folland, Gerald B.}
for a discussion.
\end{fr}
\index{Littlewood's three principles|)}

\index{approximations to the identity|(}
\begin{fr}\label{fr-approximations-to-the-identity}
The name \emph{approximations to the identity} (as in
Theorem \ref{approximations-to-the-identity} and Corollary
\ref{smooth-approximations-to-the-identity}) can be motivated by
considering a more abstract framework. A \emph{Banach algebra}\index{Banach algebra}
is a (complex) Banach space $\ms{A}$ with associative bilinear
multiplication operation satisfying the inequality
\[
 \|xy\| \leq \|x\|\|y\|
\]
for all $x,y \in \ms{A}$. \hyperref[convolution]{Young's inequality}
shows that the space $L^1$ equipped with the convolution operation
is a Banach algebra.

Given a Banach algebra $\ms{A}$,
a \emph{left approximate identity}, or a \emph{left approximation to
the identity}, is a sequence $(e_n)_{n=1}^\infty$ of elements
in $\ms{A}$ such that
\[
 \lim_{n \to \infty} \|e_nx - x\| = 0
\]
for all $x \in \ms{A}$. While $e_n$ is not quite a multiplicative
identity of the Banach algebra $\ms{A}$, it serves as one
``at infinity''---hence the name ``approximation to the identity''.
Right approximate identities can be defined analogously.
\end{fr}
\index{approximations to the identity|)}

\begin{fr}\label{fr-smooth-approximations-to-the-identity}\index{approximations to the identity!smooth}
Another useful corollary of Theorem \ref{approximations-to-the-identity}
is that, for $1 \leq p < \infty$, an $L^p$ function $u$ on an open subset
$O$ of $\bR^d$ is zero almost everywhere on $\Omega$ in case
\[
 \int fu = 0
\]
for all $f \in \ms{C}_c^\infty$. See \S\ref{fr-riesz-representation}.

A consequence of the above result is that proving the equal almost-everywhere
state of two functions amounts to integration the difference of two functions
against a small class of ``test functions''. This idea will be a recurring
theme in distribution theory, which will be developed in \S\ref{s-generalized_functions}.
\end{fr}

\begin{fr}\label{fr-sharp-inequalities}
The optimal constant in the $d$-dimensional
\hyperref[hausdorff-young]{Hausdorff-Young inequality}\index{inequality!Hausdorff-Young}
is $C_p^d$, where
\[
 C_p = \left( \frac{p^{1/p}}{(p')^{1/p'}} \right)^{1/2}.
\]
The equality is achieved if and only if $f$ is of the form
\[
 f(x) = A \exp \left( - x \cdot Ax + b \cdot x \right),
\]
where $A$ is a real symmetric positive-definite matrix and $b$ a vector in $\bC^n$.
The optimal constant was given for all $1 \leq p \leq 2$ by William Beckner
in \cite{William_Beckner:J1975}.\index{Beckner, William} The necessary and sufficient condition
for the equality, published in \cite{Elliott_H_Lieb:J1990},\index{Lieb, Elliott H.} is due to Elliott Lieb.

The optimal constant for the $d$-dimensional
\hyperref[young]{generalized Young's inequality},\index{inequality!Young} also established
by Beckner in \cite{William_Beckner:J1975},\index{Beckner, William} is $(C_pC_qC_{r'})^d$,
where the constants are defined as above. In \cite{Elliott_H_Lieb:J1990},\index{Lieb, Elliott H.}
Lieb gives a necessary and sufficient condition for an even more general version
of Young's inequality. See \cite{Lieb_Loss:B2001},\index{Lieb, Elliott H.}\index{Loss, Michael} Theorem 4.2 for a textbook
exposition of the fully generalized Young's inequality.
\end{fr}

\chapter{The Modern Theory of Interpolation}
 
The idea of the proof of the Riesz-Thorin interpolation theorem can be
generalized to a class of operators on spaces other than the Lebesgue
spaces. Alberto P. Calder\'{o}n's insight was to consider interpolation
as an operation on spaces, rather than on operators. First published
in \cite{Alberto_P_Calderon:J1964}, Calder\'{o}n's complex method
of interpolation absorbs the complex-analytic proof of the
Riesz-Thorin interpolation theorem and provides an abstract framework
in which many new interpolation theorems can be generated.

On the more concrete side, it became increasingly evident that many
useful operators simply were not bounded on $L^1$ or $L^\infty$.
Hardy spaces and the space of bounded mean oscillation were introduced
as well-behaved substitutes, and it was proven by Charles Fefferman
and Elias M. Stein in \cite{Fefferman_Stein:J1972} that operators
on these spaces can be interpolated as if they are on Lebesgue
spaces. Precisely, the Fefferman-Stein interpolation theorem is
a generalization of the Stein interpolation theorem on analytic
families of operators, where the operator on $L^1$ can be bounded
only on a subspace $H^1$ of $L^1$, and the operator on $L^\infty$
can satisfy a weaker estimate ($L^\infty \to \BMO$) than the
usual estimate ($L^\infty \to L^\infty$).

The necessary machinery for these interpolation theorems is developed
in the first and the third sections. In the fourth section, we 
study the Hilbert transform, which sets the stage for the
Fefferman-Stein theory sketched in the fifth section.
In the sixth and the last section, the Fefferman-Stein theory
is applied to the study of differential
equations and Fourier integral operators, and the complex
interpolation method is applied to Sobolev spaces.

\section{Elements of Functional Analysis}\label{s-elements_of_functional_analysis}

We begin the chapter by establishing a few key results from functional analysis,
as explicated in many references such as
\cite{Haim_Brezis:B2011},\index{Brezis, Haim} \cite{Peter_D_Lax:B2002},\index{Lax, Peter D.}
\cite{Walter_Rudin:B1991},\index{Rudin, Walter} \cite{Kosaku_Yosida:B1980},\index{Yosida, Kosaku}
and \cite{Dunford_Schwartz:B1958}.\index{Dunford, Nelson}\index{Schwartz, Jacob T.}
We have already encountered functional-analytic methods in Chapter 1, in which
we often investigated the properties of spaces of functions rather than those of
individual functions. We shall raise the level of abstraction even higher
in this section, and study various kinds of \emph{topological vector spaces}:

\begin{defin}\index{topological vector space}
A real or complex vector space $V$ is a \emph{topological vector space} if
$V$ is equipped with a topology such that the addition map $(x,y) \mapsto x+y$
and the scalar multiplication map $(a,v) \mapsto av$ are continuous.
\end{defin}

\subsection{Continuous Linear Functionals on Fr\'{e}chet Spaces}

\index{Fr\'{e}chet space|(}
$L^p$ spaces and, in general, Banach spaces are canonical examples
of topological vector spaces. Some function spaces, however, do not admit
one canonical norm. They might come with multiple natural norms; there
might not be a convenient quotient construction to turn the norm-like
map into a genuine norm. We are thus led to the following generalization:

\begin{defin}\label{defin-seminorm}\index{seminorm}\index{norm!semi-|see{seminorm}}
A \emph{seminorm} on a topological vector space $V$ is a function
$\rho:V \to [0,\infty)$ such that $\rho(a v) = |a|v$ and $\rho(v+w)
\leq \rho(v) + \rho(w)$ for all scalar $a$ and vectors $v$ and $w$.
\end{defin}

We shall see in \S\ref{s-generalized_functions} that the canonical
topology on $\ms{S}(\bR^d)$ is given by a countable collection of
seminorms. This turns $\ms{S}(\bR^d)$ into a \emph{Fr\'{e}chet space},
which we now define.

\begin{defin}\label{defin-frechet-space}
Let $V$ be a real or complex topological vector space. $V$ is a 
\emph{Fr\'{e}chet space} if $V$ is equipped with a countable
collection $\{\rho_n : n \in \bN\}$ of seminorms such that
\[
 d(v,w) = \sum_{n=1}^\infty \frac{1}{2^n}
 \left( \frac{\rho_n(v - w)}{1 + \rho_n(v - w)} \right)
\]
is a metric generating the topology of $V$. 
\end{defin}

We note that the seminorms $\rho_n$ are
continuous in the topology defined as above.
It is clear that every Banach space is a Fr\'{e}chet
space. While most function spaces we study in the present
thesis are Banach spaces, we shall see that some are more
naturally described in the language of Fr\'{e}chet spaces.

For now, we study a different problem: namely, the existence
of nontrivial continuous linear functionals. We know from linear algebra
that finite-dimensional vector spaces have as many linear functionals
as the vectors therein, and that each vector space is isomorphic to
its dual space. It is not at all clear, however, if there are nonzero
linear functionals on an infinite-dimensional space, let alone continuous
ones.

There are, of course, some infinite-dimensional spaces with nontrivial
continuous linear functionals. For example, we have seen that the
$L^p$ spaces have infinitely many bounded linear functionals.
Indeed, the \hyperref[Lp-riesz-representation]{Riesz representation theorem}
asserts that the \emph{dual} of $L^p$ spaces are highly nontrivial.

\begin{defin}
The \emph{dual space} of a topological vector space $V$ is the
vector space $V^*$ of continuous linear functionals on $V$.
\end{defin}

Our immediate goal, then, is to show that the dual of many topological vector
spaces are nontrivial. This requires a preliminary result, widely regarded
as one of the cornerstones in functional analysis.

\index{Hahn-Banach theorem|(}\index{linear functional!extension of|see{Hahn-Banach \\ theorem}}
\begin{theorem}[Hahn-Banach]\label{hahn-banach}
Let $V$ be a real vector space and $\rho$ a seminorm on $V$. If $M$ is a
linear subspace of $V$ and $l$ a real linear functional on $M$ such that
$l(v) \leq \rho(v)$ for all $v \in M$, then there exists a linear functional
$L$ on $V$ such that $L|_M = l$ and $L(v) \leq \rho(v)$ for all $v \in V$.
\end{theorem}

\begin{proof}
If $M = V$, then there is nothing to prove, and so we may suppose the existence
of a vector $z \in V \smallsetminus M$. We first show that we can extend
$l$ by one extra dimension. If $L$ is an extension of $l$
on $M \oplus \bR z$ such that $L(v) \leq \rho(v)$ for all $fv \in M \oplus
\bR z$, then, for any $y_0,y_1 \in M$, we have
\begin{eqnarray*}
 L(y_0)+L(y_1)
 &=& L(y_0+y_1) \\
 &\leq& \rho(y_0+y_1) \\
 &=& \rho(y_0 - z + z + y_1) \\
 &\leq& \rho(y_0 - z) + \rho(z+y_1).
\end{eqnarray*}
This implies that
\[
 L(y_0) - \rho(y_0-z) \leq \rho(z + y_1) - L(y_1),
\]
and so
\begin{equation}\label{hahn-banach-inequality}
 \sup_{y \in M} L(y) - \rho(y-z) \leq \inf_{y \in M} \rho(z+y) - L(y).
\end{equation}

Recall now that every $v \in M \oplus \bR z$ can
be written as the sum
\[
 v = y + \lambda z,
\]
where $y \in M$ and $\lambda \in \bR$. 
We fix a real number $\alpha$ such that
\[
 \sup_{y \in M} L(y) - \rho(y-z) \leq \alpha \leq  \inf_{y \in M} \rho(z+y) - L(y)
\]
and define
\[
 L(v) = L(y) + \lambda \alpha
\]
for each $v \in M$, so that $L$ is a linear extension of $l$ on $M$.
Furthermore, if $\lambda > 0$, then
\begin{eqnarray*}
 L(v)
 = L(y+\lambda z) 
 &=& L(y) + \lambda \alpha \\
 &=& \lambda \left[ L(\lambda^{-1} y/) + \alpha \right] \\
 &\leq& \lambda \left( L(\lambda^{-1}y)
 + \left[\rho(z + \lambda^{-1}y) - L(\lambda^{-1}y) \right] \right) \\
 &=& \lambda \rho(z + \lambda^{-1}y) \\
 &=& \rho(y + \lambda z) = \rho(v)
\end{eqnarray*}
by (\ref{hahn-banach-inequality}). If $\lambda < 0$, then
\begin{eqnarray*}
 L(v)
 = L(y + \lambda z)
 &=& L(y) - (-\lambda) \alpha \\
 &=& (-\lambda) \left[ L(-\lambda^{-1}y) - \alpha \right] \\
 &\leq& (-\lambda) \left( L(-\lambda^{-1}y)
 - \left[ L(-\lambda^{-1}y) - \rho(-\lambda^{-1}y - z)\right] \right)\\
 &=& (-\lambda) \rho(-\lambda^{-1}y - z) \\
 &=& \rho(y + \lambda z) = \rho(v).
\end{eqnarray*}
If $\lambda = 0$, then $L(y + \lambda z) = l(y)$, and there is nothing
to prove. We have thus demonstrated that we can always extend $l$ by
one extra dimension.

Let us now return to the proof of the theorem in its full generality.
We define a partial order $\leq$ on the collection of ordered pairs $(l',M')$
of linear extensions $l'$ of $l$ on $M'$ satisfying the bound $l'(v) \leq \rho(v)$
for all $v \in M'$ by setting $(l',M') \leq (l'',M'')$ if and only if
$M' \subseteq M''$. Given a chain
\begin{equation}\label{hahn-banach-chain}
 (l_1,M_1) \leq (l_2,M_2) \leq (l_3,M_3) \leq \cdots
\end{equation}
of such ordered pairs, we define the pair $(l_\infty,M_\infty)$ by setting
\[
 M_\infty = \bigcup_{n=1}^\infty M_n
 \htwo\mbox{and}\htwo
 l_\infty = l_n(v),
\]
where $n$ is chosen such that $v \in M_n$. The definition of $l_\infty$ is
unambiguous, because the linear functionals agree on common domains.
It is easy to see that $(l_\infty,M_\infty)$ is an upper bound of the
chain (\ref{hahn-banach-chain}).

We now invoke Zorn's lemma\index{Zorn's lemma} to construct a maximal element $(L,M_0)$ of
the collection given above. If $M_0 \neq V$, then we can extend $L$
by one extra dimension, whence $(L,M_0)$ is not maximal. It follows that
$L$ is the desired linear functional, and the proof is now complete.
\end{proof}

Since most functions we study in this thesis are complex-valued,
the corresponding function spaces are mostly complex vector spaces
as well. We therefore require the following generalization of the
\hyperref[hahn-banach]{Hahn-Banach theorem}, commonly referred to
as the \emph{complex Hahn-Banach theorem}.

\begin{theorem}[Bohnenblust-Sobczyk]\label{complex-hahn-banach}\index{Bohnenblust-Sobczyk \\ theorem|see{Hahn-Banach \\ theorem}}
Let $V$ be a complex vector space and $\rho$ a seminorm on $V$. If $M$ is a
linear subspace of $V$ and $l$ a complex linear functional on $M$ such that
$|l(v)| \leq \rho(v)$ for all $v \in M$, then there exists a linear
functional $L$ on $V$ such that $L|_M = l$ and $|L(v)| \leq \rho(v)$ for all
$v \in V$. 
\end{theorem}

\begin{proof}
We first note that $l$ can be written as
\[
 l = l_1 + i \, l_2
\]
where $l_1$ and $\Lambda_2$ are real linear functionals on $V$,
considered as a real vector space. For each $v \in M$, we see that
\[
 l_1(iv) + i \, l_2(iv)
 = l(iv) = i \, l(v)
 = i \left[ l_1(v) + i \, l_2(v) \right]
 = i \, l_1(v) - l_2(v).
\]
This implies that
\[
 l_2(v) = - l_1(iv)
\]
for all $v \in M$. Since we have the bound
\[
 l_1(v) \leq |l_1(v)| \leq |l(v)| \leq \rho(v)
\]
for all $v \in M$, we can invoke the \hyperref[hahn-banach]{Hahn-Banach theorem}
to construct an extension $L_1$ of $l_1$ on $V$ such that $L_1$ is still
dominated by $\rho$.

We define a complex linear functional $L$ on $V$ by setting
\[
 L(v) = L_1(v) -iL_1(iv)
\]
for each $v \in V$, which is easily seen to be an extension of $l$.
To show that $|L(v)| \leq \rho(v)$ for each $v \in V$, we
fix a $v \in V$ and find $r>0$ and $\theta \in [0,2\pi)$ such that
$L(v) = re^{i\theta}$. We then have
\[
 |L(v)| = r = re^{i\theta}e^{-i\theta} = e^{-i\theta}L(v) = L(e^{-i\theta}v).
\]
Noting that $|L(v)| \geq 0$, we conclude that
\[
 |L(v)| = L(e^{-i\theta}v) = L_1(e^{-i\theta}v).
\]
Since
\[
 -L_1(e^{-i\theta}v) = L_1(-e^{-i\theta}v) \leq \rho(-e^{-i\theta}v) = \rho(e^{-i\theta}v),
\]
we see that $|L_1(e^{-i\theta}v)| \leq \rho(e^{-i\theta}v)$. It thus
follows that
\[
 |L(v)| = L_1(e^{-i\theta}v)
 = |L_1(e^{-i\theta}v)| \leq \rho(e^{-i\theta}v)
 = |e^{-i\theta}|\rho(v) = \rho(v),
\]
as was to be shown.
\end{proof}

We are interested in \emph{bounded} linear functionals, so we now
put a topology on our vector space. The corresponding extension
theorem is the following:

\begin{theorem}\label{extension-of-linear-functionals}
Let $V$ be a real or complex topological vector space. If $\rho$ is a
continuous seminorm on $V$ and $v_0$ a vector in $V$, then there exists a
continuous linear functional $L$ on $V$ such that $L(v_0) = \rho(v_0)$
and $|L(v)| \leq \rho(v)$ for all $v \in V$.
\end{theorem}

\begin{proof}
We let $\bF$ denote either $\bR$ or $\bC$ and
define a linear functional $l$ on the one-dimensional subspace $\bF v_0$
of $V$ by setting
\[
 l(\lambda v_0) = \lambda \rho(v_0).
\]
Since $|l(\lambda v_0)| = |\lambda \rho(v_0)| = \rho(\lambda v_0)$,
we can apply either the \hyperref[hahn-banach]{real Hahn-Banach theorem}
or the \hyperref[complex-hahn-banach]{complex Hahn-Banach theorem}
to construct a linear functional $L$ on $V$ such that $L|_{\bF v_0} = l$
and $|L(v)| \leq \rho(v)$ for all $v \in V$. $L$ is a linear extension
of $l$, and so $L(v_0) = l(v_0) = \rho(v_0)$. Furthermore,
the bound $|L(v)| \leq \rho(v)$ implies that $L$ is continuous at $v=0$,
whence by linearity it is continuous on $V$.
\end{proof}

For a large class of topological vector spaces known as
\emph{locally convex spaces},\index{locally convex space} a stronger extension theorem can be
established. Since we will not need such generality, we state
and prove a special case of the theorem for Fr\'{e}chet spaces.
See \S\S\ref{fr-locally-convex-topological-spaces} for a discussion
of locally convex spaces.

\begin{cor}\label{extension-of-linear-functionals-on-frechet-spaces}
Let $V$ be a Fr\'{e}chet space. For every nonzero vector $v_0$ in $V$,
there exists a continuous seminorm $\rho$ on $V$ such that $\rho(v_0) \neq 0$.
Consequently, there is a continuous linear functional $l$ on $V$ such that
$l(v_0) \neq 0$ and $|l(v)| \leq \rho(v)$ for all $v \in V$.
\end{cor}

\begin{proof}
Let $\{\rho_n:n \in \bN\}$ be a collection of seminorms generating
the topology on $V$, all of which must be continuous.
Given a fixed nonzero vector $v_0$ in $V$, we
observe that there exists an $N \in \bN$ such that $\rho_N(v_0) \neq 0$.
If not, then the distance between $v_0$ and the zero vector in the metric
generated by $\{\rho_n:n \in \bN\}$ is zero, which is evidently absurd.
Theorem \ref{extension-of-linear-functionals} can now be applied
to $\rho_N$, whereby obtaining a continuous linear functional
continuous linear functional $l$ on $V$ such that $l(v_0) = \rho(v_0) \neq 0$
and $|l(v)| \leq \rho(v)$ for all $v \in V$.
\end{proof}

If $V$ is a Banach space, then the extension theorem can be strengthened
as follows:

\index{Banach space|(}
\begin{cor}\label{extension-of-linear-functionals-on-banach-spaces}
Let $V$ be a Banach space. For every nonzero vector $v_0$ in $V$,
there exists a continuous linear functional $l$ such that $l(v_0) = \|v_0\|$
and $\|l\|=1$. 
\end{cor}

\begin{proof}
Fix a nonzero vector $v_0 \in V$. Since the norm $\|\cdot\|$ is a continuous
seminorm on $V$ with $\|v_0\| \neq 0$, we apply Corollary
\ref{extension-of-linear-functionals-on-frechet-spaces} to construct a
continuous linear functional $l$ such that $l(v_0) = \|v_0\|$
and $|l(v)| \leq \|v\|$ for all $v \in V$. The second conclusion implies that
$\|l\| \leq 1$, whence the first conclusion implies that $\|l\| = 1$.
\end{proof}
\index{Hahn-Banach theorem|)}
\index{Fr\'{e}chet space|)}

\subsection{Complex Analysis of Banach-valued Functions}\label{complex-analysis-of-banach-valued-functions}

Let us now consider an application of the extension theorems
established in the previous subsection.
As alluded to in the beginning of the chapter, we shall extend the
Riesz-Thorin interpolation theorem (Theorem \ref{riesz-thorin})
to a more general framework in \S\ref{s-the_complex_interpolation_method}.
To do so, we shall need to consider Banach-valued functions on $\bC$.
It is therefore convenient to be able to apply complex-analytic methods to
functions on $\bC$ mapping into a complex Banach space.

\begin{defin}
Let $V$ be a complex Banach space and $O$ a connected open subset of $\bC$.
A function $f:O \to V$ is \emph{holomorphic}\index{holomorphicity!on a Banach space} if, for every bounded
linear functional $l$ on $V$, the composite map $lf$ is a holomorphic 
function on $O$ as a complex function of one variable.
\end{defin}

Since there are plenty of bounded linear functionals on $V$,
the definition is non-trivial. Indeed, it is sufficiently restrictive
that the theorems of complex analysis, such as Louiville's theorem,
continue to hold. Note that every bounded linear operator $T:V \to W$
between Banach spaces preserves the holomorphicity of $V$-valued
functions, for the composition of $T$ and an arbitrary bounded
linear functional on $W$ is a bounded linear functional on $V$.
Indeed, if $f$ is a $V$-valued holomorphic map, then, for each
linear functional $l$ on $W$, the composite map $lTf$ is a
complex-valued holomorphic map. It then follows that $Tf$
is a $W$-valued holomorphic map.
This technique is used in the proof of Theorem \ref{complex-interpolation-is-exact}.

If $V$ is a Banach space of linear operators\footnote{or, more generally,
a Banach algebra: see \S\S\ref{fr-approximations-to-the-identity}
for the definition.}, then we can talk about power series in
$V$, where the infinite sum is defined by the limit of the Cauchy sequence
of partial sums. In this case, a $V$-valued function on $O$ with a
power-series expansion in $V$ is holomorphic. We shall see an application of
this technique in the next subsection.

\subsection{Spectra of Operators on Banach Spaces}\label{s-spectra-of-operators-on-banach-spaces}

Among the central objects of study in finite-dimensional linear algebra
are eigenvalues and eigenvectors: an \emph{eigenvalue}\index{eigenvalue} of an $n$-by-$n$
matrix $A$ with complex entries is a complex number $\lambda$ such that
\[
 Av = \lambda v
\]
for some nonzero $n$-vector $v$, which is referred to as an \emph{eigenvector}\index{eigenvector}
of $A$ with respect to the eigenvalue $\lambda$. Writing $I$ to denote
the $n$-by-$n$ identity matrix, we see that the eigenvectors with
respect to an eigenvalue $\lambda$ are precisely the elements of
the nullspace of $A-\lambda I$. In other words, if a complex number $\lambda$
renders the matrix $A-\lambda I$ invertible, then the nullspace of
$A-\lambda I$ is trivial, whence there are no
``eigenvectors'' corresponding to $\lambda$. This implies that 
$\lambda$ is \emph{not} an eigenvalue of $A$.

Let us now consider a bounded linear operator
$T$ on a complex Banach space $V$. We define the \emph{resolvent set}\index{resolvent set}
$\rho(T)$ of $T$ to be the collection of complex numbers $\lambda$ such that
the operator
\[
 I - \lambda T
\]
is invertible. The \emph{spectrum}\index{spectrum} $\sigma(T)$ of $T$ is defined to be
the set $\bC \smallsetminus \rho(T)$. We observe that these
definitions are straightforward generalizations of the above observation.

Recall that finding the eigenvalues of an $n$-by-$n$ matrix $A$ amounts
to solving the polynomial equation
\[
 \det(A - \lambda I) = 0
\]
for $\lambda$. If $A$ has complex entries, then the fundamental theorem
of algebra guarantees that the roots always exist, whence every matrix
with complex entires has eigenvalues. As it turns out, our
infinite-dimensional generalization retains this property.

\begin{theorem}\label{spectrum-is-nonempty}\index{spectrum!is nonempty|(}
Let $V$ be a complex Banach space and $T:V \to V$ a bounded linear
operator. The spectrum $\sigma(T)$ of $T$ is nonempty.
\end{theorem}

To prove this result, we first observe that invertible operators
form an ``open set''.

\begin{lemma}\label{invertibility-is-open}
Let $V$ be a complex Banach space and $T:V \to V$ a bounded linear
operator. If $T$ is invertible, and if another bounded linear
operator $T':V \to V$ satisfies the norm estimate
\[
 \|T'\|_{V \to V} < \frac{1}{\|T^{-1}\|_{V \to V}}.
\]
then $T-T'$ is invertible.
\end{lemma}

\begin{proof}[Proof of lemma]
We assume for now that $T$ is the identity operator $I$.
In this case, the lemma asserts that all bounded linear operators
$I-T'$ with the norm estimate $\|T'\|_{V \to V} < 1$ is
invertible. In this case, the sequence of partial sums 
\[
 I, I+T', I+T'+(T')^2, \cdots
 \sum_{n=0}^N (T')^n, \cdots
\]
is Cauchy in the space $\ms{L}(V)$ of bounded linear endomorphisms
on $V$, which is Banach. Therefore, the operator
\[
 \sum_{n=0}^\infty (T')^n = \lim_{N \to \infty} \sum_{n=0}^N (T')^n
\]
is well-defined. We observe that
\[
 (I-T')\sum_{n=0}^\infty (T')^n = \lim_{N \to \infty} I - (T')^{n+1} = I
\]
and that
\[
 \left(\sum_{n=0}^\infty (T')^n \right)(I-T')
 = \lim_{N \to \infty} I-(T')^{n+1} = I,
\]
whence $I-T'$ is invertible and
\begin{equation}\label{power-series-inverse}
 (I-T')^{-1} = \sum_{n=0}^\infty (T')^n.
\end{equation}

We now consider an arbitrary bounded, invertible linear operator
$T:V \to V$. Fix a bounded linear operator $T':V \to V$ such that
\[
 \|T'\|_{V \to V} < \frac{1}{\|T^{-1}\|_{V \to V}}.
\]
We factor
\[
 T-T' = T(I-T^{-1}T')
\]
and observe that
\[
 \|T^{-1}T'\|_{V \to V} \leq \|T^{-1}\|_{V \to V}\|T'\|_{V \to V}
 < \frac{\|T^{-1}\|_{V \to V}}{\|T^{-1}\|_{V \to V}} = 1.
\]
Therefore, the argument carried out in the above paragraph shows that
$I-T^{-1}T'$ is invertible. Since $T$ is also invertible, it follows
that $T-T'$ is invertible, as was to be shown.
\end{proof}

We also establish the following computational result:

\begin{lemma}\label{operator-algebra-formula}
If $T$ and $T'$ are bounded, invertible linear operators on a Banach space
$V$, then
\[
 T^{-1} - (T')^{-1} = T^{-1}(T'-T)(T')^{-1}.
\]
\end{lemma}

\begin{proof}[Proof of lemma]
Observe that
\[
T^{-1} -(T')^{-1} = T^{-1}(T')(T')^{-1} - T^{-1}T(T')^{-1}
= T^{-1}(T'-T)(T')^{-1},\]
as was claimed.
\end{proof}

We now return to the proof of the main theorem.

\begin{proof}[Proof of theorem \ref{spectrum-is-nonempty}]
Let $T:V \to V$ be a bounded linear operator. If $\lambda \in \rho(T)$,
then Lemma \ref{invertibility-is-open} implies that
\[
 (\lambda-\ve)I-T = (\lambda I - T) - \ve I
\]
is invertible for sufficiently small $\ve \in \bC$. It follows that
$\rho(T)$ is open, and so $\sigma(T)$ is closed.

We claim that $\lambda \mapsto (\lambda 1 - T)^{-1}$ is
a $V$-valued holomorphic map\index{holomorphicity!on a Banach space} on $\rho(T)$. For each fixed $\mu \in \rho(T)$,
Lemma \ref{operator-algebra-formula} implies that
\begin{eqnarray*}
 \frac{(\lambda I - T)^{-1} - (\mu I - T)^{-1}}{\lambda - \mu}
 &=& \frac{(\lambda I - T)^{-1} (\mu I - \lambda I) (\mu I - T)^{-1}}{\lambda - \mu} \\
 &=& -(\lambda I - T)^{-1}(\mu I - T)^{-1}.
\end{eqnarray*}
Therefore, we have
\[
 \lim_{\mu \to \lambda} \frac{(\lambda I - T)^{-1} - (\mu I - T)^{-1}}{\lambda - \mu}
 = \lim_{\mu \to \lambda} -(\lambda I - T)^{-1}(\mu I - T)^{-1}
 = -(\lambda I - T)^{-2},
\]
and so
\[
 \lim_{\mu \to \lambda} \frac{l((\lambda I - T)^{-1}) - l((\mu I - T)^{-1})}{\lambda - \mu}
 = l((\lambda I - T)^{-2})
\]
for each bounded linear functional $l$ on $V$. It follows that
$\lambda \mapsto (\lambda I - T)^{-1}$ is holomorphic.

Let us now suppose for a contradiction that $\rho(T) = \bC$. This,
in particular, implies that $\lambda \to l((\lambda I - T)^{-1})$ is an
entire function for every bounded linear functional $l$ on $V$.
We apply the power-series expansion 
Since
continuous maps send compact sets to compact sets, this map is bounded
on every compact subset of $\rho(T)$. We apply the power-series expansion
(\ref{power-series-inverse}), employed in the proof of Lemma
\ref{invertibility-is-open}, to $(\lambda I - T)^{-1}$:
\[
 (\lambda I - T)^{-1} = \lambda^{-1}(I-\lambda^{-1}T)^{-1}
 = \lambda^{-1} \sum_{n=0}^\infty \lambda^{-n} T^n.
\]
It follows that
\[
 \|(\lambda I - T)^{-1}\|_{V \to V} \leq \frac{1}{|\lambda| - \|T\|_{V \to V}},
\]
and so $\|(\lambda I - T)^{-1}\| \to 0$ as $|\lambda| \to \infty$.

Furthermore, we have
\[
 |l((\lambda I - T)^{-1})|
 \leq \|l\|_{V \to \bC} \|(\lambda I - T)^{-1}\|_{V \to V}
 \leq \frac{\|l\|_{V \to \bC}}{|\lambda\| - \|T\|_{V \to V}}
\]
for each bounded lienar functional $l$ on $V$, whence
$\lambda \mapsto l((\lambda I - T)^{-1})$ vanishes at infinity.
Louiville's theorem can now be applied to conclude that
this map is the zero map, regardless of our choice of $l$.
Given a fixed $\lambda \in \rho(T)$, however, $(\lambda I - T)^{-1}$
is invertible and thus not the zero operator. We invoke
Theorem \ref{extension-of-linear-functionals-on-banach-spaces}
to construct a bounded linear functional $l$ on $V$ such that
$\|l\| = 1$ and
\[
 \|l((\lambda I - T))^{-1}\|_{V \to \bC}
 = \|(\lambda I - T)^{-1}\|_{V \to V}
 > 0.
\]
This is evidently absurd, since the map $\lambda \mapsto l((\lambda I - T)^{-1})$
was shown to be the zero map. It now follows that $\rho(T) \neq \bC$,
whereby we conclude that $\sigma(T)$ is nonempty.
\end{proof}

\index{spectrum!is nonempty|)}

We remark that the above theorem does \emph{not} guarantee the
existence of eigenvalues proper for all bounded operators.
Indeed, the right-shift operator $R:l^2(\bZ) \to l^2(\bZ)$ defined
on the space $l^2(\bZ)$ of square-summable sequences by
\[
 R((a_n)_{n \in \bZ}) = (a_{n+1})_{n \in \bZ}
\]
admits no eigenvalues or eigenvectors, for the identity
\[
 \lambda a_n = a_{n+1}
\]
does not hold for all sequences $(a_n)_{n \in \bZ}$, regardless of
the value of $\lambda$.

The proof of the theorem goes through \emph{verbatim} if we substitute
the bounded operators with elements of a Banach algebra. See
\ref{fr-approximations-to-the-identity} for the definition of
a Banach algebra.

It is sometimes useful to classify operators by the content of their
spectra.

\begin{defin}\label{positive-and-negative-operators}\index{operator!positive and negative}
A bounded operator $T:V \to V$ on a complex Banach space $V$
is \emph{positive} if the spectrum consists of positive real numbers,
and \emph{negative} if the spectrum consists of negative real numbers.
\end{defin}

We shall prove a theorem about positive operators in \S\ref{s-interpolation-of-hilbert-spaces}.
An example of a positive operator can be found in \S\ref{s-sobolev-spaces}.

\subsection{Compactness of the Unit Ball}

We now consider a property of finite-dimensional vector spaces that do \emph{not}
generalize to infinite-dimensional spaces. Recall the Heine-Borel theorem,\index{Heine-Borel theorem} which
guarantees that the closed and bounded sets in $\bR^d$ and $\bC^d$ are compact.
As it turns out, the norm topology on an infinite-dimensional Banach space
has ``too many open sets'' to preserve this property.

\begin{theorem}\label{noncompactness-of-the-unit-ball}\index{unit ball!noncompactness of}
The closed unit ball in a Banach space $V$ is compact if and only if
$V$ is finite-dimensional.
\end{theorem}

We will not have an occasion to use this theorem, so we omit the
proof: see, for example, Section 5.2, Theorem 6 in \cite{Peter_D_Lax:B2002}.\index{Lax, Peter D.}
The proof can be sketched easily if $V$ is a separable Hilbert space.
Recall the standard result in Hilbert-space theory that there is an
orthonormal basis $\{e_n:n \in \bN\}$ of $V$: intuitively, there are
countably many axes in $V$ that are perpendicular to one another.
If we construct an open cover of $V$ such that each ``axis'' of
the unit ball in $V$ is covered by an open set that does not overlap
much with the rest of the cover, then none of the countably many
sets covering these axes can be removed without losing the open-cover
status.

Given the usefulness of compact sets, we seek a way to reduce the
number of open sets, thereby increasing the number of compact sets.
To this end, we recall that the dual of a Banach space $V$ is a Banach
space, and that the map $v \mapsto L_v$ from $V$ to its double dual
$V^{**}$ defined by $L_v(l) = lv$ is an isometric embedding.

\begin{defin}\label{weak-star}\index{weak-* topology}
Let $V$ be a Banach space.
The \emph{weak-* topology} on $V^*$ with respect to $V$
is the topology generated by the sets $L_v^{-1}(O)$,
where $L_v$ is the linear functional in the above proposition and $O$ an arbitrary
open set in $\bC$. 
\end{defin}

With this new topology, Theorem \ref{noncompactness-of-the-unit-ball}
can now be reversed.

\begin{theorem}[Banach-Alaoglu]\label{banach-alaoglu}\index{Banach-Alaoglu theorem}\index{unit ball!weak-* compactness \\ of|see{Banach-Alaoglu theorem}}
If $V$ is a Banach space, then the unit ball in $V^*$ is compact in the weak-* topology.
\end{theorem}

\begin{proof}
Let $\bF$ denote either $\bR$ or $\bC$. We consider the product space
$\bF^V$ consisting of all functions from $V$ into $\bF$, equipped with
the standard product topology. An arbitrary element of $\bF^V$ shall be denoted
by $\omega = (\omega_v)_{v \in V}$, where $\omega_v$ is the value of $\omega$
evaluated at $v$.

We now consider the topological embedding $\Phi: V^* \to \bF^V$
defined by $\Phi(l) = (\omega_v)_{v \in V}$, where $\omega_v = l(v)$.
It suffices to check that $\Phi(B)$ is compact, where
$B$ is the closed unit ball in $V^*$.
We observe that $\Phi(B)$ is the collection of $\omega = (\omega_v)_{v \in V}$
in $\bF^V$ such that $|\omega_v| \leq \|v\|_V$, $\omega_{v+w} = \omega_v + \omega_w$,
and $\omega_{\lambda v} = \lambda \omega_v$ for all $\lambda \in \bF$ and
$v,w \in V$. By setting
\begin{eqnarray*}
 K_1 &=& \{\omega \in \bF^V : |\omega_v| \leq \|v\|_V \mbox{ for all } v \in V \} \\
 K_2 &=& \{\omega \in \bF^V : \omega_{v+w} = \omega_v + \omega_w \mbox{ and }
 \omega_{\lambda v} = \lambda \omega_v \\
 & & \mbox{ for all } \lambda \in \bF \mbox{ and }
 v,w \in V\},
\end{eqnarray*}
we see that $\Phi(B) = K_1 \cap K_2$.

Observe that $K_1$ is a product of the closed intervals
$[-\|v\|, \|v\|]$ as $v$ runs through $V$, and so Tychonoff's
theorem\index{Tychonoff's theorem} implies that $K_1$ is compact. Furthermore, the sets
\begin{eqnarray*}
 E_{v,w} &=& \{\omega \in \bF^V : \omega_{v+w} - \omega_v - \omega_w = 0\} \\
 F_{\lambda,v} &=& \{\omega \in \bF^V : \omega_{\lambda v} - \lambda \omega_v = 0\}
\end{eqnarray*}
are closed for each fixed $\lambda \in \bF$, whence
\[
 K_2 = \left( \bigcap_{v,w \in V} E_{v,w} \right) \cap
 \left( \bigcap_{\substack{ v \in E \\ \lambda \in \bF}} F_{\lambda, v} \right)
\]
is closed. Therefore, $\Phi(B) = K_1 \cap K_2$ is compact, and so is $B$.
\end{proof}

\subsection{Bounded Linear Maps between Banach Spaces}

An important property of complete metric spaces is the \emph{Baire category
theorem},\index{Baire category theorem} which states that no complete metric space can be written as a
countable union of \emph{nowhere dense sets},\index{nowhere dense sets}\index{nowhere dense sets|seealso{Baire category theorem}} viz., sets whose interiors of
their closures are empty. In this subsection, we apply the category theorem
to the study of linear maps between banach spaces and derive two powerful
consequences.

Recall that a function $f:X \to Y$ between topological spaces $X$ and $Y$
is \emph{open}\index{open map} if $f(U)$ is open in $Y$ for each open set $U$ in $Y$ and
The first result concerns open linear maps between Banach spaces.

\index{open mapping theorem|(}
\begin{theorem}[Banach-Schauder, open mapping theorem]\label{open-mapping-theorem}\index{Banach-Schauder theorem|see{open mapping theorem}}
Let $V$ and $W$ be Banach spaces and $T:V \to W$ a bounded linear
transformation. If $T$ is surjective, then $T$ is open.
\end{theorem}

\begin{proof}
We write $B^V_r(x)$ and $B^W_r(y)$ to denote the open balls of radius $r$
centered at $x \in V$ and $y \in W$, respectively. If we can show that
$T(B_1^V(0))$ contains an open ball centered at the origin, then the
linearity of $T$ establishes the desired result. To this end, we shall
first show that $\overline{T(B_1^V(0))}$ contains an open ball
centered at the origin. Since $T$ is surjective,
\[
 W = \bigcup_{n=1}^\infty T(B_n^V(0)),
\]
whence the Baire category theorem implies the existence of an integer $n_0$
such that $T(B_{n_0}^V(0))$ is \emph{not} nowhere dense. Therefore,
$\overline{T(B_{n_0}^V(0))}$ has a nonempty interior, and so the linearity
of $T$ yields a point $w_0 \in W$ and a real number $\ve>0$ such that
\[
 B^W_\ve(w_0) \subseteq \overline{T(B_1^V(0))}. 
\]

We now fix a point $v_1 \in B_1^V(0)$ such that $w_1 =T(v_1)$ satisfies
the distance estimate $\|w_1 - w_0\| < \ve/2$. If $w \in B_{\ve/2}^W(0)$,
then
\[
 \|(w-w_1)-w_0\| \leq \|w\| + \|w_1-w_0\| < \frac{\ve}{2} + \frac{\ve}{2} = \ve,
\]
and so $w-w_1 \in B^W_\ve(w_0) \subseteq \overline{T(B_1^V(0))}$. Since
$w = T(v_1) + w - w_1$, the linearity of $T$ implies that $w \in
\overline{T(B_2^V(0))}$. Once again, the linearity of $T$ establishes the
inclusion
\[
 B^W_{\ve/4}(0) \subseteq \overline{T(B_1^V(0))},
\]
and our claimed is proved. We remark that we can assume without loss
of generality that
\[
 B^W_1(0) \subseteq \overline{T(B_1^V(0))},
\]
by rescaling $T$ in the above inclusion if necessary. This, in particular,
implies that
\begin{equation}\label{open-lemma}
 B^W_{1/2^n}(0) \subseteq \overline{T(B_{1/2^n}^V(0))}.
\end{equation}
for each $n \in \bN$.

We now show that
\begin{equation}\label{open-result}
 B^W_{1/2}(0) \subseteq T(B_1^V(0)),
\end{equation}
which then establishes the theorem. Fix $w \in B^W_{1/2}(0)$, which
by (\ref{open-lemma}) is in $\overline{T(B_{1/2}^V(0))}$. We can then
find $v_1 \in B_{1/2}^V(0)$ with the distance estimate
$\|w_1 - T(v_1)\| < 2^{-2}$, or, equivalently, the inclusion
\[
w_1 - T(v_1) \in B_{1/2^2}^W(0).
\]
Applying (\ref{open-lemma}) once again, we see that 
$w_1 - T(v_1) \in \overline{T(B_{1/2^2}^V(0))}$, whence we can find
$v_2 \in B_{1/2^2}^v(0)$ such that
\[
 w_1 - T(v_1) - T(v_2) \in B_{1/2^3}^W(0).
\]
Continuing the process, we obtain a sequence $(v_n)_{n=1}^\infty$
of vectors in $V$ such that $\|v_n\| < 2^{-n}$. The sequence
$(v_1 + \cdots + v_N)_{N=1}^\infty$ of partial sums is therefore Cauchy,
and the completeness of $V$ furnishes a limit $v = \sum v_n$ with
the norm estimate
\begin{equation}\label{open-norm}
 \|v\| < \sum_{n=1}^\infty 2^{-n} = 1.
\end{equation}
Now, $T$ is continuous, and
\[
 \left\|w - \sum_{n=1}^N T(v_n) \right\| < 2^{-n-1}
\]
for each $N \in \bN$, whence it follows that $\|w-T(v)\| = 0$, or $w = T(v)$.
Combined with (\ref{open-norm}), this establishes (\ref{open-result}), and
the proof is complete.
\end{proof}
\index{open mapping theorem|)}

The second result concerns closed linear maps between Banach spaces.
An operator $T:V \to W$ between two normed linear spaces $V$ and $W$
is \emph{closed}\index{operator!closed}\index{closed linear maps|see{operator, \\ closed}} if $v_n \to v$ in $V$ and $Tv_n \to w$ implies
$Tv = w$. This is \emph{not} equivalent to the notion of a closed map
in point-set topology, which refers to a function $f:X \to Y$ between
topological spaces $X$ and $Y$ such that $f(U)$ is closed in $Y$ whenever
$U$ is closed in $X$.

\index{closed graph theorem|(}
\begin{theorem}[Closed graph theorem]\label{closed-graph-theorem}
Let $V$ and $W$ be Banach spaces and $T:V \to W$ a linear
transformation. If $T$ is closed, then $T$ is bounded.
\end{theorem}

Why the name \emph{closed graph}? Recall that the \emph{graph}\index{graph}\index{graph!closed|see{closed graph theorem}}
of a linear map $T:X \to Y$ between two Banach spaces $V$ and $W$
is defined to be the set
\[
 G_T = \{(v,T(v)) \in V \times W : v \in V\}.
\]
If $T$ is closed, then $(v_n,T(v_n)) \to (v,w)$ implies that
$(v_n,T(v_n)) \to (v,Tv)$, which is in $G_T$. Conversely, if
$G_T$ is closed, then $v_n \to v$ and $Tv_n \to w$ implies that
the limit $\lim (v_n,Tv_n) = (v,w)$ must be in $G_T$, whence
$w = Tv$. It follows that the closedness of $T$ is equivalent
to the closedness of its graph $G_T$.

\begin{proof}
We first note that the normed linear space $V \times W$
with the norm $\|(v,w)\|_{V \times W} = \|v\|_V + \|w\|_W$
is a Banach space\footnote{The second half of Proposition
\ref{plus-and-cap-are-well-defined} establishes a strenghtening
of this result for the internal sum $V + W$. To see why
this is a generalization, we recall that the product
$V \times W$ is isomorphic to the direct sum $V \oplus W$,
which equals the internal sum $V + W$ if and only if $V$ and $W$
have a trivial intersection as subspaces $V \oplus W$.} Indeed,
the norm properties are established. If $\{(v_n,w_n)\}_{n=1}^\infty$
is a Cauchy sequence in $V \times W$, then $(v_n)_{n=1}^\infty$ and
$(w_n)_{n=1}^\infty$ are Cauchy in $V$ and $W$, respectively,
and so $v_n \to v$ and $w_n \to w$ for some $v \in V$ and $w \in W$.
It now suffices to note that
\[
 \|(v_n,w_n)-(v,w)\|_{V \times W}
 = \|(v_n-v,w_n-w)\|_{V \times W}
 = \|v_n - v\|_V + \|w_n - w\|_W,
\]
which converges to 0 as $n \to \infty$. It follows that $V \times W$
is a Banach space, whence the closed subspace $G_T$ of $V \times W$
is also a Banach space.

We now consider the projection maps $P_V:G_T \to V$ and $P_W:G_T \to W$
defined by
\[
 P_V(v,Tv) = v \htwo \mbox{and} \htwo P_W(v,Tv) = Tv.
\]
Since $P_V$ is a bijective, bounded linear map, the
\hyperref[open-mapping-theorem]{open mapping theorem}\index{open mapping theorem} implies that
$P_V$ is an open map. This, in particular, shows that the
inverse $P_V^{-1}$ is a bounded linear map. Likewise, $P_W$ is
a bounded linear map, and so the composition
\[
 T = P_W \circ P_V^{-1}
\]
is bounded as well, thereby establishing the theorem.
\end{proof}

We remark that the open mapping theorem can be proven using the
closed graph theorem, thereby establishing the equivalence between
the two results. We shall see an application of the closed graph
theorem in the next subsection. See \S\S\ref{fr-fourier-inversion-problem}
for another important corollary of the Baire category theorem
and its application to the Fourier inversion problem.
\index{closed graph theorem|)}
\index{Banach space|)}

\subsection{Spectral Theory of Self-Adjoint Operators}\label{s-spectral-theory-of-self-adjoint-operators}

\index{operator!unbounded|see{self-adjoint operator}}
We now restrict our attention to Hilbert spaces, on which we can establish
substantial generalizations of many results from finite-dimensional
linear algebra. In particular, we shall focus on the problem of diagonalization
in this subsection.

Recall that an $n$-by-$n$ matrix $A$ with complex entries is \emph{self-adjoint}\index{self-adjoint!matrix}
if $A$ equals the conjugate transpose $A^*$ of $A$, and \emph{unitary}\index{unitary!matrix}
if $A^{-1}$ equals $A^*$. A standard result in linear algebra is that
every self-adjoint matrix $A$ is \emph{unitarily equivalent} to a diagonal matrix,
viz., there exists a diagonal matrix $D$ and a unitary matrix $U$ such that
$D = U^{-1}AU$. Since the columns of a unitary matrix form an orthonormal basis,
it follows that every self-adjoint matrix can be diagonalized with respect to
an orthonormal basis. This is the \emph{spectral theorem}.\index{spectral theorem!on finite-dimensional \\ vector spaces}

The spectral theorem can be generalized considerably. To this end, we first
define the infinite-dimensional generalization of self-adjoint matrices.

\begin{defin}
Let $H$ be a complex Hilbert space. A linear operator $T:H \to H$ is
\emph{self-adjoint}\index{self-adjoint!operator} if
\[
 \langle Tv,w \rangle_H = \langle v, Tw \rangle_H
\]
for all $v,w \in H$.
\end{defin}

We show that self-adjoint operators on Hilbert spaces must be bounded .

\begin{prop}[Hellinger-Toeplitz]\label{hellinger-toeplitz}\index{Hellinger-Toeplitz theorem}\index{self-adjoint!operators are bounded|see{\\ Hellinger-Toeplitz theorem}}
If $T$ is a self-adjoint operator on a Hilbert space $H$, then $T$ is bounded.
\end{prop}

\begin{proof}
Fix vectors $v,w,v_1,v_2,\ldots,v_n,\ldots$ in $H$ such that
$v_n \to v$ and $Tv_n \to w$. By the self-adjointness of $T$, we have
\[
 \langle Tv_n,x \rangle_H = \langle v_n, Tx \rangle_H
\]
for each $n \in \bN$ and every $x \in H$, whence taking the limit yields
\[
 \langle w,x \rangle_H = \langle v, Tx \rangle_H.
\]
Applying the self-adjointness of $T$ once again, we have the identity
\[
 \langle w,x \rangle_H = \langle Tv, x \rangle_H
\]
for all $x \in H$. Therefore,
\[
 \langle w - Tv , x \rangle_H = 0
\]
for all $x \in H$, and, in particular,
\[
 \langle w - Tv, w- Tv \rangle_H = \|w- Tv\|^2_H = 0.
\]
It follows that $w = Tv$, and so $T$ is closed.
We now invoke the closed graph theorem (Theorem \ref{closed-graph-theorem})\index{closed graph theorem}
to conclude that $T$ is bounded.
\end{proof}

Note, however, that there is nothing ``topological'' about the definition
of self-adjointness, and so we seek to generalize the definition to
unbounded operators. In light of the above proposition, we are forced to
define unbounded self-adjoint operators only on proper subspaces of the
Hilbert space in question.

\begin{defin}
Let $H_1$ and $H_2$ be complex Hilbert spaces,
$D$ a dense subspace of $H_1$,
and $T:D \to H_2$ a linear operator. We define $D^*$ to be the
collection of all $w \in H_2$ such that, for each $v \in H_1$,
there exists a vector $w^* \in H_1$ satisfying the identity
\[
 \langle Tv,w \rangle_{H_2} = \langle v, w^* \rangle_{H_1}.
\]
The \emph{adjoint}\index{adjoint}\index{operator!adjoint of an|see{adjoint}} of $T$ is the operator $T^*:D^* \to H_1$
defined to be
\[
 T^*w = w^*
\]
at each $w \in D^*$, so that
\[
 \langle Tv,w \rangle = \langle v, T^*w \rangle
\]
for each $v \in D$ and every $w \in D^*$. $T$ is \emph{self-adjoint}\index{self-adjoint!operator}\index{operator!self-adjoint|see{self-adjoint}}
if $D = D^*$ and $T = T^*$.
\end{defin}

A remark is in order. The density of $D$ guarantees that
there is only one $w^*$ for each $w \in D$, whence $T^*w$ is
unambiguously defined. Of course, $D^*$ is a linear subspace of
$H_2$, and $T^*$ a linear operator. Henceforth, we shall
write $\mc{D}(T)$ to denote $D$, and $\mc{D}(T^*)$ to denote $D^*$.
Furthermore, we shall speak of an unbounded operator $T$
\emph{on} $H_1$, although it is defined only on $\mc{D}(T)$.

Let us now return to the task at hand. A \emph{unitary operator}\index{unitary!operator}\index{operator!unitary|see{unitary}}
on a Hilbert space $H_1$ into another Hilbert space $H_2$
is a bounded linear operator $U:H_1 \to H_2$
such that $T^{-1} = T^*$. The spectral theorem for linear operators
on \emph{separable} Hilbert spaces can now be stated as follows:

\index{spectral theorem!on separable Hilbert spaces|(}
\begin{theorem}[Spectral theorem]\label{spectral-theorem}
Let $T$ be an unbounded self-adjoint operator on a separable
Hilbert space $H$. There exists a measure space $(X,\mf{M},\mu)$,
a unitary operator $U:L^2(X,\mu) \to H$, and a real-valued
$\mu$-measurable function $a$ on $X$ such that
\[
 U^{-1}TU u(x) = a(x) u(x)
\]
for all $Uu \in \mc{D}(T)$. Furthermore, $Uu \in \mc{D}(T)$
if and only if $au \in L^2(X,\mu)$. 
\end{theorem}

This formulation of the spectral theorem is that of
Theorem 1.7 in Chapter 8 of \cite{Michael_E_Taylor:B2010-2}.\index{Taylor, Michael E.}
The proof is rather elaborate and requires heavy machinery,
so we omit it. See Chapter 8 of \cite{Michael_E_Taylor:B2010-2},\index{Taylor, Michael E.}
Chapter 32 of \cite{Peter_D_Lax:B2002},\index{Lax, Peter D.} Chapter 13 of
\cite{Walter_Rudin:B1991},\index{Rudin, Walter} or Chapter XI of \cite{Kosaku_Yosida:B1980}\index{Yosida, Kosaku}
for an exposition of spectral theory of unbounded operators.
For our purposes, it suffices to consider an extension of
the spectral theory of bounded operators on Banach spaces
discussed in \S\S\ref{s-spectra-of-operators-on-banach-spaces}.

The \emph{resolvent set}\index{resolvent set!of a self-adjoint operator} $\rho(T)$ of a self-adjoint operator
$T:H_1 \to H_2$ between two complex Hilbert spaces $H_1$ and
$H_2$ is the collection of all complex numbers $\lambda$ such
that $T-\lambda I$ is a bijective map from $\mc{D}(T)$ onto $H_2$.
The \emph{spectrum}\index{spectrum!of a self-adjoint operator} $\sigma(T)$ of $T$ is the set
$\bC \smallsetminus \rho(T)$. It can be shown\footnote{See,
for example, Theorem 5 in Chapter 31 of \cite{Peter_D_Lax:B2002}.}
that the spectrum of a self-adjoint operator consists of
real numbers. This framework allows us to extend Definition
\ref{positive-and-negative-operators} and consider positive
or negative unbounded self-adjoint operators.\index{operator!positive and negative} We shall prove
a theorem about positive unbounded self-adjoint operators in
\S\S\ref{s-interpolation-of-hilbert-spaces}.
\index{spectral theorem!on separable Hilbert spaces|)}
 
\section{The Complex Interpolation Method}\label{s-the_complex_interpolation_method}

 \index{Calder\'{o}n, Alberto P.!interpolation|see{complex interpolation}}
\index{complex interpolation|(}
In this section, we study the basics of
Alberto Calder\'{o}n's complex method of interpolation,
following \cite{Alberto_P_Calderon:J1964}, \cite{Bergh_Lofstrom:B1976},\index{Bergh, J\"{o}ran}\index{L\"{o}fstr\"{o}m, J\"{o}rgen}
and \cite{Michael_E_Taylor:B2010}.\index{Taylor, Michael E.}
Calder\'{o}n's theory serves as a turning point for the development of interpolation
theory, providing a new abstract framework that allows
the theory to grow beyond the realm of classical harmonic analysis.
Despite its prevalence in standard expositions
of modern interpolation theory, the language of category theory
is avoided in this section. See \S\S\ref{fr-abstract-interpolation-theory}
for a brief sketch of the categorical formulation.

\subsection{Complex Interpolation}

The main novelty of the theory is the shift in focus from
interpolation of \emph{operators} to interpolation of \emph{spaces}.
Take the Fourier transform operator, for example. 
We have used the Riesz-Thorin interpolation theorem (Theorem \ref{riesz-thorin})\index{Riesz-Thorin interpolation \\ theorem}
to the $L^1$ Fourier transform and the $L^2$ Fourier transform to obtain
a new Fourier transform operator on, say, $L^{1.5}$. It can also be said,
however, that we have determined $L^{1.5}$ to be an ``interpolation space''
between $L^1$ and $L^2$. To make this notion precise, we first pick out
the pairs of spaces which we can interpolate.

\begin{defin}\label{banach-couples}\index{Banach couple}
A (complex) \emph{Banach couple} is an order pair $(B_0,B_1)$ of (complex) Banach spaces
such that both $A_0$ and $A_1$ are continuously embedded into a Hausdorff
topological vector space $V$.\index{topological vector space}
\end{defin}

Recall from \S\ref{s-the_fourier_transform} that we have defined the
$L^p$ Fourier transform by defining the $L^1+L^2$ Fourier transform\index{Fourier transform!on L1 + L2@on $L^1+L^2$|(}
operator and restricting it onto the ``interpolation spaces'' $L^p$.
Had $L^1+L^2$ not be well-defined, this line of reasoning would have been
nonsensical. The continuous-embedding criterion guarantees that $B_0+B_1$
is well-defined, as we shall see in Proposition \ref{plus-and-cap-are-well-defined}
below.

We also recall that Proposition \ref{norm-estimates-of-intermediate-lp-spaces}
provided another way of defining the $L^p$ Fourier transform: namely,
extending via Theorem \ref{norm-preserving-extension} the Fourier transform
on $L^1 \cap L^2$. This extension, moreover, agreed with the restriction of
the $L^1 + L^2$ Fourier transform.\index{Fourier transform!on L1 + L2@on $L^1+L^2$|)}

We would like to model our abstract framework on the two modes of interpolation
discussed above. This, above all, requires the two spaces $B_0 + B_1$ and
$B_0 \cap B_1$ to be well-defined and well-behaved, which we establish promptly.

\begin{prop}\label{plus-and-cap-are-well-defined}
If $(B_0,B_1)$ is a Banach couple, then $B_0 \cap B_1$ is a Banach space with
the norm
\[
 \|v\|_{B_0 \cap B_1} = \max\{\|v\|_{B_0} , \|v\|_{B_1}\},
\]
and $B_0 + B_1$ is a Banach space with the norm 
\[
 \|v\|_{B_0 + B_1} = \inf_{v = v_0 + v_1} \|v_0\|_{B_0} + \|v_1\|_{B_1}.
\]
Furthermore, both $B_0 \cap B_1$ and $B_0 + B_1$ are continuously embedded
into the ambient Hausdorff topological vector space.
\end{prop}

\begin{proof}
We first show that $B_0 \cap B_1$ is a Banach space.
It is easy to check that $\|\cdot\|_{B_0 \cap B_1}$ is a norm.
If $(v_n)_{n=1}^\infty$ is a Cauchy
sequence in $B_0 \cap B_1$, then we can find $v \in B_0$ and $v' \in B_1$
such that $\|v_n - v\|_{B_0} \to 0$ and $\|v_n - v'\|_{B_1} \to 0$ as $n \to \infty$.
Norm convergences in $B_0$ and $B_1$ must agree with convergence in
the topology of the ambient Hausdorff topological vector space, whence the limit
must be unique. Therefore, $v$ equals $v'$ and is consequently in $B_0 \cap B_1$.
Furthermore, it is now evident that $(v_n)_{n=1}^\infty$ converges to $v$ in
the norm topology of $B_0 \cap B_1$, thus establishing the completeness of
$\|\cdot\|_{B_0 \cap B_1}$.

We now turn to $B_0 + B_1$. Clearly, $\|\cdot\|_{B_0 + B_1}$ is a norm.
If $(v_n)_{n=1}^\infty$ is a Cauchy sequence in $B_0 + B_1$, then
we can find sequences $(w_n)_{n=1}^\infty$ and $(x_n)_{n=1}^\infty$ in
$B_0$ and $B_1$, respectively, such that $v_n = w_n + x_n$ for each $n \in \bN$.
Since $(w_n)_{n=1}^\infty$ and $(x_n)_{n=1}^\infty$ are Cauchy sequences in
$B_0$ and $B_1$, respectively, we can find $w \in B_0$ and $x \in B_1$
such that $\|w_n - w\|_{B_0} \to 0$ and $\|x_n - x\|_{B_1} \to 0$
as $n \to \infty$. It now suffices to observe that
\[
 \lim_{n \to \infty} \|v_n - (w+x)\|_{B_0 + B_1}
 \leq \lim_{n \to \infty} \|v_n - w_n\|_{B_0} + \lim_{n \to \infty} \|v_n-x_n\|_{B_1}
 = 0,
\]
whence $B_0+B_1$ is complete.

We now let $V$ be the ambient space in which $B_0$ and $B_1$ are continuously embedded.
We can take the embedding $B_0 \cap B_1 \hookrightarrow V$ to be the restriction
of either the embedding $B_0 \hookrightarrow V$ or the embedding $B_1 \hookrightarrow
V$ on $B_0 \cap B_1$. As for $B_0 + B_1$ we note that both $B_0$ and $B_1$ can be
identified with complete and thus closed subsets of $V$,
whence the standard gluing lemma\footnote{See, for example, Theorem 18.3 in
\cite{James_R_Munkres:B2000}.\index{Munkres, James R.}} of point-set topology applied to
the embeddings $B_0 \hookrightarrow V$ and $B_1 \hookrightarrow V$ furnishes a
continuous embedding of $B_0 + B_1$ into $V$.
\end{proof}

We now set out to generalize the \hyperref[riesz-thorin]{Riesz-Thorin interpolation theorem}
in the new framework.
Let us recall that the proofs of the interpolation theorems involved placing the
two endpoint operators on the two boundaries of the strip
\[
 S = \{z \in \bC : 0 \leq \Re z \leq 1\}
\]
and examining what happens in the middle. This was done by encoding the operators
in a function that is continuous on $S$, holomorphic in the interior of $S$, and
suitably bounded on the two boundaries of the strip, and then establishing the
intermediate bound via Hadamard's three-lines theorem (Theorem \ref{hadamard}).\index{Hadamard's three-lines theorem}
The key is to leave behind
the complex-valued functions on $S$, 
and to consider instead the Banach-valued functions
on $S$. The same argument will then produce new \emph{spaces}---as opposed to
\emph{operators}---in the middle of the strip.

\begin{defin}\label{space-generating-function}
A \emph{space-generating function}\footnote{This is not a standard notion
and will not be used beyond this subsection.} for a complex Banach couple
$(B_0,B_1)$ is a function $f:S \to B_0+B_1$ such that
\begin{enumerate}[(a)]
 \item $f$ is continuous and bounded on $S$ with respect to the norm of
 $B_0+B_1$;
 \item $f$ is holomorphic in the interior of $S$, as per the definition of
 holomorphicity in \S\S\ref{complex-analysis-of-banach-valued-functions};
 \item $f$ maps into $B_0$ and is continuous with respect to the norm
 of $B_0$ on the line $\Re z = 0$ and decays to zero as $|\Im z| \to \infty$
 on this line;
 \item $f$ maps into $B_1$ and is continuous with respect to the norm
 of $B_1$ on the line $\Re z = 1$ and decays to zero as $|\Im z| \to \infty$
 on this line.
\end{enumerate}
The collection of all space-generating functions for $(B_0,B_1)$ is denoted
by $\mc{F}(B_0,B_1)$.
\end{defin}

The interpolation spaces, which we shall define in due course, will be
subspaces of $B_0+B_1$ isomorphic to a quotient space of $\mc{F}(B_0,B_1)$.
We first check that $\mc{F}(B_0,B_1)$ is indeed a well-behaved space of functions.

\begin{prop}
$\mc{F}(B_0,B_1)$ is a Banach space, with the norm
\[
 \|f\|_{\mc{F}} = \max\left\{
 \sup_{-\infty < y < \infty} \|f(iy)\|_{B_0},
 \sup_{-\infty < y < \infty} \|f(1+iy)\|_{B_1}
 \right\}.
\]
\end{prop}

\begin{proof}
It is trivial to check that $\|\cdot\|$ is a norm, so it suffices
to check that $\mc{F}(B_0,B_1)$ is complete. To this end,
we suppose that $(f_n)_{n=1}^\infty$ is a Cauchy sequence in
$\mc{F}(B_0,B_1)$. For each $z = x + i y$ in the strip $S$,
the \hyperref[hadamard]{three-lines lemma} provides the estimate
\begin{eqnarray*}
 \|f_n(z) - f_m(z)\|_{B_0+B_1}
 &\leq& \max \left\{ \sup_{y} \|f(iy)\|_{B_0+B_1},
 \sup_{y} \|f(1+iy)\|_{B_0+B_1}
 \right\} \\
 &\leq& \|f_n - f_m\|_{\mc{F}},
\end{eqnarray*}
whence $(f_n)_{n=1}^\infty$ converges uniformly to a function
$f \in B_0+B_1$. By the uniformity, the function $f$ is
continuous and bounded on $S$ and holomorphic in the
interior of $S$.

Since the $B_0+B_1$ norm is determined by the maximum
of the $B_0$ norm and the $B_1$ norm, we see that
$(f_n)_{n=1}^\infty$ converges uniformly to a limit
$g_0$ in $B_0$ and another limit $g_1$ in $B_1$. The Hausdorff
condition of the ambient topological vector space guarantees that
$f = g_0 = g_1$. The uniformity once again guarantees that
the conditions (c) and (d) in Definition \ref{space-generating-function}
are satisfied, whence $f$ is in $\ms{F}(B_0,B_1)$. It now follows
from the definition of the $\mc{F}$-norm that $\|f_n - f\|_{\mc{F}} \to 0$,
and our claim is established.
\end{proof}

We are now ready to give the main definition of the section.

\begin{defin}\index{interpolation space!complex}\index{interpolation space}
Let $(B_0,B_1)$ be a complex Banach couple and fix $\theta \in [0,1]$. The
\emph{complex interpolation space of order $\theta$} between $B_0$ and
$B_1$ is the normed linear subspace
\[
 B_\theta = [B_0,B_1]_\theta
 = \{v \in B_0+B_1 : v = f(\theta) \mbox{ for some } f \in \mc{F}(B_0,B_1)\}
\]
of $B_0+B_1$, with the norm
\[
 \|v\|_\theta = \|v\|_{B_\theta}
 = \inf_{\substack{ f \in \ms{F} \\ f(\theta) = v}} \|f\|_{\mc{F}}.
\]
\end{defin}

We first check that the interpolation spaces are well-behaved.
For this purpose, let us recall that the quotient norm on the quoient space
$B/N$ of a Banach space $B$ is given by
\[
 \|[v]\|_{B/N} = \inf_{w \in [v]} \|w\|_{B}
 = \inf_{w \in N} \|v+w\|_{B}. 
\]
It is a standard result that the closedness of $N$ guarantees the completeness
of the quotient norm,\index{norm!quotient} thereby turning $B/N$ into a Banach space.\index{quotient space}\index{quotient space!norm on a|see{norm, quotient}}

\begin{prop}
For each $\theta \in [0,1]$, the interpolation space $[B_0,B_1]_\theta$ is
isometrically isomorphic to the quotient Banach space
$\mc{F}(B_0,B_1)/\mc{N}_\theta$, where $\mc{N}_\theta$ is the
subspace of $\mc{F}(B_0,B_1)$ consisting of all functions $f \in
\mc{F}(B_0,B_1)$ such that $f(\theta) = 0$. 
\end{prop}

\begin{proof}
First, we observe that $\mc{N}_\theta$ is a closed subspace of
$\mc{F}(B_0,B_1)$, so that the quotient space $\mc{F}(B_0,B_1)/\mc{N}_\theta$
is a Banach space.
We consider the mapping $f \mapsto f(\theta)$ from $\mc{F}(B_0,B_1)$
to $B_0 + B_1$. Clearly, the image of the mapping is $[B_0,B_1]_\theta$,
and the kernel $\mc{N}_\theta$. Since
\[
 \|f(\theta)\|_{B_0+B_1}
 \leq \max \left\{
 \sup_y \|f(iy)\|_{B_0+B_1} , \sup_y \|f(1+iy)\|_{B_0+B_1} \right\}
 \leq \|f\|_{\mc{F}}
\]
for each $f \in \mc{F}(B_0,B_1)$, the mapping is bounded, and so
$\mc{F}(B_0,B_1)/\mc{N}_\theta$ is isomorphic to $[B_0,B_1]_\theta$.
\end{proof}

We now observe that the complex interpolation method can be
flipped in a natural way; the proof is a straightforward
application of the definitions and is thus omitted.

\begin{prop}
For each $\theta \in [0,1]$, we have the isomorphism
\[
 [B_0,B_1]_\theta \cong [B_1,B_0]_{1-\theta}.
\]
\end{prop}

The complex interpolation method behaves well under
reiteration. The proof is rather elaborate and we omit it:
see \S\S\ref{fr-dual-of-complex-interpolation-space}
for a discussion.

\begin{theorem}[Reiteration theorem]\label{complex-interpolation-reiteration}\index{complex interpolation!reiteration theorem}
Let $(B_0,B_1)$ be a Banach couple. Fix $\theta_0,\theta_1 \in [0,1]$
and set
\[
 X_j = [B_0,B_1]_{\theta_j}.
\]
If $B_0 \cap B_1$ is dense in each of the spaces $B_0$, $B_1$, and $X_0 \cap X_1$,
then we have the isomorphism
\[
 [X_0,X_1]_\Theta \cong [B_0,B_1]_{(1-\Theta)\theta_0 + \Theta \theta_1}
\]
for all $\Theta \in [0,1]$.
\end{theorem}

Finally, we show that the method of complex interpolation is
truly a generalization of the \hyperref[riesz-thorin]{Riesz-Thorin
interpolation theorem}.\index{Riesz-Thorin interpolation \\ theorem}

\begin{theorem}[Complex interpolation is exact]\label{complex-interpolation-is-exact}\index{interpolation space!exact}
Let $(B_0,B_1)$ and $(C_0,C_1)$ be two complex Banach couples
and $T:B_0 + B_1 \to C_0 + C_1$ a bounded linear operator.
For each $\theta \in [0,1]$, the restriction of $T$
to $[B_0,B_1]_\theta$ maps boundedly into $[C_0,C_1]_\theta$
and satisfies the norm estimate
\[
 \|T\|_{B_\theta \to C_\theta} \leq
 \|T\|_{B_0 \to C_0}^{1-\theta} \|T\|_{B_1 \to C_1}^\theta.
\]
\end{theorem}

\begin{proof}
Fix $\theta \in [0,1]$, $v \in [B_0,B_1]_\theta$, and
$\ve>0$. We shall show that
\[
 \|Tv\|_{C_\theta}
 \leq k_0^{1-\theta}k_1^\theta\|v\|_{B_\theta},
\]
where
\[
 k_0 = \|T\|_{B_0 \to C_0} \htwo \mbox{and}\htwo
 k_1 = \|T\|_{B_1 \to C_1}.
\]
By the definition of the $B_\theta$ norm, we can find
an $f \in \mc{F}(B_0,B_1)$ such that
$f(\theta) = v$ and $\|f\|_{\mc{F}} \leq \|v\|_{B_\theta} + \ve$.

We claim that
\[
 g(z) = k_0^{z-1}k_1^{-z}[Tf(z)]
\]
belongs to $\mc{F}(C_0,C_1)$ and satisfies the norm estimate
$\|g\|_{\mc{F}} \leq \|v\|_{B_\theta} + \ve$. The continuity
of $g$ is clear. Since $f \in \mc{F}(B_0,B_1)$, we have
the bound $M \geq \|f(z)\|_{B_0+B_1}$ for all $z \in S$,
and so we have the estimate.
\begin{eqnarray*}
 \|g(z)\|_{C_0+C_1}
 &=& k_0^{z-1}k_1^{-z}\|Tf(z)\|_{C_0+C_1} \\
 &\leq& k_0^{z-1}k_1^{-z}\|T\|_{B_0+B_1 \to C_0+C_1} \|f(z)\|_{B_0+B_1} \\
 &\leq& k_0^{z-1}k_1^{-z}\|T\|_{B_0+B_1 \to C_0+C_1} C,
\end{eqnarray*}
Therefore, (a) in Definition \ref{space-generating-function} is satisfied.
Given any bounded linear functional $l$ on $C_0+C_1$, the map
\[
 lg(z) = l(k_0^{z-1}k_1^{-z}[Tf(z)])
 = k_0^{z-1}k_1^{-z} [lTf(z)]
\]
is holomorphic in the interior of $S$, for $lT$ is a bounded linear
functional on $B_0+B_1$ and $f \in \mc{F}(B_0,B_1)$.
This establishes
(b). (c) and (d) follow from the rapid decay of $k_0^{z-1}k_1^{-z}$,
and so $g$ is in $\mc{F}(C_0,C_1)$. We also observe that
\begin{eqnarray*}
 \|g\|_{\mc{F}}
 &=& \max\left\{ \sup_y \|g(iy)\|_{C_0} , \sup_y \|g(1+iy)\|_{C_1} \right\} \\
 &\leq& \max\left\{ \sup_y |k_0^{iy}||k_1^{-iy}|\|f(iy)\|_{B_0},
 \sup_y |k_0^{iy}||k_1^{-iy}|\|f(1+iy)\|_{B_1} \right\} \\
 &=& \max\left\{ \sup_y \|f(iy)\|_{B_0},
 \sup_y \|f(1+iy)\|_{B_1} \right\} \\
 &=& \|f\|_{\mc{F}} \\
 &\leq& \|v\|_{B_\theta} + \ve,
\end{eqnarray*}
as was claimed.

It now follows that
\[
 \|v\|_{B_\theta} + \ve
 \geq \|g\|_{\mc{F}}
 \geq \|g(\theta)\|_{B_\theta}
 = \|k_0^{\theta-1}k_1^{-\theta}[Tf(\theta)]\|_{C_\theta}
 = k_0^{\theta-1}k_1{-\theta}\|Tv\|_{C_\theta}, 
\]
whereby we have the estimate
\[
 \|Tv\|_{C_\theta}
 \leq k_0^{1-\theta}k_1^{\theta}\left(\|v\|_{B_\theta} + \ve\right).
\]
Since $\ve>0$ was arbitrary, we have the inequality
\[
 \|Tv\|_{C_\theta}
 \leq k_0^{1-\theta}k_1^{\theta}\|v\|_{B_\theta}
\]
for all $v \in B_\theta$, and the proof is now complete.
\end{proof}

See \S\ref{fr-abstract-interpolation-theory} for the definition
of an exact interpolation space. Calder\'{o}n presents another
complex interpolation method in \cite{Alberto_P_Calderon:J1964},
which leads to a study of dual spaces of complex interpolation
spaces. See \S\S\ref{fr-dual-of-complex-interpolation-space} for
a quick sketch.

\subsection{Interpolation of \texorpdfstring{$L^p$}{Lp} Spaces}

Since we have motivated the method of complex interpolation
as a generalization of the \hyperref[riesz-thorin]{Riesz-Thorin
interpolation theorem}, it is natural to expect that
Riesz-Thorin has been incorporated into the theory as a
special case thereof. For simplicity's sake, we
prove the theorem only on $\bR^d$.

\begin{theorem}\label{interpolation-of-lp-spaces}\index{interpolation space!between Lp spaces@between $L^p$ spaces|(}
Given $p_0,p_1 \in [1,\infty]$, we have
\[
 [L^{p_0}(\bR^d),L^{p_1}(\bR^d)]_\theta = L^{p_\theta}(\bR^d),
\]
for each $\theta \in (0,1)$, where
\[
 p_\theta^{-1} = (1-\theta)p_0^{-1} + \theta p_1^{-1}.
\]
\end{theorem}

\begin{proof}
We first show that $\mc{C}^\infty_c(\bR^d)$ is a dense
subspace of $L^{p_0}(\bR^d)+L^{p_1}(\bR^d)$ in the $L^{p_0}+L^{p_1}$
norm given in Proposition \ref{plus-and-cap-are-well-defined}.
To see this, we fix an arbitrary $f \in L^{p_0}+L^{p_1}$ and
find $f_0 \in L^{p_0}$ and $f_1 \in L^{p_1}$ such that
$f = f_0+f_1$. By Corollary \ref{approximation-by-smooth-functions},
we can find two sequences $(\varphi_n)_{n=1}^\infty$ and
$(\phi_n)_{n=1}^\infty$ such that $\|f_0 - \varphi_n\|_{p_0} \to 0$
and $\|f_1 - \phi_n\|_{p_1} \to 0$ as $n \to \infty$. Therefore,
we have
\[
 \lim_{n \to \infty} \|f - (\varphi_n + \phi_n)\|_{L^{p_0}+L^{p_1}}
 \leq \lim_{n \to \infty} \|f_0 - \varphi_n\|_{p_0} + \|f_1 - \phi_n\|_{p_1}
 = 0,
\]
as desired.

It now suffices to show that
\[
 \|f\|_{[\theta]} = \|f\|_{[L^{p_0}(\bR^d),L^{p_1}(\bR^d)]_\theta} = \|f\|_{p_\theta} 
\]
for all $f \in \mc{C}^\infty_c(\bR^d)$. To this end, we fix an
$f \in \mc{C}^\infty_c(\bR^d)$, pick an $\ve>0$, and define
\[
 \Phi_z(x) = e^{\ve z^2 - \ve \theta^2} \frac{|f(x)|^{p/p(z)}f(x)}{|f(x)}
\]
for each $z \in S$, where
\[
 \frac{1}{p(z)} = \frac{1-z}{p_0} + \frac{z}{p_1}.
\]
We assume without loss of generality that $\|f\|_{p_\theta} = 1$ by
renormalizing if necessary. Observe that $\Phi \in
\mc{F}(L^{p_0}(\bR^d),L^{p_1}(\bR^d))$ and $\|\Phi\|_{\mc{F}} \leq e^\ve$.
Since $\Phi(\theta) = f$, we have $\|f\|_{[\theta]}$, and so
$\|f\|_{[\theta]} \leq \|f\|_{p_\theta}$. 

To establish the reverse inequality, we recall the following consequence
of the \hyperref[Lp-riesz-representation]{Riesz representation theorem}:\index{representation theorem!F. Riesz, Lp-space version@F. Riesz, $L^p$-space version}
\[
 \|f\|_{p_\theta}
 = \sup_{\substack{ g \in \mc{C}^\infty_c(\bR^d) \\ \|g\|_{p_\theta'} = 1}}
 \int fg.
\]
We fix such a $g$ and set
\[
 \Psi_z(x) = e^{\ve z^2 - \ve \theta^2}
 \frac{|g(x)|^{p_\theta' / p'(z)} g(x)}{|g(x)|},
\]
for each $z \in S$, where
\[
 \frac{1}{p'(z)} = \frac{1-z}{p_0'} + \frac{z}{p_1'}.
\]
We assume without loss of generality that $\|f\|_{[\theta]} = 1$
by renormalizing if necessary. Observe that
\[
 \Xi(z) = \int \Phi_z(x) \Psi_z(x) \, dx
\]
satisfies the bounds
\[
 |\Xi(iy)| \leq e^\ve \htwo \mbox{and} \htwo
 |\Xi(1+iy)| \leq e^{2\ve}
\]
for all $y \in \bR$. It now follows from the \hyperref[hadamard]{three-lines lemma}
that $|\Xi(z)| \leq e^{2\ve}$ for all $z \in S$, whence
$\|f\|_{p_\theta} \leq \|f\|_{[\theta]}$. This completes the proof.
\end{proof}
\index{interpolation space!between Lp spaces@between $L^p$ spaces|)}

The \hyperref[riesz-thorin]{Riesz-Throin interpolation theorem} now follows
as a direct consequence of Theorem \ref{complex-interpolation-is-exact} and
Theorem \ref{interpolation-of-lp-spaces}. 

\subsection{Interpolation of Hilbert Spaces}\label{s-interpolation-of-hilbert-spaces}

We conclude the section by studying another example of interpolation spaces,
which we shall have an occasion to use later in the chapter.
Let $H$ be a separable Hilbert space and $T$ a self-adjoint operator on $H$.
We assume furthermore that $T$ is positive
(see \S\S\ref{s-spectral-theory-of-self-adjoint-operators}).
The spectral theorem (Theorem \ref{spectral-theorem})\index{spectral theorem!on separable Hilbert spaces} implies that
there is a unitary operator\index{unitary!operator}
$U:H \to L^2(X,\mu)$ and a real-valued $\mu$-measurable function $a$ on $X$
such that
\[
 D = UTU^{-1}u(x) = a(x)u(x)
\]
for all $u \in L^2(X,\mu)$.
It then follows that $\mc{D}(T) = U^{-1}(\mc{D}(D))$, where
\[
 \mc{D}(D) = \{u \in L^2(X,\mu) : au \in L^2(X,\mu)\}.
\]

We shall assume that $a(x) \geq 1$, which is equivalent to the assumption that
\[
 \langle Tv,v \rangle \geq \|v\|^2.
\]
Note that if $a$ is bounded, then $\mc{D}(D) = L^2(X,\mu)$ and
$\mc{D}(T) = H$, whence $T$ must be a bounded operator
(Proposition \ref{hellinger-toeplitz}).\index{Hellinger-Toeplitz theorem} For each $\theta$
in the strip
\[
 S = \{z \in \bC : 0 \leq \Re z \leq 1\},
\]
we define $T^\theta$ to be the operator $U^{-1}D^\theta U$, where
\[
 D^\theta u(x) = a(x)^\theta u(x).
\]
The associated domain $\mc{D}(T^\theta)$ is the preimage $U^{-1}(\mc{D}(D^\theta))$,
where
\[
 \mc{D}(D^\theta) = \{u \in L^2(X,\mu) : a^\theta u \in L^2(X,\mu)\}.
\]

We now characterize the interpolation spaces between $H$ and $\mc{D}(T)$

\index{interpolation space!between Hilbert spaces|(}
\begin{theorem}\label{interpolation-of-domains-of-self-adjoint-operators-on-a-hilbert-space}
Let $H$ and $T$ be defined as above. For each $\theta \in [0,1]$, we have
\[
 [H,\mc{D}(T)]_\theta = \mc{D}(T^\theta).
\]
\end{theorem}

\begin{proof}
Fix $v \in \mc{D}(T^\theta)$. If we let
\[
 f(z) = T^{-z + \theta}v,
\]
then $f \in \mc{F}(H,\mc{D}(T))$ and $v = f(\theta)$. Conversely,
for each $f \in \mc{F}(H,\mc{D}(T))$ and every $\ve>0$, we
note that
\begin{eqnarray*}
 \|T^z(I - i \ve T)^{-1} f(z)\|_H
 &\leq& \sup_y \max \left\{
 \|(I-i \ve T)^{-1} T^{iy} z(iy)\|_H , \right. \\
 & & \left. \|(T^{1+iy}(I+i \ve T)^{-1} z(1+iy)\|_H
 \right\} \\
 &\leq& C
\end{eqnarray*}
by the maximum modulus principle,\index{maximum modulus principle} where $C$ is a constant
independent of $\ve$. It thus follows that
$f(\theta) \in \mc{D}(T^\theta)$, and the proof is complete.
\end{proof}

We shall use this result in \S\S\ref{s-sobolev-spaces},
when we characterize fractional-order Sobolev spaces
via the Fourier transform: see the proof of
Theorem \ref{l2-sobolev-fourier-characterization}.
\index{complex interpolation|)}
\index{interpolation space!between Hilbert spaces|)}
  
\section{Generalized Functions}\label{s-generalized_functions}

\index{Schwartz, Laurent!distribution theory|see{tempered distributions}}
\index{generalized functions|see{tempered distributions}}
\index{distribution theory|see{tempered distributions}}
In the following two sections, we set up the stage for
the interpolation theorem of C. Fefferman and E. Stein,
which we study in \S\ref{s-hardy-spaces-and-bmo}. We introduce
Laurent Schwartz's\index{Schwartz, Laurent}\index{tempered!distributions} theory of distributions 
in this section and apply it in the next section to the study
of a classical singular operator called the Hilbert transform.

We recall the remark from \S\ref{s-the_fourier_transform}
that the output of the Fourier transform sometimes cannot be described
as a function. To provide a rigorous explanation for this remark,
we introduce a particular kind of ``generalized functions'',
known as \emph{tempered distributions}.

\subsection{The Schwartz Space}

\index{Schwartz, Laurent!space|(}
The starting point of distribution theory is that linear functionals
acting on a space of ``nice functions'' is easier to deal with than
the functions they represent. In the context of harmonic analysis,
the natural candidate for such a space is the
\hyperref[defin-schwartz-space]{Schwartz space}, which
we have seen to be closed under differentiation and the Fourier
transform. Even better, the Schwartz space is also closed under
other fundamental operations in harmonic analysis: translation,
rotation, reflection, dilation, and convolution. The first four
assertions follow from trivial computations, so we omit the proof.

\begin{prop}\index{translation!of Schwartz functions}\index{reflection!of Schwartz functions}\index{Fourier transform!on the Schwartz space}\index{rotation!of Schwartz functions}
If $\varphi \in \ms{S}(\bR^d)$, then $\tau_h \varphi, e_h \varphi,
\tilde{\varphi}$, and $\widehat{\delta_a \varphi}$ are
in $\ms{S}(\bR^d)$ for each $h \in \bR^d$ and every $a>0$.
\end{prop}

To see that the Schwartz space is closed under convolution, we recall
Theorem \ref{convolution-theorem}, which states that
the Fourier transform turns convolution into pointwise multiplication.
Since the Schwartz space is closed under the Fourier transform,
pointwise multiplication, and the inverse Fourier transform,
the convolution theorem implies that the Schwartz space is closed
under convolution.

\begin{prop}\index{convolution!of Schwartz functions}
If $\varphi,\phi \in \ms{S}(\bR^d)$, then $\varphi*\phi \in \ms{S}(\bR^d)$.
\end{prop}

\begin{proof}
$\varphi * \phi = (\widehat{\varphi * \phi})^\vee = (\hat{\varphi}\hat{\phi})^\vee$.
\end{proof}

Let us now return to the task of examining linear functionals on the Schwartz
space. We are only interested in the bounded ones, for they are the ones
that are well-behaved under limiting operations. This,
in turn, requires us to give a topology on $\ms{S}(\bR^d)$.

To do so, we consider the natural \hyperref[defin-seminorm]{seminorms}
\[
 \rho_{\alpha \beta}(\varphi)
 = \sup_{x \in \bR^d} \left| x^\alpha D^\beta \varphi(x) \right|
 = \|x^\alpha D^\beta\varphi(x)\|_\infty
\]
for each pair of multi-indices $\alpha$ and $\beta$. We observe that
each $\rho_{\alpha\beta}$ is complete, in the sense that $\rho_{\alpha\beta}
(\varphi_n - \varphi_m) \to 0$ as $n,m \to 0$ implies that there
exists a $\varphi \in \ms{S}(\bR^d)$ such that $\rho_{\alpha\beta}
(\varphi_n - \varphi) \to 0$ as $n \to \infty$. Indeed, we see by
setting $\alpha = \beta = 0$ that
$(\varphi_n)_{n=1}^\infty$ is uniformly Cauchy, whence we can find a function
$\varphi:\bR^d \to \bC$ to which the sequence converges uniformly. All other
convergences are uniform as well, and so the decay and smoothness conditions are
trivially established.

We now arrange the seminorms into a sequence $(\rho_n)_{n=1}^\infty$ and
consider the metric
\[
 d(\varphi,\phi) = \sum_{n=1}^\infty \frac{1}{2^n}
 \left( \frac{\rho_n(\varphi - \phi)}{1 + \rho_n(\varphi - \phi)} \right)
\]
on $\ms{S}(\bR^d)$. It is easy to see that
$d(\varphi_k,\phi) \to 0$ as $k \to \infty$ if and only 
 $\rho_n(\varphi_k - \varphi) \to 0$ for all $n$. This, in particular,
implies that $d$ is a complete metric.
It therefore follows that $\ms{S}(\bR^d)$ is a
\hyperref[defin-frechet-space]{Fr\'{e}chet space}.\index{Fr\'{e}chet space|(} The Fr\'{e}chet-space
topology on $\ms{S}(\bR^d)$ is commonly referred to as the
\emph{strong topology} on $\ms{S}(\bR^d)$. Consequently, a sequence
converging in the strong topology is said to \emph{converge strongly}.

We now recall that the Fourier transform is a linear automorphism on
$\ms{S}(\bR^d)$. If $(\varphi_n)_{n=1}^\infty$ is a sequence in $\ms{S}(\bR^d)$
converging strongly to $\varphi \in \ms{S}(\bR^d)$, then
\begin{eqnarray*}
 \lim_{n \to \infty} |\hat{\varphi}_n(\xi) - \hat{\varphi}(\xi)|
 &=& \lim_{n \to \infty} \left| \int (\varphi_n(x) - \varphi(x))
 e^{-2 \pi i x \cdot \xi} \, dx \right| \\
 &\leq& \lim_{n \to \infty} \int |\varphi_n(x) - \varphi(x)|
 |e^{-2 \pi i x \cdot \xi}| \, dx \\
 &=& \int \lim_{n \to \infty} |\varphi_n(x) - \varphi(x)| \, dx
 = 0
\end{eqnarray*}
by the uniform convergence of $(\varphi_n)_{n=1}^\infty$. The same
convergence result can easily be established for all $\rho_{\alpha\beta}$.
Carrying out an analogous computation for the inverse Fourier transform,
we see that the Fourier transform is a homeomorphism on $\ms{S}(\bR^d)$
onto itself. In other words, the Fourier transform is a topological
automorphism on $\ms{S}(\bR^d)$.

To wrap up the above discussion, we now collect some basic properties of
$\ms{S}(\bR^d)$. Of course, the topology on $\ms{S}(\bR^d)$ in
the following proposition is the strong topology.

\begin{prop}\label{properties-of-schwartz-functions}
The following are basic properties of the Schwartz space:
\begin{enumerate}[(a)]
 \item $\ms{S}(\bR^d)$ is closed under translation, rotation, reflection,
 differentiation, convolution, and the Fourier transform.
 \item The Fourier transform is a linear automorphism and
 a topological automorphism of $\ms{S}(\bR^d)$.
 \item For each pair of multi-indices $\alpha$ and $\beta$, the map
 $\varphi(x) \mapsto x^\alpha D^\beta \varphi(x)$ is continuous.
 \item If $\varphi \in \ms{S}(\bR^d)$, then $\tau_h \varphi \to
 \varphi$ as $h \to 0$.
 \item Let $h = (0,\ldots,h_n,\ldots,0)$ lie on the $n$th coordinate
 axis of $\bR^d$. For each $\varphi \in \ms{S}(\bR^d)$, we have
 \[
  \lim_{|h| \to 0} \frac{\varphi - \tau_h \varphi}{h_n}
  = \frac{\partial}{\partial x_n} \varphi.
 \]
\end{enumerate}
\end{prop}

\begin{proof}
(a) and (b) were already established in this subsection. (c) is
a trivial consequence of the $\rho_{\gamma \delta}$-convergence
implying the $\rho_{(\gamma+\alpha) (\delta+\beta)}$-convergence.
Likewise, (d) and (e) are easy consequences of the definition
of strong convergence in $\ms{S}(\bR^d)$.
\end{proof}

It is important to note that strong convergence implies
$L^p$ convergence. In fact, the $L^p$-norm of a Schwartz function
is dominated by a finite linear combination of Schwartz norms:

\begin{theorem}\label{Linfty-estimate}
If $\varphi \in \ms{S}(\bR^d)$ and $p \in (0,\infty]$, then,
for some constant $C_{p,d}$ depending only on $p$ and $d$,
\[
 \|D^\beta \varphi\|_p \leq C_{p,d}
 \sum_{|\alpha| \leq \llbracket\frac{d+1}{p}\rrbracket+1} \rho_{\alpha\beta}(\varphi)
\]
whenever the right-hand side is finite. Here $\llbracket\frac{d+1}{p}\rrbracket$
is the greatest integer smaller than or equal to $\frac{d+1}{p}$.
\end{theorem}

\begin{proof}
The proof is trivial for $p = \infty$, so we assume that $p < \infty$.
We let $\omega_d$ denote the volume of the $d$-dimensional unit ball
and break the $L^p$-norm of $D^\beta \varphi$ into two pieces:
\begin{eqnarray*}
 & & \|D^\beta \varphi\|_p \\
 &=& \left( \int_{|x| \leq 1} |D^\beta \varphi(x)|^p \, dx
 + \int_{|x| \geq 1} |D^\beta \varphi(x)|^p \, dx \right)^{1/p} \\
 &=& \left( \int_{|x| \leq 1} |D^\beta \varphi(x)|^p \, dx
 + \int_{|x| \geq 1} |x|^{d+1}|D^\beta \varphi(x)|^p|x|^{-(d+1)} \, dx \right)^{1/p} \\
 &\leq& \left( \int_{|x| \leq 1} \|D^\beta \varphi(x)\|_\infty^p \, dx
 + \int_{|x| \geq 1} |x|^{d+1}|D^\beta \varphi(x)|^p|x|^{-(d+1)} \, dx \right)^{1/p} \\
 &\leq& \left( \omega_d \|D^\beta \varphi(x)\|_\infty^p
 + \sup_{x \in \bR^d} |x|^{d+1}|D^\beta \varphi(x)|^p
 \int_{|x| \geq 1} |x|^{-(d+1)} \, dx \right)^{1/p}.
\end{eqnarray*}

We set $C_{p,d}'$ to be the maximum of $\omega_d$ and
$\int_{|x| \geq 1} |x|^{-(d+1)} \, dx$.
For an appropriately chosen constant $C_p''>0$, we have the following estimate:

\begin{eqnarray*}
 \|D^\beta \varphi\|_p
 &\leq&  (C_{p,d}')^p
 \left( \|D^\beta \varphi\|^p_\infty + \sup_{x \in \bR^d} |x|^{d+1} |D^\beta \varphi(x)|^p
 \right)^{1/p} \\
 &\leq& (C_{p,d}')^p C_p'' \left[
 \left( \|D^\beta \varphi\|^p_\infty \right)^{1/p}
 + \left(\sup_{x \in \bR^d} |x|^{d+1} |D^\beta \varphi(x)|^p\right)^{1/p} \right] \\
 &=& (C_{p,d}')^p C_p'' \left( \|D^\beta \varphi\|_\infty +
 \sup_{x \in \bR^d} |x|^{\frac{d+1}{p}} |D^\beta \varphi(x)| \right) \\
 &\leq& (C_{p,d}')^p C_p'' \left( \|D^\beta \varphi\|_\infty +
 \sup_{x \in \bR^d} |x|^{\llbracket\frac{d+1}{p}\rrbracket+1} |D^\beta \varphi(x)| \right)
\end{eqnarray*}

We now set set $C_{p,d}'''$ to be the minimum of the map
\[
 x \mapsto \sum_{|\alpha| = \llbracket\frac{d+1}{p}\rrbracket+1} |x^\alpha|
\]
on the unit sphere $|x|=1$. $C_{p,d}'''$ is positive, for this map has no zero on
the unit sphere. This, in particular, implies that
\begin{eqnarray*}
 |x|^{\llbracket\frac{d+1}{p}\rrbracket+1}
 \leq C_{p,d}'''\sum_{|\alpha| = \llbracket\frac{d+1}{p}\rrbracket+1} |x^\alpha|
\end{eqnarray*}
for all $|x|=1$, and we can extend the inequality to all $|x|>0$ by renormalizing
$x$. Since
\[
 \|D^\beta \varphi\|_p
 \leq (C_{p,d}')^p C_p'' \left( \|D^\beta \varphi\|_\infty +
 \sup_{x \in \bR^d} |x|^{\llbracket\frac{d+1}{p}\rrbracket+1} |D^\beta \varphi(x)| \right),
\]
we let
\[
 C_{p,d} = (C_{p,d}'(^p C_p'' \times \max\{1,C_{p,d}'''\}
\]
to conclude that
\begin{eqnarray*}
 \|D^\beta \varphi\|_p
 &\leq& C_{p,d}
 \left( \|D^\beta \varphi\|_\infty +
 \sup_{x \in \bR^d} \sum_{|\alpha| = \llbracket\frac{d+1}{p}\rrbracket+1} |x^\alpha|
 |D^\beta \varphi(x)| \right) \\
 &\leq& C_{p,d} \left( \rho_{0\beta}(\varphi)
 + \sum_{|\alpha| = \llbracket\frac{d+1}{p}\rrbracket+1}
 \rho_{\alpha\beta}(\varphi) \right) \\
 &\leq& C_{p,d} \sum_{|\alpha| \leq \llbracket\frac{d+1}{p}\rrbracket+1}
 \rho_{\alpha\beta}(\varphi).
\end{eqnarray*}
Therefore, the $L^p$-norm of a Schwartz function is dominated by a finite
linear combination of $\rho_{\alpha\beta}$-norms, as was to be shown.
\end{proof}
\index{Schwartz, Laurent!space|)}
\index{Fr\'{e}chet space|)}

\subsection{Tempered Distributions}

\index{Schwartz, Laurent!space, the dual of|see{tempered distributions}}
\index{tempered!distributions|(}
We are now ready to consider the continuous linear functionals
on $\ms{S}(\bR^d)$.

\begin{defin}
The space of \emph{tempered distributions} is the continuous
dual space $\ms{S}'(\bR^d)$ of $\ms{S}(\bR^d)$. In other words,
$\ms{S}'(\bR^d)$ consists of the bounded linear functionals
on $\ms{S}(\bR^d)$.
\end{defin}

In what sense are tempered distributions ``generalized functions''?
Given a function $f \in L^p(\bR^d)$ for some $1 \leq p \leq \infty$,
we set
\[
 l_f(\varphi) = \int f(x) \varphi(x) \, dx
\]
for each $\varphi \in \ms{S}(\bR^d)$. $L_f$ is clearly finite
and is a linear functional on $\ms{S}(\bR^d)$. To show that
$L_f$ is continuous, it suffices to establish the continuity of
$L_f$ at the origin. To this end, we pick a sequence
$(\varphi_n)_{n=1}^\infty$ in $\ms{S}(\bR^d)$ converging strongly to 0.
This, in particular, implies that $\rho_{\alpha0}(\varphi_n) \to 0$
as $n \to \infty$, whence Theorem \ref{Linfty-estimate} implies that
$\|\varphi_n\|_{p'} \to 0$ as $n \to \infty$. It now follows from
H\"{o}lder's inequality that
\[
 \lim_{n \to \infty} |l_f(\varphi_n)|
 \leq  \lim_{n \to \infty} \|f\|_p \|\varphi_n\|_{p'}
 = 0,
\]
and so $l_f \in \ms{S}'(\bR^d)$.

At this point, we take a moment to introduce a new notation.
Given a tempered distribution $u$ and a Schwartz function $\varphi$,
we shall denote the action of $u$ on $\varphi$ by
\[
 u(\varphi) = \langle \varphi, u \rangle.
\]
To see why we use the inner-product notation, we recall that the
\hyperref[hilbert-riesz-representation]{Hilbert-space
F. Riesz representation theorem}\index{representation theorem!F. Riesz, Hilbert-space version} yields an isomorphism
\[
 u \mapsto \langle \cdot, u \rangle_{\mc{H}}.
\]
from a Hilbert space $\mc{H}$ to its dual $\mc{H}^*$. Since
each bounded linear functional $l$ on $\mc{H}$ corresponds uniquely
to an element of $\mc{H}$, we may consider $L$ as an element of
$\mc{H}$ and write
\[
 lv = \langle v, L\rangle.
\]
Following this identification, we use the same inner-product
notation for the Schwartz space and bounded linear functionals thereon.
With this notation, we identify the $L^p$-function $f$ with the associated
bounded linear functional $L_f$ on $\ms{S}(\bR^d)$ and write $f$ to
denote the tempered distribution. In other words, we write
$\langle \varphi, f \rangle$ to denote $L_f(\varphi)$.

Let us consider a few more canonical examples of tempered distributions.
For each $f \in L^1(\bR^d)$, we can associate a complex Borel measure\index{Borel, \'{E}mile!measure}\index{measure!complex}
$\mu_f$ defined by
\[
 \mu_f(E) = \int_E f(x) dx
\]
for all Borel sets\index{Borel, \'{E}mile!set} $E \subseteq \bR^d$. In this sense, the space
$\ms{M}(\bR^d)$ of complex Borel measures on $\bR^d$ contains $L^1(\bR^d)$.
Now, for each $\mu \in \ms{M}(\bR^d)$, we consider the linear functional
\[
 l_\mu(\varphi) = \int \varphi(x) \, d\mu(x)
\]
on $\ms{S}(\bR^d)$. To show that $l_\mu$ is bounded, we pick a
sequence $(\varphi_n)_{n=1}^\infty$ of Schwartz functions converging strongly
to 0 and observe that
\[
 \lim_{n \to \infty} |l_\mu(\varphi_n)|
 \leq \lim_{n \to \infty} \|\varphi_n\|_1 \mu(\bR^d)
 = 0
\]
by H\"{o}lder's inequality. Similarly as above, we identify the measure $\mu$
with the associated tempered distribution.

An important special case is the \emph{Dirac
$\delta$-distribution},\index{Dirac $\delta$-distribution} defined for a fixed point $x \in \bR^d$
to be the measure
\[
 \delta^x(E) =
 \begin{cases}
  1 & \mbox{ if } x \in E \\
  0 & \mbox{ if } x \notin E.
 \end{cases}
\]
For each $\varphi \in \ms{S}(\bR^d)$, the tempered distribution
$\delta^x$ satisfies the identity
\[
 \langle \varphi, \delta^x  \rangle = \varphi(x).
\]

A basic property of the Dirac $\delta$-distribution is that
$\delta^0 * \psi = \psi$ for all Schwartz functions $\psi$.
See \S\S\ref{s-operations-on-tempered-distributions} for the
definition of convolution of a tempered distribution and
a Schwartz function. The property is a trivial consequence
of the definitions presented in the subsection.
Moreover, the Dirac $\delta$-distributions are ``atomic'' examples of distributions
with point support. See \S\S\ref{fr-tempered-distributions-with-point-support}\index{support!of tempered distributions}
for the precise statement of the theorem,
as well as the definition of support of a distribution.

Yet wider classes of functions and measures are tempered distributions.
We recall that $\varphi:\bR^d \to \bC$ is a Schwartz function if
and only if $\varphi$ satisfies the growth condition
\[
 \sup_{x \in \bR^d} \langle x \rangle^n |D^\beta \varphi(x)| < \infty
\]
for each positive integer $n$ and every multi-index $\beta$, where
\[
 \langle x \rangle = \sqrt{1+x^2}.
\]
Therefore, if a measurable function $f:\bR^d \to \bC$ satisfies
$\langle x \rangle^{-n}f(x) \in L^p(\bR^d)$ for some postive integer $n$
and $1 \leq p < \infty$, then
\[
 \langle \varphi, f \rangle = \int f(x) \varphi(x) \, dx
 = \int [ \langle x \rangle^{-n} f(x) ] [ \langle x \rangle^n \varphi(x) ] \, dx
\]
is a tempered distribution. In this case, we say that
the function $f$ is a \emph{tempered $L^p$-function}.\index{tempered!Lp-function@$L^p$-function} Similarly, we say
that a Borel measure\index{Borel, \'{E}mile!measure} $\mu$, real or complex, is a \emph{tempered measure}\index{tempered!measure}
if the total variation\index{measure!total variation of} $|\mu|$ of $\mu$ satisfies the bound
\[
 \int \langle x \rangle^{-n} \, d|\mu|(x) < \infty
\]
for some positive integer $n$. If $\mu$ is a tempered measure, then
\[
 \langle \varphi, \mu \rangle = \int \varphi(x) \, d\mu(x)
\]
is a tempred distribution.

In general, a linear functional $l$ on $\ms{S}(\bR^d)$ is a tempered
distribution precisely in case $l$ is bounded by a finite linear
combination of Schwartz norms.

\begin{theorem}\label{iff-tempered-distribution}\index{tempered!distributions, necessary and sufficient condition for}
A linear functional $l$ on $\ms{S}(\bR^d)$ is a tempered distribution
if and only if there are integers $m$ and $n$ and a constant $C>0$
such that
\[
 |\langle \varphi, l \rangle| \leq
 C \sum_{\substack{|\alpha| \leq m \\ |\beta| \leq n}}
 \rho_{\alpha\beta}(\varphi)
\]
for all $\varphi \in \ms{S}(\bR^d)$.
\end{theorem}

\begin{proof}
It is clear that the existence of such $m$ and $n$ implies the continuity
of $l$. Conversely, we suppose that $l$ is continuous. Observe that
the collection of sets
\[
 N_{\ve,m,n}
 = \left\{\varphi : \sum_{\substack{|\alpha| \leq m \\ |\beta| \leq n}}
 \rho_{\alpha\beta}(\varphi) < \ve \right\}
\]
for all $\ve>0$ and $m,n \in \bN$ and their translates form a subbasis
of the strong topology on $\ms{S}(\bR^d)$. Therefore, we can find $\ve$,
$m$, and $n$ such that $|\langle \varphi, l \rangle| \leq 1$ on
$N_{\ve,m,n}$.

For each $\varphi \in \ms{S}(\bR^d)$, we set
\[
 \|\varphi\| = \sum_{\substack{|\alpha| \leq m \\ |\beta| \leq n}}
 \rho_{\alpha\beta}(\varphi).
\]
Fix $\ve_0 \in (0,\ve)$ and note that
\[
 \varphi_0 = \frac{\ve_0}{\|\varphi\|} \varphi \in N_{\ve,m,n}
\]
for all $\varphi \in \ms{S}(\bR^d)$. Therefore,
we have
\[
 |\langle \varphi, l \rangle|
 = \frac{\ve_0}{\|\varphi\|} |\langle \varphi_0, l \rangle|
 \leq 1,
\]
or
\[
 |\langle \varphi, l \rangle| \leq
 \frac{1}{\ve_0} \|\varphi\|
 = \sum_{\substack{|\alpha| \leq m \\ |\beta| \leq n}}
 \rho_{\alpha\beta}(\varphi).
\]
It follows that $C = \ve_0^{-1}$ is the desired constant,
and the proof is complete.
\end{proof}

\subsection{Operations on Tempered Distributions}\label{s-operations-on-tempered-distributions}

Generalized functions, of course, would be mere abstract nonsense
if they existed only for the sake of generalization. Before we can
study applications of tempered distributions, we must extend
the familiar operations in analysis to this general context.
We thus return to the six operations in harmonic analysis that
we have discussed in the beginning of this section: translation,
rotation, reflection, dilations, convolution, differentiation,
and the Fourier transform.

The guiding principle for the extensions
we shall study is that \emph{an operation on a generalized function
manifests itself by operating on the test functions.}
The action of a tempered distribution is
studied by ``testing it out'' on each Schwartz function.
It is therefore natural to define operations on tempered
distributions by the action of the operations on Schwartz functions:

\index{Fourier transform!of tempered distributions|(}
\begin{defin}
Let $u$ be a tempered distribution.
\begin{enumerate}[(a)]
 \item Given $h \in \bR^d$, the \emph{translation}\index{translation!of tempered distributions} of $u$
 with respect to $h$ is defined by
 \[
  \langle \varphi, \tau_h u \rangle
 = \langle \tau_{-h} \varphi, u \rangle
 \]
 for each $\varphi \in \ms{S}(\bR^d)$.
 \item Given $h \in \bR^d$, the \emph{rotation}\index{rotation!of tempered distributions} of $u$
 with respect to $h$ is defined by
 \[
  \langle \varphi, e_h u \rangle
 = \langle e_{-h} \varphi, u \rangle
 \]
 for each $\varphi \in \ms{S}(\bR^d)$.
 \item The \emph{reflection}\index{reflection!of tempered distributions} of $u$ is defined by
 \[
  \langle \varphi, \tilde{u} \rangle
 = \langle \tilde{\varphi}, u \rangle
 \]
 for each $\varphi \in \ms{S}(\bR^d)$.
 \item The \emph{dilation}\index{dilation!of tempered distributions} of $u$ is defined by
 \[
  \langle \varphi, \delta_a u \rangle
 = \langle a^{-d}\delta_{a^{-1}} \varphi , u \rangle
 \]
 for each $a > 0$.
 \item Given $\psi \in \ms{S}(\bR^d)$, the \emph{convolution}\index{convolution!of tempered distributions} of
 $u$ and $\psi$ is defined by
 \[
  \langle \varphi, u * \psi \rangle
 = \langle \tilde{\psi} * \varphi, u \rangle
 \]
 for each $\varphi \in \ms{S}(\bR^d)$.
 \item Given a multi-index $\beta$, the \emph{partial derivative}\index{differentiation!of tempered distributions} of $u$
 with respect to $\beta$ is defined by
 \[
  \langle \varphi, D^\beta u \rangle
 = (-1)^{|\beta|} \langle D^\beta \varphi , u \rangle 
 \]
 for each $\varphi \in \ms{S}(\bR^d)$.
 \item The \emph{Fourier transform} of $u$ is defined by
 \[
  \langle \varphi, \hat{u} \rangle
 = \langle \hat{\varphi} , u \rangle
 \]
 for each $\varphi \in \ms{S}(\bR^d)$.
\end{enumerate}
\end{defin}

The above definitions are direct generalizations of their
analogues for ordinary functions. If $u$ is an $L^p$-function, then
\[
 \langle \varphi, \hat{u} \rangle
 = \int \hat{u}(t) \varphi(t) \, dt
 = \int u(t) \hat{\varphi}(t) \, dt
 = \langle \hat{\varphi}, u \rangle
\]
by the multiplication formula (Theorem \ref{multiplication-formula}).\index{multiplication formula}
With this new definition of Fourier transform,
translation, rotation, and reflection are defined in a way that
makes Proposition \ref{symmetry-invariance-of-fourier-transform}\index{Fourier transform!symmetry invariance of|(} true.
Differentiation of tempered distribution is also straightforward, as
we have
\begin{eqnarray*}
 \langle \varphi, D^\beta u \rangle
 &=& \int [D^\beta u(x)] \varphi(x) \, dx \\
 &=& (-1)^{|\beta|} \int u(x) [D^\beta \varphi(x)] \, dx \\
 &=& (-1)^{|\beta|} \langle D^\beta \varphi , u \rangle
\end{eqnarray*}
via integration by parts, provided that $u$ and $\varphi$ 
have enough smoothness conditions. Finally, \hyperref[fubini-tonelli]{Fubini's theorem}
shows that
\[
  \langle \varphi, u * \psi \rangle
 = \int (u*\psi)(x)\varphi(x) \, dx = \int u(x)(\tilde{\psi} * \varphi)(x) \, dx
 = \langle \tilde{\psi} * \varphi, u \rangle.
\]

It is easy to see that the space of tempered distributions is
closed under the six operations. Furthermore,
the basic properties of the six operations are preserved as well. In the
following proposition, we state the generalizations of Proposition
\ref{symmetry-invariance-of-fourier-transform}, Proposition
\ref{differentiation-to-multiplication},\index{Fourier transform!differentiation of} Proposition
\ref{fourier-transform-of-schwartz-is-schwartz}, and Theorem
\ref{fourier-inversion-formula}\index{Fourier inversion! of tempered distributions} in the context of tempered distributions.
The proof of the following proposition is a direct adaptation of
the corresponding statement on the Schwartz space and is thus omitted.

\begin{prop}
Let $u \in \ms{S}'(\bR^d)$, $\psi \in \ms{S}(\bR^d)$,
and $h \in \bR^d$.
\begin{enumerate}[(a)]
 \item $\widehat{\tau_h u} = e_{-h} \hat{u}$.
 \item $\widehat{e_h u} = \tau_h \hat{u}$.
 \item $\hat{\tilde{u}} = \tilde{\hat{u}}$
 \item If $P$ is a polynomial in $d$ variables, then
  \[
  P(D)\hat{f}(\xi) = \ms{F}(P(-2 \pi i x)f(x))(\xi)
  \hone \mbox{and} \hone
  \ms{F}(P(D)f)(\xi) = P(2 \pi i \xi)\hat{f}(\xi).
 \]
 \item The \emph{inverse Fourier transform}, defined by
 \[
  \langle \varphi, u^\vee \rangle
 = \langle \varphi^\vee , \varphi \rangle
 \]
 for each $\varphi \in \ms{S}(\bR^d)$, is well-defined
 on $\ms{S}'(\bR^d)$. We have $(\hat{u})^\vee = u$, \linebreak
 and the Fourier transform is a linear automorphism
 on $\ms{S}'(\bR^d)$.
\end{enumerate}
\end{prop}
\index{Fourier transform!symmetry invariance of|)}
\index{Fourier transform!of tempered distributions|)}

\subsection{Convolution Operators and Fourier Multipliers}\label{s-convolution-operators-and-fourier-multipliers}

The reader may have noted that we have not said anything
about the convolution operation. It is an odd one,
indeed: for one, convolution is an operation between
a tempered distribution and a Schwartz function,
whereas all the other operations were those of
two tempered distributions. As such, the convolution
operation possesses a number of special properties
we now study.

We first observe that the convolution of a tempered
distribution and a Schwartz function is not only
a tempered distribution, but also a function.

\index{convolution!of tempered distributions|(}
\begin{theorem}\label{convolution-with-tempered-distributions}
If $u \in \ms{S}'(\bR^d)$ and $\psi \in \ms{S}(\bR^d)$, then
$u * \psi$ is a $\ms{C}^\infty$ map, in the sense that the function
\[
 f(x) = (u * \psi)(x) = \langle \tau_x \tilde{\psi} , u \rangle
\]
is in $\ms{C}^\infty(\bR^d)$. Furthermore, each multi-index $\beta$ admits
constants $C_\beta$ and $n_\beta$ such that
\[
 |D^\beta (u * \psi)(x)| \leq C_\beta \langle x \rangle^{n_\beta},
\]
whence $u * \psi$ is a tempered distribution.
\end{theorem}

\begin{proof}
We first show that $f \in \mc{C}^\infty(\bR^d)$. For each $1 \leq n \leq d$,
we write $h^n = (0,\ldots,h_n,\ldots,0)$. Proposition
\ref{properties-of-schwartz-functions}(e) implies that
\[
 \lim_{|h| \to 0} \frac{\tau_{x+h^n} \tilde{\psi} - \tau_x \tilde{\psi}}{h_n}
 = -\tau_x\left( \frac{\partial \tilde{\psi}}{\partial x_n} \right).
\]
in the strong topology. By continuity of $u$, we have
\[
 \lim_{h_n \to 0} \frac{f(x)}{h_n}
 = \lim_{h_n \to 0}
 \frac{\langle \tau_{x+h^n} \tilde{\psi} - \tau_x \tilde{\psi}, \rangle}{h_n}
 = \left< -\tau_x \left( \frac{\partial \tilde{\psi}}{\partial x_n} \right)
 , u \right>.
\]
We carry out an analogous argument for each $n$ and conclude from
Proposition \ref{properties-of-schwartz-functions}(d)
that $f \in \mc{C}^1(\bR^d)$.
Since $\frac{\partial \tilde{\psi}}{\partial x_d}$
is a Schwartz function, we can reiterate the argument to prove that
$D^\beta f$ exists and is continuous for each multi-index $\beta$.

We now show that each $D^\beta f$ satisfies the slow-growth
condition for a tempered distribution.
Theorem \ref{iff-tempered-distribution} furnishes a constant $C>0$
and integers $m$ and $n$ such that
\begin{equation}\label{oh-hi}
 |f(x)| = |\langle \tau_x \tilde{\psi}, u \rangle|
 \leq \sum_{\substack{|\alpha| \leq m \\ |\beta| \leq n}}
 \rho_{\alpha\beta}(\tau_x \tilde{\psi}).
\end{equation}
Since
\[
 \rho_{\alpha\beta}(\tau_x\tilde{\psi})
 = \sup_{y \in \bR^d} |y^\alpha (D^\beta \tilde{\psi})(y - x)|
 = \sup_{y \in \bR^d} |(y+x)^\alpha (D^\beta \tilde{\psi})(y)|
\]
is bounded by a polynomial in $x$,
\[
 (D^\beta f)(x) = (-1)^{|\beta|} \langle \tau_x D^\beta \tilde{\psi},
 u \rangle
\]
combined with (\ref{oh-hi}) establishes the claim.

It now remains to show that $u * \psi$ \emph{is} the function $f$.
To this end, it suffices to prove the identity
\[
 \langle \varphi, u*\psi \rangle
 = \int \varphi(t) f(t) \, dt.
\]
To this end, we observe that
\begin{eqnarray*}
\langle \varphi, u*\psi \rangle
&=& \langle \tilde{\psi} * \varphi , u \rangle \\
&=& \left< \int \tilde{\psi}(x-t)\varphi(t) \, dt , u \right> \\
&=& \left< \int (\tau_t \tilde{\psi})(x) \varphi(t) \, dt , u \right>
\end{eqnarray*}
We now note that the Riemann sums of the last integral converge
in the strong topology, whence
\[
 \left< \int (\tau_t \tilde{\psi})(x) \varphi(t) \, dt , u \right>
 = \int \langle \tau_t \tilde{\psi}, u \rangle , \varphi(t) \, dt
 = \int f(t) \varphi(t) \, dt.
\]
The claim follows from the linearity and continuity of $u$.
\end{proof}
\index{convolution!of tempered distributions|)}

We now recall the convolution theorem (Theorem \ref{convolution-theorem}),
which describes the close relationship between pointwise multiplication
and the convolution operation. We shall need a generalization of these theorems
in the context of distribution theory. To this end,
we must make sense of pointwise multiplication
between a tempered distribution and a Schwartz function.

\begin{defin}\index{pointwise multiplication of tempered distributions}
If $u \in \ms{S}'(\bR^d)$ and $\psi \in \ms{S}(\bR^d)$, then
the \emph{pointwise multiplication} of $u$ and $\psi$ is defined to be
\[
 \langle \varphi, u\psi \rangle = \langle \psi\varphi, u \rangle
\]
for each $\varphi \in \ms{S}(\bR^d)$. 
\end{defin}

The following result is now a trivial consequence of the definition and the
properties of tempered distributions:

\begin{theorem}[Convolution theorem on $\ms{S}'(\bR^d)$]\label{convolution-theorem-on-distributions}\index{convolution theorem!on tempered distributions}
If $u \in \ms{S}'(\bR^d)$ and $\psi \in \ms{S}(\bR^d)$, then
\[
 \widehat{u * \psi} = \hat{u}\hat{\psi}
 \htwo \mbox{and} \htwo
 \widehat{u\psi} = \hat{u} * \hat{\psi}. 
\]
\end{theorem}

A higher, more abstract viewpoint is now in order.
Since the convolution between a tempered distribution and a Schwartz function
produces a function, we can consider the convolution operation as an
operator on $\ms{S}(\bR^d)$. Since $\ms{S}(\bR^d)$ is a dense subspace in
many function spaces, this operator often can be extended via
Theorem \ref{norm-preserving-extension}.

\index{convolution!operator|(}
\begin{defin}
Let $V$ be a normed linear space of measurable functions on $\bR^d$ containing
$\ms{S}(\bR^d)$ as a dense subspace. A \emph{convolution operator} on
$V$ is a linear operator on $V$ that can be written as
\[
 T\psi = u * \psi
\]
for all $f \in \ms{S}(\bR^d)$, where $u$ is a tempered distribution determined
uniquely by $T$.
\end{defin}

Of course, the most important cases of convolution operators are those
between Lebesgue spaces.

\begin{defin}
Given $1 \leq p,q \leq \infty$, the space $\ms{M}_{p,q}$ is defined to be
the collection of bounded convolution operators on $L^p(\bR^d)$ into
$L^q(\bR^d)$.
\end{defin}

Since there is a one-to-one correspondence between each convolution
operator and its associated tempered distribution,
we shall abuse the notation and speak of tempered distributions
\emph{in} $\ms{M}_{p,q}$.

We shall primarily be concerned with the space $\ms{M}_{p,p}$.
Observe that $\ms{M}_{p,p}$ is precisely the collection of
bounded operators that are \emph{basically} the Fourier transform:

\index{Fourier multiplier|(}
\begin{defin}
Fix $p \in [1,\infty)$. A bounded linear operator
$T:L^p(\bR^d) \to L^p(\bR^d)$ is an \emph{$L^p$ Fourier multiplier}
if there exists a bounded function $m$ such that
\[
 T \psi = (m \hat{\psi})^\vee
\]
for all $\psi \in \ms{S}(\bR^d)$. $m$ is referred to as
the \emph{symbol} of the Fourier multiplier $T$, and
the collection of all symbols of $L^p$ Fourier multipliers
is denoted by $\ms{M}^p$.
\end{defin}

\index{Fourier multiplier!as a convolution operator}
It follows from the convolution theorem (Theorem \ref{convolution-theorem-on-distributions})
that $m \in \ms{M}^p$ if and only if the associated
Fourier multiplier $T_m$ is in $\ms{M}_{p,p}$.
We shall study an important example of a Fourier multiplier
in the next section. For now, we content ourselves
with concrete characterizations of two important special cases.

\begin{theorem}\label{M11}
A convolution operator $T \psi = u * \psi$ is in $\ms{M}_{1,1}$ if and only if
$u$ is a complex Borel measure.\index{measure!complex}\index{Borel, \'{E}mile!measure}\index{Fourier multiplier!L1@$L^1$} In this case,
\[
 \|T\|_{L^1 \to L^1} = |u|,
\]
the total variation of $u$.
It thus follows that $\mc{M}^1$ is the space of
the Fourier transforms of complex Borel measures.
\end{theorem}

To prove the above theorem, we shall need the following representation theorem:

\begin{lemma}[Riesz-Markov representation theorem]\label{riesz-markov}\index{representation theorem!Riesz-Markov}\index{representation theorem!on C0@on $C_0(\bR^d)$|see{Riesz-Markov}}\index{Riesz-Markov theorem|see{representation theorem}}
The dual of the space $\mc{C}_0(\bR^d)$ of continuous functions on $\bR^d$
vanishing at infinity is isomorphic to the space $M(\bR^d)$ of complex
Borel measures on $\bR^d$, in the sense that each bounded linear functional
$l$ on $\mc{C}_0(\bR^d)$ furnishes a complex Borel measure $\mu$ on $\bR^d$
such that
\[
 l(g) = \int g(x) \, d\mu(x)
\]
for each $g \in \mc{C}_0(\bR^d)$.
\end{lemma}

The proof of the representation theorem can be found in many standard
textbooks in measure theory: see, for example, Theorem 6.19 in
\cite{Walter_Rudin:B1986}\index{Rudin, Walter} or Theorem 7.17 in \cite{Gerald_B_Folland:B1999}.\index{Golland, Gerald B.}

Let us now return to the task at hand:

\begin{proof}[Proof of Theorem \ref{M11}]
If $\mu$ is a complex Borel measure, then Young's inequality
\[
 \|\mu * \psi| \leq |\mu|\|\psi\|_1
\]
continues to hold, whence $\mu \in \ms{M}_{1,1}$. To prove the
converse assertion for an arbitrary $u \in \ms{M}_{1,1}$,
we define the \emph{Gauss-Weierstrass kernel}\index{Gauss-Weierstrass kernel}
\[
 W(t,\ve) = (4 \pi \ve)^{-\pi/2} e^{-|t|^2/4\ve}
\]
for each $\ve>0$: we have used the kernel in Proposition \ref{gaussian}\index{Fourier transform!of the Gaussian}
and the proof of the Fourier inversion formula
(Theorem \ref{fourier-inversion-formula}) already. Note that
\[
 \int W(t,\ve) \, dt = 1
\]
for all $\ve>0$. Therefore, $(W(t,\ve))_{\ve > 0}$ is an
\hyperref[approximations-to-the-identity]{approximation to the
identity},\index{approximations to the identity} and $u_\ve(x) = (u * W(t,\ve))(x)$ satisfies the estimate
\[
 \|u_\ve\|_1 \leq A\|W(t,\ve)\|_1 = C
\]
for some constant $C>0$ that does not depend on $\ve$.

We now consider $L^1(\bR^d)$ as an embedded subspace of the space
$M(\bR^d)$ of complex Borel measures on $\bR^d$. By the
\hyperref[riesz-markov]{Riesz-Markov representation theorem},
$M(\bR^d)$ is the dual of $\mc{C}_0(\bR^d)$, whence
we can endow $M(\bR^d)$ with the weak-* topology\index{weak-* topology} with respect
to $\mc{C}_0(\bR^d)$ (Definition \ref{weak-star}). The
Banach-Alaoglu theorem\index{Banach-Alaoglu theorem} (Theorem \ref{banach-alaoglu}) now
implies that the unit ball in $M(\bR^d)$ is compact, and
so a ``subsequence'' $(u_{\ve_k})_{k=1}^\infty$ of the
approximation to the identity converges to
a measure $\mu \in M(\bR^d)$ in this topology. In other words,
we have
\begin{equation}\label{M11-eq}
 \lim_{k \to \infty} \int \psi(x) u_{\ve_k} (x) \, dx
 = \int \psi(x) \, d\mu(x)
\end{equation}
for each $\psi \in \ms{S}(\bR^d)$.

It remains to show that $\mu$ \emph{is} the distribution $u$.
To this end, it suffices to show that
\[
 \langle \psi, u \rangle = \int \psi(x) \, d\mu(x)
\]
for each $\psi \in \ms{S}(\bR^d)$. We fix a $\psi \in \ms{S}(\bR^d)$
and set
\[
 \psi_\ve(x) = (\varphi(t) * W(t,\ve))(x) = \int \psi(x-t)W(t,\ve) \, dt
\]
for each $\ve>0$. For each multi-index $\beta$, we have
\[
 (D^\beta \psi_\ve)(x) = (D^\beta \psi(t) * W(t,\ve))(x)
 = \int (D^\beta \psi)(x-t) W(t,\ve) \, dt.
\]
The calculation in the proof of the Fourier inversion theorem
(Theorem \ref{fourier-inversion-formula})\index{Fourier inversion!on L1@on $L^1$} shows that $D^\beta \psi_\ve$
converges uniformly to $D^\beta \psi$. Therefore,
$(\psi_\ve)_{\ve>0}$ converges strongly to $\psi$ as $\ve \to 0$,
and so $\langle \psi_\ve,u \rangle \to \langle \psi, u \rangle$
as $\ve \to 0$. Since $\tilde{W}(t,\ve) = W(t,\ve)$, we have
\[
 \langle \psi_\ve , u \rangle
 = \langle W(t,\ve) * \psi(t) , u \rangle
 = \langle \psi, u * W(t,\ve) \rangle
 = \int \psi(x) u_\ve(x) \, dx.
\]
It now follows from (\ref{M11-eq}) that
\[
 \langle \psi, u \rangle
 = \lim_{k \to \infty} \langle \psi_{\ve_k} , u \rangle
 = \lim_{k \to \infty} \int \psi(x) u_{\ve_k} \, dx
 = \int \psi(x) \, d\mu(x),
\]
as was to be shown.
\end{proof}

We also have a nice characterization for $\ms{M}_{2,2}$.

\begin{theorem}\index{Fourier multiplier!L2@$L^2$}
A convolution operator $T\psi = u * \psi$ is in $\ms{M}_{2,2}$ if and only
if $\hat{u} \in L^\infty(\bR^d)$. In this case,
\[
 \|T\|_{L^2 \to L^2} = \|\hat{u}\|_\infty.
\]
It thus follows that $\mc{M}^2 = L^\infty(\bR^d)$.
\end{theorem}

\begin{proof}
Let $u \in \ms{M}_{2,2}$. For each $\psi \in \ms{S}(\bR^d)$,
the convolution theorem (Theorem \ref{convolution-theorem-on-distributions})\index{convolution theorem!on tempered distributions}
implies that
\[
 \hat{u}\hat{\psi} = \widehat{u * \psi},
\]
for each $\psi \in \ms{S}(\bR^d)$, whence by Plancherel's theorem
(Theorem \ref{plancherel})\index{Plancherel's theorem}
we have
\[
 \|\hat{u}\hat{\psi}\|_2
 = \|\widehat{u * \psi}\|_2
 = \|u * \psi\|_2.
\]
Since $u \in \ms{M}_{2,2}$, there exists a constant $C>0$ such that
\[
 \|u * \psi\|_2 \leq C\|\psi\|_2
\]
for all $\psi \in \ms{S}(\bR^d)$. Applying Plancherel's theorem
once more, we obtain the norm estimate
\[
 \|\hat{u}\hat{\psi}\|_2 \leq C\|\psi\|_2 = C\|\hat{\psi}\|_2
\]
for all $\psi \in \ms{S}(\bR^d)$

The Fourier transform is an automorphism on $\ms{S}(\bR^d)$,
and so the above is equivalent to
\[
 \|\hat{u}\psi\|_2 \leq C\|\psi\|_2
\]
for all $\psi \in \ms{S}(\bR^d)$, where
\[
 \hat{u}\psi = \hat{u} \widehat{(\psi^\vee)}.
\]
We now invoke Theorem \ref{norm-preserving-extension} to extend
the multiplication operator
\[
 \psi \mapsto \hat{u}
\]
as a bounded operator from $L^2$ into itself. By the 
norm-preserving property of the extension, it follows
that $\|\hat{u}\|_\infty \leq C$.

Conversely, if $\hat{u} \in L^\infty(\bR^d)$, then
Plancherel's theorem and the convolution theorem imply that
\[
 \|\hat{u} * \psi\|_2 = \|\widehat{u\psi^\vee}\|_2
 = \|u\psi^\vee\|_2 \leq \|\hat{u}\|_\infty \|\psi^\vee\|_2
 = \|\hat{u}\|_\infty \|\psi\|_2
\]
for each $\psi \in \ms{S}(\bR^d)$. It follows that
$u \in \mc{M}_{2,2}$, and the operator norm of the associated
convolution operator is easily seen to be $\|\hat{u}\|_\infty$.
\end{proof}

Characterization of the space $\ms{M}_{p,q}$ is trickier.
It is a standard result in harmonic analysis that $\ms{M}_{p,q}$
coincides with the space of bounded linear operators $T:L^p(\bR^d)
\to L^q(\bR^d)$ that \emph{commute with translations},\index{operator!commuting with translations} viz.,
\[
 T(\tau_h f) = \tau_h Tf
\]
for all $f \in L^p(\bR^d)$; see Chapter I, Theorem 3.16
in \cite{Stein_Weiss:B1971}\index{Stein, Elias M.}\index{Weiss, Guido} or Theorem 2.5.2 in
\cite{Loukas_Grafakos:B2008}\index{Grafakos, Loukas} for a proof. The only other known
result so far is the duality relation
\[
 \ms{M}_{p,q} = \ms{M}_{q',p'}.
\]
that holds for $p,q \in [1,\infty]$. A proof can be found in
\cite{Stein_Weiss:B1971}, Chapter I, Theorem 3.20, 
or in \cite{Loukas_Grafakos:B2008}, Theorem 2.5.7.

Fourier multipliers are special cases of pseudodifferential operators,
which, in turn, are special cases of Fourier integral operators.
We shall briefly study pseudodifferential operators
in \S\S\ref{s-sobolev-spaces} and Fourier integral operators in
\S\S\ref{s-analysis-of-the-homogeneous-wave-equation}.\index{pseudodifferential operator}\index{Fourier integral operator}\index{operator!pseudodifferential|see{pseudodifferential operator}}\index{operator!Fourier integral|see{Fourier integral operator}}
\index{tempered!distributions|)}
\index{convolution!operator|)}
\index{Fourier multiplier|)}
 
\section{The Hilbert Transform}\label{s-the_hilbert_transform}

\index{Hilbert transform|(}
In this section, we study a particularly important example of
a convolution operator, the Hilbert transform.
The Hausdorff-Young inequality tells us that the Fourier transform
is bounded on $L^p$ for $p \in [1,2]$, but it is not very clear
how the $L^p$-theory of the Fourier transform should be developed
for $p > 2$. To this end, we shall consider a Fourier multiplier
whose symbol is a constant, suitably normalized to retain
the validity of Plancherel's theorem. We shall see that
this operator is bounded on $L^p$ for all $p \in (1,\infty)$.

\subsection{The Distribution \texorpdfstring{$\pv(\frac{1}{x})$}{pv(1/x)} and the \texorpdfstring{$L^2$}{L2} Theory}

Recall that a measurable function on a subset $E$ of $\bR^d$ is \emph{locally integrable}\index{locally integrable}
if $f$ is integrable on every compact subset of $E$. If a locally integrable
function $f$ barely fails to be integrable near the origin, then it is often possible
to compute the
\emph{principal value}\index{principal value} of its integral, defined by
\[
 \pv \int f(x) \, dx = \lim_{\ve \to 0} \int_{|x| \geq \ve} f(x) \, dx.
\]

\begin{defin}
Let $f$ be measurable on $\bR^d$ and locally integrable on $\bR^d
\smallsetminus \{0\}$. The \emph{principal-value distribution}\index{principal-value distribution}
$\pv(f)$ is defined to be
\[
 \langle \varphi, \pv(f) \rangle
 = \lim_{\ve \to 0} \int_{|x| \geq \ve} f(x) \varphi(x) \, dx
\]
for each $\varphi \in \ms{S}(\bR^d)$, provided that the limit
exists for all such $\varphi$.
\end{defin}

\begin{prop}
The one-dimensional principal-value distribution $\pv(\frac{1}{x})$ is
a well-defined tempered distribution. 
\end{prop}

\begin{proof}
Fix $\varphi \in \ms{S}(\bR)$ and suppose that $\ve \leq 1$.
We write
\begin{equation}\label{decomposing-hilbert}
 \int_{|x| \geq \ve} \frac{\varphi(x)}{x} \, dx
 = \int_{1 \geq |x| \geq \ve} \frac{\varphi(x)}{x} \, dx
 + \int_{|x| > 1} \frac{\varphi(x)}{x} \, dx.
\end{equation}
Since the Schwartz function $\varphi$ decays rapidly at infinity,
the second integral in the right-hand side of (\ref{decomposing-hilbert})
converges. As for the first one, we note that
\[
 \int_{1 \geq |x| \geq \ve} \frac{1}{x} \, dx = 0.
\]
Therefore,
\[
 \int_{1 \geq |x| \geq \ve} \frac{\varphi(x)}{x} \, dx
 = \int_{1 \geq |x| \geq \ve} \frac{\varphi(x)-\varphi(0)}{x} \, dx
 \leq \int_{1 \geq |x| \geq \ve}
 \frac{\ds\left(\sup_x |\varphi'(x)|\right) |x|}{x} \, dx,
\]
and so the integral in the left-hand side of (\ref{decomposing-hilbert})
converges. The above computation also yields the estimate
\[
 \left| \lim_{\ve \to 0} \int_{|x| \geq \ve} \frac{\varphi(x)}{x} \, dx \right|
 \leq k \sum_{\substack{|\alpha| \leq 1 \\ |\beta| \leq 1}}
 \rho_{\alpha\beta}(\varphi)
\]
for some constant $k$, whence it follows from Theorem \ref{iff-tempered-distribution}
that $\pv(\frac{1}{x})$ is a tempered distribution.
\end{proof}

\begin{defin}
The \emph{Hilbert transform} $\mc{H}$ is the convolution operator
\[
 \mc{H} \psi = \frac{1}{\pi}\pv\left(\frac{1}{x}\right) * \psi.
\]
\end{defin}

As was alluded to above, we now show that the Hilbert transform
is a Fourier multiplier.

\begin{theorem}\label{L2-hilbert-transform}\index{Hilbert transform!as a Fourier multiplier|(}
$\mc{H}$ is an $L^2$ Fourier multiplier with the symbol $-i \sgn(x)$.
\end{theorem}

\begin{proof}
For each $\psi \in \ms{S}(\bR^d)$, we observe that
\begin{eqnarray*}
 \left< \psi , \widehat{\pv\left( \frac{1}{x} \right)} \right>
 &=& \left< \widehat{\psi} , \pv\left( \frac{1}{x} \right) \right> \\
 &=& \lim_{\ve \to 0} \int_{|\xi| \geq \ve}
 \frac{\hat{\psi}(\xi)}{\xi} \, d\xi \\
 &=& \lim_{\ve \to 0} \int_{|\xi| \geq \ve}
 \int_{-\infty}^\infty \frac{\psi(x)e^{-2 \pi i \xi x}}{\xi} \, dx \, d\xi \\
 &=& \lim_{\ve \to 0} \int_{\frac{1}{\ve} \geq|\xi| \geq \ve}
 \int_{-\infty}^\infty \frac{\psi(x)e^{-2 \pi i \xi x}}{\xi} \, dx \, d\xi \\
 &=& \lim_{\ve \to 0} \int_{-\infty}^\infty
 \int_{\frac{1}{\ve} \geq|\xi| \geq \ve}
 \frac{\psi(x)e^{-2 \pi i \xi x}}{\xi} \, d\xi \, dx \\
 &=& \lim_{\ve \to 0} \int_{-\infty}^\infty \psi(x)
 \left(-i \int_{\frac{1}{\ve} \geq |\xi| \geq \ve}
 \frac{\sin(2 \pi \xi x)}{\xi} \, d\xi \right) \, dx.
\end{eqnarray*}

Since
\[
 \left|\int_a^b \frac{\sin \xi}{\xi} \, dx \right| \leq 4
\htwo\mbox{and}\htwo
 \int_{-\infty}^\infty \frac{\sin(b \xi)}{\xi} \, dx = \pi \sgn(b)
\]
for all $0<a<b<\infty$, we see that the quantity
\[
 -i \int_{\frac{1}{\ve} \geq |\xi| \geq \ve}
 \frac{\sin(2 \pi \xi x)}{\xi} \, d\xi
\]
are uniformly bounded by 8\index{8} and converges to $\pi \sgn(x)$
as $\ve \to 0$. We now apply the dominated convergence theorem
to conclude that
\[
 \lim_{\ve \to 0} \int_{-\infty}^\infty \psi(x)
 \left(-i \int_{\frac{1}{\ve} \geq |\xi| \geq \ve}
 \frac{\sin(2 \pi \xi x)}{\xi} \, d\xi \right) \, dx
 = \pi\int_{-\infty}^\infty \psi(x)(-i \sgn(x)) \, dx,
\]
whence the Fourier transform of $\pv(\frac{1}{x})$ is
$-i \pi \sgn(\xi)$.

The convolution theorem (\ref{convolution-theorem-on-distributions})\index{convolution theorem!on tempered distributions}
now implies that
\[
 \widehat{\mc{H}\psi}(\xi)
 = \frac{1}{\pi} \mc{F}\left(\pv \left( \frac{1}{x} \right) * \psi \right)(\xi) 
 = \frac{1}{\pi} \widehat{\pv \left( \frac{1}{x} \right)} \widehat{\psi}(\xi)
 = -i \sgn(\xi) \hat{\psi}(\xi),
\]
and Plancherel's theorem (Theorem \ref{plancherel})\index{Plancherel's theorem} yields the equality
\[
 \|\mc{H}\psi\|_2 = \|\widehat{\mc{H}\psi}\|_2
 = \|-i \sgn(\xi) \hat{\psi}(\xi)\|_2
 = \|\hat{\psi}\|_2 = \|\psi\|_2.
\]
It follows that $\mc{H}$ is an $L^2$ Fourier multiplier with the symbol
$-i \sgn(\xi)$, as was to be shown.
\end{proof}
\index{Hilbert transform!as a Fourier multiplier|)}

\subsection{The \texorpdfstring{$L^p$}{Lp} Theory}

We now tackle the main theorem of the present section,
usually attributed to Marcel Riesz.
While Riesz himself never studied the problem himself,
the subject of his
1927 paper \cite{Marcel_Riesz:J1927-2}\index{Riesz, Marcel} is now known to be
directly relevant to the study of the Hilbert transform.
Riesz's original proof of the related result
exploits the close relationship between the Hilbert
transform and the Cauchy integral in complex analysis.
A suitably modified
version of this proof that serves directly as the proof of
the $L^p$-boundedness of the Hilbert transform
can be found in Chapter 2, Section 3 of
\cite{Stein_Shakarchi:B2011}.\index{Stein, Elias M.}\index{Shakarchi, Rami}

For the sake of brevity, we do not pursue the connections
with complex analysis. Instead, we follow the approach
in \S\S4.1.3 of \cite{Loukas_Grafakos:B2008}\index{Grafakos, Loukas} and
base our proof on the following identity:

\begin{lemma}\label{interesting-hilbert-identity}
$(\mc{H}\psi)^2 = \psi^2 + 2\mc{H}(\psi(\mc{H}\psi))$
for each $\psi \in \ms{S}(\bR^d)$.
\end{lemma}

\begin{proof}[Proof of lemma]
We let $m(\xi) = - i \sgn(\xi)$ be the symbol of the
Hilbert transform. Taking the Fourier transform of the right-hand side,
we obtain
\begin{eqnarray*}
 \ms{F}\left[ \psi^2 + 2\mc{H}(\psi(\mc{H}\psi)) \right](\xi) 
 &=& \widehat{\psi^2}(\xi) + 2\ms{F}\left[ \mc{H}(\psi (\mc{H} \psi)) \right](\xi) \\
 &=& \left( \hat{\psi} * \hat{\psi} \right) (\xi) + 
 2m(\xi) \left( \hat{\psi} * \widehat{\mc{H}\psi} \right)(\xi) \\
 &=& \int \hat{\psi}(\eta) \hat{\psi}(\xi - \eta) \, d \eta \\
 & & + 2m(\xi) \int \hat{\psi}(\eta) m(\eta) \hat{\psi}(\xi-\eta) \, d\eta \\
 &=& \int \hat{\psi}(\eta) \hat{\psi}(\xi - \eta) \, d \eta \\
 & & + 2m(\xi) \int \hat{\psi}(\eta) m(\xi - \eta) \hat{\psi}(\xi-\eta) \, d\eta;
\end{eqnarray*}
here we have used the convolution theorem (Theorem \ref{convolution-theorem}).\index{convolution theorem!on Lp, for 1 p 2@$L^p$, for $1 \leq p \leq 2$}
We average the last two quantities in the above inequality to conclude that
\[
 \ms{F}[ \psi^2 + 2\mc{H}(\psi(\mc{H}\psi)) ](\xi) 
 = \int \hat{\psi}(\eta) \hat{\psi}(\xi - \eta)
 \left( 1 + m(\xi) \left[ m(\eta) + m(\xi - \eta) \right] \right) \, d\eta.
\]
Since
\begin{eqnarray*}
 1 + m(\xi) m(\eta) + m(\xi) m(\xi - \eta)
 &=& 1 - \sgn(\xi)\sgn(\eta) - \sgn(\xi)\sgn(\xi - \eta) \\
 &=&
 \begin{cases}
  1 & \mbox{ if } \eta > \xi > 0 \\
  0 & \mbox{ if } \xi = \eta > 0 \\
  -1 & \mbox{ if } \xi > \eta > 0 \\
  0 & \mbox{ if } \xi > \eta = 0 \\
  1 & \mbox{ if } \xi = 0 \\
  0 & \mbox{ if } \xi < \eta = 0 \\
  -1 & \mbox{ if } \xi < \eta < 0 \\
  0 & \mbox{ if } \xi = \eta < 0 \\
  1 & \mbox{ if }  \eta < \xi < 0
 \end{cases} \\
 &=&
 \begin{cases}
  1 & \mbox{ if } \frac{\eta}{\xi} > 1 \mbox{ or } \xi = 0 \\
  0 & \mbox{ if } \frac{\eta}{\xi} = 1 \mbox{ or } \eta = 0 \\
  -1 & \mbox{ if } \frac{\eta}{\xi} < 1
 \end{cases} \\
 &=& m(\eta)m(\xi-\eta),
\end{eqnarray*}
it follows from the convolution theorem (Theorem \ref{convolution-theorem-on-distributions})\index{convolution theorem!on tempered distributions}
that
\begin{eqnarray*}
  \ms{F}[ \psi^2 + 2\mc{H}(\psi(\mc{H}\psi)) ](\xi) 
 &=& \int \hat{\psi}(\eta)\hat{\psi}(\xi - \eta) m(\eta) m(\xi - \eta) \, d\eta \\
 &=& \int \widehat{\mc{H}\psi}(\xi) \widehat{\mc{H}\psi}(\xi - \eta) \, d\eta \\
 &=& \left( \widehat{\mc{H}\psi} * \widehat{\mc{H}\psi} \right)(\xi) \\
 &=& \widehat{(\mc{H}\psi)^2}(\xi).
\end{eqnarray*}
Applying the inverse Fourier transform on both sides, we obtain the
desired equality.
\end{proof}

We are now ready to establish the $L^p$-boundedness
of the Hilbert transform.

\begin{theorem}[M. Riesz]\label{Lp-boundedness-of-the-hilbert-transform}\index{Hilbert transform!Lp-boundedness of@$L^p$-boundedness of|(}
$\mc{H} \in \ms{M}_{p,p}$ for all $1 < p <\infty$.
\end{theorem}

\begin{proof}
We have already established that $\mc{H} \in \ms{M}_{2,2}$ (Theorem \ref{L2-hilbert-transform}).
Fix a positive integer $n$ and assume inductively that
\[
 \|\mc{H}\psi\|_p \leq A_p \|\psi\|_p
\]
for $p=2^n$. Lemma \ref{interesting-hilbert-identity} implies that
\[
 \|\mc{H}\psi\|_{2p}
 = \|(\mc{H}\psi)^2\|^{1/2}_p \\
 = \| \psi^2 + 2\mc{H}(\psi(\mc{H}\psi)) \|^{1/2}_p,
\]
and so
\begin{eqnarray*}
 \|\mc{H}\psi\|_{2p}
 &\leq& \left( \|\psi^2\|_p  + 2\|\mc{H}(\psi(\mc{H}\psi))\|_p \right)^{1/2} \\
 &\leq& \left( \|\psi^2\|_p  + 2A_p\|\psi(\mc{H}\psi)\|_p \right)^{1/2} \\
 &\leq& \left( \|\psi^2\|_p  + 2A_p\|\psi\|_{2p}\|\mc{H}\psi)\|_{2p} \right)^{1/2};
\end{eqnarray*}
the last inequality follows from the Cauchy-Schwarz inequality.

It follows that
\begin{eqnarray*}
 0
 &\geq&
 \left( \frac{\|\mc{H}\psi\|_{2p}}{\|\psi\|_{2p}} \right)^2
 - 2 A_p \left( \frac{\|\mc{H}\psi\|_{2p}}{\|\psi\|_{2p}} \right) 
 - 1 \\
 &=& \left[ \left( \frac{\|\mc{H}\psi\|_{2p}}{\|\psi\|_{2p}} \right)
 - A_p + \sqrt{1 + A_p^2} \right]
 \left[ \left( \frac{\|\mc{H}\psi\|_{2p}}{\|\psi\|_{2p}} \right)
 - A_p - \sqrt{1 + A_p^2} \right],
\end{eqnarray*}
whence
\[
 A_p - \sqrt{1 + A_p^2} \leq
  \frac{\|\mc{H}\psi\|_{2p}}{\|\psi\|_{2p}}
 \leq A_p + \sqrt{1 + A_p^2}.
\]
In particular, we conclude that
\[
 \|\mc{H}\psi\|_{2p} \leq \left( A_p + \sqrt{1 + A_p^2} \right) \|\psi\|_{2p},
\]
which establishes $\mc{H} \in \ms{M}_{2^m,2^m}$ for all positive integers $m$.
We now invoke the Riesz-Thorin interpolation theorem (Theorem \ref{riesz-thorin}),\index{Riesz-Thorin interpolation \\ theorem}
to obtain the $L^p$-boundedness for all $p \geq 2$.

It remains to show that $\mc{H} \in \ms{M}_{p,p}$ for $1 < p \leq 2$, and this
requires a standard duality argument. By Plancherel's theorem\index{Plancherel's theorem}
(Theorem \ref{plancherel}), we have
\begin{eqnarray*}
 \langle \mc{H} \psi, \varphi \rangle
 &=& \langle \widehat{\mc{H} \psi}, \widehat{\varphi} \rangle
 = \langle -i\sgn(\xi) \psi , \hat{\varphi} \rangle \\
 &=& \langle  \hat{\psi} , -i\sgn(\xi) \hat{\varphi} \rangle
 = \langle \widehat{\psi}, \widehat{\mc{H} \varphi} \rangle \\
 &=& \langle \psi,  \mc{H} \varphi \rangle
\end{eqnarray*}
with respect to the standard inner product in $L^2(\bR)$. 
This, combined with the \hyperref[Lp-riesz-representation]{Riesz
representation theorem},\index{representation theorem!F. Riesz, Lp-space version@F. Riesz, $L^p$-space version} implies that
\begin{eqnarray*}
 \|\mc{H}\psi\|_p
 &=& \sup_{\|\varphi\|_{p'} \leq 1}
 \left| \int (\mc{H}\psi) \varphi \right|  \\
 &=& \sup_{\|\varphi\|_{p'} \leq 1}
 \left| \int (\mc{H}\psi) \bar{\varphi} \right| 
 = \sup_{\|\varphi\|_{p'} \leq 1}
 \langle \mc{H} \psi, \varphi \rangle \\
 &=& \sup_{\|\varphi\|_{p'} \leq 1}
 \langle \psi,  \mc{H} \varphi \rangle;
\end{eqnarray*}
here the density of $\mc{S}(\bR)$ in $L^{p'}(\bR)$
allows us to use only the Schwartz function to compute
the operator norm of the functional
\[
 f \mapsto \int \mc{H}\psi f.
\]
Since $p' \geq 2$, we can apply what we have proved above
to conclude that
\[
 \|\mc{H}\psi\|_p
 \leq \sup_{\|\varphi\| \leq 1} \|\varphi\|_{p'} \|\psi\|_p 
 \leq A_{p'} \|\psi\|_p.
\]
This completes the proof.
\end{proof}
\index{Hilbert transform!Lp-boundedness of@$L^p$-boundedness of|)}

\subsection{Singular Integral Operators}

One crucial drawback of the theory developed in this section
is that the Hilbert transform is only defined on $\bR$. In order
to study multidimensional Fourier analysis in this setting, we
would expect that it is
necessary to define higher-dimensional analogues of
the Hilbert transform.

\begin{defin}
Given an integer $1 \leq j \leq d$, we define the \emph{$n$th Riesz transform}\index{Riesz transform}
$\mc{R}_n$ on $\bR^d$ to be the convolution operator
\[
 \mc{R}_n \psi = \frac{1}{\omega_d} \pv\left( \frac{x_n}{|x|^{d+1}} \right) * \psi,
\]
where $\omega_d$ is the volume of the $d$-dimensional ball. 
\end{defin} 

A minor modification of the argument given in the proof of Theorem
\ref{L2-hilbert-transform} shows that the Riesz transforms are $L^2$
Fourier multipliers. How about the $L^p$ boudedness?
To this end, we remark that both the Hilbert transform and the
Riesz transforms are integral operators of the form
\[
 \int K(x-y)f(y) \, dy,
\]
where the kernel $K$ barely fails to be integrable on the diagonal $x=y$.
The transforms, therefore, are instances of \emph{singular integral
operators}, and their $L^p$-theory is subsumed to that of a much wider
class of operators. We give one such example:

\begin{defin}
$K \in L^2(\bR^d)$ is a \emph{Calder\'{o}n-Zygmund kernel}\index{Calder\'{o}n-Zygmund!kernel}
if the following conditions are met:
\begin{enumerate}[(a)]
 \item The Fourier transform $\hat{K}$ of $K$ is in $L^\infty(\bR^d)$.
 \item $K \in \mc{C}^1(\bR^d \smallsetminus \{0\})$ and
 \[
  |\nabla K(x)| \leq \frac{\|\hat{K}\|_\infty}{|x|^{d+1}}.
 \]
\end{enumerate}
\end{defin}

\index{Calder\'{o}n-Zygmund!singular integral \\ operator}
\begin{theorem}[Calder\'{o}n-Zygmund]\label{calderon-zygmund}\index{Calder\'{o}n-Zygmund!theorem}
Let $p \in (1,\infty)$.
If $K$ is a Calder\'{o}n-Zygmund kernel, then the \emph{singular integral
operator of Calder\'{o}n-Zygmund type}
\[
 Tf = \int K(x-y)f(y) \, dy,
\]
initially defined for $f \in L^1(\bR^d) \cap L^p(\bR^d)$, satisfies
the norm estimate
\[
 \|Tf\|_p \leq A_p \|f\|_p
\]
for some constant $A_p$ independent of $f$ and $\|K\|_2$.
Therefore, $T$ can be extended to a bounded operator $T:L^p(\bR^d) \to L^p(\bR^d)$.
\end{theorem}

The classical proof of the theorem makes use of the
technique known as \emph{Calder\'{o}n-Zygmund decomposition}\index{Calder\'{o}n-Zygmund!decomposition} and can be
found in Chapter 2, Section 2 of \cite{Elias_M_Stein:B1970}.
A key element in the proof is the interpolation theorem of Marcinkiewicz,
which is discussed briefly in \S\ref{fr-marcinkiewicz-and-real-interpolation}.
In the present thesis, we shall derive this result as a consequence
of the Fefferman-Stein interpolation theorem, which we take up in the next section.
\index{Hilbert transform|)}

\section{Hardy Spaces and \texorpdfstring{$\BMO$}{BMO}}\label{s-hardy-spaces-and-bmo}

\index{Fefferman, Charles|(}\index{Stein, Elias M.|(}
We have seen in the last section that
a clever use of the \hyperref[riesz-thorin]{Riesz-Thorin interpolation theorem}
at intermediate points establishes the
\hyperref[Lp-boundedness-of-the-hilbert-transform]{$L^p$ boundedness of the Hilbert transform}
without the endpoint estimates. This is not entirely satisfying, for there is
no clear way to generalize the proof to a wider class of operators.

In this section, we take a stroll through a theory of interpolation that allows us
to interpolate operators that are not necessarily bounded operators between
Lebesgue spaces. In particular, we shall consider two new Banach spaces,
$H^1$ and $\BMO$, which serve as substitutes for $L^1$ and $L^\infty$,
respectively. The section will culminate in an interpolation
theorem between $L^2$ and $\BMO$ and its dual theorem between $H^1$ and $L^2$.

First presented by Charles Fefferman and Elias Stein in
\cite{Fefferman_Stein:J1972}, the theory of interpolation on $H^1$ and
$\BMO$ is laden with intricate technical details and requires more than
a mere section for a full development. We shall therefore confine ourselves
to stating the main definition and theorems, with occasional sketches of
proofs. 

\subsection{The Hardy Space \texorpdfstring{$H^1$}{H1}}

\index{H1@$H^1$|(}
\index{Hardy space|see{H1@$H^1$}}
Since the Hilbert transform of an $L^1$-function is not necessarily
in $L^1$, it is natural to consider the subspace $H^1(\bR)$ of $L^1(\bR)$
consisting of $L^1$-functions whose Hilbert transforms are in $L^1$ as well.
Many important operators, however, take functions on a multidimensional
Euclidean space as their input, and it is thus of interest to consider the
$d$-dimensional generalization $H^1(\bR^d)$ of $H^1(\bR)$, consisting
of $L^1$-functions whose Riesz transforms are in $L^1$ as well. More generally,
we define, for each $p \geq 1$, the subspace $H^p(\bR^d)$ of $L^p(\bR^d)$
as follows:

\begin{defin}\label{riesz-transform-characterization-of-H1}
The (real) \emph{Hardy space $H^p(\bR^d)$ of order $p$} consists of functions
$f \in L^p(\bR^d)$ such that the sum
\[
 \|f*\rho_\ve\|_p + \sum_{n=1}^d \|(\mc{R}_n f) * \rho_\ve\|_p 
\]
is bounded for each approximations to the identity
$(\rho_\ve)_{n=1}^\infty$.
\end{defin} 

By the $L^p$-boundedness of the Riesz transforms, the
Hardy space $H^p(\bR^d)$ coincides with the Lebesgue space $L^p(\bR^d)$ for
all $1 < p < \infty$. As for $p=1$, the above discussion implies that
the Hardy space $H^1(\bR^d)$ is strictly smaller than the Lebesgue space
$L^1(\bR^d)$. Moreover, we can write $H^1$-functions as linear combinations
of particularly basic functions in $H^1$, known as \emph{atoms}.

\index{H1@$H^1$!atomic decomposition of|(}
\begin{defin}
A $d$-dimensional \emph{$H^1$-atom} is a measurable function $\mf{a}$
supported in a ball $B$ in $\bR^d$ such that $|\mf{a}(x)| \leq |B|^{-1}$
for almost every $x \in \bR^d$ and $\int \mf{a}(x) \, dx = 0$.
\end{defin}

The atomic decomposition of $H^1$ can be stated as follows:

\begin{theorem}[Atomic decomposition of $H^1$]\label{atomic-decomposition-characterization-of-H1}
Every $d$-dimensional $H^1$-atom belongs to $H^1(\bR^d)$,
and every function $f \in H^1(\bR^d)$ can be written as
the infinite linear combination
\[
 f = \sum_{n=1}^\infty \lambda_n \mf{a}_n
\]
of $H^1$-atoms $(\mf{a}_n)_{n=1}^\infty$ with
$\sum |\lambda_n| < \infty$, where the sum
is understood as the limit of partial sums in $L^1(\bR^d)$.
Therefore, a function $f \in L^1(\bR^d)$ is in $H^1(\bR^d)$
if and only if $f$ admits an atomic decomposition.
\end{theorem}
\index{H1@$H^1$!atomic decomposition of|)}

Stein, in \cite{Elias_M_Stein:B1993}, proves the equivalence
of the Riesz-transform definition and the maximal-function
definition---which we do not discuss---in Chapter III, Section 4.3.
The equivalence of the maximal-function definition and the
atomic-decomposition definition is proved in Chapter III, Section 2.2. 
The $L^1$-norm-convergence characterization is taken from
Chapter 2, Section 5.1 in \cite{Stein_Shakarchi:B2011},\index{Shakarchi, Rami} which
takes the atomic-decomposition characterization as the definition
of $H^1$.

With the above theorem, we can now define the \emph{$H^1$-norm}\index{norm of}
of $f \in H^1(\bR^d)$ as
\[
 \|f\|_{H^1} = \inf \sum_{n=1}^\infty |\lambda_n|,
\]
where the infimum is taken over all atomic decompositions of $f$.
$H^1$ is a Banach space with this norm. Furthermore, $H^1$
behaves much nicer than $L^1$, because the singular integral operators\index{Calder\'{o}n-Zygmund!singular integral \\ operator}
of Calder\'{o}n-Zygmund type are bounded operators from $H^1$ to $L^1$:
this is proved in \cite{Elias_M_Stein:B1993}, Chapter III, Section 3.1.

We can, therefore, consider $H^1$ as a better substitute for
$L^1$ in many cases. As is the case with $L^1$, the dual of
$H^1$ can be realized as a concrete space of functions. This
is the space of bounded mean oscillations, which we shall
define in due course.
\index{H1@$H^1$|)}

\subsection{Interlude: The Maximal Function}

In order to chasracterize the dual of $H^1$, we must
find a suitable substitute space for $L^\infty$.
To this end, we shall shift our
focus from controlling the functions themselves to dealing with
their mean values instead. We review the basic theory of integral
mean values in this section, focusing in particular on a dominating
function of the mean-value function, the Hardy-Littlewood maximal function.

For each $\delta>0$, we recall that the \emph{mean value}\index{mean value}\index{integral mean value|see{mean value}} $(A_\delta f)(x)$ of a
measurable function $f$ at $x \in \bR^d$ is defined by
\[
 (A_\delta f)(x) = \frac{1}{|B_\delta(x)|} \int_{B_\delta(x)} f(y) \, dy
\]
The \emph{Lebesgue differentiation theorem}\index{Lebesgue, Henri!differentiation theorem} guarantees that the mean value
is a good approximation of the actual function. More precisely, if
$f:\bR^d \to \bC$ is locally integrable, then 
\[
 \lim_{\delta \to 0} (A_\delta f)(x) = f(x)
\]
for almost every $x$. Proving such a convergence result often requires
an estimate on the \emph{size} of the operator. For the mean-value
operator, we introduce the following maximal-estimate operator:

\index{Hardy-Littlewood maximal function|(}
\begin{defin}
The \emph{Hardy-Littlewood maximal function} of a measurable function $f$
on $\bR^d$ is
\[
 (\mc{M}f)(x) = \sup_{\delta > 0} (A_\delta |f|)(x)
 = \sup_{\delta > 0} \frac{1}{|B_\delta(x)|} \int_{B_\delta(x)} |f(y)| \, dy.
\]
\end{defin}

Similar to the Hilbert transform and the Riesz transforms, the maximal
function does not map $L^1$ into $L^1$. Indeed, if $f \in L^1(\bR^d)$
is not of $L^1$-norm zero, then there exists a ball $B$ in $\bR^d$
such that $\int_B |f| > 0$. We fix some $\delta>0$ such that
$B$ is contained in $B_\delta(1)$ and set
\[
 k = \frac{1}{|B_\delta(1)|} \int_B |f(y)| \, dy.
\]
For each $|x| \geq 1$, we can now establish the following lower bound:
\begin{eqnarray*}
 (\mc{M}f)(x)
 &\geq& \frac{1}{|B_{\delta+(|x|-1)}(x)|} \int_{B_{\delta+(|x|-1)}(x)} |f(y)| \, dy \\
 &\geq& \frac{1}{|x|^d|B_\delta(1)|} \int_{B_\delta(1)} |f(y)| \, dy \\
 &\geq& \frac{1}{|x|^d} \frac{1}{|B_\delta(1)|} \int_{B} |f(y)| \, dy \\
 &=& \frac{k}{|x|^d}.
\end{eqnarray*}
It follows that $\mc{M}f$ is not integrable on $\bR^d$.

As a substitute for the norm estimate, we have the following inequality:

\begin{theorem}[Hardy-Littlewood maximal inequality]\label{hary-littlewood-maximal-inequality}\index{Hardy-Littlewood maximal function!maximal inequality}\index{Hardy-Littlewood maximal function!weak-type (1,1) estimate of|see{maximal inequality}}
If $f \in L^1(\bR^d)$, then we have the \emph{weak-type} estimate
\[
 m\{x: (\mc{M}f) (x) > \alpha \} \leq \frac{A_d}{\alpha} \|f\|_1,
\]
where $A_d$ is a constant that depends only on the dimension $d$.
\end{theorem}

The standard proof of the inequality makes use of the \emph{Vitali covering
lemma} and is covered in many standard textbooks in real analysis. See,
for example, Chapter 3, Theorem 1.1 in \cite{Stein_Shakarchi:B2005}
Theorem 3.17 in \cite{Gerald_B_Folland:B1999}, or Section 7.5 in
\cite{Walter_Rudin:B1986}. Even though $\mc{M}$ is not of type $(1,1)$,
we have the $L^\infty$-norm estimate
\[
 \|\mc{M}f\|_\infty = \|f\|_\infty
\]
and, using these endpoint estimates, we can establish the following
interpolated bounds:

\begin{theorem}[$L^p$-boundedness of the maximal function]\label{lp-boundedness-of-the-maximal-function}\index{Hardy-Littlewood maximal function!Lp-boundedness of@$L^p$-boundedness of}
If $f \in L^p(\bR^d)$ for some $1 < p < \infty$, then $\mc{M}f \in
  L^p(\bR^d)$ and
  \[
   \|\mc{M}f\|_p \leq A_{p,d} \|f\|_p,
  \]
where $A_{p,d}$ is a constant that depends only on $p$ and the dimension $d$.
\end{theorem}

The norm estimates are established by splitting $\mc{M}f$ into its large
and small parts: see pages 4-7 in \cite{Elias_M_Stein:B1970} for the proof.
The idea of the proof can be generalized to establish an interpolation theorem
for ``weak-type'' endpoint estimates. This is the theorem of Marcinkiewicz\index{Marcinkiewicz, J\'{o}zef!interpolation theorem} and
serves as a starting point for the real method of interpolation, which is
discussed in \S\ref{fr-marcinkiewicz-and-real-interpolation}.
\index{Hardy-Littlewood maximal function|)}

\subsection{Functions of Bounded Mean Oscillation}

\index{BMO@$\BMO$|(}\index{bounded mean oscillations|see{$\BMO$}}
We now consider a space of functions whose integral mean values\index{mean value} are
controlled.

\begin{defin}
The space $\BMO(\bR^d)$ of \emph{bounded mean oscillations} on $\bR^d$ consists
of locally integrable functions $f$ on $\bR^d$ such that
\begin{equation}\label{bmo-definition}
 \frac{1}{|B|} \int_B |f(x)-f_B| \, dx
\end{equation}
is uniformly bounded for all balls $B$, where $f_B$ is the integral mean value
\[
 \frac{1}{|B|} \int_B f(x) \, dx
\]
over $B$.
\end{defin}

The infimum of all uniform bounds of (\ref{bmo-definition}) is denoted
by $\|f\|_{\BMO}$.\index{BMO@$\BMO$!norm of} Provided that we consider the quotient space\index{quotient space} given by the equivalence
relation
\[
 f \sim g \htwo \Leftrightarrow \htwo f = g+k
 \hone \mbox{ for some constant } \hone k,
\]
$\|\cdot\|_{\BMO}$ is a complete norm on $\BMO(\bR^d)$. Functions of bounded mean
oscillation are ``nearly bounded'' in the following sense: $f \in \BMO(\bR^d)$
if and only if its \emph{sharp maximal function}\index{sharp maximal function}
\[
 f^\sharp(x) = \sup_{B \ni x} \frac{1}{|B|} \int_B |f(y) - f_B| \, dy
\]
is bounded. Since the sharp maximal function is dominated by the
Hardy-Littlewood maximal function at each point, we have the norm estimate
\[
 \|f^\sharp\|_p \leq \|\mc{M}f\|_p \leq A_{p,d} \|f\|_p
\]
for all $1 < p < \infty$.

We now turn to the remarkable observation of C. Fefferman that
there is a duality relationship between $H^1$ and $\BMO$,
analogous to that of $L^1$ and $L^\infty$.
In other words, each bounded linear functional $l$ on $H^1(\bR^d)$ 
can be represented as
\begin{equation}\label{h1-bmo}
 l(f) = \int f(x) u(x) \, dx
\end{equation}
for a unique $u \in \BMO(\bR^d)$. We do not have a substitute for H\"{o}lder's
inequality, however, and so the integral in (\ref{h1-bmo}) need not be
well-defined. To bypass the problem, we first consider bounded linear functionals
on the space $H^1_a(\bR^d)$ of bounded, compactly-supported $H^1$-functions
with integral zero. $H^1_a(\bR^d)$ is precisely the collection of finite linear
combinations of $H^1$-atoms, which is dense by the argument in
Chapter 3, Section 2.4 of \cite{Elias_M_Stein:B1993}. The integral in
(\ref{h1-bmo}) then converges and remains the same for all representatives
of the equivalence class $[u]$ in $\BMO$. Furthermore, we can invoke
Theorem \ref{norm-preserving-extension} to extend the linear functionals
of the form (\ref{h1-bmo}) onto $H^1$.

With this, we can now state the fundamental theorem of Charles Fefferman,
which is the first main result in \cite{Fefferman_Stein:J1972}:

\begin{theorem}[C. Fefferman duality]\label{fefferman-duality}\index{Fefferman, Charles!duality theorem}\index{BMO@$\BMO$!dual of|see{Fefferman duality \\ theorem}}\index{representation theorem!on BMO@on $\BMO$|see{Fefferman duality theorem}}
If $u \in \BMO(\bR^d)$, then the linear functional of the form (\ref{h1-bmo}),
initially defined on $H^1_a(\bR^d)$ and extended onto $H^1(\bR^d)$, is
bounded and satisfies the norm inequality
\[
 \|l\| \leq k\|u\|_{\BMO}
\]
for some constant $k$.
Conversely, every bounded linear functional $l$ on $H^1(\bR^d)$ can be
written in the form (\ref{h1-bmo}) and satisfies the norm inequality
\[
 \|u\|_{\BMO} \leq k' \|l\|
\]
for some constant $k'$.
\end{theorem}

See Chapter 3, Theorem 1 in \cite{Elias_M_Stein:B1993} for a proof. With
the duality theorem, we can establish a reverse norm estimate
\[
 \|f\|_p \leq A_{p,d}'\|f^\sharp\|_p
\]
for all $1 < p < \infty$. A proof of the inequality can be found in
Chapter 3, Section 2 of \cite{Elias_M_Stein:B1993}.
\index{BMO@$\BMO$|)}

\subsection{Interpolation on \texorpdfstring{$H^1$}{H1} and \texorpdfstring{$\BMO$}{BMO}}

We now come to the main interpolation theorem of the chapter.

\index{Fefferman-Stein interpolation \\ theorem!on BMO@on $\BMO$|(}
\begin{theorem}[Fefferman-Stein interpolation, $\BMO$ version]\label{fefferman_stein-bmo}
For each $\theta$ in $(0,1)$, we have
\[
 [L^2(\bR^d),\BMO(\bR^d)]_\theta = L^{p_\theta}(\bR^d)
\]
in the language of complex interpolation, where
\[
 p_\theta = \frac{2}{1-\theta}.
\]
This generalizes the \hyperref[stein]{Stein interpolation theorem}
in the following sense:

For each $z$ in the closed strip
\[
 \{z \in \bC : 0 \leq \Re z \leq 1\},
\]
let us assume that we have a bounded linear operator
$T_z:L^2(\bR^d) \to L^2(\bR^d)$ such that each
\[
 z \mapsto \int_{\bR^d} (T_zf)g
\]
is a holomorphic function in the interior of $S$ and is continuous
on $S$ and that the norms $\|T_z\|_{L^2 \to L^2}$ of the operators
are uniformly bounded. If there exists a constant $k$ such that
\[
 \|T_{iy}f\|_2 \leq k\|f\|_L^2 
\]
for all $f \in L^2(\bR^d)$ and $y \in \bR$ and that
\[
 \|T_{1+iy}f\|_{\BMO} \leq k\|f\|_\infty
\]
for all $f \in L^2(\bR^d) \cap L^\infty(\bR^d)$ and $y \in \bR$,
then we have the interpolated bound
\[
 \|T_\theta f\|_{p_\theta} \leq k_\theta \|f\|_{p_\theta}
\]
for each $\theta \in (0,1)$ and every $f \in L^2 \cap L^p$,
where
\[
 p_\theta = \frac{2}{1-\theta}.
\]
\end{theorem}
\index{Fefferman-Stein interpolation \\ theorem!on BMO@on $\BMO$|)}

The uniform boundedness hypothesis in the above theorem is
for clarity's sake can be relaxed to resemble properly the hypothesis
of the \hyperref[stein]{Stein interpolation theorem}.\index{Stein, Elias M.!interpolation theorem} The proof
makes uses of the properties of the sharp maximal function and
can be found in Chapter 3, Section 5.2 of \cite{Elias_M_Stein:B1993}.

Using the \hyperref[fefferman-duality]{Fefferman duality theorem},\index{Fefferman, Charles!duality theorem}
we can now modify the argument given in the proof of the
\hyperref[Lp-boundedness-of-the-hilbert-transform]{$L^p$-boundedness
of the Hilbert transform}\index{Hilbert transform!Lp-boundedness of@$L^p$-boundedness of} to establish the following dual result:

\index{Fefferman-Stein interpolation \\ theorem!on H1@on $H^1$|(}
\begin{theorem}[Fefferman-Stein interpolation, $H^1$ version]\label{fefferman_stein-h1}
For each $\theta$ in $(0,1)$, we have
\[
 [H^1(\bR^d),L^2(\bR^d)]_\theta = L^{p_\theta}(\bR^d)
\]
in the language of complex interpolation, where
\[
 p_\theta = \frac{2}{2-\theta}
\]
This generalizes the \hyperref[stein]{Stein interpolation theorem}
in the following sense:

For each $z$ in the closed strip
\[
 \{z \in \bC : 0 \leq \Re z \leq 1\},
\]
let us assume that we have a bounded linear operator
$T_z:L^2(\bR^d) \to L^2(\bR^d)$ such that each
\[
 z \mapsto \int_{\bR^d} (T_zf)g
\]
is a holomorphic function in the interior of $S$ and is continuous
on $S$ and that the norms $\|T_z\|_{L^2 \to L^2}$ of the operators
are uniformly bounded. If there exists a constant $k$ such that
\[
 \|T_{iy}f\|_2 \leq k\|f\|_L^2 
\]
for all $f \in L^2(\bR^d)$ and $y \in \bR$ and that
\[
 \|T_{1+iy}f\|_{1} \leq k\|f\|_{H^1}
\]
for all $f \in L^2(\bR^d) \cap H^1(\bR^d)$ and $y \in \bR$,
then we have the interpolated bound
\[
 \|T_\theta f\|_{p_\theta} \leq k_\theta \|f\|_{p_\theta}
\]
for each $\theta \in (0,1)$ and every $f \in L^2 \cap L^p$,
where
\[
 p_\theta = \frac{2}{2-\theta}.
\]
\end{theorem}

The $H^1$ interpolation theorem
can be used to study
linear operators that are not necessarily of type $(1,1)$.
In particular, the Fefferman-Stein theory settles the
$L^p$-boundedness problem
of the singular integral operators of Calder\'{o}n-Zygmund\index{Calder\'{o}n-Zygmund!singular integral \\ operator}
type, which include the Riesz transforms. We conclude
the section with a remark that
the development of the Fefferman-Stein theory does \emph{not}
make use of the $L^p$-boundedness of the Riesz transforms,
lest our argument be circular.
\index{Fefferman-Stein interpolation \\ theorem!on H1@on $H^1$|)}

We remark that the $L^\infty \to \BMO$ boundedness
hypothesis in Theorem \ref{fefferman_stein-bmo}
can be replaced by the more general $L^p \to \BMO$ boundedness
hypothesis for
some $p \in (1,\infty]$. The same proof then establishes the
corresponding interpolation theorem, and the duality argument
shows that the $H^1 \to L^1$ boundedness hypothesis
in Theorem \ref{fefferman_stein-h1} can be replaced
by the more general $H^1 \to L^p$ hypothesis for some
$p \in [1,\infty)$. We shall have an occasion to use this
formulation of the $H^1$-theorem in the next section.
\index{Fefferman, Charles|)}\index{Stein, Elias M.|)}
 
\section{Applications to Differential Equations}\label{s-applications-to-pde}
\index{PDE|see{differential equation}}
We have developed a number of tools for studying linear operators
on function spaces throughout the present thesis. In this section,
we shall apply them to the study of linear partial differential
equations and derive a few results. Standard theorems that do
not make use of interpolation theory will simply be cited with
references for proofs. 

What are linear partial differential equations?
Recall that the higher-order derivatives
of a function $f$ on $\bR^d$ can be written as $D^\alpha f(x)$, where $\alpha$
is an appropriate $d$-dimensional multi-index. The symbol $D^\alpha$ can
be thought of as an \emph{operator} on a suitably defined space of functions
on $\bR^d$: for example, the symbol
\[
 D^{(1,0,1)} = \frac{\partial}{\partial x_1} \frac{\partial}{\partial x_3}
\]
takes functions in $\mc{C}^2(\bR^3)$ to functions in $\mc{C}(\bR^3)$.
We can then define a $d$-dimensional \emph{linear differential operator}\index{differential!operator}
to be an operator $L$ on $\ms{S}'(\bR^d)$ that takes the form
\begin{equation}\label{do}
 (Lu)(x) = \sum_{n=1}^N f_n(x) D^{\alpha_n} u(x),
\end{equation}
where each $f_n$ is a function on $\bR^d$ and $\alpha_n$ a $d$-dimensional
multi-index. Since $u$ is a tempered distribution, the derivatives are
taken to be weak derivatives.

Given a linear differential operator $L$, we consider a
\emph{linear partial differential equation}\index{differential!equation}
\begin{equation}\label{pde}
 Lu = C
\end{equation}
for some constant $C$. We say that a linear PDE (\ref{pde}) is
\emph{homogeneous}\index{differential!equation, homogeneous and inhomogeneous} if $C=0$, and \emph{inhomogeneous} otherwise. 

\subsection{Fundamental Solutions}

We begin by examining \emph{constant-coefficient linear partial
differential equations}.\index{differential!equation, constant-coefficient} These are linear partial differential equations
such that the coefficients $f_n$ of the associated differential operator
(\ref{do}) are constants.
As such, every constant-coefficient linear PDE can be written in the form
\begin{equation}\label{constant-coefficent-linear-pde}
 P(D)u = v,
\end{equation}
where $P$ is a polynomial in $d$ variables and $v$ a function on $\bR^d$.

We have seen that the Fourier transform turns differentiation into multiplication
by a polynomial (Proposition \ref{differentiation-to-multiplication}).\index{Fourier transform!differentiation of} It
is natural to expect that every differential equation of the form
(\ref{constant-coefficent-linear-pde}) can be solved by taking the Fourier
transform of both sides, dividing through by the resulting polynomial factor,
and taking the inverse Fourier transform. To make this heuristic reasoning
precise, we introduce the following notion:

\begin{defin}\index{differential!equation, fundamental solution \\ of}\index{fundamental solution|see{\\ differential equation}}
A \emph{fundamental solution} of a constant-coefficient linear PDE
(\ref{constant-coefficent-linear-pde}) is a tempered distribution
$E \in \ms{S}'(\bR^d)$ such that
\[
 P(D)E = \delta^0.
\]
\end{defin}

If $v$ is a Schwartz function, then Theorem \ref{convolution-with-tempered-distributions}
implies that $E*v \in \mc{C}^\infty$. Furthermore,
\[
 P(D)(E*v) = (P(D)E) * v = \delta^0 * v = v,
\]
whence $u=E*v$ is a solution to the constant-coefficient linear PDE
(\ref{constant-coefficent-linear-pde}). A basic existence result is the following:

\begin{theorem}[Malgrange-Ehrenpreis]\index{fundamental solution!existence of|see{\\ Malgrange-Ehrenpreis \\ theorem}}\index{Malgrange-Ehrenpreis theorem}
Every constant-coefficient linear PDE has a fundamental solution.
\end{theorem}

\noindent See Theorem 8.5 in \cite{Walter_Rudin:B1991}\index{Rudin, Walter} for a proof, which makes use of
the rigorous version of the heuristic reasoning presented above.

\subsection{Regularity of Solutions}

Recall that the $d$-dimensional \emph{Laplacian}\index{Laplacian} is defined to be the differential
operator
\[
 \Delta = \sum_{n=1}^d \frac{\partial^2}{\partial x_n^2}.
\]
\emph{Laplace's equation} is the associated homogeneous PDE
\[
 \Delta u = 0,
\]
and a solution to Laplace's equation is referred to as a \emph{harmonic function}.\index{harmonic function}

It is a standard result in complex analysis that every two-dimensional harmonic
function is in $\mc{C}^2(\bR^d)$, as it is the real part of a holomorphic\index{holomorphicity!on the complex plane}
function.\index{regularity theorem!for Cauchy-Riemann \\ equations} A higher-dimensional generalization of the harmonic function theory,
which can be found in Chapter 2 of \cite{Stein_Weiss:B1971}, establishes the
analogous result for higher dimensions. Such a result is called a
\emph{regularity theorem}\index{regularity theorem} and establishes a higher degree of regularity for
the solutions of a certain differential equation than is \emph{a priori} assumed.

To this end, we shall consider a class of differential operators that generalizes
the Laplacian.

\begin{defin}\index{elliptic!differential operator}
Let $N \in \bN$. A linear differential operator $L$ is \emph{elliptic of
order $N$} if all multi-indices $\alpha_n$ in the expression (\ref{do})
satisfy the inequality $|\alpha_n| \leq N$ and if at least one coefficient
function $f_{n_0}$ with $|\alpha_{n_0}|=N$ is nonzero.
\end{defin}

An \emph{elliptic linear partial differential equation}\index{elliptic!PDE} is a differential
equation whose associated differential operator is elliptic. Examples
include Laplace's equation and its complex-variables counterpart,
the Cauchy-Riemann equations. The ``once differentiable, forever
differentiable'' regularity result of complex analysis generalizes to
elliptic operators:

\begin{theorem}[Elliptic regularity, $\mc{C}^\infty$-version]\label{elliptic-regularity-c-inf}\index{elliptic!regularity on Cinfty@regularity on $\mc{C}^\infty$}\index{regularity theorem!elliptic|see{elliptic}}
Let $L$ be an elliptic operator of order $N$ with coefficients in
$\mc{C}^\infty$. In addition, we assume that all coefficients of order $N$
are constants. If $v \in \mc{C}^\infty(\bR^d)$, then
every $u \in \mb{S}'(\bR^d)$ satisfying the identity
\[
 Lu = v
\]
is in $\mc{C}^\infty(\bR^d)$. This, in particular, implies that
every solution $u$ of the homogeneous differential equation
\[
 Lu = 0
\]
is in $\mc{C}^\infty(\bR^d)$.
\end{theorem}

The above theorem is a consequence of a more general theorem, which will be
presented in the next subsection.

\subsection{Sobolev Spaces}\label{s-sobolev-spaces}

\index{differentiation!of tempered distributions|(}\index{weak derivative|see{differentiation of tempered distributions}}
\index{Sobolev space|(}
So far, we spoke of regularity only in the sense of strong derivatives.
Linear differential operators were defined in terms of weak derivatives,
however, and we have yet to discuss the notion of regularity appropriate
for this general context. Our immediate goal is to make precise the notion of
``well-behaved weak derivatives''. This is done by introducing \emph{Sobolev
spaces}, which are subspaces of $L^p$ containing functions whose weak
derivatives are also in $L^p$. We shall then apply Calder\'{o}n's complex
interpolation method to characterize ``in-between'' smoothness,
which we require in order to state the general version of the
\hyperref[elliptic-regularity-c-inf]{elliptic regularity theorem}.

We begin by defining integer-order Sobolev spaces.

\index{complex interpolation|(}
\begin{defin}\index{Sobolev space!Lp@$L^p$}
Let $1 \leq p < \infty$ and $k \in \bN$. The \emph{$L^p$-Sobolev space
of order $k$} on $\bR^d$ is defined to be the collection $W^{k,p}(\bR^d)$
of tempered distributions $u$ such that $D^\alpha u \in L^p(\bR^d)$ for every
multi-index $|\alpha| \leq k$. The derivatives are taken to be weak derivatives.
\end{defin}

\index{Sobolev space!intermediate}
Note that $W^{0,p} = L^p$. $W^{k,p}(\bR^d)$ is a Banach space, with the norm
\[
 \|u\|_{W^{k,p}} = \left( \sum_{|\alpha| \leq k} \|D^\alpha u\|_p \right)^{1/p}.
\]
Therefore, we can consider the intermediate spaces
\[
 W^{s,p}(\bR^d) = [W^{k,p}(\bR^d),W^{k+1,p}(\bR^d)]_\theta
\]
in the sense of complex interpolation, where $k \leq s \leq k+1$ and
\[
 \frac{1}{s} = \frac{(1-\theta)}{k} + \frac{\theta}{k+1}.
\]
This allows us to speak of ``fractional derivatives'', in the sense that
$f \in L^p$ has derivatives ``up to order $s$'' in case $f \in W^{s,p}(\bR^d)$.
\index{Sobolev space!intermediate}

\index{Sobolev space!L2@$L^2$|(}\index{H2@$H^s$|see{Sobolev space, $L^2$}}
For $p=2$, we can find an alternate, more concrete characterization of
$L^2$ Sobolev spaces, which are Hilbert subspaces of $L^2$.
Given a real number $s$, we define\footnote{It is an
unfortunate coincidence in the history of mathematical analysis that
$H^p$ refers to the real Hardy space and $H^s$ the $L^2$-Sobolev space.
The convention is to use $p$ for the order of a real Hardy space and
$s$ for that of an $L^2$-Sobolev space.
More often than not, the context makes it clear which of the two spaces is in
use.}
$H^s(\bR^d)$ to be the space of tempered distributions $u$ such that
\[
 \langle\xi\rangle^s \hat{u}(\xi) \in L^2(\bR^d),
\]
where
\[
 \langle \xi \rangle = \sqrt{1+|\xi|^2},
\]
as was defined in \S\ref{s-the_fourier_transform}. Since
the Fourier transform turns differentiation into multiplication
by a polynomial, the equality $H^s = W^{s,2}$ is easily
established for each positive integer $s$.
We now show that the identification remains true for all $s \geq 0$.

\begin{theorem}\label{l2-sobolev-fourier-characterization}
$W^{s,2}(\bR^d) = H^s(\bR^d)$ for all $s \geq 0$.
\end{theorem}

\begin{proof}
We have already established the theorem for $t \in \bN$. Fix $s>0$ and
define the multiplication operator\index{multiplication operator}
\[
 (M_{\langle x \rangle^t}u)(x) = \langle x \rangle^t u(x)
\]
on $H^s(\bR^d)$. Since $\langle x \rangle^s \in \ms{C}^\infty(\bR^d)$
and $D^\alpha \langle x \rangle^t \in L^\infty(\bR^d)$ for each
multi-index $\alpha$, we see that $M_{\langle x \rangle^t}$ maps $H^t(\bR^d)$
into $H^t(\bR^d)$. Furthermore\footnote{See
\S\S\ref{s-spectral-theory-of-self-adjoint-operators}\index{spectral theory!on separable Hilbert spaces} for notation.}, the operator
\[
 \Lambda^t = \ms{F}^{-1} M_{\langle x \rangle^t} \ms{F}
\]
is an unbounded self-adjoint operator\index{self-adjoint!operator} on $L^2(\bR^d)$ with
$\mc{D}(\Lambda^t) = H^t(\bR^d)$. Indeed, $u \in H^t(\bR^d)$ is, by definition,
equivalent to $\langle \xi \rangle^t \hat{u} \in L^2(\bR^d)$, and
Plancherel's theorem (Theorem \ref{plancherel}) shows that
this happens if and only if $\ms{F}^{-1}(\langle \xi \rangle^t \hat{u})
\in L^2(\bR^d)$.
It now follows from Theorem \ref{interpolation-of-domains-of-self-adjoint-operators-on-a-hilbert-space}
that
\begin{equation}\label{L2-sobolev-interpolation-lemma-1}
 [L^2(\bR^d), H^t(\bR^d)]_\theta = \mc{D}(\Lambda^{t\theta}) = H^{t\theta}(\bR^d)
\end{equation}
for each $\theta \in [0,1]$. The desired result can now be established either
by invoking Theorem \ref{complex-interpolation-reiteration}
or by establishing the identity
\begin{equation}\label{sobolev-identity-yay}
 [H^{t_0}(\bR^d),H^{t_1}(\bR^d)]_\theta = H^{(1-\theta) t_0 + \theta t_1}(\bR^d),
\end{equation}
which follows from a minor but more elaborate modification of the above argument.
\end{proof}

Note that the space $H^s(\bR^d)$ remains well-defined for the negative
values of $s$. Formula (\ref{sobolev-identity-yay}) shows that the negative-order
$L^2$-Sobolev spaces can be interpolated in the same manner. Furthermore,
the negative-order $L^2$-Sobolev spaces can be used to establish a representation
thereom: indeed, for each $s \geq 0$, the dual of $H^s(\bR^d)$ is isomorphic to
$H^{-s}(\bR^d)$.\index{Sobolev space!of negative order}\index{Sobolev space!dual of|see{Sobolev space of \\ negative order}}\index{representation theorem!on Sobolev spaces|see{Sobolev \\ space of negative order}}

We now state a generalization of the \hyperref[elliptic-regularity-c-inf]{elliptic
regularity theorem}, which can be applied to a $L^2$-Sobolev space of any order:
see, for example, Theorem 8.12 in \cite{Walter_Rudin:B1991}\index{Rudin, Walter} for a proof.

\begin{theorem}[Elliptic regularity, Sobolev-space verison]\label{elliptic-regularity-sobolev}\index{elliptic!regularity on Sobolev spaces}
Let $L$ be an elliptic operator of order $N$ with coefficients in
$\mc{C}^\infty$. In addition, we assume that all coefficients of order $N$
are constants. If $v \in H^s(\bR^d)$, then
every $u \in \mb{S}'(\bR^d)$ satisfying the identity
\[
 Lu = v
\]
is in $\mc{C}^\infty(\bR^d)$. This, in particular, implies that
every solution $u$ of the homogeneous differential equation
\[
 Lu = 0
\]
is in $H^{s+N}(\bR^d)$.
\end{theorem}
\index{Sobolev space!L2@$L^2$|)}
\index{complex interpolation|)}

\index{pseudodifferential operator|(}
The $H^s$-characterization of the $L^2$-Sobolev spaces is explained most naturally
in the framework of \emph{pseudodifferential operators}, which are operators
$f \mapsto Tf$ given by
\[
 (Tf)(x)
 = \mc{F}^{-1}\left( a(x,\xi) \hat{f}(\xi) \right)(x)
 = \int_{\bR^d} a(x,\xi) \hat{f}(\xi) e^{2 \pi i x \cdot \xi} \, d\xi.
\]
$a(x,\xi)$ is called the \emph{symbol} of the pseudodifferential operator $T$,
and we often write $T_a$ to denote $T$ in order to emphasize the symbol of $T$.
If $a(x,\xi)=a_1(\xi)$ is independent of $x$, then $T$ is a Fourier multiplier\index{Fourier multiplier}\index{multiplier operator|see{Fourier multiplier}}
operator
\[
 \widehat{T_a f}(\xi) = a_1(\xi) \hat{f}(\xi)
\]
discussed in \S\S\ref{s-convolution-operators-and-fourier-multipliers};
if $a(x,\xi) = a_2(x)$ is independent of $\xi$, then $T$ is a \emph{multiplication
operator}\index{multiplication operator}
\[
 (T_af)(x) = a_2(x) f(x)
\]
used in the $H^s$-characterization of the $L^2$-Sobolev spaces. We have seen
above that the pseudodifferential operators $T_a$ with symbols
$a(x,\xi) = \langle \xi \rangle^n$ mimic the behavior of $n$th-order partial
differential operator, and so pseudodifferential operators can rightly be
seen as generalizations of regular differential operators.

The most commonly used symbol class, denoted by $S^m$, consists of\linebreak
$\mc{C}^\infty(\bR^d \times \bR^d)$-maps $a(x,\xi)$ satisfying the
estimate
\[
 |\partial_x^\beta \partial_\xi^\alpha a(x,\xi)|
  \leq A_{\alpha,\beta} (1 + |\xi|)^{m-|\alpha|}
\]
for each pair of multi-indices $\alpha$ and $\beta$. In particular,
the class $S^m$ includes all polynomials of degree $m$. It can be
shown that pseudodifferential operators with symbols in $S^m$ map
the Schwartz space $\ms{S}(\bR^d)$ into itself. Furthermore,
if $a \in S^{m_1}$ and $b \in S^{m_2}$, then there is a symbol
$c \in S^{m_1+m_2}$ such that
\[
 T_c = T_a \circ T_b.
\]
If $m=0$, then the pseudodifferential operators defined on $\ms{S}(\bR^d)$
can be extended to a bounded operator from $L^2(\bR^d)$ into itself.

See Chapters VI and VII in \cite{Elias_M_Stein:B1993}\index{Stein, Elias M.} for an exposition
of pseudodifferential operators in the context of the theory of
singular integral operators, and Chapter 7 in \cite{Michael_E_Taylor:B2010-2}\index{Taylor, Michael E.}
for an exposition in the context of partial differential equations.
Further developments in the theory of Sobolev spaces can be found in
Chapter 4 in \cite{Michael_E_Taylor:B2010} and Chapter 13 in
\cite{Michael_E_Taylor:B2010-3}. 
\index{Sobolev space|)}
\index{pseudodifferential operator|)}

\subsection{Analysis of the Homogeneous Wave Equation}\label{s-analysis-of-the-homogeneous-wave-equation}

\index{wave equation|(}
We now turn to the wave equation, an archetypal example of another class of
differential equations referred to as hyperbolic partial differential equations.
Unlike Laplace's equation, the wave equation takes two different kinds of variable:
the one-dimensional time variable $t$ and the $d$-dimensional space variable $x$.
Therefore, the domain space is $(1+d)$-dimensional and we denote it by $\bR^{d+1}$.

\begin{defin}\index{D'Alembertian}
The \emph{D'Alembertian} on $\bR^{1+d}$ is the differential operator
\[
 \Box = \partial^2_t - \Delta_x
 = \frac{\partial^2}{\partial t^2}
 - \sum_{n=1}^d \frac{\partial^2}{\partial x^2}.
\]
A (linear) \emph{wave equation} is a differential equation whose
associated differential operator is the D'Alembertian.
\end{defin}

We shall study the \emph{Cauchy problem} for the wave equation,\index{wave equation!Cauchy problem for} which means
that we will analyze the solutions of the equation with respect to given
initial conditions at $t = 0$ or $x = 0$. To this end, we consider
$u(t,x)$ to be an $\ms{S}'(\bR^d)$-valued function for each fixed $t$
and a $\mc{C}^2(\bR^d)$-map for each fixed $x$.

\begin{theorem}
The Cauchy problem for the homogeneous linear wave equation admits
the following solution:
\begin{enumerate}[(a)]
\item The solution to the homogeneous linear wave equation $\Box u = 0$
with initial conditions $u(0,0) = f$ and $u_t(0,0) = 0$ is
\begin{equation}\label{wave-1}
 u(t,x) = \ms{F}^{-1}\left[\cos(t|\xi|) \hat{f}(\xi)\right](x)
 =  \int e^{2 \pi i x \cdot \xi} \cos(t|\xi|) \hat{f}(\xi) \, d\xi.
\end{equation}
\item The solution to the homogeneous linear wave equation $\Box u = 0$
with initial conditions $u(0,0) = 0$ and $u_t(0,0) = g$ is
\begin{equation}\label{wave-2}
 u(t,x) = \ms{F}^{-1} \left[ \frac{\sin(t|\xi|)}{|\xi|} \hat{f}(\xi) \right]
 = \int e^{2 \pi i x \cdot \xi} \frac{\sin(t|\xi|)}{|\xi|} \hat{f}(\xi) \, d\xi.
\end{equation}
\end{enumerate}
\end{theorem}

Combining the above result with Plancherel's theorem (Theorem \ref{plancherel}),\index{Plancherel's theorem}
we conclude that the operator that takes the initial condition $f$
in (a) and produces the solution (\ref{wave-1}) is an $L^2$ Fourier multiplier\index{Fourier multiplier!as a solution operator for the \\ wave equation} with
the smooth symbol $\cos(t|\xi|)$. The same might be said about the operator
that takes the initial condition $g$ in (b) and produces the solution (\ref{wave-2}),
whose symbol
\[
 \frac{\sin(t|\xi|)}{|\xi|}
\]
is integrated in the usual fashion at $|\xi|=0$; the result is analytic in $|\xi|$.

Let us take a closer look at this operator. We consider a family of operators
\[
 T_k g = \int e^{2 \pi i x \cdot \xi} \left(\frac{\sin(t|\xi|)}{|\xi|}\right)
 \langle \xi \rangle^{1-k} \hat{g}(\xi) \, d\xi
\]
for each $k \in \bN$, which are defined \emph{a priori} for Schwartz functions
$g$. In this family, $T_1$ is the ``solution operator'' we have discussed above.
We shall apply the Fefferman-Stein theory to establish a boundedness result for $T_1$.

If we set $k=0$, then $|\xi|^{-1} \langle \xi \rangle$ is bounded at infinity,
and $|\xi|^{-1} \sin (t|\xi|) \langle \xi \rangle$ is smooth at $|\xi|=0$,
and so we obtain an $L^2$-Fourier multiplier with a smooth symbol.
How about the other end? If $k=d+\ve$ for some $\ve>0$, then
the function
\[
 \left(\frac{\sin(t|\xi|)}{|\xi|} \right) \langle \xi \rangle^{1-k}
\]
is in $L^1$, and so the operator $T_k$ maps $L^1$ boundedly to $L^\infty$.
We can then apply the Stein interpolation theorem (Theorem \ref{stein})\index{Stein, Elias M.!interpolation theorem}
to obtain a boundedness result for $T_1$. This result is not ideal, however:
the index is off by $\ve$ from the optimal index,
and so we do not get the $L^p \to L^{p'}$ bound.

We are therefore forced to consider the index $k=d$ directly. The integral
\[
 T_d g = \int e^{2 \pi i x \cdot \xi} \left(\frac{\sin(t|\xi|)}{|\xi|}\right)
 \langle \xi \rangle^{1-d} \hat{g}(\xi) \, d\xi
\]
now does not make sense for an arbitrary $L^1$-function $g$. The function
\[
 \left(\frac{\sin(t|\xi|)}{|\xi|} \right) \langle \xi \rangle^{1-d}
\]
only barely fails to be integrable, however, and the operator is actually
bounded from $H^1$ to $L^\infty$. We now apply $H^1$ version of
the Fefferman-Stein interpolation theorem (Theorem \ref{fefferman_stein-h1})\index{Fefferman-Stein interpolation \\ theorem!on H1@on $H^1$}
to obtain the desired boundedness result.
\index{wave equation|)}

\index{Fourier integral operator|(}
This is a special case of a 1982 theorem of Michael Beals,\index{Beals, R. Michael|(} which is stated
in terms of Fourier integral operators
\[
 (T_\lambda f)(x) = \int_{\bR^{d-1}} e^{i\lambda \Phi(x,\xi)} \psi(x,\xi) f(x) \, dx.
\]
Here $\xi$ is a $d$-dimensional real variable, $\lambda$ a positive parameter,
$\psi$ a smooth function of compact support in $x$ and $\xi$, and $\Phi$
a real-valued smooth function in $x$ and $\xi$. Furthermore, the Hessian
\[
 \det\left( \frac{\partial^2 \Phi(x,\xi)}{\partial x_i \partial \xi_j} \right)
\]
is assumed to be nonzero on the support of $\psi$. With these assumtions,
we have the $L^2$-bound
\[
 \|T_\lambda f\|_2 \leq A \lambda^{-d/2} \|f\|_2,
\]
which subsumes the $L^2$-boundedness property of the Fourier transform.\linebreak
Fourier integral operators generalize pseudodifferential operators
discussed in the previous subsection and comprise an important class
of integral operators known as \emph{oscillatory integrals}. See
Chapters VIII and IX in \cite{Elias_M_Stein:B1993}\index{Stein, Elias M.} or the second half
of \cite{Thomas_H_Wolff:B2003}\index{Wolff, Thomas H.} for an exposition
on oscillatory integrals.\index{oscillatory integral} 

The theorem of Beals, first stated in his 1980 thesis and subsequently
published in \cite{R_Michael_Beals:J1982}, establishes the above boundedness
result for every Fourier integral operator of the form
\[
 (Tf)(x) = \int e^{2 \pi i x \cdot (\xi - \varphi(\xi))} a(\xi) \hat{f}(\xi) \, d\xi,
\]
where $a$ satisfies the growth condition
\[
 |\partial^\beta a(\xi) | \leq k_\beta \left(1+|\xi|\right)^{-m-|\beta|}
\]
for each multi-index $\beta$ and $\varphi$ is \emph{homogeneous of degree 1},
viz., $\varphi(\lambda x) = \lambda \varphi(x)$ and analytic on
$\bR^d \smallsetminus \{0\}$. See Section 1 of \cite{R_Michael_Beals:J1982}
for the precise statement and the proof. This result can be applied to the
study of a wide class of hyperbolic partial differential equations: the
details can be found in Section 5 of \cite{R_Michael_Beals:J1982}.
\index{differentiation!of tempered distributions|)}
\index{Beals, R. Michael|)}
\index{Fourier integral operator|)}
  
\section{Additional Remarks and Further Results}

In this section, we collect miscellaneous comments that provide
further insights or extension of the material discussed in the
chapter. No result in the main body of the thesis relies on
the material presented herein.

\begin{fr}\label{fr-locally-convex-topological-spaces}
Let $V$ be a Hausdorff topological vector space, $\rho$ a seminorm\index{topological vector space}\index{seminorm}
on $V$, and $k$ a positive real number.
It can be shown\footnote{See, for example, Proposition
2 in Chapter I of \cite{Kosaku_Yosida:B1980}.\index{Yosida, Kosaku}} that the
``closed ball'' $M = \{x \in X : \rho(x) \leq k\}$ satisfies the
following properties:
\begin{enumerate}[(a)]
 \item $M$ contains the zero vector;
 \item $M$ is \emph{convex}, viz., $v,w \in M$ and $0 < \lambda < 1$
 implies $(1-\lambda) v + \lambda w \in M$;
 \item $M$ is \emph{balanced}, viz., $v \in M$ and $|\lambda| \leq 1$
 imply $\lambda v \in M$;
 \item $M$ is \emph{absorbing}, viz., each $v \in V$ furnishes a
 scalar $\lambda > 0$ such that $\lambda^{-1} v \in M$.
\end{enumerate}
We say that $V$ is a \emph{locally convex (topological vector) space}\index{locally convex space}
if every open set in $V$ containing the zero vector also contains
a convex, balanced, and absorbing open subset of $V$. Locally convex
spaces behave much like Fr\'{e}chet spaces, for the most part,
except that they are not required to be metrizable.
See Chapter 13 of \cite{Peter_D_Lax:B2002}\index{Lax, Peter D.} for an exposition of
the standard theorems in the theory of locally convex spaces.
\end{fr}

\begin{fr}\label{fr-tempered-distributions-with-point-support}\index{support!of tempered distributions}
The \emph{support} of a tempered distribution $u$ is defined to be
the intersection of all closed sets $K \subset \bR^d$ such that
$\langle \varphi, u \rangle = 0$ whenever the support of
$\varphi \in \ms{S}(\bR^d)$ is in $\bR^d \smallsetminus K$.
It is easy to see that the Dirac $\delta$-distribution\index{Dirac $\delta$-distribution}
$\delta^{x_0}$ is supported in the singleton set $\{x_0\}$.
More is true, however: if $u$ is \emph{any} tempered distribution
supported in the singleton set $\{x_0\}$, then there exists
an integer $n$ and complex numbers $\lambda_\alpha$ such that
\[
 u = \sum_{|\alpha| \leq n} \lambda_\alpha D^\alpha \delta^{x_0}.
\]
In this sense, the Dirac $\delta$-distributions are ``atoms''
of all distributions supported at a point. See Proposition 2.4.1
in \cite{Loukas_Grafakos:B2008}\index{Grafakos, Loukas} for a proof.
\end{fr}

\begin{fr}\label{fr-abstract-interpolation-theory}
\index{category theory|(}
We now place the method of complex interpolation in its fully
general context, which requires the language of category theory.
A \emph{category} $\mc{C}$ is a collection of \emph{objects} and
\emph{morphisms}. Contained in $\mc{C}$ is a class\footnote{We
use the term \emph{class} here, for many standard collections of
objects, such as the collection of all vector spaces, do not form
sets under the standard axioms of set theory. We do not attempt
to discuss the theory of classes---or any foundational issue,
for that matter---in this thesis.}
$\Ob\mc{C}$ of objects and, for each pair of objects
$A$ and $B$, a set $\Hom_{\Ob\mc{C}}(A,B)$ of morphisms.
Furthermore, any three objects $A$, $B$, and $C$ furnish
a function
\[
 \Hom_{\Ob \mc{C}}(B,C) \times \Hom_{\Ob\mc{C}}(A,B)
 \to \Hom_{\Ob \mc{C}}(A,C)
\]
called a \emph{law of composition}, satisfying the following
properties:
\begin{enumerate}[(i)]
 \item If $A$, $B$, $A'$, and $B'$ are objects of $\mc{C}$,
 then the two sets $\Hom_{\Ob\mc{C}}(A,B)$ and $\Hom_{\Ob\mc{C}}(A',B')$
 are disjoint provided that $A \neq A'$ or $B \neq B'$. If $A=A'$ and
 $B = B'$, then the two sets must coincide.
 \item For each object $A$ of $\mc{C}$, there is a morphism $\id_A \in \Hom_{\Ob\mc{C}}(A,A)$
 such that $(f,\id_A) \mapsto f$ for all $f \in \Hom_{\Ob\mc{C}}(A,B)$
 and $(\id_A,g) \mapsto f$ for all $g \in \Hom_{\Ob\mc{C}}(B,A)$, regardless of
 $B \in \Ob\mc{C}$.
 \item If $A$, $B$, $C$, and $D$ are objects of $\mc{C}$, then $f \in \Hom_{\Ob\mc{C}}(A,B)$,
 $g \in \Hom_{\Ob\mc{C}}(B,C)$, and $h \in \Hom_{\Ob\mc{C}}(C,D)$ satisfy the associativity relation
 \[
  ((h,g),f) = (h,(g,f)).
 \]
\end{enumerate}
Since laws of composition behave much like the composition operation between two functions,
we leave behind the cumbersome notation $(f,g) \mapsto h$ and simply write
\[
 f \circ g = h \htwo \mbox{or} \htwo fg = h.
\]

We shall work in the category $\mc{T}$ of topological vector spaces,\index{topological vector space} whose objects are
topological vector spaces (either over $\bR$ or $\bC$) and whose morphisms
are continuous linear transformations. A category $\mc{C}_1$ is a \emph{subcategory}
of another category $\mc{C}_2$ if every object and every morphism in $\mc{C}_1$ is
an object and a morphism, respectively, in $\mc{C}_2$. We shall primarily be concerned
one subcategory of $\mc{T}$: namely, the category $\mc{B}$ of Banach spaces\index{Banach space} with
bounded linear maps as morphisms.

Given two categories $\mc{C}_1$ and $\mc{C}_2$, we define a \emph{covariant functor}\index{functor|see{covariant functor}}\index{covariant functor}
$F:\mc{C}_1 \to \mc{C}_2$ to be a ``map of categories'' that sends each object
$A \in \Ob \mc{C}_1$ to another object $F(A) \in \Ob \mc{C}_2$ and each
morphism $f \in \Hom_{\Ob\mc{C}}(A,B)$ in $\mc{C}_1$ to another morphism
$F(f) \in \Hom_{\Ob\mc{C}}(F(A),F(B))$ in $\mc{C}_2$. Furthermore, a covariant
function must satisfy the following properties:
\begin{enumerate}[(i)]
 \item $F(fg) = F(f)F(g)$ for all morphisms $f$ and $g$ in $\mc{C}_1$;
 \item $F(\id_A) = \id_{F(A)}$ for all objects $A$ in $\mc{C}_1$.
\end{enumerate}
\index{category theory|)}

We now let $\overline{\mc{B}}$ denote the category of all Banach couples
(Definition \ref{banach-couples}), in which a morphism
$T:(A_0,A_1) \to (B_0,B_1)$ is a bounded linear map $T:A_0+A_1 \to B_0+B_1$
such that $T|_{A_0}$ is a bounded linear map into $B_0$ and $T|_{A_1}$ a
bounded linear map into $B_1$. We shall need two covariant functors
from $\overline{\mc{B}}$ to $\mc{B}$: the \emph{summation functor}\index{functor!summation}
$\Sigma$ and the \emph{intersection functor}\index{functor!intersection} $\Delta$. $\Sigma$
sends $(A_0,A_1)$ to $A_0 + A_1$ and $T:(A_0,A_1) \to (B_0,B_1)$
to its direct-sum counterpart $T:A_0 +A_1 \to B_0 + B_1$.
$\Delta$ sends $(A_0,A_1)$ to $A_0 \cap A_1$ and $T:(A_0,A_1) \to (B_0,B_1)$
to the restriction $T|_{A_0 \cap A_1}:A_0 \cap A_1 \to B_0 \cap B_1$.

Given a Banach couple $(A_0,A_1)$, we say that a Banach space $A$ is
an \emph{interpolation space}\index{interpolation space} between $A_0$ and $A_1$ if
\begin{enumerate}[(i)]
 \item $\Delta((A_0,A_1))$ continuously embeds into $A$;
 \item $A$ continuously embeds into $\Sigma((A_0,A_1))$;
 \item Whenever $T:(A_0,A_1) \to (A_0,A_1)$ is a morphism
 in $\overline{\mc{B}}$, the restriction
 $T|_A$ is a bounded linear map into $A$.
\end{enumerate}
More generally, two Banach spaces $A$ and $B$ are said to be
\emph{interpolation spaces with respect to} Banach couples
$(A_0,A_1)$ and $(B_0,B_1)$ if
\begin{enumerate}[(i)]
 \item $\Delta((A_0,A_1))$ continuously embeds into $A$;
 \item $A$ continuously embeds into $\Sigma((A_0,A_1))$;
 \item $\Delta((B_0,B_1))$ continuously embeds into $B$;
 \item $B$ continuously embeds into $\Sigma((B_0,B_1))$;
 \item Whenever $T:(A_0,A_1) \to (B_0,B_1)$ is a morphism
 in $\overline{\mc{B}}$, the restriction $T|_A$ is a bounded
 linear map into $B$.
\end{enumerate}
We remark that interpolation spaces $A$ and $B$ with respect
to Banach couples $(A_0,A_1)$ and $(B_0,B_1)$ are not
necessarily interpolation spaces between $A_0$ and $A_1$
or between $B_0$ and $B_1$, respectively.

\index{interpolation functor}
An \emph{interpolation functor}\index{functor!interpolation|see{interpolation functor}} is a covariant functor
$F:\overline{\mc{B}} \to \mc{B}$ such that
if $(A_0,A_1)$ and $(B_0,B_1)$ are Banach couples, then
$F((A_0,A_1))$ and $F((B_0,B_1))$ are interpolation
spaces with respect to the two Banach couples. Furthermore,
$F$ must send each morphism $T:(A_0,A_1) \to (B_0,B_1)$
in $\overline{\mc{B}}$ to its direct-sum counterpart
$T:A_0 + A_1 \to B_0+B_1$. An interpolation functor $F$
is said to be \emph{exact}\index{interpolation functor!exact} if every pair of
interpolation spaces $A$ and $B$ with respect to
a given pair of Banach couples $(A_0,A_1)$ and $(B_0,B_1)$
satisfies the following norm estimate
\[
 \|T\|_{A \to B} \leq \max\left( \|T\|_{A_0 \to B_0}, \|T\|_{A_1 \to B_1} \right).
\]
The \emph{Aronszajn-Gagliardo theorem}\index{Aronszajn-Gagliardo theorem}\index{interpolation functor!existence of|see{\\ Aronszajn-Gagliardo theorem}}\footnote{See Theorem 2.5.1
in \cite{Bergh_Lofstrom:B1976}\index{Bergh, J\"{o}ran}\index{L\"{o}fstr\"{o}m, J\"{o}rgen} for a proof.} states that every
Banach couple $(A_0,A_1)$ and its interpolation space $A$ furnish
an exact interpolation functor $F_0$ such that $F_0((A_0,A_1)) = A$.
In this sense, we can always find a suitable ``interpolation theorem''.

Given $\theta \in [0,1]$, an interpolation functor $F$ is said to be
\emph{exact of order $\theta$} in case
\[
 \|T\|_{A \to B} \leq \|T\|_{A_0 \to B_0}^{1-\theta} \|T\|_{A_1 \to B_1}^\theta
\]
for every pair of interpolation spaces $A$ and $B$ with respect to
a given pair of Banach couples $(A_0,A_1)$ and $(B_0,B_1)$. We can now see that
the complex interpolation functor
\[
 C_\theta((B_0,B_1)) = [B_0,B_1]_\theta
\]
is exact of exponent $\theta$. Calder\'{o}n introduces another exact
interpolation functor of exponent $\theta$ in \cite{Alberto_P_Calderon:J1964},\index{Calder\'{o}n, Alberto P.|(}
which we shall discuss in the next subsection.
\index{interpolation functor|)}
\end{fr}

\begin{fr}\label{fr-dual-of-complex-interpolation-space}
\index{interpolation space!dual of|(}\index{representation theorem!on interpolation spaces|see{interpolation space}}
\index{complex interpolation|(}
Here we use the framework established in the preceding subsection: see
\S\S\ref{fr-abstract-interpolation-theory} for relevant definitions.
In this subsection, we shall define an exact interpolation functor
that serves as a dual to the interpolation functor we studied in
the main body of the thesis.

As usual, we let $S$ be the closed strip
\[
 S = \{z \in \bC : 0 \leq \Re z \leq 1\}.
\]
A \emph{*-space-generating function}\footnote{In tune with
Definition \ref{space-generating-function}, this is not a standard
term and will not be used beyond this subsection.} for a complex
Banach couple $(B_0,B_1)$ is a function $g:S \to B_0 + B_1$
such that
\begin{enumerate}[(a)]
 \item $g$ is continuous on $S$ with respect to the norm of $B_0 + B_1$;
 \item $g$ is holomorphic in the interior of $S$, as per the definition of
  holomorphicity in \S\S\ref{complex-analysis-of-banach-valued-functions}.
 \item $\|g(z)\|_{B_0+B_1} \leq k(1+|z|)$ for some constant $k$ independent
 of $z \in S$;
 \item $g(iy_1)-g(iy_2) \in B_0$ and $g(1+iy_1) - g(iy_2) \in B_1$
 for all $y_1,y_2 \in \bR$ and
 \end{enumerate}
 \[
  \|g\|_{\mc{G}} = \max \left\{
   \sup_{y_1,y_2} \left\| \frac{g(iy_1)-g(iy_2)}{y_1-y_2} \right\|_{B_0},
   \sup_{y_1,y_2} \left\| \frac{g(1+iy_1)-g(1+iy_2)}{y_1 - y_2} \right\|_{B_1}
   \right\}
 \]
\indent is finite.

We denote by $\mc{G}(B_0,B_1)$ the collection of all *-space-generating functions.
As was the case with $\BMO$ introduced in \S\ref{s-hardy-spaces-and-bmo},
we must take the quotient space of $\mc{G}(B_0,B_1)$ modulo the space of
constant functions to obtain a Banach space (\cite{Alberto_P_Calderon:J1964}, \S5).
Given $\theta \in [0,1]$, the \emph{dual complex inteporlation space of order
$\theta$} between $B_0$ and $B_1$ is defined to be the normed linear subspace
\[
 B^\theta = [B_0,B_1]^\theta
 = \{v \in B_0 + B_1 : v = g'(\theta) \mbox{ for some } g \in \mc{G}(B_0,B_1)\}
\]
of $B_0+B_1$, with the norm
\[
 \|v\|^\theta = \|v\|_{B^\theta}
  = \inf_{\substack{ g \in \mc{G} \\ g'(\theta) = v }} \|g\|.
\]
Here $g'(\theta)$ is the derivative of $g$ at $\theta$, defined in the usual
way as the limit of the difference quotient. For each $\theta \in [0,1]$,
the interpolation space $[B_0,B_1]^\theta$ is isometrically isomorphic to
the quotient Banach space $\mc{G}(B_0,B_1)/\mc{N}^\theta$, where
$\mc{N}^\theta$ is the subspace of $\mc{G}(B_0,B_1)$ consisting of
all functions $g \in \mc{G}(B_0,B_1)$ such that $g'(\theta) = 0$
(\cite{Alberto_P_Calderon:J1964}, \S6). Furthermore, the functor
\[
 C^\theta((B_0,B_1)) = [B_0,B_1]^\theta
\]
is an exact interpolation functor of order $\theta$
(\cite{Alberto_P_Calderon:J1964}, \S7).

If either $B_0$ or $B_1$ is reflexive, then
\[
 [B_0,B_1]_\theta = [B_0,B_1]^\theta
 \htwo \mbox{and} \htwo
 \|f\|_\theta = \|f\|^\theta
\]
for all $\theta \in (0,1)$ (\cite{Alberto_P_Calderon:J1964}, \S9.5).
In general, however, we have the inclusion
\[
 [B_0,B_1]_\theta \subseteq [B_0,B_1]^\theta
\]
and the inequality
\[
 \|f\|^\theta \leq \|f\|_\theta
\]
for all $\theta \in [0,1]$; this is the \emph{equivalence theorem}\index{complex interpolation!equivalence theorem}
(\cite{Alberto_P_Calderon:J1964}, \S8). Furthermore,
for each $\theta \in (0,1)$, we have the representation theorem
\[
 \left([B_0,B_1]_\theta\right)^* \cong [B_0^*,B_1^*]^\theta,
\]
where the isomorphism is isometric;
this is the \emph{duality theorem}\index{complex interpolation!duality theorem} (\cite{Alberto_P_Calderon:J1964},
\S12.1). The equivalence theorem and the duality theorem
can be used to prove the reiteration theorem, which
is stated as Theorem \ref{complex-interpolation-reiteration}\index{complex interpolation!reiteration theorem} in
the present thesis (\cite{Alberto_P_Calderon:J1964}, \S12.3).
\index{complex interpolation|)}
\index{Calder\'{o}n, Alberto P.|)}
\index{interpolation space!dual of|)}
\end{fr}

\begin{fr}\label{fr-marcinkiewicz-and-real-interpolation}
The standard passage from the Hardy-Littlewood maximal inequality
(Theorem \ref{hary-littlewood-maximal-inequality}) to
the $L^p$-boundedness of the Hardy-Littlewood maximal function
(Theorem \ref{lp-boundedness-of-the-maximal-function}) can
be generalized to establish another interpolation theorem,
due to J\'{o}zef Marcinkiewicz. We note the crucial fact that
the Hardy-Littlewood maximal function\index{Hardy-Littlewood maximal function} is \emph{not} linear.
Indeed, the Marcinkiewicz interpolation theorem applies to
sublinear operators, which we shall define in due course.

Let $(X,\mu)$ and $(Y,\nu)$ be $\sigma$-finite measure
spaces. We fix a vector space $D$ of $\mu$-measurable
functions on that $X$ that contains the simple functions with
finite-measure support. Unlike the \hyperref[riesz-thorin]{Riesz-Thorin
interpolation theorem}, the proof of the Marcinkieiwciz interpolation
theorem makes use of real-variable methods and does not require
the functions to be complex-valued. We also assume that $D$ is
closed under truncation\footnote{See Definition \ref{strong-type}
for the definition.}.

We say that an operator $T$ on $D$ into the vector space $\mc{M}(Y,\mu)$
of $\nu$-measurable functions on $Y$ is \emph{subadditive}\index{operator!subadditive} if, for
all $f_1,f_2 \in D$, we have the inequality
\[
 |T(f_1+f_2)(x)| \leq |(Tf_1)(x)| + |(Tf_2)(x)|
\]
at almost every $x \in Y$. Given $1 \leq p \leq \infty$ and
$1 \leq q < \infty$, a sublinear operator $T:D \to \mc{M}(Y,\mu)$
is said to be \emph{of weak type $(p,q)$}\index{operator!of weak type (p,q)@of weak type $(p,q)$} if there exists
a constant $k > 0$ such that
\[
 |\{x \in \bR^d : (Tf)(x) > \alpha\}| \leq \left(\frac{k\|f\|_p}{\alpha}\right)^q.
\]
The infimum of all such $k$ is referred to as the \emph{weak $(p,q)$ norm}
of $T$. If $q = \infty$, then $T$ is of weak type $(p,\infty)$ if
$T$ is of type $(p,\infty)$ (Definition \ref{strong-type}\index{operator!of type (p,q)@of type $(p,q)$}); in this case,
the weak $(p,\infty)$ norm of $T$ is defined to be the $(p,\infty)$ norm
of $T$.

\begin{theorem}[Marcinkiewicz interpolation]\label{marcinkiewicz}
\index{Marcinkiewicz, J\'{o}zef!interpolation theorem|(}
Let $1 \leq p_0 \leq q_0 \leq \infty$ and $1 \leq p_1 \leq q_1 \leq \infty$
and assume that $q_0 \neq q_1$. If $T$ is a subadditive operator simultaneously
of weak type $(p_0,q_0)$ and of weak type $(p_1,q_1)$, then
$T$ is of type $(p_\theta,q_\theta)$ with the norm estimate
\[
 \|T\|_{L^{p_\theta} \to L^{q_\theta}}
 \leq \|T\|^{1-\theta}_{L^{p_0} \to L^{q_0}} \|T\|^\theta_{L^{p_1} \to L^{q_1}}
\]
for each $\theta \in [0,1]$, where
\begin{eqnarray*}
 p_\theta^{-1} &=& (1-\theta)p_0^{-1} + \theta p_1^{-1};\\
 q_\theta^{-1} &=& (1-\theta)q_0^{-1} + \theta q_1^{-1}.
\end{eqnarray*}
\end{theorem}

We remark that the interpolation theorem holds only in the lower triangle
of the Riesz diagram \index{Riesz, Marcel!diagram} (Figure \ref{riesz-diagram}). The interpolation
theorem was announced, without proof, on the main diagonal $p_j = q_j$
by J\'{o}zef Marcinkiewicz in his 1939 paper \cite{Jozef_Marcinkiewicz:J1939}.\index{Marcinkiewicz, J\'{o}zef}
The extension to the lower triangle was first given by Zygmund in
\cite{Antoni_Zygmund:J1956}.\index{Zygmund, Antoni} A proof of the theorem on $\bR^d$ can
be found in Appendix B of \cite{Elias_M_Stein:B1970},\index{Stein, Elias M.} and a minor modification
of this proof yields the theorem as stated above.

The key element in the proof is the \emph{rearrangement} of the function $Tf$.
The theory of rearrangement-invariant spaces encodes the main idea of the proof
and generalizes the Marcinkiewicz interpolation theorem to a much wider
class of spaces known as \emph{Lorentz spaces}. See Chapter V, Section 3
of \cite{Stein_Weiss:B1971}\index{Stein, Elias M.}\index{Weiss, Guid} or \S1.4 in \cite{Loukas_Grafakos:B2008}\index{Grafakos, Loukas} for
an introduction to the theory of Lorentz spaces and the proof of the generalized
Marcinkiewicz interpolation theorem in this setting. Analogous to
the \hyperref[riesz-thorin]{Riesz-Thorin interpolation theorem}, the Marcinkiewicz
interpolation theorem can be generalized in the framework of interpolation of
spaces as well. This is the method of \emph{real interpolation}, first
developed by Jacques-Louis Lions and Jaak Peetre. A brief introduction to
the real method of interpolation is given in Chapter 2 of \cite{Bergh_Lofstrom:B1976}.\index{Bergh, J\"{o}ran}\index{L\"{o}fstr\"{o}m, J\"{o}rgen}
For a more detailed treatment, any one of the many monographs on the subject
of real interpolation can be consulted: the classical one is \cite{Bennett_Sharpley:B1988}.\index{Bennett, Colin}\index{Sharpley, Robert}
\index{Marcinkiewicz, J\'{o}zef!interpolation theorem|)}
\end{fr}

\begin{fr}\label{fr-fourier-inversion-problem}
\index{Fourier inversion!on L1@on $L^1$|(}
In this subsection, we provide a quick survey of the Fourier inversion problem.
Recall that we have defined the Fourier transform on $L^p(\bR^d)$ for all
$1 \leq p \leq 2$. For $p > 2$, the $L^p$ Fourier transform in general
is a tempered distribution, so let us restrict our attention to the
usual $L^1$ Fourier transform on $L^1(\bR^d) \cap L^p(\bR^d)$.
A natural question to ask is as follows: to what extent does the identity
\[
 f = (\hat{f})^\vee
\]
holds for general $f \in L^p$? More precisely, if we define
\[
 S_R f(x) = \int_{|\xi| \leq R} \hat{f}(\xi) e^{2 \pi i \xi \cdot x} \, d\xi,
\]
we may ask ourselves whether $S_R f$ converge to $f$ as $R \to \infty$, and,
if so, in what sense.

If $f$ and $\hat{f}$ are both in $L^1(\bR^d)$, the
\hyperref[fourier-inversion-formula]{Fourier inversion formula} tells us that
$S_R f$ converges to $f$ pointwise. A strong regularity condition, such as
continuous differentiability with compact support, will guarantee that
the convergence will be uniform.
\index{Fourier inversion!on L1@on $L^1$|)}

\index{Fourier inversion!on Lp, pointwise@on $L^p$, pointwise|(}
What can we say about the Fourier transform of general $L^p$ functions?.
If $p<1$, then the Lebesgue spaces are pathological, and so the outlook
is bleak.
$p=1$ is hopeless as well, for Andrey Kolmogorov\index{Kolmogorov, Andrey!divergence of Fourier series|(} exhibited an $L^1$ function
whose Fourier inversion formula fails to converge at every point.
Chapter 3, Section 2.2 of \cite{Stein_Shakarchi:B2005} contains an
exmaple for Fourier series. This result can be understood as
a consequence of the \emph{uniform boundedness principle}, a corollary
of the Baire category theorem\index{Baire category theorem}:

\begin{theorem}[Banach-Steinhaus, Uniform boundedness principle]\label{uniform-boundedness-principle}\index{uniform boundedness principle}\index{Banach-Steinhaus theorem|see{uniform boundedness principle}}
Let $V$ be a Banach space and $\mc{L}$ a collection of bounded linear
functionals on $V$. If
\begin{equation}\label{ubp-1}
 \sup_{l \in \mc{L}} |l(f)| < \infty
\end{equation}
for each $f \in \mc{B}$, then
\[
 \sup_{l \in \mc{L}} \|l\| < \infty.
\]
The conclusion continues to hold if we assume that (\ref{ubp-1}) holds
on a subset of $\mc{B}$ that is not a countable union of nowhere dense
sets (``of second category'').
\end{theorem}

The \hyperref[uniform-boundedness-principle]{uniform boundedness principle}
produces many functions whose \linebreak Fourier series diverge on a dense
subset of $[-\pi,\pi]$. This divergence result for the Fourier series can
be morphed into a divergence result for the Fourier transform, via the so-called
\emph{transference principle}: see \S3.6 in \cite{Loukas_Grafakos:B2008}\index{Grafakos, Loukas} for
a discussion. 
\index{Kolmogorov, Andrey!divergence of Fourier series|)}

How about $1<p<\infty$? For $n=1$, there is the pointwise
almost-everywhere result, established for $L^2$ functions by Lennart Carleson
and extended to $L^p$ functions for all $1 < p < \infty$ by Richard Hunt two
years later:

\begin{theorem}[Carleson-Hunt, 1966 \& 1968]\index{Carleson-Hunt theorem}
Fix $1 <p < \infty$, let $f \in L^p(\bR)$, and assume that $\hat{f}$ exists.
Then
\[
 f(x) = \lim_{R \to \infty} S_R f(x)
\]
almost everywhere.
\end{theorem}

The proof is notoriously intricate and is still considered to be one of the
most difficult proofs in mathematical analysis.
As of May 2012, the analogous result for higher dimensions is still open.
\index{Fourier inversion!on Lp, pointwise@on $L^p$, pointwise|)}

\index{Fourier inversion!on Lp, norm-convergence@on $L^p$, norm-convergence|(}
How about the $L^p$ convergence? For $n=1$, we have the classical theorem
of Marcel Riesz (Theorem \ref{Lp-boundedness-of-the-hilbert-transform}).
This result, combined with the uniform boundedness principle, implies that
the $L^p$ convergence does take place for all $1 < p < \infty$.
As for $n>1$, the $L^2$ convergence follows from Plancherel's theorem
(Theorem \ref{plancherel}) and the uniform boundedness principle.

It remains to investigate the $L^p$ convergence of the Fourier series
for $1 < p < \infty$, $p \neq 2$. In \cite{Elias_M_Stein:J1956},
E. Stein provides the first progress towards solving the problem,
as an application of his interpolation theorem. Recall
from the proof of the \hyperref[fourier-inversion-formula]{Fourier inversion formula}
that multiplying a nice function---the Gaussian in the proof---to $\hat{f}$ in the
integrand facilitated the convergence of $S_R f$. Taking a cue from this, we consider
the \emph{Bochner-Riesz mean}\index{Bochner-Riesz!means}
\[
 S^\delta_R f(x)
 = \int_{|\xi| \leq R} \hat{f}(\xi) e^{2 \pi i \xi \cdot x}
 \left( 1- \frac{|\xi|^2}{R^2} \right)^\delta \, d\xi
\]
for each $R>0$ and $\delta \geq 0$. Note that $S^0_R$ is the regular spherical
summation. Does $S^\delta_R f$ converges nicely to $f$
for $\delta >0$? Stein obtained a partial result in \cite{Elias_M_Stein:J1956}
as an application of the \hyperref[stein]{Stein interpolation theorem}:

\begin{theorem}[Stein, 1956]\label{stein-1956}\index{Stein, Elias M.!estimate on the Bochner-Riesz \\ means}
$S^\delta_R$ is a linear operator of type $(p,p)$ for $1 < p <2$,
provided that $\delta > (2/p) - 1$.
\end{theorem}

Perhaps surprisingly, the above result did not extend to a proof of the $L^p$
convergence of multidimensional Fourier transform. Instead, Fefferman
obtained the following negative result in \cite{Charles_Fefferman:J1971}:

\begin{theorem}[Fefferman, 1971]\index{Fefferman, Charles!ball multiplier theorem}
The spherical summation of the Fourier inversion formula
converges in the norm topology of $L^p$ if only if $p=2$ viz.,
\[
 \lim_{R \to \infty}
 \left\| f - S_R f \right\|_p
 = 0.
\]
for all $f \in L^p$ whose Fourier transform exist if only if $p = 2$.
\end{theorem}

What can be salvaged from Stein's result?
For technical reasons, the result does not hold unless

\[
 \left| \frac{1}{p} - \frac{1}{2} \right| < \frac{2\delta + 1}{2n}.
\]
This, nevertheless, does not say anything about when the estimate does hold.
We might hope for the following:

\begin{conjecture}[Bochner-Riesz conjecture]\index{Bochner-Riesz!conjecture}
If $\delta > 0$, $1 \leq p \leq \infty$, and
\[
 \left| \frac{1}{p} - \frac{1}{2} \right| < \frac{2\delta + 1}{2n},
\]
then
\[
 \lim_{R \to \infty} \|S_R^\delta f - f\|_p = 0
\]
for all $f \in L^p$.
\end{conjecture}

This conjecture is tied to many important problems
in mathematical analysis, among which are the \emph{Stein restriction conjecture}\index{Stein, Elias M.!restriction conjecture}
and the \emph{Kakeya conjuecture}.\index{Kakeya conjecture} 
A quick exposition leading up to these conjectures can be found in
\cite{Thomas_H_Wolff:B2003}.\index{Wolff, Thomas H.} For a more detailed
survey, see Chapters VIII, IX, and X of
\cite{Elias_M_Stein:B1993}\index{Stein, Elias M.} or Chapter 10 of \cite{Loukas_Grafakos:B2008-2}.\index{Grafakos, Loukas}
\index{Fourier inversion!on Lp, norm-convergence@on $L^p$, norm-convergence|)}
\end{fr}

\cleardoublepage
\phantomsection
\addcontentsline{toc}{chapter}{Bibliography}
\nocite{*}
\bibliographystyle{amsalpha}
\bibliography{capstone}

\providecommand{\bysame}{\leavevmode\hbox to3em{\hrulefill}\thinspace}
\providecommand{\MR}{\relax\ifhmode\unskip\space\fi MR }
\providecommand{\MRhref}[2]{%
  \href{http://www.ams.org/mathscinet-getitem?mr=#1}{#2}
}
\providecommand{\href}[2]{#2}
\begin{thebibliography}{Tay10b}

\bibitem[Ahl79]{Lars_V_Ahlfors:B1979}
Lars~V. Ahlfors, \emph{Complex analysis: An introduction to the theory of
  analytic functions of one complex variable}, third ed., McGraw-Hill, 1979.

\bibitem[Ash76]{J_Marshall_Ash:B1976}
J.~Marshall Ash (ed.), \emph{Studies in harmonic analysis}, Studies in
  Mathematics, vol.~13, The Mathematical Association of America, 1976.

\bibitem[Bea82]{R_Michael_Beals:J1982}
R.~Michael Beals, \emph{{$L^p$} boundedness of fourier integral operators},
  Memoirs of the American Mathematical Society \textbf{38} (1982), no.~264,
  1--57.

\bibitem[Bec75]{William_Beckner:J1975}
William Beckner, \emph{Inequalities in fourier analysis}, Annals of Mathematics
  \textbf{102} (1975), 159--182.

\bibitem[BK91]{Brudnyi_Krugljak:B1991}
Yu.~A. Brudny\u{i} and N.~Ya. Krugljak, \emph{Interpolation functors and
  interpolation spaces}, vol.~1, North-Holland, 1991.

\bibitem[BL76]{Bergh_Lofstrom:B1976}
J\"{o}ran Bergh and J\"{o}rgen L\"{o}fstr\"{o}m, \emph{Interpolation spaces: An
  introduction}, Springer-Verlag, 1976.

\bibitem[Bre11]{Haim_Brezis:B2011}
Ha\"{i}m Brezis, \emph{Functional analysis, sobolev spaces and partial
  differential equations}, Springer, 2011.

\bibitem[BS88]{Bennett_Sharpley:B1988}
Colin Bennett and Robert Sharpley, \emph{Interpolation of operators}, Academic
  Press, 1988.

\bibitem[Cal64]{Alberto_P_Calderon:J1964}
Alberto~P. Calder\'{o}n, \emph{Intermediate spaces and interpolation, the
  complex method}, Studia Mathematica \textbf{24} (1964), 113--190.

\bibitem[DS58]{Dunford_Schwartz:B1958}
Nelson Dunford and Jacob~T. Schwartz, \emph{Linear operators}, vol.~1,
  Interscience Publishers, 1958.

\bibitem[Fal85]{Kenneth_J_Falconer:B1985}
Kenneth~J. Falconer, \emph{The geometry of fractal sets}, Cambridge University
  Press, 1985.

\bibitem[Fef71]{Charles_Fefferman:J1971}
Charles Fefferman, \emph{The multiplier problem for the ball}, Annals of
  Mathematics \textbf{94} (1971), no.~2, 330--336.

\bibitem[Fef95]{Fefferman_Fefferman_Wainger:B1995}
Essays on Fourier Analysis in Honor of Elias M. Stein (Charles Fefferman,
  Robert Fefferman, and Stephen Wainger, eds.), 1995.

\bibitem[Fol99]{Gerald_B_Folland:B1999}
Gerald~B. Folland, \emph{Real analysis: Modern techniques and their
  applications}, second ed., John Wiley \& Sons, 1999.

\bibitem[FS72]{Fefferman_Stein:J1972}
Charles Fefferman and Elias~M. Stein, \emph{${H}^p$ spaces of several
  variables}, Acta Mathematica \textbf{129} (1972), no.~3-4, 137--193.

\bibitem[Gra08a]{Loukas_Grafakos:B2008}
Loukas Grafakos, \emph{Classical fourier analysis}, second ed., Springer, 2008.

\bibitem[Gra08b]{Loukas_Grafakos:B2008-2}
\bysame, \emph{Modern fourier analysis}, second ed., Springer, 2008.

\bibitem[HK71]{Hoffman_Kunze:B1971}
Kenneth Hoffman and Ray Kunze, \emph{Linear algebra}, second ed.,
  Prentice-Hall, 1971.

\bibitem[HLP52]{Hardy_Littlewood_Polya:B1999}
G.~Hardy, J.~E. Littlewood, and G.~P\'{o}lya, \emph{Inequalities}, second ed.,
  Cambridge University Press, 1952.

\bibitem[HS65]{Hewitt_Stromberg:B1965}
Edwin Hewitt and Karl Stromberg, \emph{Real and abstract analysis},
  Springer-Verlag, 1965.

\bibitem[Lan02]{Serge_Lang:B2002}
Serge Lang, \emph{Algebra}, revised third ed., Springer-Verlag, 2002.

\bibitem[Lax02]{Peter_D_Lax:B2002}
Peter~D. Lax, \emph{Functional analysis}, Wiley-Interscience, 2002.

\bibitem[Lie90]{Elliott_H_Lieb:J1990}
Elliott~H. Lieb, \emph{Gaussian kernels have only guassian maximizers},
  Inventiones Mathematicae \textbf{102} (1990), 179--208.

\bibitem[LL01]{Lieb_Loss:B2001}
Elliott~H. Lieb and Michael Loss, \emph{Analysis}, second ed., American
  Mathematical Society, 2001.

\bibitem[Mar39]{Jozef_Marcinkiewicz:J1939}
J\'{o}zef Marcinkiewicz, \emph{Sur l'interpolation d'operations}, Comptes
  rendus de l'Acad\'{e}mie des sciences, Paris \textbf{208} (1939), 1272--1273.

\bibitem[Mir95]{Rick_Miranda:B1995}
Rick Miranda, \emph{Algebraic curves and riemann surfaces}, American
  Mathematical Society, 1995.

\bibitem[Mun00]{James_R_Munkres:B2000}
James~R. Munkres, \emph{Topology}, second ed., Prentice-Hall, Upper Saddle
  River, NJ, 2000.

\bibitem[Rie27a]{Marcel_Riesz:J1927-2}
Marcel Riesz, \emph{Sur les fonctions conjugu\'{e}es}, Mathematische
  Zeitschrift \textbf{49} (1927), 465--497.

\bibitem[Rie27b]{Marcel_Riesz:J1927}
\bysame, \emph{Sur les maxima des formes bilin\'{e}aires et sur les
  fonctionnelles lin\'{e}aires}, Acta Mathematica \textbf{49} (1927), 465--497.

\bibitem[Rud76]{Walter_Rudin:B1976}
Walter Rudin, \emph{Principles of mathematical analysis}, McGraw-Hill, 1976.

\bibitem[Rud86]{Walter_Rudin:B1986}
\bysame, \emph{Real and complex analysis}, third ed., McGraw-Hill, 1986.

\bibitem[Rud91]{Walter_Rudin:B1991}
\bysame, \emph{Functional analysis}, second ed., McGraw-Hill, 1991.

\bibitem[SS03a]{Stein_Shakarchi:B2003-2}
Elias~M. Stein and Rami Shakarchi, \emph{Complex analysis}, Princeton
  University Press, 2003.

\bibitem[SS03b]{Stein_Shakarchi:B2003}
\bysame, \emph{Fourier analysis: An introduction}, Princeton University Press,
  2003.

\bibitem[SS05]{Stein_Shakarchi:B2005}
\bysame, \emph{Real analysis: Measure theory, integration, and hilbert spaces},
  Princeton University Press, 2005.

\bibitem[SS11]{Stein_Shakarchi:B2011}
\bysame, \emph{Functional analysis: Introduction to further topics in
  analysis}, Princeton University Press, 2011.

\bibitem[Ste56]{Elias_M_Stein:J1956}
Elias~M. Stein, \emph{Interpolation of linear operators}, Transactions of the
  American Mathematical Society \textbf{83} (1956), 482--492.

\bibitem[Ste70]{Elias_M_Stein:B1970}
\bysame, \emph{Singular integrals and differentiability properties of
  functions}, Princeton University Press, 1970.

\bibitem[Ste93]{Elias_M_Stein:B1993}
\bysame, \emph{Harmonic analysis}, Princeton University Press, 1993.

\bibitem[SW71]{Stein_Weiss:B1971}
Elias~M. Stein and Guido Weiss, \emph{Introduction to fourier analysis on
  euclidean spaces}, Princeton University Press, 1971.

\bibitem[Tay10a]{Michael_E_Taylor:B2010-2}
Michael~E. Taylor, \emph{Partial differential equations}, second ed., vol.~2,
  Springer, 2010.

\bibitem[Tay10b]{Michael_E_Taylor:B2010}
\bysame, \emph{Partial differential equations}, second ed., vol.~1, Springer,
  2010.

\bibitem[Tay10c]{Michael_E_Taylor:B2010-3}
\bysame, \emph{Partial differential equations}, second ed., vol.~3, Springer,
  2010.

\bibitem[Tho48]{Olof_Thorin:T1948}
G.~Olof Thorin, \emph{Convexity theroems generalizing those of m. riesz and
  hadamard with some applications}, Ph.D. thesis, Lund University, 1948.

\bibitem[Tre75]{Francois_Treves:B1975}
Fran\c{c}ois Treves, \emph{Basic linear partial differential equations},
  Academic Press, 1975.

\bibitem[TZ44]{Tamarkin_Zygmund:J1944}
Jacob~David Tamarkin and Antoni Zygmund, \emph{Proof of a theorem of thorin},
  Bulletin of the American Mathematical Society \textbf{50} (1944), 279--282.

\bibitem[Wol03]{Thomas_H_Wolff:B2003}
Thomas~H. Wolff, \emph{Lectures on harmonic analysis}, American Mathematical
  Society, 2003.

\bibitem[Yos80]{Kosaku_Yosida:B1980}
K\^{o}saku Yosida, \emph{Functional analysis}, sixth ed., Springer, 1980.

\bibitem[Zyg56]{Antoni_Zygmund:J1956}
Antoni Zygmund, \emph{On a theorem of marcinkiewicz concerning interpolation of
  operations}, Journal de Math\'{e}matiques Pures et Appliqu\'{e}es \textbf{35}
  (1956), 223--248.

\end{thebibliography}

\index{dual space|seealso{representation theorem}}
\index{mollifiers|seealso{approximations to the identity}}
\index{approximations to the identity!smooth|seealso{mollifiers}}
\index{Calder\'{o}n, Alberto P.!Calder\'{o}n-Zygmund|see{\\ Calder\'{o}n-Zygmund}}
\index{Zygmund, Antoni|seealso{\\ Calder\'{o}n-Zygmund}}
\index{Minkowski, Hermann|seealso{\\ inequality}}
\index{Fourier, Joseph|seealso{Fourier transform}}
\index{Thorin, Olof|seealso{\\ Riesz-Thorin interpolation \\ theorem}}
\index{Lebesgue, Henri!measure|seealso{Littlewood's three principles}}
\index{differential!operator|seealso{elliptic}}
\index{interpolation space!exact|seealso{interpolation functor}}
\index{Bochner-Riesz!means|seealso{Stein, \\ Elias M.}}

\cleardoublepage
\phantomsection
\addcontentsline{toc}{chapter}{Index}
\printindex
\end{document}